\documentclass[11pt,reqno]{amsart}
\usepackage{amsmath,amsbsy,amssymb,amscd,amsfonts,latexsym,amstext,delarray, amsmath,color,caption,comment}
\usepackage{tikz,lscape}
\usetikzlibrary{matrix,arrows}
\usepackage{graphicx}
\usepackage{multicol}
\usepackage{scrextend}
\usepackage{mathtools}
\usepackage{amssymb} 
\usepackage{calligra}
\usepackage{calrsfs}
\usepackage{eucal}
\usepackage{calrsfs}
\usepackage{epsfig}
\usepackage{xcolor,import}
\usepackage{transparent}
\usepackage{tikz-cd}
\usepackage{epstopdf}
\usepackage{subcaption}
\usepackage{enumitem}
\usepackage{float}
\numberwithin{equation}{section}
\numberwithin{figure}{section}
\usepackage{hyperref}
\usepackage{amsthm}
\usepackage{thmtools}
\usepackage{cleveref}
\usepackage{mathrsfs} 
\DeclareFontFamily{U}{mathb}{\hyphenchar\font45}
\DeclareFontShape{U}{mathb}{m}{n}{
      <5> <6> <7> <8> <9> <10> gen * mathb
      <10.95> mathb10 <12> <14.4> <17.28> <20.74> <24.88> mathb12
      }{}
\DeclareSymbolFont{mathb}{U}{mathb}{m}{n}

\DeclareMathSymbol{\precneq}{3}{mathb}{"AC}

\usepackage{mathtools}

\usepackage{graphicx}
\usepackage{tikz}
\tikzset{
  font={\fontsize{10pt}{12}\selectfont}}
  \usetikzlibrary{arrows.meta,arrows}
  \tikzset{>=latex}
\usepackage{caption, cleveref} 
\usepackage{xfrac}  
\usetikzlibrary{arrows, fit, backgrounds, patterns, positioning, calc, intersections, decorations.markings, decorations.pathmorphing, decorations.pathreplacing, shapes.misc}
\usetikzlibrary{shapes.multipart}
\usepackage[utf8]{inputenc}
\usepackage{multicol}
\usepackage{subcaption}
\usepackage{enumitem}

\usepackage[normalem]{ulem} 

\DeclareMathAlphabet{\pazocal}{OMS}{zplm}{m}{n}

\input xy

\xyoption{all}
    \advance \topmargin by -\headheight
    \advance \topmargin by -\headsep

    \evensidemargin \oddsidemargin
\makeatletter
\newcommand{\leqnomode}{\tagsleft@true\let\veqno\@@leqno}
\newcommand{\reqnomode}{\tagsleft@false\let\veqno\@@eqno}
\makeatother

\newtheorem {theorem}    {Theorem}[section]

\newtheorem {question}    {Question}
\theoremstyle{definition}
\newtheorem {lemma}      [theorem]    {Lemma}
\newtheorem {corollary}  [theorem]    {Corollary}

\theoremstyle{definition}
\newtheorem{definition}[theorem]{Definition}
\theoremstyle{definition}
\newtheorem{remark}[theorem]{Remark}

\newcommand{\defeq}{\vcentcolon=}
\newcommand{\gen}[1]{\langle #1 \rangle}

\def\TM{{\rm TM}}

\def\a{\alpha}                
\def\b{\beta}

\def\eps{\varepsilon}

\def\a{\alpha}
\def\b{\beta}

\def\N{\mathbb{N}}     
\def\lab{{\text{Lab}}}

\DeclareMathOperator{\CL}{CL}

\newcommand{\tm}{\mathrm{tm}}

\newcommand{\REACH}{\mathsf{REACH}}

\renewcommand{\int}{\mathrm{int}}

\def\b{\beta}

\def\vertexradius{.1}
\def\vertex(#1){\fill (#1) circle (\vertexradius)}

\begin{document}

\title{\bf Conjugator Lengths and Isoperimetric functions}
\maketitle
\begin{center}

Conan Gillis, Francis Wagner

\end{center}

\bigskip

\begin{center}

\textbf{Abstract}

\end{center}

We show that for any recognizing $S$-machine $\textbf{S}$ with superadditive time function $f$, there exists a finitely presented group whose conjugator length function is quadratic and whose Dehn function grows like the square of $f$.  
This result is dual to that of a previous paper of the authors, which constructed a family of groups with cubic Dehn function that have various conjugator length functions.
We thereby give the first known example of a finitely presented group whose conjugator length function is recursive, but whose Word and Conjugacy Problems are both undecidable.  This answers an analogue of a question of Rips. 
Moreover, given a finitely presented group $G$ with decidable Word Problem, we obtain a finitely presented group with decidable Conjugacy Problem whose Dehn function grows faster than that of $G$.  
Finally, combining this result with its dual, for a wide array of pairs of functions $(f,g)$ we furnish an example of a finitely presented group with Dehn function equivalent to $f$ and conjugator length function equivalent to $g$.  This shows that the two invariants are very strongly independent and making significant progress on a question of Bridson, Riley, and Sale.

\bigskip

\bigskip


\section{Introduction}

This article studies two invariants of finitely presented groups: the Dehn function and the conjugator length function.  When combined with the present authors' previous paper \cite{GW}, the results developed herein may be summarized as exhibiting that the two invariants are independent in a very strong sense.


Let $G=\langle X\mid R\rangle$ be a finitely presented group. The Dehn and conjugator length functions, $\delta_G$ and $\CL_G$, are invariants of $G$ which respectively quantify the algorithmic notions of the Word and Conjugacy problems in terms of the intrinisic geometry of $G$. The Dehn function in particular has been studied for its relationship to the computational complexity of the Word Problem: $\delta_G$ is a recursive function if and only if the Word Problem for $G$ is decidable \cite{Gersten} (see \cite{Sapir06} for a survey on additional connections). Assuming the Word Problem for $G$ is decidable, a similar Gersten-style argument gives that $\CL_G$ is recursive if and only if the Conjugacy Problem for $G$ is decidable \cite{BRS}. Indeed, just as the complexity of the Conjugacy Problem for $G$ gives an upper bound for the complexity of its Word Problem, $\CL_G$ and the complexity of the Word Problem for $G$ can be used together to bound the complexity of the Conjugacy Problem for $G$. 

This paper's main motivation is to exhibit the necessity of the assumption that the Word Problem for $G$ is decidable, as the above complexity bounds are destroyed when considering only the functions $\delta_G$ and $\CL_G$. In particular, we show there exist finitely presented groups with recursive (indeed quadratic) conjugator length function but non-recursive Dehn function and, therefore, undecidable Word and Conjugacy Problem. This result is, in a sense, analogous to the fact that a group with quadratic Dehn function need not have solvable conjugacy problem, as shown by Ol'shanskii and Sapir in answering a question of Rips \cite{OS20} (see \cite{OS06} as well). 
These same methods, however, can also be used to show that for any recursive function $f$ which may be realized as the Dehn function of a finitely presented group, there exists a finitely presented group whose Dehn function grows at least as fast as $f$ and whose Conjugacy Problem is decidable.

Beyond decidability questions, our second motivation is to answer a question of Bridson, Riley, and Sale relating to the image of the map $\mu: G\mapsto (\delta_G, \CL_G)$ defined for finitely presented groups, where both functions are considered up to a suitable equivalence $\sim$ defined below \cite{BRS}. It is well-known that $\delta_G\sim n$  if and only if $G$ is Gromov-hyperbolic, and this in turn implies $\CL_G(n)\sim n$ \cite{BridsonHaefliger1999}. On the other hand, the authors showed in \cite{GW} that, for any $f$ which is the Dehn function of \textit{some} finitely presented group and grows at least as fast as $n^2$, there is a group $G_f$ with $\CL_{G_f}(n)\sim f$ and $\delta_{G_f}(n)\sim n^3$. We elucidate the image of $\mu$ further, adding all pairs where the coordinates are, at least, polynomials of not-too-small degree.

\subsection{Formulation of results} \

In order to state the main results effectively, we require some definitions. 

\begin{definition}

Fix a presentation $\pazocal{P}=\gen{X\mid\pazocal{R}}$ for a group $G$ with $|X|<\infty$.  Given a string $w=x_1x_2\cdots x_n$ with $x_i\in X\cup X^{-1}$, henceforth referred to as a \textit{word over $X$}, the \textit{area} of $w$ is the minimal number of relations in $\pazocal{R}\cup \pazocal{R}^{-1}$ that must be inserted in or deleted from $w$ to take it to the empty word, perhaps with intermediate free cancellations or insertions; if no such sequence exists, then $w$ does not represent the identity in $G$ and is assigned area 0. The Dehn function $\delta_\pazocal{P}$ is then the function whose output $\delta_\pazocal{P}(n)$ is the maximal area of all words of length at most $n$. 

\end{definition}

\begin{definition}

Fix a generating set $X$ for a group $G$ with $|X|<\infty$.  For two words $u,v$ over $X$, the \textit{conjugator distance} is the length of the shortest word $\gamma$ over $X$ such that $\gamma u \gamma^{-1}=v$ in $G$; if no such $\gamma$ exists, then analogous to the previous definition the conjugator distance is assigned to be $0$.  The \textit{conjugator length function} $\CL_{G,X}(n)$ is then the maximal conjugator distance over all pairs of words of total length at most $n$. 

\end{definition}

For expositions of both functions defined above, and additional perspectives, we refer the reader to \cite{Riley2017} and \cite{BRS} respectively.

To make these notions into invariants of the group itself and not dependent on either the particular generating set or presentation, we consider these functions up to a particular equivalence relation.

\begin{definition} For non-decreasing functions $f,g:\N\to[0,\infty)$, we say $f\preceq g$ if there exists a constant $C\in\N$ such that $f(n)\leq Cg(Cn+C)+Cn+C$.  As this defines a preorder on the family of non-decreasing functions $\N\to[0,\infty)$, it induces an equivalence relation $\sim$ on this family, {\frenchspacing i.e. $f\sim g$ if and only if $f\preceq g$ and $g\preceq f$.}\end{definition}

The notions of $\sim$ and $\preceq$ make precise the notions of a function ``growing as" or ``growing faster than" another, phrasing used in the introductory discussion of our results above.  It is crucial to observe the following facts regarding the functions of interest:

\begin{itemize}

\item If $X$ and $Y$ are two finite generating sets of the same group, then $\CL_{G,X}\sim\CL_{G,Y}$.

\item If $\pazocal{P}$ and $\pazocal{S}$ are two finite presentations of the same group (or even quasi-isometric groups), then $\delta_\pazocal{P}\sim\delta_\pazocal{S}$.

\end{itemize}

As such, given a finitely presented group $G$, we may define the functions $\delta_G$ and $\CL_G$ to be the equivalence class of the corresponding function taken with respect to any finite presentation, yielding $\sim$-invariants of the group alone.

With these preliminaries defined, the main result of this article may be stated.

\begin{theorem} \label{main-theorem}

Let $\textbf{S}$ be a recognizing $S$-machine with time function $\TM_\textbf{S}$.  Suppose $\TM_\textbf{S}(n)\geq n$ for all $n$ and $\TM_\textbf{S}\preceq f$ for some positive superadditive function $f$, {\frenchspacing i.e. a function $f:\N\to[0,\infty)$ such that $f(n+m)\geq f(n)+f(m)$ for all $n,m$ and $f(n)\geq1$ for all $n\neq0$}.  Then there exists a finitely presented group $G_\textbf{S}$ such that $\CL_{G_\textbf{S}}(n)\sim n^2$ and $\TM_\textbf{S}(n)^2\preceq\delta_{G_\textbf{S}}(n)\preceq f(n)^2$.

\end{theorem}

Note that, since time functions are considered up to $\sim$-equivalence, the assumption $\TM_\textbf{S}(n)\geq n$ is not at all restrictive.  Moreover, the existence of $f$ is also not restrictive, as one may take the `superadditive closure' defined to be $f(n)=\max(\TM_\textbf{S}(n_1)+\dots+\TM_\textbf{S}(n_k))$ for all non-negative partitions $n=n_1+\dots+n_k$ (see Remark 1.5 of \cite{OS12}).

An immediate consequence of \Cref{main-theorem} is the first known example of a finitely presented group with recursive conjugator length function and undecidable Word and Conjugacy Problems.


\begin{corollary} There exists a finitely presented group with quadratic conjugator length function and undecidable Word Problem.
\end{corollary}
\begin{proof}

Take $\pazocal{T}$ to be any non-deterministic multi-tape Turing machine which recognizes a language which is recursively enumerable but not recursive.  The time function $\TM_\pazocal{T}$ of $\pazocal{T}$ must then be non-recursive, while the main result of \cite{SBR} yields an $S$-machine $\textbf{S}$ with $\TM_\textbf{S}\sim\TM_\pazocal{T}^3$.  But then $\TM_\textbf{S}^2$ is also non-recursive, so that the group $G_\textbf{S}$ given by \Cref{main-theorem} satisfies the statement.


\end{proof}

Another consequence of our main result is the realization of a finitely presented group with decidable Conjugacy Problem whose Dehn function grows at least as fast as a prescribed recursive Dehn function.  

%
%


\begin{corollary} \label{cor Dehn+CP}

For any finitely presented group $G$ with decidable Word Problem, there exists a finitely presented group $H$ with $\delta_H\succeq\delta_G^2$ such that the Conjugacy Problem for $H$ is decidable.

\end{corollary}

\begin{proof}

The proof of Theorem 1.4 of \cite{GW} exhibits a recognizing $S$-machine $\textbf{S}$ such that $\TM_\textbf{S}\sim \delta_G$.  Hence, taking $H$ to be the group $G_\textbf{S}$ given by \Cref{main-theorem} satisfies the statement.

\end{proof}

A conjecture of Guba and Sapir (Conjecture 1 of \cite{GubaSapir99}) posits that every Dehn function is equivalent to a superadditive function.  Assuming the veracity of this conjecture, the statement of \Cref{cor Dehn+CP} may be sharpened, replacing the asymptotic bound with equivalence.  Indeed, given another analogous conjecture of \cite{SBR} regarding what functions may be realized as time functions, one may obtain a finitely presented group with any prescribed recursive Dehn function and decidable Conjugacy Problem.  

As a concrete example of this, for any $\a\geq2$ computable in double-exponential time, combining \Cref{cor Dehn+CP} with the constructions of \cite{O18} yields a finitely presented group with decidable Conjugacy Problem and Dehn function equivalent to $n^\a$.

Of course, \Cref{cor Dehn+CP} is of particular interest for its application to the opposite end of the recursive spectrum, showing that finitely presented groups with decidable Conjugacy Problem can still have exceedingly large Dehn function, {\frenchspacing e.g. growing as fast as the Ackermann function \cite{Dison-Einstein-Riley}}.

Finally, we obtain a consequence relating the possible pairs of functions which may be realized as the Dehn and conjugator length functions of a finitely presented group.

\begin{corollary}\label{DehnCLIndep} Fix a pair of non-decreasing functions $f\succeq n^3$ and $g\succeq n^2$ such that:

\begin{itemize}

\item $f$ is a positive superadditive function which is the time function of some $S$-machine $\textbf{S}$

\item $g\sim \delta_H$ for some finitely presented $H$.

\end{itemize}

Then there exists a group $G_{f,g}$ such that $\delta_{G_{f,g}}(n)\sim f$ and $\CL_{G_{f,g}}(n)\sim g$.
\end{corollary}
\begin{proof}
Take the direct product of $G_\textbf{S}$, as given in Theorem \ref{main-theorem}, with $G_H$, as constructed in \cite[Corollary 1.4]{GW}. The Dehn and conjugator length functions of the product are then the maxima of the respective functions of the factors. 

\end{proof}

For example, as above for $\a\geq3$ and $\b\geq2$ computable in double-exponential time, there exists a finitely presented group with Dehn function equivalent to $n^\a$ and conjugator length function equivalent to $n^\b$.  Indeed, assuming the conjectures referenced above, there exists a finitely presented group whose Dehn and conjugator length functions are equivalent to just about any pair of viable functions.

\medskip

\subsection{Paramaterized Complexity} \

An illustrative interpretation of our results comes from the field of parameterized complexity. Fix a finitely presented group $G=\gen{X\mid\pazocal{R}}$. Define the \textit{Parameterized Conjugacy Problem} for $G$ to be the decision problem which asks, for a pair of input words $(u,v)$ over $X$ and a parameter $k$, does there exist a word $\gamma$ over $X$ of length at most $k$ such that $\gamma u\gamma^{-1}=v$ in $G$?  

Note that the Word Problem for $G$ can be reduced to the Parameterized Conjugacy Problem by fixing one of the inputs as the empty word and taking $k=1$.  On the other hand, given an algorithm $\mathsf{W}$ solving the Word Problem for $G$, one has an algorithm $\overline{\mathsf{W}}$ which simply checks the above equation for each possible $\gamma$. 

If $\delta_G$ is a computable function, there exists an algorithm $\mathsf{Ger}$ due to Gersten
\cite{Gerstenl1} which solves the Word Problem for $G$ by computing $\delta_G(n)$, where $n$ is the length of the input word, and (essentially) reading the boundary label of all van Kampen diagrams of smaller or equal area.

Our results, particularly Corollary \ref{DehnCLIndep}, show that there are finitely presented groups where the algorithm $\overline{\mathsf{Ger}}$ has arbitrarily bad time complexity, and indeed groups where the algorithm can fail altogether, even when the conjugator length is given as a parameter of the problem. Moreover, $\overline{\mathsf{Ger}}$ fails precisely when the Word Problem for $G$ is undecidable. Thus, the failure of $\overline{\mathsf{Ger}}$ implies that the Parameterized Conjugacy Problem is, itself, undecidable. 

With these interpretations, one can see that a `direct attack' approach to the Conjugacy Problem for the constructed groups is something of a `needle in a haystack': Given two words $u$ and $v$ over the generating set which represent conjugate elements of the group, it is difficult to find a conjugator even though there is one which is not too long compared to the lengths of the two input words.

\medskip

\subsection{Open Problems} \

There are just a few gaps remaining to fully understanding the image of $\mu$.  

First, a question of \cite{BRS} asks whether there exists a finitely presented group whose conjugator length is sub-quadratic but not linear (see \cite[Problem 1']{GW}). Even if such a group exists, though, it is unclear what the growth of the Dehn function might be, or if there is a similar independence as displayed by \Cref{DehnCLIndep}.

\begin{question}
If a finitely presented group has sub-quadratic conjugator length function, what (if any) constraints does this impose on the Dehn function?
\end{question}

Such constraints are likely to be quite loose, given that BS$(1,2)$ has linear conjugator length \cite[Theorem 2]{sale2016conjugacy} and exponential Dehn Function. The most promising approach to finding conjugator length functions between linear and quadratic seems to be analyzing the Anderaa-Higman-Rotman-Rips construction of a group with undecidable Word Problem \cite{Rotman}, however this group has a BS$(1,2)$ subgroup, which will inflate the Dehn function to at least $2^n$.

On the other side, while \Cref{DehnCLIndep} leaves open the possible conjugator length functions which may be realized by groups with subcubic Dehn function.

\begin{question}

If a finitely presented group $G$ satisfies $n^2\preceq\delta_G(n)\prec n^3$, what (if any) constraints does this impose on $\CL_G$?

\end{question}

Finally, while the proof of Theorem 1.4 of \cite{GW} restricts Dehn functions to a subset of the class of functions which may be realized as the time function of some $S$-machine (an analogue of a result of \cite{SBR} which applies to multi-tape non-deterministic Turing machines), the same is not immediate for conjugator length functions.  In particular, it is unclear if the results of \cite{GW} exhibit \textit{all} possible functions which may be realized as the conjugator length function of a finitely presented group.

\begin{question}

Let $f\succeq n^2$ be the conjugator length function of some finitely presented group. Must it be the time function of some $S$-machine? In particular, must $f$ also be a Dehn function of some finitely presented group?

\end{question}

\medskip


\subsection{Discussion of contents} \

As in \cite{GW}, our primary tool is the computational model of $S$-machines.  We refer the reader to our previous paper for a conceptual overview dedicated to our purpose. For a standard general exposition of the power of this tool, see \cite{Sapir06}. 

The article is organized as follows. In Section \ref{sec-S-machines} we restate the definition of $S$-machines, as well as several basic lemmas concerning their operation which are particularly useful in our context. The only major deviation here from the setup of \cite{GW} is the definition of the `$k$-primitive machines' $\textbf{LR}_k$, $\textbf{RL}_k$ obtained from the primitive machines $\textbf{LR}$, $\textbf{RL}$ by composing $k$ copies; the practice of using these embellishments is established in \cite{OS20}, \cite{WEmb}, and others. Section \ref{sec-parameters} introduces the \textit{highest parameter principle}, a standard technique in the long and technical inductive arguments of this nature (see \cite{O}). In \Cref{sec-enhanced}, we largely repeat the main construction of  $\textbf{E}_{\textbf{S}}$ from \cite{GW} to obtain the `$k$-enhanced machine' $\textbf{E}_{\textbf{S},k}$, differing from its predecessor only in that one of the submachines which operates as a primitive machine is replaced with the corresponding $k$-primitive machine. A large choice of $k$ ensures that the execution of primitive machines will comprise the vast majority of the runtime for $\textbf{E}_{\textbf{S},k}$, which will be essential in Section \ref{sec-minimal-diagrams}. Finally, the main machine $\textbf{M}_\textbf{S}$ is obtained by concatenating a very large number of copies of the $k$-enhanced machine together, allowing for arguments in the sections that follow which resemble the standard tools of small-cancellation theory.

The typical constructions of the finitely presented groups associated to $S$-machines are recited in Section \ref{sec-groups}, along with many standard tools used for their study. The key point here, and throughout the literature, is that the execution of $S$-machines is intimately related to the structure of van Kampen (disk) and Schupp (annular) diagrams over these presentations, and so directly dictate the decidability and complexity of the Word and Conjugacy Problems of these associated groups.

In Sections 6, 8, and 9, we adapt many of the arguments of \cite{O18} and \cite{OS20} to this setting to obtain an upper bound on the area of a generic class of van Kampen diagrams over the finite presentation of the main group, effectively finding an upper bound on the group's Dehn function.  On the other side, in Sections 7, 8, 9, and 10, we adapt the authors' arguments in \cite{GW} and combine them with those above to study a generic class of Schupp diagrams over this presentation, effectively obtaining an upper bound on the group's conjugator length function.  These arguments are all combined to form a proof of \Cref{main-theorem} in \Cref{sec-main-proof}.

A major technical innovation of this paper is the concept of spears and directed designs introduced in \Cref{sec-spears}.  This generalizes the concept of shafts and designs introduced by Ol'shanskii in \cite{O18} for the study of van Kampen diagrams.  Crucially, spears and directed designs adapt these tools to the setting of annular diagrams, facilitating the main arguments which obtain the desired bounds on the conjugator length function.

\bigskip

\textbf{Acknowledgements:} The authors express their deep gratitude to Tim Riley and Alexei Miasnikov for many helpful suggestions and discussions relating to this work.


\section{\texorpdfstring{$S$}--machines} \label{sec-S-machines}

In this section, we give a formal definition of the computational model of $S$-machines.  This definition is laid out in much the same way as it is in \cite{GW} and \cite{WEmb}.

\subsection{Rewriting systems} \

Following \cite{BORS}, \cite{O18}, \cite{OS01}, \cite{OSconj}, \cite{OS06}, \cite{OS20}, \cite{SBR}, \cite{WEmb}, \cite{W} among others, we give here a description of $S$-machines as rewriting system for group words.


The hardware of an $S$-machine is defined by a set $Q$ of state letters and a set $Y$ of tape letters, each of which are decomposed into a number of subsets, $Q=\sqcup_{i=0}^NQ_i$ and $Y=\sqcup_{i=1}^NY_i$, called parts. Elements of the inverse sets $Q^{-1},Y^{-1}$ are also called state or tape letters, respectively. Here $N$ is a fixed number and part of the data of the hardware.






Let $W$ be a reduced word over $Y\cup Q$ of the form $q_0^{\eps_0}w_1q_1^{\eps_1}\dots w_sq_s^{\eps_s}$, where $w_i$ is a word over $Y_{k(i)}$, $q_i\in Q_{j(i)}$, and $\eps_i\in \{\pm1\}$. We call $W$ an \textit{admissible word} if, for all $i$, the following conditions are satisfied:

\begin{enumerate}

\item $k(i)=j(i)+1$ if $\eps_i=1$ and $k(i)=j(i)$ if $\eps_i=-1$.

\item Either $j(i+1)=j(i)+\eps_i$ and $\eps_{i+1}=\eps_i$ or $j(i+1)=j(i)$, $q_{i+1}=q_i$, and $\eps_{i+1}=-\eps_i$.

\end{enumerate}



The \textit{base} of an admissible word $W$ is the (perhaps unreduced) word $Q_{j(0)}^{\eps_0}Q_{j(1)}^{\eps_1}\ldots Q_{j(s)}^{\eps_s}$, and the subword $q_{i-1}^{\eps_{i-1}}w_iq_i^{\eps_i}$ is called an \textit{$Q_{j(i)-1}^{\eps_{i-1}}Q_{j(i)}^{\eps_i}$-sector} ($W$ may contain several such sectors). The base $Q_0Q_1\ldots Q_N$ is called the \textit{standard base}, and a \textit{configuration} is an admissible word with the standard base.


%
%
%
%
%


The \textit{$a$-length} and \textit{$q$-length} of an admissible word $W$, denoted $|W|_a$ and $|W|_q$, are the number of $a$-letters and $q$-letters, respectively, comprising $W$.







An \textit{$S$-rule} on the hardware $(Y,Q)$ is a rewriting rule $\theta$ defined by:

\begin{itemize}

\item A pair $q_i,q_i'\in Q_i$ for all $i$

\item A subset $ Y_i(\theta)\subseteq Y_i$ for all $i$.

\item A set of words $\{\a_{i,\theta},\omega_{i,\theta}\in F(Y_i(\theta)):i=1,\dots,N\}$, with the proviso that $\omega_{0,\theta}=\a_{N,\theta}=1$.

\end{itemize}







An admissible word $W$ is \textit{$\theta$-admissible} if every letter comprising it is a letter of $Q(\theta)=\{q_0,\dots, q_N\}$, a letter of $Y(\theta)=\sqcup_{i=0}^N Y_i(\theta)$, or an inverse of one of these letters.  Applying $\theta$ to $W$ returns the admissible word $W\cdot\theta$ defined by:

\begin{itemize}

\item Replacing any occurrence of $q_i^{\pm1}$ with $(\omega_{i,\theta}q_i'\a_{i+1,\theta})^{\pm1}$.

\item Completely freely reducing the result of the previous step

\item Removing any prefix or suffix consisting of tape letters, so that the result begins and ends with $q$-letters.

\end{itemize}



We abbreviate $\theta$ by:
$$\theta=[q_0\to q_0'\a_{1,\theta}, \ q_1\to \omega_{1,\theta}q_1'\a_{2,\theta}, \ \dots, \ q_{N-1}\to \omega_{N-1,\theta}q_{N-1}'\a_{N,\theta}, \ q_N\to \omega_{N,\theta}q_N']$$
If $Y_i(\theta)=\emptyset$, then $\theta$ is said to \textit{lock the $Q_{i-1}Q_i$-sector}.  In this case, we take $\a_{i,\theta}=\omega_{i,\theta}=1$ and write  $q_{i-1}\xrightarrow{\ell}\omega_{i-1,\theta}q_i'$ in the representation of the rule.

Note that this notation omits the \textit{domain} $Y(\theta)$; unless otherwise noted we adopt the convention that $Y_i(\theta)$ equals $Y_i$ or $\emptyset$, with the latter case reflected by the locking notation.







The \textit{inverse $S$-rule} of $\theta$ has:

\begin{itemize}

\item the same defining $q$-letters, with the roles of $q_i$ and $q_i'$ switched,

\item $Y_i(\theta^{-1})=Y_i(\theta)$, and

\item $\omega_{\theta^{-1},i}=\omega_{\theta,i}^{-1}$ and $\a_{\theta^{-1},i}=\a_{\theta,i}^{-1}$.

\end{itemize}


An admissible word $W$ is $\theta$-admissible if and only if $(W\cdot\theta)\cdot\theta^{-1}$ is defined and equals $W$.


A finite symmetric set $\Theta(\textbf{S})$ of $S$-rules over $(Y,Q)$ defines the \textit{software} of the given $S$-machine $\textbf{S}$ with hardware $(Y,Q)$.  We refine this further by arbitrarily partitioning $\Theta(\textbf{S})=\Theta^+(\textbf{S})\sqcup\Theta^-(\textbf{S})$ such that $\theta\in\Theta^+(\textbf{S})$ if and only if $\theta^{-1}\in\Theta^-(\textbf{S})$; elements of $\Theta(\textbf{S})^+$ are called \textit{positive rules} of $\textbf{S}$ and elements of $\Theta(\textbf{S})^-$ are called \textit{negative rules}.


We call a finite sequence $W_0,\dots, W_t$ of admissible words, along with a sequence of with fixed rules $\theta_1,\dots,\theta_t\in \Theta(\textbf{S})$ such that $W_{i-1}\cdot\theta_i\equiv W_i$, a \textit{computation} of \textit{length} $t$ with \textit{history }$\theta_1\cdots\theta_t$. We denote a computation by $\pazocal{C}:W_0\to\dots\to W_t$ and call it \textit{reduced} if its history is a reduced word over $\Theta^+(\textbf{S})$; clearly, any computation can be transformed into a reduced computation with identical beginning and ending words.



If $Q=\sqcup_{i=0}^NQ_i$ is the set of state letters of an $S$-machine, we fix a \textit{start} and \textit{end} letter in each $Q_i$, and call a configuration a \textit{start} or \textit{end} configuration if all its state letters are, respectively, start or end letters.  A \textit{recognizing} $S$-machine is a machine with a fixed subset $I$ of $\{1,\dots,N\}$, where the $Q_{i-1}Q_i$-sector is deemed an \textit{input sector} if and only if $i\in I$.  In this case, an \textit{input configuration} of such a machine is a start configuration whose only tape letters are in input sectors; the \textit{input} of such a configuration is its projection onto the tape alphabet. The \textit{accept configuration} is the end configuration with no tape letters at all.

A configuration $W$ is \textit{accepted} if there exists an \textit{accepting computation} starting from $W$ and ending with the accept configuration.  The input of an accepted input configuration is, by abuse of notation, also said to be accepted.

For a configuration $W$ accepted by $\textbf{S}$, $\tm_\textbf{S}(W)$ is the length of the shortest accepting computation for $W$.  The \textit{time function} $\TM_{\textbf{S}}:\N\to\N$ of $\textbf{S}$ is then: $$\TM_{\textbf{S}}(n)=\max\{\tm_\textbf{S}(W): W\text{ is an accepted input configuration of } \textbf{S}, \ |W|_a\leq n\}$$
The \textit{generalized time function} of $\textbf{S}$ is defined similarly, but with the maximum ranging over all accepted configurations rather than simply input ones.


Two recognizing $S$-machines are \textit{equivalent} if they accept the same set of input words and have $\Theta$-equivalent time functions.



The next statement allows us to simplify many of the combinatorial arguments that follow.

\begin{lemma} [Lemma 2.1 of \cite{O18}] \label{simplify rules}
 
Every recognizing $S$-machine $\textbf{S}$ is equivalent to a recognizing $S$-machine in which $|\a_{\theta,i}|_a,|\omega_{\theta,i}|_a\leq1$ for every rule $\theta$.

\end{lemma}

As we assume the words $\a_{\theta,i}$ and $\omega_{\theta,i}$ are formed over the domain of $\theta$, \Cref{simplify rules} allows us to assume that each part of every rule of an $S$-machine is of the form $q_i\to bq_i'a$ where $\|a\|,\|b\|\leq1$.  Hence, by an abuse of notation we interpret the corresponding part of $\theta^{-1}$ as $q_i'\to b^{-1}q_ia^{-1}$. 



\medskip


\subsection{Elementary properties} \

This subsection lists, with references to the literature, some statements that are very frequently useful in analyzing $S$-machines.


\begin{lemma}[Lemma 2.2 of \cite{WEmb}] \label{locked sectors}

If the rule $\theta$ locks the $Q_iQ_{i+1}$-sector, 
then the base of any $\theta$-admissible word has no subword of the form $Q_iQ_i^{-1}$ or $Q_{i+1}^{-1}Q_{i+1}$.

\end{lemma}



As is common in the $S$-machine literature (e.g. \cite{GW,WEmb}), we consider words over disjoint alphabets $A$ and $B$, where the alphabets are connected by an injection (and/or bijection) $\varphi:A\to B$, usually clear from context.  For any word $w$ over $A$, its \textit{copy} over the alphabet $B$ (with respect to $\varphi$) is the word obtained by applying $\varphi$ to each letter (and preserving exponents). The map $\varphi$ usually arises from re-indexing ($x_i\mapsto x_{i+1}$) or the use of primes ($x\mapsto x'$).


\begin{lemma} [Lemma 2.7 of \cite{O18}] \label{multiply one letter}

Let $X_i$ be a subset of $Y_i\cup Y_i^{-1}$ with $X_i\cap X_i^{-1}=\emptyset$.  Let $\Theta_i^+$ be a set of positive rules in correspondence with $X_i$ such that each rule multiplies the $Q_{i-1}Q_i$-sector by the corresponding letter on the left (respectively on the right).  Let $\pazocal{C}:W_0\to\dots\to W_t$ be a reduced computation with base $Q_{i-1}Q_i$ and history $H\in F(\Theta_i^+)$.  Denote the tape word of $W_j$ as $u_j$ for each $0\leq j\leq t$.  Then:


\begin{enumerate} [label=({\alph*})]

\item $H$ is the natural copy of the reduced form of the word $u_tu_0^{-1}$ read from right to left (respectively the word $u_0^{-1}u_t$ read left to right). In particular, if $u_0\equiv u_t$, then the computation is empty

\item $\|H\|\leq\|u_0\|+\|u_t\|$

\item if $\|u_{j-1}\|<\|u_j\|$ for some $1\leq j\leq t-1$, then $\|u_j\|<\|u_{j+1}\|$

\item $\|u_j\|\leq\max(\|u_0\|,\|u_t\|)$

\end{enumerate}

\end{lemma}

For any two words $w_1,w_2\in F(Y_i)$ over $X_i$, there exists a computation of the machine in \Cref{multiply one letter} which has initial and terminal tape words $w_1$ and $w_2$, respectively.

%
%
%
%
%
%
%

\begin{lemma} [Lemma 2.8 of \cite{O18}] \label{multiply two letters}

Let $X_\ell$ and $X_r$ be disjoint subsets of $Y_i$ which are copies of some set $X$.  Let $\Theta_i^+$ be a set of positive (or negative) rules in correspondence with $X$ such that each rule multiplies the $Q_{i-1}Q_i$-sector on the left by the corresponding letter's copy in $X_\ell$ and on the right by the copy in $X_r$.  Let $\pazocal{C}:W_0\to\dots\to W_t$ be a reduced computation with base $Q_{i-1}Q_i$ and history $H\in F(\Theta_i^+)$.  Denote the tape word of $W_j$ as $u_j$ for each $0\leq j\leq t$.  Then:


\begin{enumerate}[label=({\alph*})]

\item if $\|u_{j-1}\|<\|u_j\|$ for some $0\leq j\leq t-1$, then $\|u_j\|<\|u_{j+1}\|$

\item $\|u_j\|\leq\max(\|u_0\|,\|u_t\|)$ for each $j$

\item $\|H\|\leq\frac{1}{2}(\|u_0\|+\|u_t\|)$.

\end{enumerate}

\end{lemma}

%
%
%
%
%
%

\begin{lemma} [Lemma 3.6 of \cite{OS12}] \label{unreduced base}

Suppose $\pazocal{C}:W_0\to\dots\to W_t$ is a reduced computation of an $S$-machine with base $Q_iQ_i^{-1}$ (respectively $Q_i^{-1}Q_i$). For $0\leq j\leq t$, let $u_j$ be the tape word of $W_j$. Suppose each rule of $\pazocal{C}$ multiplies the $Q_iQ_{i+1}$-sector (respectively the $Q_{i-1}Q_i$-sector) by a letter from the left (respectively from the right), with different rules corresponding to different letters. Then there exists a factorization $H\equiv H_1H_2^\ell \bar{H}_2H_3$ such that:

\begin{enumerate}

\item $\ell\in\N$

\item $\bar{H_2}$ is a proper prefix of $H_2$.

\item $|W_i|_a=|W_{i-1}|_a-2$ for all $1\leq i\leq\|H_1\|$

\item $|W_i|_a=|W_{i-1}|_a=\|H_2\|$ for all $\|H_1\|+1\leq i\leq t-\|H_3\|$

\item $|W_i|_a=|W_{i-1}|_a+2$ for all $t-\|H_3\|+1\leq i\leq t$

\end{enumerate}


\end{lemma}

Despite the different phrasing, which is due to our specific purpose in this paper, Lemma \ref{unreduced base} has the same proof as in \cite{OS12}.  Further, each instance of $H_2$ cyclically permutes the tape word, while $\bar{H}_2$ cyclically permutes just a proper prefix or suffix of the word.  The former fact implies, for $h_i=\|H_i\|$ and $0\leq k\leq \ell$, that $W_{h_1}\equiv W_{h_1+kh_2}$.


\begin{lemma}[Lemma 2.6 of \cite{GW}] \label{unreduced base quadratic}

Let $\textbf{S}$ be an $S$-machine satisfying \Cref{simplify rules} and suppose every rule of $\textbf{S}$ multiplies the $Q_iQ_{i+1}$ sector (respectively the $Q_{i-1}Q_i$-sector) by a letter from the left (respectively from the right), with different rules corresponding to different letters.  Let $\pazocal{C}:W_0\to\dots\to W_t$ be a nonempty reduced computation whose base $B$ contains a subword of the form $Q_iQ_i^{-1}$ (respectively $Q_i^{-1}Q_i)$.  Suppose $\pazocal{C}$ has minimal length among computations between $W_0$ and $W_t$.  Then there exists a two-letter subword $UV$ of $B$ such that:

\begin{enumerate}

\item There exists a rule in the history of $\pazocal{C}$ which multiplies the $UV$-sector by a letter on the left or on the right

\item Letting $\pazocal{C}':W_0'\to\dots\to W_t'$ the restriction of $\pazocal{C}$ to the base $UV$, we have $t\leq 12n^2+2n$ for $n=\max(|W_0'|_a,|W_t'|_a)$.

\end{enumerate}

\end{lemma}

\medskip


\subsection{Primitive machines} \

This section introduces the \textit{primitive machines} discussed in \cite{GW} as well as some generalizations introduced in previous literature.  


The first primitive machine is $\textbf{LR}(Y)$, where $Y$ is some alphabet.  The standard base of the machine is $PQR$, with $P=\{p_1,p_2\}$, $Q=\{q_1,q_2\}$, and $R=\{r_1,r_2\}$; $p_1$, $q_1$, and $r_1$ are the start sate letters, while $p_2$, $q_2$, $r_2$ are the end state letters. The tape letters for the $PQ$ and $QR$ sectors are (respectively) two copies $Y_1,Y_2$ of $Y$, created by affixing the subscripts $y_1,y_2$ to each $y\in Y$.

The positive rules of $\textbf{LR}(Y)$ are defined as follows:

\begin{itemize}
	
	\item For every $y\in Y$, a leftward-translation $\tau_1(y)=[p_1\to p_1, \ q_1\to y_1^{-1}q_1y_2, \ r_1\to r_1]$
	
	\item The \textit{connecting rule} $\zeta=[p_1\xrightarrow{\ell} p_2, \ q_1\to q_2, \ r_1\to r_2]$
	
	\item For every $y\in Y$, a rightward-translation $\tau_2(y)=[p_2\to p_2, \ q_2\to y_1q_2y_2^{-1}, \ r_2\to r_2]$.

\end{itemize}

For any configuration $W$ of $\textbf{LR}(Y)$, let $w$ be the concatenation of all copies of the tape words over $Y$ (i.e. the pre-images of the mappings $Y\to Y_i$); applying any rule leaves the value of $w$ in $F(Y)$ unchanged. The use of this or a similar observation is called a \textit{projection argument}.



Projection arguments give the following useful description of this machine's operation:

\begin{lemma}[Lemma 3.1 of \cite{O18}] \label{primitive computations}

Let $\pazocal{C}:W_0\to\cdots\to W_t$ be a reduced computation of $\textbf{LR}(Y)$ in the standard base. Then:

\begin{enumerate}[label=({\arabic*})]

\item if $|W_{i-1}|_a<|W_i|_a$ for some $1\leq i\leq t-1$, then $|W_i|_a<|W_{i+1}|_a$

\item $|W_i|_a\leq\max(|W_0|_a,|W_t|_a)$ for each $i$

\item Suppose $W_0\equiv p_1~u~q_1r_1$ and $W_t\equiv p_2~v~q_2r_2$ for some $u,v\in F(Y_1)$.  Then $u\equiv v$, $|W_i|_a=\|u\|\defeq\ell$ for each $i$, $t=2\ell+1$, and the $PQ$-sector is locked in the rule $W_\ell\to W_{\ell+1}$. Moreover, letting $\bar{u}$ be the word obtained by reading $u$ from right to left, the history $H$ of $\pazocal{C}$ is a copy of $\bar{u}\zeta u$.

\item if $W_0\equiv p_j~u~q_jr_j$ (resp $W_0\equiv p_jq_j~u~r_j$) and $W_t\equiv p_j~v~q_jr_j$ (resp $W_t\equiv p_jq_j~v~r_j$) for some $u,v$ and $j\in\{1,2\}$, then $u\equiv v$ and the computation is empty (i.e $t=0$)


\item if $W_0$ is of the form $p_j~u~q_jr_j$ or $p_jq_j~v~r_j$ for $j\in\{1,2\}$, then $|W_i|_a\geq|W_0|_a$ for every $i$.

\end{enumerate}

\end{lemma}

It is exceedingly useful to note that there exists a reduced computation of $\textbf{LR}(Y)$ of the form described in \Cref{primitive computations}(3) for any $u\in F(Y)$.


Two more useful statements are:


\begin{lemma}[Lemma 2.8 of \cite{GW}] \label{primitive time}

Let $\pazocal{C}:W_0\to\dots\to W_t$ be a reduced computation of $\textbf{LR}(Y)$ in the standard base.  Then $t\leq2\max(|W_0|_a,|W_t|_a)+1$.

\end{lemma}

%
%
%
%
%


\begin{lemma} [Lemma 3.4 of \cite{OS20}] \label{primitive unreduced} Suppose $W_0\to\dots\to W_t$ is a reduced computation of $\textbf{LR}(Y)$ with base $PQQ^{-1}P^{-1}$ (or $R^{-1}Q^{-1}QR$) such that $W_0\equiv p_jq_j~u~q_j^{-1}p_j^{-1}$ (or $W_0\equiv r_j^{-1}q_j^{-1}~v~q_jr_j$) for $j\in\{1,2\}$ and some word $u$ (or $v$). Then $|W_0|_a\leq\dots\leq|W_t|_a$ and all state letters of $W_t$ have the index $j$.  In particular, if $W_t\equiv p_jq_j~u'~q_j^{-1}p_j^{-1}$ (or $W_t\equiv r_j^{-1}q_j^{-1}~v'~q_jr_j$), then $t=0$.

\end{lemma}


We also call upon the `flipped' version of $\textbf{LR}(Y)$, denoted $\textbf{RL}(Y)$. Formally, identifying the hardware of this machine with that of $\textbf{LR}(Y)$, the positive rules of $\textbf{RL}(Y)$ are the same as $\textbf{LR}(Y)$, save that the subscripts of the state letters in each $\tau_1(y)$ and $\tau_2(y)$ are switched. The obvious analogues of Lemmas \ref{primitive computations}--\ref{primitive unreduced} hold for $\textbf{RL}(Y)$.

We will find it useful to concatenate these primitive machines together. Formally, for a fixed $k\geq1$, the primitive machine $\textbf{LR}_k(Y)$ has standard base $PQR$ with $P=\{p_i\}_{i=1}^{2k}$, $Q=\{q_i\}_{i=1}^{2k}$, and $R=\{r_i\}_{i=1}^{2k}$, while the tape alphabets are again copies of $Y$.  The state letters with subscript $1$ are again the start state letters, while those with subscript $2k$ are now taken as the end state letters.  

The positive rules of $\textbf{LR}_k(Y)$ are then given as follows:

\begin{itemize}
	
	\item For every $i\in\{1,\dots,k\}$ and every $y\in Y$, the leftward-translation
	\begin{align*}
	\tau_{2i-1}(y)&=[p_{2i-1}\to p_{2i-1}, \ q_{2i-1}\to y_1^{-1}q_{2i-1}y_2, \ r_{2i-1}\to r_{2i-1}] 
	\intertext{\item For every $i\in\{1,\dots,k\}$, the connecting rule}
	\zeta_{2i-1}&=\{p_{2i-1}\xrightarrow{\ell} p_{2i}, \ q_{2i-1}\to q_{2i}, \ r_{2i-1}\to r_{2i}] 
	\intertext{\item For every $i\in\{1,\dots,k\}$ and every $y\in Y$, the rightward-translation}
	\tau_{2i}(y)&=[p_{2i}\to p_{2i}, \ q_{2i}\to y_1q_{2i}y_2^{-1}, \ r_{2i}\to r_{2i}] \\\intertext{\item For every $i\in\{1,\dots,k-1\}$, the connecting rule}
	\zeta_{2i}&=[p_{2i}\to p_{2i+1}, \ q_{2i}\xrightarrow{\ell} q_{2i+1}, \ r_{2i}\to r_{2i+1}]
	\end{align*}
	
\end{itemize}

Note that, per this convention, $\textbf{LR}_1(Y)=\textbf{LR}(Y)$.


The analogues of Lemmas \ref{primitive computations} and \ref{primitive unreduced} hold for the machines $\textbf{LR}_k(Y)$. We will make specific use of the analogue of Lemma \ref{primitive computations}(3):

\begin{lemma}[Lemma 2.8 of \cite{WEmb}] \label{LR_k analogue}

Let $\pazocal{C}:W_0\to\dots\to W_t$ be a reduced computation of $\textbf{LR}_k(Y)$ in the standard base. If $W_0\equiv p_1uq_1~r_1$ and $W_t\equiv p_{2k}vq_{2k}~r_{2k}$ for some $u,v\in F(Y_1)$, then $u\equiv v$, $|W_i|_a=|W_0|_a\defeq l$ for each $i$, and $t=2lk+2k-1$.

\end{lemma}

As above, it is very useful to observe that there exists a computation of $\textbf{LR}_k(Y)$ of the form detailed in \Cref{LR_k analogue} for any $u\in F(Y)$.

The next statement is the analogue of \Cref{primitive time} for this setting.  Its proof can be adapted from the proof presented for the analogous statement in \cite{GW}.

\begin{lemma} \label{LR_k primitive time}

Let $\pazocal{C}:W_0\to\dots\to W_t$ be a reduced computation of $\textbf{LR}_k(Y)$ in the standard base.  Then $t\leq 2k\max(|W_0|_a,|W_t|_a)+2k-1$.

\end{lemma}



We include $\textbf{LR}_k(Y)$ as a primitive machine as well and, when the alphabet $Y$ is contextually clear, we denote all the primitive machines constructed above by $\textbf{LR}$, $\textbf{RL}$, and $\textbf{LR}_k$.

\medskip
	

\section{Parameters} \label{sec-parameters} \

The arguments spanning the rest of this paper are reliant on the \textit{highest parameter principle}, the dual to the lowest parameter principle described in \cite{O}. In particular, we introduce the relation $<<$ on parameters defined as follows:

If $\a_1,\a_2,\dots,\a_n$ are parameters with $\a_1<<\a_2<<\dots<<\a_n$, then for all $2\leq i\leq n$, it is understood that $\a_1,\dots,\a_{i-1}$ are assigned prior to the assignment of $\a_i$ and that the assignment of $\a_i$ is dependent on the assignment of its predecessors. The resulting inequalities are then understood as `$\a_i\geq$(any expression used henceforth involving $\a_1,\dots,\a_{i-1}$)'.

Specifically, the assignment of parameters we use here is:
\begin{align*}
\lambda^{-1}<&<k<<N<<c_0<<c_1<<c_2<<c_3<<L_0<< L<<K_0<<K\\
&<<J<<\delta^{-1}<<C_1<<C_2<<C_3<<N_1<<N_2<<N_3<<N_4<<N_5
\end{align*}

\medskip
	

\section{Computational constructions} \label{sec-enhanced}

Our goal in this section is to construct an $S$-machine whose computational makeup is sufficient for the proof of \Cref{main-theorem}.  Per the hypotheses of that statement, this construction begins with an arbitrary recognizing $S$-machine $\textbf{S}$ with designated time function $\TM_\textbf{S}$. 

An analogous objective was accomplished in Section 3 of \cite{GW} through the construction of the `enhanced machine' $\textbf{E}_\textbf{S}$.  This section begins with a very similar construction, but with two crucial alterations: (1) The five submachines are composed in a different order, and (2) The parallel function of the primitive machine in the `fourth step' is replaced with the parallel function of $\textbf{LR}_k$ (where $k$ is the parameter listed in \Cref{sec-parameters}).  Despite these changes, the proofs of many of the statements pertaining to this initial construction are virtually identical to those that were presented for the analogous statements in \cite{GW}, and are thus omitted or reduced to a brief discussion of any necessary deviations.

From there, the constructed machine is concatenated with a `mirror copy' of itself to introduce a certain level of symmetry which is necessary for the arguments of \Cref{sec-scopes}. Finally, the resulting machine is concatenated with a large number of copies of itself, enabling a version of small-cancellation in the corresponding diagrams (see \Cref{sec-minimal-diagrams}) which aids with the necessary combinatorial arguments.

\subsection{Historical sectors} \

As in \cite{GW}, the first step is to introduce `historical sectors' that keep track of the rules applied in a computation and in turn control the time and space complexity of the machine.

For this, let $Q_0Q_1\dots Q_s$ be the standard base of the machine $\textbf{S}$ and $\Phi^+$ be its positive rules.  For each $i\in\{0,\dots,s\}$, let $Q_{i,\ell}$ and $Q_{i,r}$ be two disjoint copies of $Q_i$ and let $X_{i,\ell}$, $X_{i,r}$ be two disjoint copies of $\Phi^+$.  We then construct the machine $\textbf{S}_h$ with standard base
$$Q_{0,\ell}Q_{0,r}Q_{1,\ell}Q_{1,r}\dots Q_{s,\ell}Q_{s,r}$$
The tape alphabet of the $Q_{i,\ell}Q_{i,r}$-sector is $X_{i,\ell}\sqcup X_{i,r}$, and the tape alphabet of the $Q_{i-1,r}Q_{i,\ell}$-sector is identified with that of the $Q_{i-1}Q_i$-sector of $\textbf{S}$.

The positive rules of $\textbf{S}_h$ are identified with $\Phi^+$, with the rule $\theta\in\Phi^+$ operating in the $Q_{i-1,r}Q_{i,\ell}$-sector as it does in the $Q_{i-1}Q_i$-sector of $\textbf{S}$ and multiplying the tape word of the \textit{historical sector} $Q_{i,\ell}Q_{i,r}$ on the left by the copy of $\theta^{-1}$ over the \textit{left historical alphabet} $X_{i,\ell}$ and on the right by the copy of $\theta$ over the \textit{right historical alphabet} $X_{i,r}$.

The rest of the details of the machine are defined in the obvious way: The input sectors are exactly the \textit{working sectors} $Q_{i-1,r}Q_{i,\ell}$ for which the $Q_{i-1}Q_i$-sector of $\textbf{S}$ is an input sector, while the start and end letters in $Q_{i,\ell},Q_{i,r}$ are the copies of those of $Q_i$.

The following statements exhibit the purpose and utility of this construction.

\begin{lemma}[Lemma 3.9 of \cite{O18}] \label{one alphabet historical words}

Let $W_0\to\dots\to W_t$ be a reduced computation of $\textbf{S}_h$ with base $Q_{i,\ell}Q_{i,r}$ and history $H$.  Suppose the $a$-letters of $W_0$ are all from the left (respectively right) historical alphabet.  Then $\|H\|\leq|W_t|_a$ and $|W_0|_a\leq|W_t|_a$.  Moreover, if $t\geq1$, then the tape word of $W_t$ contains letters from the right (respectively left) historical alphabet.

\end{lemma}


\begin{lemma}[Lemma 5.14 of \cite{WEmb}] \label{one alphabet historical words unreduced}

Let $W_0\to\dots\to W_t$ be a reduced computation of $\textbf{S}_h$ with base $Q_{i,\ell}Q_{i,\ell}^{-1}$ (respectively $Q_{i,r}^{-1}Q_{i,r}$) and history $H$.  Suppose the $a$-letters of $W_0$ are all from the corresponding right (respectively left) historical alphabet.  Then $|W_0|_a=|W_t|_a-2\|H\|$.  Moreover, if $t\geq1$, then the tape word of $W_t$ contains letters from both the left and right historical alphabets.

\end{lemma}

\begin{lemma}[Lemma 3.12 of \cite{O18}] \label{M_2 bound}

Let $\pazocal{C}:W_0\to\dots\to W_t$ be a reduced computation of $\textbf{S}_h$.  If the base of $\pazocal{C}$ has length at least 3, then $|W_i|_a\leq9(|W_0|_a+|W_t|_a)$ for all $0\leq i\leq t$.

\end{lemma}

\medskip


\subsection{Composition of machines} \

The next machine in our construction, called the \textit{standard $k$-enhanced machine} and denoted $\textbf{E}_{\textbf{S},k}^0$, functions as the \textit{composition} of (a `padded' version of) $\textbf{S}_h$ with four other recognizing machines.

This construction is carried out in much the same way as in \cite{GW}: The machine is made up of five recognizing `submachines' $\textbf{E}_{\textbf{S},k}^0(1),\dots,\textbf{E}_{\textbf{S},k}^0(5)$ which are connected through the \textit{transition rules} $\sigma(i,i+1)$ which switch the state letters from the end letters of $\textbf{E}_{\textbf{S},k}^0(i)$ to the start letters of $\textbf{E}_{\textbf{S},k}^0(i+1)$ (with some restricted domains).  Naturally, the start letters of $\textbf{E}_{\textbf{S},k}^0(1)$ and the end letters of $\textbf{E}_{\textbf{S},k}^0(5)$ are taken to be the start and end letters of the machine, respectively, while the input sectors of $\textbf{E}_{\textbf{S},k}^0(1)$ are the input sectors of the machine.  The standard base of $\textbf{E}_{\textbf{S},k}^0$ is:
$$B_0\equiv(P_0Q_{0,\ell}Q_{0,r}R_0)(P_1Q_{1,\ell}Q_{1,r}R_1)\dots(P_sQ_{s,\ell}Q_{s,r}R_s)$$
In what follows, we will take the length of $B_0$ to be $(N-1)/2$, where $N$ is the parameter listed in \Cref{sec-parameters}.  Note that this assignment is possible since we may add as many locked sectors to $\textbf{S}$ as we desire without changing its computational makeup, and so can take $N$ as large as necessary.

The tape alphabet of the $R_{i-1}P_i$-sector is identified with that of the $Q_{i-1,r}Q_{i,\ell}$-sector of $\textbf{S}_h$, while the tape alphabets of the $P_iQ_{i,\ell}$-, $Q_{i,\ell}Q_{i,r}$-, and $Q_{i,r}R_i$-sectors consist of \textit{left} and \textit{right historical alphabets} which are copies of $\Phi^+$.  As such, we adopt the terminology from the previous section that the $R_{i-1}P_i$-sectors are called \textit{working} while the others are called \textit{historical}.

Both the parts of the state letters and the positive rules of the machine are the disjoint union of those of the five submachines.  As such, we fully understand the machine by describing the makeup of the five submachines and the domains of the transition rules.  To this end:

\begin{itemize}

\item Each part of the state letters of the submachine $\textbf{E}_{\textbf{S},k}^0(1)$ consists of a single letter, functioning as both the start and end state letter of the corresponding part.  The input sectors are the working sectors corresponding to the input sectors of $\textbf{S}$.  The positive rules of the machine are in one-to-one correspondence with $\Phi^+$, with the rule corresponding to $\theta\in\Phi^+$ multiplying each $Q_{i,\ell}Q_{i,r}$-sector on the left by the copy of $\theta$ in the left historical alphabet (the left alphabet being its domain in this sector) and locking all other non-input sectors.

\item The submachine $\textbf{E}_{\textbf{S},k}^0(2)$ operates `in parallel' as the primitive machine $\textbf{RL}=\textbf{RL}(\Phi^+)$ on each of the subwords $P_iQ_{i,\ell}Q_{i,r}$ of the standard base, taking the corresponding left historical alphabets as the copies of $\Phi^+$ and locking all other non-input sectors.  To match the makeup of $\textbf{RL}$, each part of the state letters contains two letters, with the transition from the start to the end letters given by the copy of the connecting rule.

\item $\textbf{E}_{\textbf{S},k}^0(3)$ functions as a copy of $\textbf{S}_h$, with each rule operating on the working sectors and the $Q_{i,\ell}Q_{i,r}$-sectors as its analogue and locking all other sectors.

\item The submachine $\textbf{E}_{\textbf{S},k}^0(4)$ operates in parallel as the primitive machine $\textbf{LR}_k=\textbf{LR}_k(\Phi^+)$ on each of the subwords $Q_{i,\ell}Q_{i,r}R_i$ of the standard base, taking the corresponding right historical alphabets as the copies of $\Phi^+$ and locking all other sectors.  Each part of the state letters contains $2k$ letters.

\item The submachine $\textbf{E}_{\textbf{S},k}^0(5)$ functions in the same way as $\textbf{E}_{\textbf{S},k}^0(1)$, except that the rule corresponding to $\theta\in\Phi^+$ multiplies each $Q_{i,\ell}Q_{i,r}$-sector on the right by the copy of $\theta^{-1}$ in the right alphabet (its domain in this sector) and locks all other (including input) sectors.

\item The transition rules $\sigma(12)$ and $\sigma(23)$ have identical domains: In each input sector it is the entire tape alphabet, in each $Q_{i,\ell}Q_{i,r}$-sector it is the left historical alphabet, and in any other sector it is empty.

\item The transition rules $\sigma(34)$ and $\sigma(45)$ also have identical domains.  However, their domain in each $Q_{i,\ell}Q_{i,r}$-sector is the right historical alphabet, while that in all other sectors (including the input sectors) is empty.

\end{itemize}

As discussed in the introduction to this section, this construction can be interpreted as a generalization of that of the standard enhanced machine in \cite{GW}.  Indeed, the standard enhanced machine can be interpreted as the machine obtained from $\textbf{E}_{\textbf{S},k}^0$ by:

\begin{itemize}

\item Taking $k=1$.

\item Switching the side on which the rules of the submachines $\textbf{E}_{\textbf{S},1}^0(1)$ and $\textbf{E}_{\textbf{S},1}^0(5)$ multiply the tape words of the $Q_{i,\ell}Q_{i,r}$-sectors.

\item Altering $\textbf{E}_{\textbf{S},1}^0(2)$ so that it operates in parallel as $\textbf{LR}$ on the subwords $Q_{i,\ell}Q_{i,r}R_i$.

\item Performing the analogous modification to $\textbf{E}_{\textbf{S},1}^0(4)$.

\end{itemize}  

One may understand these points as `swapping' the submachines (1),(5) and (2),(4), though strictly speaking this is not quite the case since we must also swap the roles of the tape alphabets in each.

As in \cite{GW}, the composition above functions to control the computational makeup of the machine: 


\begin{itemize}

\item The submachines $\textbf{E}_{\textbf{S},k}^0(1)$ and $\textbf{E}_{\textbf{S},k}^0(5)$ function to add and delete history words, making the time function equivalent to $\TM_\textbf{S}$ (whereas the time function of $\textbf{S}_h$ is linear).

\item The submachines $\textbf{E}_{\textbf{S},k}^0(2)$ and $\textbf{E}_{\textbf{S},k}^0(4)$, on the other hand, give tight control on how a reduced computation with unreduced base can proceed.

\end{itemize}

Moreover, as will be made precise in the next section, the introduction of the parameter $k$ allows us to also control how long certain computations operate as $\textbf{E}_{\textbf{S},k}^0(4)$.

Note that by the assignment of the input sectors, there is a correspondence between the input configurations of $\textbf{S}$ and those of $\textbf{E}_{\textbf{S},k}^0$ given by simply adding several locked sectors.  Given an input configuration $W$ of $\textbf{S}$, we denote the corresponding input configuration of $\textbf{E}_{\textbf{S},k}^0$ by $I_0(W)$.  Note that the input of $I_0(W)$ is the same as that of $W$.

\medskip


\subsection{Computations of the standard $k$-enhanced machine} \label{sec-standard-computations} \

In order to understand the computational makeup of $\textbf{E}_{\textbf{S},k}^0$, it is useful to break down any reduced computation into its subcomputations which operate as a submachines $\textbf{E}_{\textbf{S},k}^0(j)$.  To keep track of this, we consider the computation's `step history', a concept used in the same way in \cite{GW} and developed in \cite{O18}, \cite{OS20}.



By construction, the history $H$ of a reduced computation of $\textbf{E}_{\textbf{S},k}^0$ is the product of transition rules and maximal nonempty products of rules of one of the five defining submachines $\textbf{E}_{\textbf{S},k}^0(j)$. Replacing each transition rule $\sigma(ij)$ by $(ij)$ and condensing each (maximal) product of rules of $\textbf{E}_{\textbf{S},k}^0(i)$ to $(i)$ gives the \textit{step-history} of $H$. We refer the reader to \cite{WEmb} for a more detailed description, and here merely set out the convention that the step $(ij)$ may be omitted whenever its presence is implied by context (i.e. $(2)(23)(3)$ will be written as $(2)(3)$).


A key step in our argument is studying which subwords are \textit{impossible} to have in a step history:
\begin{lemma}[Compare with Lemma 3.3 of \cite{GW}] \label{E primitive step history}

Let $\pazocal{C}:W_0\to\dots\to W_t$ be a reduced computation of $\textbf{E}_{\textbf{S},k}^0$ with base $B$.

\begin{enumerate}[label=(\alph*)]

\item If $B$ contains a subword $B'$ of the form $(P_iQ_{i,\ell}Q_{i,r})^{\pm1}$, of the form $P_iQ_{i,\ell}Q_{i,\ell}^{-1}P_i^{-1}$, or of the form $Q_{i,r}^{-1}Q_{i,\ell}^{-1}Q_{i,\ell}Q_{i,r}$, then the step history of $\pazocal{C}$ is not $(12)(2)(21)$ or $(32)(2)(23)$.

\item If $B$ contains a subword $B'$ of the form $(Q_{i,\ell}Q_{i,r}R_i)^{\pm1}$, of the form $Q_{i,\ell}Q_{i,r}Q_{i,r}^{-1}Q_{i,\ell}^{-1}$, or of the form $R_i^{-1}Q_{i,r}^{-1}Q_{i,r}R_i$, then the step history of $\pazocal{C}$ is not $(34)(4)(43)$ or $(54)(4)(45)$.

\end{enumerate}

\end{lemma}

%
%

\begin{proof}

The proof is identical to that presented in \cite{GW}, using the analogues of \Cref{primitive computations}(4) and \Cref{primitive unreduced} for $\textbf{LR}_k$.

\end{proof}

\begin{lemma}[Lemma 3.4 of \cite{GW}] \label{E run step history}

Let $\pazocal{C}:W_0\to\dots\to W_t$ be a reduced computation of $\textbf{E}_{\textbf{S},k}^0$ with base $B$.

\begin{enumerate}[label=(\alph*)]

\item If $B$ contains a subword $UV$ of the form $(Q_{i,\ell}Q_{i,r})^{\pm1}$ or of the form $Q_{i,r}^{-1}Q_{i,r}$, then the step history of $\pazocal{C}$ is not $(23)(3)(32)$.

\item If $B$ contains a subword $UV$ of the form $(Q_{i,\ell}Q_{i,r})^{\pm1}$ or of the form $Q_{i,\ell}Q_{i,\ell}^{-1}$, then the step history of $\pazocal{C}$ is not $(43)(3)(34)$.

\end{enumerate}

\end{lemma}

%
%
%

\begin{lemma}[Compare with Lemma 3.5 of \cite{GW}] \label{inputs accepted}

Let $\pazocal{D}$ be a reduced computation of $\textbf{S}$ accepting the input configuration $W$.  Letting $H\in F(\Phi^+)$ be the history of $\pazocal{D}$, there exists a reduced computation of $\textbf{E}_{\textbf{S},k}^0$ with step history $(1)(2)(3)(4)(5)$ and length $(2k+5)\|H\|+(2k+4)$ which accepts the input configuration $I_0(W)$.

\end{lemma}

\begin{proof}

The discrepancy between this statement and its analogue in \cite{GW} arises from the application of \Cref{LR_k analogue} in place of \Cref{primitive computations}(3) for the subcomputation with step history (4).  Otherwise, the proof is analogous.

\end{proof}

\begin{lemma}[Lemma 3.6 of \cite{GW}] \label{E0 language}

The language of accepted inputs of $\textbf{E}_{\textbf{S},k}^0$ is the same as that of $\textbf{S}$.  Moreover, the time function of $\textbf{E}_{\textbf{S},k}^0$ is $\sim$-equivalent to $\TM_\textbf{S}$.

\end{lemma}

\begin{lemma}[Compare with Lemma 3.7 of \cite{GW}] \label{E standard one-step}

Let $\pazocal{C}:W_0\to\dots\to W_t$ be a reduced computation of $\textbf{E}_{\textbf{S},k}^0$ with step history $(j)$ for some $j\in\{1,\dots,5\}$.  Suppose the base of $\pazocal{C}$ is of the form $(P_iQ_{i,\ell}Q_{i,r}R_i)^{\pm1}$.

\begin{enumerate}[label=(\alph*)]

\item If $j\in\{1,5\}$, then $t\leq2\max(|W_0|_a,|W_t|_a)$.

\item If $j=3$, then $t\leq\max(|W_0|_a,|W_t|_a)$.  Moreover, if $W_0$ is $\sigma(32)$- or $\sigma(34)$-admissible, then $|W_0|_a\leq|W_t|_a$.

\item If $j=2$, then $t\leq2\max(|W_0|_a,|W_t|_a)+1$.  Moreover, if $W_0$ is $\sigma(j,j-1)$- or $\sigma(j,j+1)$-admissible, then $|W_0|_a\leq|W_t|_a$.

\item If $j=4$, then $t\leq2k\max(|W_0|_a,|W_t|_a)+2k-1$.  Moreover, if $W_0$ is $\sigma(j,j-1)$- or $\sigma(j,j+1)$-admissible, then $|W_0|_a\leq|W_t|_a$.

\end{enumerate}

\end{lemma}

\begin{proof}

%
%

The proof presented in \cite{GW} immediately implies (a)-(c).  For (d), we apply \Cref{LR_k primitive time} and the analogue of \Cref{primitive computations}(5) for $\textbf{LR}_k$.

\end{proof}

\begin{lemma}[Lemma 3.8 of \cite{GW}] \label{E time (12)}

Let $\pazocal{C}:W_0\to\dots\to W_t$ be a reduced computation of $\textbf{E}_{\textbf{S},k}^0$ whose base is of the form $(P_iQ_{i,\ell}Q_{i,r}R_i)^{\pm1}$.  If the step history of $\pazocal{C}$ or its inverse is a subword of $(2)(3)(4)$, then $t\leq c_0\max(|W_0|_a,|W_t|_a)+c_0$.  Moreover, if $W_0$ is $\sigma(21)$- or $\sigma(45)$-admissible, then $|W_0|_a\leq|W_t|_a$.

\end{lemma}

\begin{proof}

The difference between the estimate provided in this statement and the one in its analogue in \cite{GW} is accounted for by the alteration in \Cref{E standard one-step}.  The proof follows in just the same way, using the parameter assignment $c_0>>k$.

%
%
%
%

\end{proof}


\begin{lemma}[Compare with Lemma 3.9 of \cite{GW}] \label{E standard no (1)}

Let $\pazocal{C}:W_0\to\dots\to W_t$ be a reduced computation of $\textbf{E}_{\textbf{S},k}^{0}$ whose base is of the form $(P_iQ_{i,\ell}Q_{i,r}R_i)^{\pm1}$.  If the step history of $\pazocal{C}$ does not contain the letters $(1)$, $(12)$, or $(21)$, then $t\leq 3c_0\max(|W_0|_a,|W_t|_a)+2c_0$.

\end{lemma}

\begin{proof}

By Lemmas \ref{E primitive step history} and \ref{E run step history}, the step history of $\pazocal{C}$ must be a subword of $(2)(3)(4)(5)(4)(3)(2)$.  Moreover, by \Cref{E time (12)} it suffices to assume that $\pazocal{C}$ contains a nonempty maximal subcomputation $W_x\to\dots\to W_y$ with step history $(5)$.

\Cref{E standard one-step}(a) then implies $y-x\leq2\max(|W_x|_a,|W_y|_a)$, so that we may assume $x>0$ or $y<t$.  In the former case, the subcomputation $W_0\to\dots\to W_{x-1}$ satisfies the hypotheses of \Cref{E time (12)}, so that $x\leq c_0|W_0|_a+c_0$ and $|W_x|_a\leq|W_0|_a$.  In the latter, the same argument applies to the subcomputation $W_{y+1}\to\dots\to W_t$, so that $t-y\leq c_0|W_t|_a+c_0$ and $|W_y|_a\leq|W_t|_a$.  Thus, $t\leq c_0|W_0|_a+c_0|W_t|_a+2\max(|W_0|_a,|W_t|_a)+2c_0$, and so the statement is given by taking $c_0\geq2$.

\end{proof}

The next statement follows in just the same way:

\begin{lemma}[Compare with Lemma 3.10 of \cite{GW}] \label{E standard no (5)}

Let $\pazocal{C}:W_0\to\dots\to W_t$ be a reduced computation of $\textbf{E}_{\textbf{S},k}^{0}$ whose base is of the form $(P_iQ_{i,\ell}Q_{i,r}R_i)^{\pm1}$.  If the step history of $\pazocal{C}$ does not contain the letters $(5)$, $(54)$, or $(45)$, then $t\leq3c_0\max(|W_0|_a,|W_t|_a)+2c_0$.

\end{lemma}





Similarly, the next two statements are proved in just the same way as their analogues, using \Cref{LR_k analogue} in place of \Cref{primitive computations}(3) for the maximal subcomputations with step history (4).

\begin{lemma}[Lemma 3.11 of \cite{GW}] \label{E standard accepted (4)}

If $\pazocal{C}:W_0\to\dots\to W_t$ is a reduced computation of $\textbf{E}_{\textbf{S},k}^0$ in the standard base with step history $(34)(4)(45)$, then $W_t$ is accepted by a reduced computation of $\textbf{E}_{\textbf{S},k}^0(5)$.

\end{lemma}

%
%
%

\begin{lemma}[Lemma 3.12 of \cite{GW}] \label{E standard accepted (3)}

If $\pazocal{C}:W_0\to\dots\to W_t$ is a reduced computation of $\textbf{E}_{\textbf{S},k}^0$ in the standard base with step history $(23)(3)(34)$, then $W_t$ is accepted by a reduced computation with step history $(4)(5)$.

\end{lemma}

%
%
%
%

Let $W$ be a configuration of $\textbf{E}_{\textbf{S},k}^0$.  Similar to the setup in \cite{GW}, we define:

\begin{itemize}

\item $W^{(i)}$ to be the admissible subword of $W$ with base $P_iQ_{i,\ell}Q_{i,r}R_i$.

\item The \textit{working length} of $W$, $|W|_{wk}$, to be the number of tape letters from the working sectors of $W$, {\frenchspacing i.e. $|W|_{wk}=|W|_a-\sum|W^{(i)}|_a$}.

\end{itemize}

Note that if $W$ is accepted, then since each rule of $\textbf{E}_{\textbf{S},k}^0$ operates in parallel on the relevant sectors, $W^{(i)}$ and $W^{(j)}$ are copies of one other (as the same is true for the accept configuration).  



\begin{lemma}[Compare with Lemma 3.13 of \cite{GW}] \label{E standard (1)}

Let $W$ be an accepted configuration of $\textbf{E}_{\textbf{S},k}^0$ whose state letters belong to the hardware of the submachine $\textbf{E}_{\textbf{S},k}^0(1)$.  Then there exists a reduced computation accepting $W$ with step history $(1)(2)(3)(4)(5)$ and length at most $$(2k+5)\TM_\textbf{S}(|W|_{wk})+|W^{(i)}|_a+2k+4$$

\end{lemma}

\begin{proof}

As in the proof of the analogous statement in \cite{GW}, there exists a reduced computation of $\textbf{E}_{\textbf{S},k}^0(1)$ of length $|W^{(i)}|_a$ which erases all historical sectors in parallel.  As this results in an input configuration, the statement follows by applying \Cref{inputs accepted} and concatenating the computations.

%
%
%

\end{proof}

\begin{lemma}[Compare with Lemma 3.14 of \cite{GW}] \label{E generalized time}

For any accepted configuration $W$ of $\textbf{E}_{\textbf{S},k}^0$, there exists an accepting computation of length at most $c_0\TM_\textbf{S}(2c_0\|W\|)+c_0|W|_a+2c_0$.  

\end{lemma}

\begin{proof}

Let $\pazocal{C}:W\equiv W_0\to\dots\to W_t$ be an accepting computation whose step history is of minimal length and let $\pazocal{C}^{(i)}:W_0^{(i)}\to\dots\to W_t^{(i)}$ be its restriction to the subword $P_iQ_{i,\ell}Q_{i,r}R_i$ of the standard base.  By Lemmas \ref{E primitive step history}, \ref{E run step history}, and \ref{E standard accepted (3)}, the step history of $\pazocal{C}$ is then a suffix of $$(3)(32)(2)(21)(1)(12)(2)(23)(3)(34)(4)(45)(5)$$
Let $W_y\to\dots\to W_t$ be the maximal subcomputation whose step history does not contain the letter $(1)$ or $(12)$.  By \Cref{E standard no (1)}, it follows that $t-y\leq3c_0\max(|W_y^{(i)}|_a,|W_t^{(i)}|_a)+2c_0$.  But $W_t$ is the accept configuration, so that $|W_t|_a=0$, and hence $t-y\leq3c_0|W_y^{(i)}|_a+2c_0$.  

Hence, taking $N$ sufficiently large implies $t-y\leq c_0|W_y|_a+2c_0$, and so as a result it suffices to assume $y>0$.

Now, let $W_x\to\dots\to W_{y-1}$ be the maximal subcomputation with step history $(1)$.  By \Cref{E standard (1)}, we may assume without loss of generality that $t-x\leq(2k+5)\TM_\textbf{S}(|W_x|_{wk})+|W_x^{(i)}|_a+2k+4$.  Hence, by the parameter assignment $c_0>>k$ it suffices to assume $x>0$.

As a result, the step history of the initial subcomputation $W_0\to\dots\to W_{x-1}$ is a suffix of $(3)(2)$.  But then since $W_{x-1}$ is $\sigma(21)$-admissible, applying \Cref{E time (12)} to the inverse computation $W_{x-1}^{(i)}\to\dots\to W_0^{(i)}$ implies $|W_x^{(i)}|_a\leq|W_0^{(i)}|_a$ and $x-1\leq c_0|W_0^{(i)}|_a+c_0$ for all $i\in\{0,\dots,s\}$.

By \Cref{simplify rules}, it follows that $|W_x|_{wk}\leq|W_0|_{wk}+2sx\leq|W_0|_{wk}+2c_0\sum (|W_0^{(i)}|_a+2)\leq2c_0\|W\|$.  The statement thus follows from the parameter assignment $c_0>>k$.

\end{proof}

The next statements have no analogue in \cite{GW} and begin to speak to the purpose of the generalization from enhanced machines to $k$-enhanced machines.

\begin{lemma} \label{start to end mostly (4)}

If $\pazocal{C}:W_0\to\dots\to W_t$ is a reduced computation of $\textbf{E}_{\textbf{S},k}^0$ in the standard base with step history $(12)(2)(23)(3)(34)(4)(45)$, then:

\begin{enumerate}[label=(\alph*)]

\item There exists a word $H\in F(\Phi^+)$ such that $W_0$ has the copy of $H$ over the corresponding left historical alphabet written in each $Q_{i,\ell}Q_{i,r}$-sector

\item $t=(2k+3)\|H\|+(2k+4)$

\item The length of the subcomputation of $\pazocal{C}$ with step history $(34)(4)(45)$ is $2k\|H\|+2k+1$.

\item $W_t$ has the copy of $H$ over the corresponding right alphabet written in each $Q_{i,\ell}Q_{i,r}$-sector.

\end{enumerate}

\end{lemma}

\begin{proof}

By \Cref{E standard accepted (4)} (or \Cref{E standard accepted (3)}), each $W_i$ is an accepted configuration.  So since each rule acts in parallel on each $Q_{i,\ell}Q_{i,r}$-sector, $W_0$ has a copy of the same word written in each such sector.  Hence, (a) follows by noting that $W_0$ is $\sigma(12)$-admissible.

Now, decompose $\pazocal{C}$ into the subcomputations $\pazocal{C}_2:W_1\to\dots\to W_x$, $\pazocal{C}_3:W_{x+1}\to\dots\to W_y$, and $\pazocal{C}_4:W_{y+1}\to\dots\to W_{t-1}$ such that $\pazocal{C}_j$ is a maximal subcomputation with step history $(j)$.  

As a result:

\begin{itemize}

\item \Cref{primitive computations}(3) implies $x-1=2\|H\|+1$,

\item \Cref{one alphabet historical words} implies $y-x-1=\|H\|$, and

\item \Cref{LR_k analogue} implies $t-y-2=2k\|H\|+2k-1$.

\end{itemize}

(b) and (c) thus follow immediately, while (d) follows from applications of Lemmas \ref{primitive computations}(3), \ref{one alphabet historical words}, and \ref{LR_k analogue}.

\end{proof}

\begin{lemma} \label{(1) to (5) mostly (4)}

Let $\pazocal{C}:W_0\to\dots\to W_t$ be a reduced computation of $\textbf{E}_{\textbf{S},k}^0$ in the standard base.  Suppose the step history of $\pazocal{C}$ has first letter $(12)$ or $(54)$, and last letter $(21)$ or $(45)$.  Then:

\begin{enumerate}[label=(\alph*)]

\item $W_0$ is an accepted configuration.

\item Letting $W_i'$ be the admissible subword of $W_i$ with base $Q_{i,\ell}Q_{i,r}$, $t\geq2k\max(|W_0'|_a,|W_t'|_a)$.

\item The sum of the lengths of the subcomputations of $\pazocal{C}$ with step history $(34)(4)(45)$ or $(54)(4)(43)$ is at least $(1-\frac{5}{2k})t$.

\end{enumerate}

\end{lemma}

\begin{proof}

In any case, Lemmas \ref{E primitive step history} and \ref{E run step history} imply that the step history of $\pazocal{C}$ contains a subword of the form $(34)(4)(45)$ or $(54)(4)(43)$.  Hence, (a) follows by \Cref{E standard accepted (4)}.

Moreover, these lemmas imply there exists a factorization of $\pazocal{C}$ given by $\pazocal{C}_1\pazocal{D}_1\dots\pazocal{C}_{m-1}\pazocal{D}_{m-1}\pazocal{C}_m$ with $m\geq1$ such that:

\begin{itemize}

\item Each subcomputation $\pazocal{C}_j$ has step history $(12)(2)(23)(3)(34)(4)(45)$ or has step history $(54)(4)(43)(3)(32)(2)(21)$.

\item Each subcomputation $\pazocal{D}_i$ has step history $(1)$ or $(5)$.

\end{itemize}

For all $1\leq j\leq m$, fix $0\leq x_j<y_j\leq t$ such that $\pazocal{C}_j:W_{x_j}\to\dots\to W_{y_j}$.  We may then apply \Cref{start to end mostly (4)} to $\pazocal{C}_j$ (or its inverse), meaning there exists $H_j\in F(\Phi^+)$ such that $W_{x_j}$ and $W_{y_j}$ both have copies of $H_j$ written in each $Q_{i,\ell}Q_{i,r}$-sector.  \Cref{start to end mostly (4)}(b) then implies $y_j-x_j=(2k+3)\|H_j\|+(2k+4)$, while \Cref{start to end mostly (4)}(c) says that each $\pazocal{C}_j$ has a subcomputation $\pazocal{C}_j'$ with step history $(34)(4)(45)$ or $(54)(4)(43)$ of length $2k\|H_j\|+2k+1$.  

What's more, for $1\leq j\leq m-1$, the restriction of $\pazocal{D}_j$ to any $Q_{i,\ell}Q_{i,r}$-sector satisfies the hypotheses of \Cref{multiply one letter}, and so has length at most $\max(\|H_j\|,\|H_{j+1}\|)$.

Hence, $\displaystyle t\leq\sum_{j=1}^m((2k+3)\|H_j\|+2k+4)+\sum_{j=1}^{m-1}\max(\|H_j\|,\|H_{j+1}\|)\leq(2k+5)\sum_{j=1}^m(\|H_j\|+1)$ while the sum $\ell$ of the lengths of the subcomputations of $\pazocal{C}$ with step history $(34)(4)(45)$ or $(54)(4)(43)$ is $\displaystyle\sum_{j=1}^m(2k\|H_j\|+2k+1)\geq2k\sum_{j=1}^m(\|H_j\|+1)$.  Thus, $\ell\geq\frac{2k}{2k+5}t\geq(1-\frac{5}{2k})t$.

\end{proof}

\begin{lemma} \label{long history mostly (4)}

Let $\pazocal{C}:W_0\to\dots\to W_t$ be a reduced computation of $\textbf{E}_{\textbf{S},k}^0$ in the standard base.  Suppose $t>c_1\max(\|W_0\|,\|W_t\|)$.  Then:

\begin{enumerate}[label=(\alph*)]

\item $W_0$ is an accepted configuration.

\item The sum of the lengths of the subcomputations of $\pazocal{C}$ with step history $(34)(4)(45)$ or $(54)(4)(43)$ is at least $\left(1-\frac{4}{k}\right)t$.

\end{enumerate}

\end{lemma}

\begin{proof}

By the parameter choice $c_1>>c_0$, \Cref{E standard no (1)}, and \Cref{E standard no (5)}, the step history of $\pazocal{C}$ must contain both:

\begin{itemize}

\item a letter $(1)$, $(12)$, or $(21)$; and

\item a letter $(5)$, $(54)$, or $(45)$.

\end{itemize}

Lemmas \ref{E primitive step history} and \ref{E run step history} then imply the existence of a maximal subcomputation $\pazocal{C}':W_r\to\dots\to W_s$ satisfying the hypotheses of \Cref{(1) to (5) mostly (4)}.  Hence, $W_r$ is an accepted configuration, implying (a).

Now, the step history of the subcomputation $\pazocal{D}:W_s\to\dots\to W_t$ must be a prefix of either:

\begin{itemize}

\item $(1)(12)(2)(23)(3)(34)(4)$, or

\item $(5)(54)(4)(43)(3)(32)(2)$.

\end{itemize}

Let $\pazocal{E}:W_s\to\dots\to W_x$ be the maximal subcomputation with step history $(1)$ or $(5)$ and let $\pazocal{E}':W_s'\to\dots\to W_x'$ be its restriction to a $Q_{i,\ell}Q_{i,r}$-sector.   \Cref{(1) to (5) mostly (4)}(b) implies $s-r\geq2k|W_s'|_a$.  Further, $\pazocal{E}'$ satisfies the hypotheses of \Cref{multiply one letter}, so that $x-s\leq\max(|W_s'|_a,|W_x'|_a)$.  

Suppose $t>x$.  Then for the subcomputation $\pazocal{F}:W_x\to\dots\to W_t$ following $\pazocal{E}$, for any $i$ the restriction $\pazocal{F}^{(i)}:W_x^{(i)}\to\dots\to W_t^{(i)}$ satisfies the hypotheses of \Cref{E time (12)}.  In particular, $|W_x'|_a=|W_x^{(i)}|_a\leq|W_t^{(i)}|_a$ and $t-x\leq c_0|W_t^{(i)}|_a+c_0$.

The parameter choice $c_0>>N$ then implies $t-s\leq\max(|W_s'|_a,|W_t|_a)+c_0\|W_t\|$.  But $\|W_t\|<t/c_1$ and $|W_s'|_a\leq\frac{1}{2k}(s-r)\leq t/2k$, so that the parameter choices $c_1>>c_0>>k$ imply $t-s\leq3t/4k$.


A symmetric argument implies $r\leq 3t/4k$, so that $s-r\geq(1-\frac{3}{2k})t$.

Thus, \Cref{(1) to (5) mostly (4)}(c) implies the sum of the lengths of the subcomputations of $\pazocal{C}$ with step history $(34)(4)(45)$ or $(54)(4)(43)$ is at least $(1-\frac{5}{2k})(s-r)\geq(1-\frac{5}{2k})(1-\frac{3}{2k})t\geq(1-\frac{4}{k})t$.

\end{proof}

\medskip


\subsection{Mirror copies} \

The next step in our construction is to concatenate $\textbf{E}_{\textbf{S},k}^0$ with a `mirror copy' of itself, introducing a certain level of symmetry into the computational model.  This is performed in much the same way as in \cite{OS20}, \cite{WCubic}, and \cite{WMal} (and to some extent \cite{WEmb}).

Let $B_0'\equiv P_0'Q_{0,\ell}'Q_{0,r}'R_0'\dots P_s'Q_{s,\ell}'Q_{s,r}'R_s'$ be a copy of the standard base $B_0$ of $\textbf{E}_{\textbf{S},k}^0$.  The \textit{mirror $k$-enhanced machine} $\textbf{E}_{\textbf{S},k}^1$ then has standard base $B_{std}\equiv B_0(B_0')^{-1}$.  Note that, in accordance with the assignment of the previous sections the length of the standard base $B_{std}$ of $\textbf{E}_{\textbf{S},k}^1$ is taken to be the parameter $N-1$.

The tape alphabets of the sectors formed by letters of $B_0$ are identified with those of $\textbf{E}_{\textbf{S},k}^0$, those of the sectors formed by letters of $(B_0')^{-1}$ are taken to be copies of their analogues, and that of the $R_s(R_s')^{-1}$-sector is empty.

The start and end letters of each part correspond to those of $\textbf{E}_{\textbf{S},k}^0$.  On the other hand, the number of input sectors is doubled, with any `mirror' of an input sector taken to be another input sector.

The software of $\textbf{E}_{\textbf{S},k}^1$ is in one-to-one correspondence with that of $\textbf{E}_{\textbf{S},k}^0$, with each rule operating on admissible subwords with base a reduced subword of $B_0$ or of $B_0'$ in the same way as its analogue (and obviously locking the $R_s(R_s')^{-1}$-sector).  As such, a positive rule operates in the standard base in symmetry on the `mirror copies' of the standard base of $\textbf{E}_{\textbf{S},k}^0$.

For example, if a rule of $\textbf{E}_{\textbf{S},k}^0$ multiplies the $Q_{i,r}R_i$-sector on the left by the tape letter $a$, then its analogue in $\textbf{E}_{\textbf{S},k}^1$ does the same to the $Q_{i,r}R_i$-sector, but also multiplies the $(R_i')^{-1}(Q_{i,r}')^{-1}$ sector on the right by $(a')^{-1}$, where $a'$ is the copy of $a$ in the corresponding tape alphabet.

By construction, we may view the standard $k$-enhanced machine as the composition of the five submachines $\textbf{E}_{\textbf{S},k}^1(1),\dots,\textbf{E}_{\textbf{S},k}^1(5)$ through the transition rules $\sigma(i,i+1)$.  

For any configuration $W$ of $\textbf{E}_{\textbf{S},k}^1$, there exist configurations $W_1$ and $W_2$ of $\textbf{E}_{\textbf{S},k}^0$ such that $W$ is the concatenation of $W_1$ and the copy of $W_2^{-1}$ with base $(B_0')^{-1}$.  In this case, we write $W=\mathscr{C}_1(W_1,W_2)$.

If $W$ is accepted, then the symmetry of the accept configuration implies $W_1\equiv W_2$.  Hence, the next statement follows immediately.

\begin{lemma} \label{E1 language}

A configuration $W$ of $\textbf{E}_{\textbf{S},k}^1$ is accepted if and only if there exists an accepted configuration $W_0$ of $\textbf{E}_{\textbf{S},k}^0$ such that $W\equiv\mathscr{C}_1(W_0,W_0)$.  In this case, any computation of $\textbf{E}_{\textbf{S},k}^0$ accepting $W_0$ corresponds to a computation of $\textbf{E}_{\textbf{S},k}^1$ of the same length accepting $W$.

\end{lemma}

The following analogue of \Cref{E generalized time} thus follows immediately.

\begin{lemma} \label{E1 generalized time}

For any accepted configuration $W$ of $\textbf{E}_{\textbf{S},k}^1$, there exists an accepting computation of length at most $c_0\TM_\textbf{S}(c_0\|W\|)+c_0|W|_a+2c_0$.

\end{lemma}

The other statements of the previous section also have natural analogues for $\textbf{E}_{\textbf{S},k}^1$.  However, the next statement exhibits the main reason for the introduction of this mirror symmetry.

\begin{lemma} \label{standard one-step}

Let $\pazocal{C}:W_0\to\dots\to W_t$ be a reduced computation of $\textbf{E}_{\textbf{S},k}^1(j)$ in the standard base. Suppose there exists $z\in\{1,\dots,t\}$ such that $|W_z|_a>3|W_0|_a$. Then there exist two-letter subwords $U_\ell V_\ell$ and $U_rV_r$ of the standard base such that:

\begin{enumerate}

\item Letting $\pazocal{C}_\ell:W_{0,\ell}\to\dots\to W_{t,\ell}$ and $\pazocal{C}_r:W_{0,r}\to\dots\to W_{t,r}$ be the restrictions of $\pazocal{C}$ to the $U_\ell V_\ell$- and $U_rV_r$-sectors, respectively, $|W_{z,\ell}|_a<\dots<|W_{t,\ell}|_a$ and $|W_{z,r}|_a<\dots<|W_{t,r}|_a$.

\item Every rule of the subcomputation $W_z\to\dots\to W_t$ multiplies the $U_\ell V_\ell$-sector by one letter on the left and the $U_rV_r$-sector by one letter on the right.

\end{enumerate}

\end{lemma}

\begin{proof}

First, suppose $|W_0|_a\geq|W_1|_a$.  Then we must of course have $z>1$, and so $|W_z|_a>3|W_1|_a$.  So, in this case the shorter subcomputation $W_1\to\dots\to W_t$ satisfies the hypotheses of the statement.  Hence, as the conclusion applies only to a `tail' of the computation, the statement follows by inducting on the length of the computation.

Thus, we may assume without loss of generality that $|W_0|_a<|W_1|_a$.  

We now proceed in five cases based on the value of $j$.

\textbf{1.} Suppose $j=1$.

For each $x\in\{0,\dots,s\}$ and $0\leq i\leq t$, let $W_{i,x}$ and $W_{i,x}'$ be the admissible subwords with base $Q_{x,\ell}Q_{x,r}$ and $(Q_{x,r}')^{-1}(Q_{x,\ell}')^{-1}$, respectively.  Then, setting $\a_x=|W_{1,x}|_a-|W_{0,x}|_a$ and $\a_x'=|W_{1,x}'|_a-|W_{0,x}'|_a$, by construction we have:

\begin{itemize}

\item $\a_x,\a_x'\in\{\pm1\}$ for all $x$

\item $\displaystyle|W_1|_a-|W_0|_a=\sum_{x=0}^s\a_x+\a_x'$ 

\end{itemize}

Hence, $|W_1|_a>|W_0|_a$ implies there exists $x,y\in\{0,\dots,s\}$ such that $\a_x=\a_y'=1$.  But the restriction of $\pazocal{C}$ to the $Q_{x,\ell}Q_{x,r}$- and $(Q_{y,r}')^{-1}(Q_{y,\ell}')^{-1}$-sectors satisfy the hypotheses of \Cref{multiply one letter}, so that the statement is satisfied for $U_\ell V_\ell\equiv Q_{x,\ell}Q_{x,r}$ and $U_rV_r\equiv (Q_{y,r}')^{-1}(Q_{y,\ell}')^{-1}$.

%
%
%
%
%
%
%

\textbf{2.} Suppose $j=2$.

This time, let $V_{i,x}$ be the admissible subword of $W_i$ with base $P_xQ_{x,\ell}Q_{x,r}$ and $V_{i,x}'$ that with base $(Q_{x,r}')^{-1}(Q_{x,\ell}')^{-1}(P_x')^{-1}$.  Then for $\b_x=|V_{1,x}|_a-|V_{0,x}|_a$ and $\b_x'=|V_{1,x}'|_a-|V_{0,x}'|_a$, we have:

\begin{itemize}

\item $\b_x,\b_x'\in\{-2,0,2\}$ for all $x$

\item $\displaystyle |W_1|_a-|W_0|_a=\sum_{x=0}^s\b_x+\b_x'$.

\end{itemize}

Hence, $|W_1|_a>|W_0|_a$ implies there exists $x$ such that $\b_x=2$ or $\b_x'=2$.  Assuming $\b_x=2$, then \Cref{primitive computations} implies the statement for $U_\ell V_\ell\equiv Q_{x,\ell}Q_{x,r}$ and $U_rV_r\equiv P_xQ_{x,\ell}$; otherwise, \Cref{primitive computations} again implies the statement for $U_\ell V_\ell\equiv (Q_{x,\ell}')^{-1}(P_x')^{-1}$ and $U_rV_r\equiv (Q_{x,r}')^{-1}(Q_{x,\ell}')^{-1}$.

%
%

\textbf{3.} Suppose $j=3$.

For $UV\equiv Q_{i,\ell}Q_{i,r}$ or $UV\equiv(Q_{i,r}')^{-1}(Q_{i,\ell}')^{-1}$, let $\pazocal{C}_{UV}:W_{0,UV}\to\dots\to W_{t,UV}$ be the restriction of $\pazocal{C}$ to the $UV$-sector.  Then $\pazocal{C}_{UV}$ satisfies the hypotheses of \Cref{multiply two letters}, so that $z\leq\max(|W_{0,UV}|_a,|W_{z,UV}|_a)$.

Suppose $|W_{0,UV}|_a\geq|W_{z,UV}|_a$ for all $UV$.  Then $|W_0|_a\geq2sz$.  Further, note that in any of the working sectors, a transition alters the $a$-length by at most $2$.  So, since there are $2s$ working sectors and all sectors not already mentioned are locked, $|W_z|_a-|W_0|_a\leq4sz$.  But then $|W_z|_a\leq3|W_0|_a$, contradicting the hypothesis.

Hence, there exists $UV$ such that $|W_{z,UV}|_a>|W_{0,UV}|_a$, and so there exists $0\leq i\leq z$ such that $|W_{i,UV}|_a>|W_{i-1,UV}|_a$.  But then \Cref{multiply two letters} implies the statement for $U_\ell V_\ell=UV=U_rV_r$.

\textbf{4.} Suppose $j=4$.

The argument follows in much the same way as for $j=2$, considering the restrictions of $\pazocal{C}$ to the subwords $Q_{x,\ell}Q_{x,r}R_x$ and the mirror copies.

\textbf{5.} Suppose $j=5$.

The argument then follows in much the same way as for $j=1$, but this time by taking $U_\ell V_\ell\equiv (Q_{y,r}')^{-1}(Q_{y,\ell}')^{-1}$ and $U_rV_r\equiv Q_{x,\ell}Q_{x,r}$.

\end{proof}

\medskip


\subsection{The $k$-enhanced machine} \

The penultimate machine in our construction, the \textit{$k$-enhanced machine} $\textbf{E}_{\textbf{S},k}$, is the `circular analogue' of (a small tweak to) the mirror $k$-enhanced machine $\textbf{E}_{\textbf{S},k}^1$.  The definition of this machine is given in much the same way as in \cite{GW}, \cite{O18}, \cite{WEmb}, \cite{W}, and many others.

In particular, the standard base of $\textbf{E}_{\textbf{S},k}$ is $\{t\}B_{std}$, where each part comprising $B_{std}$ is identical to its counterpart in $\textbf{E}_{\textbf{S},k}^1$ and the part $\{t\}$ is a singleton.  Note that the length of this base is the parameter $N$.

The tape alphabet of the $\{t\}P_0$-sector is empty, while those of all other sectors are carried over from the hardware of $\textbf{E}_{\textbf{S},k}^1$.
The novelty of this machine is that it has an additional tape alphabet for the sector $(P_0')^{-1}\{t\}$-sector, representing the space after $B_{std}$, making $$(R_s')^{-1}(Q_{s,r}')^{-1}(Q_{s,\ell}')^{-1}(P_s')^{-1}\{t\}P_0Q_{0,\ell}Q_{0,\ell}^{-1}P_0^{-1}\{t\}^{-1}P_s'Q_{s,\ell}'$$ a legitimate base for an admissible word.
This sort of $S$-machine is called a \textit{cyclic machine}. 

In this setting (and analogous to most other settings in previous literature), the tape alphabet assigned to the $(P_s')^{-1}\{t\}$-sector is empty.

The positive rules of $\textbf{E}_{\textbf{S},k}$ correspond to those of $\textbf{E}_{\textbf{S},k}^1$, operating on the copy the hardware of $\textbf{E}_{\textbf{S},k}^1$ in the same way and, naturally, locking the new sector.  The input sectors are also assigned in just the same way as for $\textbf{E}_{\textbf{S},k}^1$. The corresponding definitions (for example, submachines, historical sectors, working sectors, {\frenchspacing etc.) then} extend in the obvious way, as do all statements pertaining to the machine.

Now, fix an arbitrary cyclic $S$-machine $\textbf{M}$. The base of an admissible word in the hardware of $\textbf{M}$ is said to be \textit{circular} if it starts and ends with the same base letter (and the base isn't a single letter).   Specifically, an unreduced circular base is called \textit{defective}.  As in previous literature, a circular base is said to be \textit{revolving} if none of its proper subwords is circular, while a revolving defective base is called \textit{faulty}.

Suppose $W$ is an admissible word in the hardware of $\textbf{M}$ whose base $B\equiv xvx$ is circular.  Let $v\equiv v_1yv_2$ where $y$ is a single letter and $v_1,v_2$ are (perhaps empty) base words.  Then the circular base word $B'\equiv yv_2xv_1y$ is called a \textit{cyclic permutation} of $B$.

As long as the first and last state letters of $W$ are identical, we can then form from $W$ a corresponding admissible word $W'$ with base $B'$ by:

\begin{itemize}

\item Cyclically permuting $W$ so that its state letter from $y$ is the first letter.

\item Merging the two identical state letters of $x$ into one letter.

\item Adding an identical copy of the state letter of $y$ to the end of the word.

\end{itemize}

Through an abuse of notation we also call the admissible word $W'$ a \textit{cyclic permutation} of $W$ in this case.

Note that if $U$ is an admissible word with circular base which is $\theta$-admissible for some $S$-rule $\theta$, then by the definition of $S$-rules $U$ must have identical first and last letters.  As such, given a reduced computation $\pazocal{C}:W_0\to\dots\to W_t$ with circular base $B$, for any cyclic permutation $B'$ of $B$ we may form the cyclic permutation $W_i'$ of $W_i$ with base $B'$.  It is then a consequence of the definition that there exists a reduced computation $\pazocal{C}':W_0'\to\dots\to W_t'$ with the same history as $\pazocal{C}$.  Through another abuse of notation, we also call $\pazocal{C}'$ a \textit{cyclic permutation} of $\pazocal{C}$.



As in \cite{CW} and \cite{GW}, we define the \textit{universal reach relation} of $\textbf{M}$ to be the binary relation $\REACH^{uni}_\textbf{M}$ on the set of admissible words of $\textbf{M}$ with circular base given by $(W_1,W_2)\in \REACH^{uni}_\textbf{M}$ if and only if there exists a reduced computation between $W_1$ and $W_2$.

Given $(W_1,W_2)\in\REACH^{uni}_\textbf{M}$, a reduced computation between $W_1$ and $W_2$ is said to be \textit{minimal} if the length of its history is minimal amongst all computations realizing the relation.

\medskip


\subsection{Controlled history and reduced circular bases} \

We restrict our attention in this section to elements $(W_1,W_2)\in\REACH^{uni}_{\textbf{E}_{\textbf{S},k}}$ such that $W_1$ and $W_2$ have reduced circular bases.  Unlike in \cite{GW}, though, we do not bound the length of a minimal computation between $W_1$ and $W_2$, but rather analyze the consequences of the existence of a reduced computation between $W_1$ and $W_2$ whose history is long.  

To aid with this, we first adapt the following definition from previous literature {\frenchspacing (e.g. \cite{O18}, \cite{OS20}, \cite{WEmb}, and \cite{W})}.

\begin{definition}

The history $H$ of a reduced computation $\pazocal{C}$ of $\textbf{E}_{\textbf{S},k}$ is called \textit{controlled} if it (or its inverse) can be factored $H\equiv\zeta_{i-1}H'\zeta_i$ such that:

\begin{itemize}

\item $H'$ is the history of a subcomputation with step history $(4)$ whose history has no occurrences of transition or connecting rules

\item $i=1,\dots,2k$ so that $\zeta_{i-1}$ and $\zeta_i$ are the copies of consecutive connecting rules of $\textbf{LR}_k$, with $\zeta_0$ and $\zeta_{2k}$ taken to be $\sigma(34)$ and $\sigma(45)$, respectively.

\end{itemize}

\end{definition}

As in \Cref{sec-standard-computations}, given a configuration $W$ of $\textbf{E}_{\textbf{S},k}$ and $0\leq j\leq s$, let $W^{(j)}$ be the admissible subword with base $P_jQ_{j,\ell}Q_{j,r}R_j$.  Extend this to let $\bar{W}^{(j)}$ be the admissible subword with base $(R_j')^{-1}(Q_{j,r}')^{-1}(Q_{j,\ell}')^{-1}(P_j')^{-1}$.

The next statement then follows in just the same way as its analogue in \cite{W}, applying Lemmas \ref{locked sectors}, \ref{multiply one letter}, \ref{primitive computations}, \ref{E standard accepted (4)}, and \ref{E1 language}.

\begin{lemma}[Compare with Lemma 4.18 of \cite{W}] \label{enhanced controlled}

Let $\pazocal{C}:W_0\to\dots\to W_t$ be a reduced computation of $\textbf{E}_{\textbf{S},k}$ with controlled history $H$.  Then the base $B$ of the computation is reduced and the computation is uniquely determined by $H$ and $B$.  Moreover, if $\pazocal{C}$ is a computation in the standard base, then $|W_i|_a=|W_0|_a$ for all $0\leq i\leq t$, $\|H\|=|W_0^{(j)}|_a+2=|\bar{W}_0^{(j)}|_a+2$ for all $0\leq j\leq s$, and $W_0$ is accepted by $\textbf{E}_{\textbf{S},k}$.

\end{lemma}

\begin{lemma} \label{long history controlled}

Let $\pazocal{C}:W_0\to\dots\to W_t$ be a reduced computation of $\textbf{E}_{\textbf{S},k}$ in the standard base with a fixed subcomputation $\pazocal{D}:W_x\to\dots\to W_y$ of $\pazocal{C}$.  Suppose $y-x\geq\frac{1}{4}\lambda t$ and $t>c_1\max(\|W_0\|,\|W_t\|)$.  Then for any factorization $H\equiv H'H''H'''$ of the history $H$ of $\pazocal{D}$ such that $\|H'\|+\|H'''\|\leq\lambda(y-x)$, $H''$ contains a controlled subword.

\end{lemma}

\begin{proof}

Note that the restriction $\pazocal{C}_0:V_0\to\dots\to V_t$ of $\pazocal{C}$ to the base $B_0$ satisfies the hypotheses of \Cref{long history mostly (4)}.

By construction, the subcomputation $\pazocal{D}'':W_r\to\dots\to W_s$ of $\pazocal{D}$ with history $H''$ has length $\|H''\|=s-r\geq(1-\lambda)(y-x)\geq\frac{1}{4}\lambda(1-\lambda)t$.  Taking $\lambda^{-1}\geq5$ then implies $s-r\geq\lambda^2t$.  The parameter assignment $k>>\lambda^{-1}$ then allows us to assume $s-r\geq\frac{6}{k}t$.

\Cref{long history mostly (4)} then implies there exists a subcomputation $\pazocal{E}$ of $\pazocal{C}$ with history $H_0$ satisfying:

\begin{itemize}

\item The step history of $\pazocal{E}$ is $(34)(4)(45)$ or $(54)(4)(43)$

\item There exists a subword $H_0''$ of $H_0$ shared with $H''$ such that $\|H_0''\|\geq\frac{1}{k}\|H_0\|$.

\end{itemize}

By \Cref{LR_k analogue}, $\|H_0\|=2kl+2k+1$ where $l$ is the $a$-length of the admissible subword with base $P_iQ_{i,\ell}Q_{i,r}R_i$ of any admissible word comprising the defining sequence of $\pazocal{E}$ (indeed, this is the same if we instead consider the admissible subword with base $Q_{i,\ell}Q_{i,r}R_i$ or a mirror copy of one of these).

\Cref{primitive computations} also implies $H_0$ consists of the concatenation along connecting rules of $2k$ controlled histories of length $l+2$.  So, any subword of $H_0$ of length $2l+2$ contains a controlled subword.  

But then $\|H_0''\|\geq\frac{1}{k}(2kl+2k+1)\geq2l+2$, implying the statement.

\end{proof}

\begin{lemma} \label{reduced revolving controlled}

Let $\pazocal{C}:W_0\to\dots\to W_t$ be a reduced computation of $\textbf{E}_{\textbf{S},k}$ whose base $B$ is circular and contains a reduced revolving subword $B'$.  Then letting $\pazocal{C}':W_0'\to\dots\to W_t'$ be the restriction of $\pazocal{C}$ to the base $B'$, either $t\leq c_1\max(\|W_0'\|,\|W_t'\|)$ or:

\begin{enumerate}[label=(\alph*)]

\item The history $H$ of $\pazocal{C}$ contains a controlled subword

\item $B$ or its inverse is a cyclic permutation of $(\{t\}B_{std})^\ell\{t\}$ for some $\ell\geq1$

\item $W_0^{\pm1}$ is a cyclic permuation of $V^\ell t$ for some accepted configuration $V$ of $\textbf{E}_{\textbf{S},k}$.

\end{enumerate}

\end{lemma}

\begin{proof}

As the $(P_s')^{-1}\{t\}$-sector has empty tape alphabet, there exists a reduced computation $\pazocal{C}'':W_0''\to\dots\to W_t''$ in the standard base such that $|W_i'|_a=|W_i''|_a$ for all $i$.

If $t>c_1\max(\|W_0'\|,\|W_t'\|)$, then $t>c_1\max(\|W_0''\|,\|W_t''\|)$, so that \Cref{long history controlled} implies (a).  Conditions (b) and (c) then follow from \Cref{enhanced controlled}.

\end{proof}

\medskip


\subsection{Defective bases} \label{sec-defective} \

In this section, we consider pairs $(W_1,W_2)\in\REACH_{\textbf{E}_{\textbf{S},k}}^{uni}$ such that each $W_i$ has defective base $B$, with the goal to bound the length of a minimal computation between $W_1$ and $W_2$.  In light of \Cref{reduced revolving controlled}, it suffices to restrict our attention to the case where $B$ has no reduced revolving subword.
%
As a first step toward this, we recall the following statement from \cite{GW}, which is given in just the same way.

\begin{lemma}[Lemma 3.19 of \cite{GW}] \label{Defective (34)(4)(45)}

Let $\pazocal{C}:W_0\to\dots\to W_t$ be a reduced computation of $\textbf{E}_{\textbf{S},k}$ with base $B$.  If the step history of $\pazocal{C}$ is $(34)(4)(45)$ or $(54)(4)(43)$, then $B$ must be reduced.

\end{lemma}
%
%
%

%
%
%
%
%
%
%
%

\begin{lemma}[Compare with Lemma 3.21 of \cite{GW}] \label{Defective PQQR}

Let $\pazocal{C}:W_0\to\dots\to W_t$ be a reduced computation of $\textbf{E}_{\textbf{S},k}$ with defective base $B$.  Suppose $B$ contains a subword $B'$ of the form $(P_iQ_{i,\ell}Q_{i,r}R_i)^{\pm1}$ or a mirror copy of such a word.  Then letting $\pazocal{C}':W_0'\to\dots\to W_t'$ be the restriction of $\pazocal{C}$ to $B'$, we have $t\leq3c_0\max(|W_0'|_a,|W_t'|_a)+2c_0$.

\end{lemma}

\begin{proof}

Noting that every rule of $\textbf{E}_{\textbf{S},k}$ operates on the mirror copy of $B_0$ in the symmetric manner, it suffices to assume $B'\equiv P_iQ_{i,\ell}Q_{i,r}R_i$ for some $i$.

Lemmas \ref{E primitive step history}, \ref{E run step history}, and \ref{Defective (34)(4)(45)} imply that the step history of $\pazocal{C}$ either:

\begin{itemize}

\item does not have any occurrence of the letters $(1)$, $(12)$, or $(21)$; or 

\item does not have any occurrence of the letters $(5)$, $(54)$, and $(45)$.  

\end{itemize}

The statement then follows by \Cref{E standard no (1)} or \Cref{E standard no (5)}.

\end{proof}

As a result of \Cref{Defective PQQR}, it suffices to restrict our attention to defective bases that contain no subword of the form $(P_iQ_{i,\ell}Q_{i,r}R_i)^{\pm1}$ or a mirror copy of such a word.  Indeed, even if a cyclic permutation of a defective base contains such a subword, then \Cref{Defective PQQR} may be applied.  

To this end, a defective base is said to be \textit{strongly defective} if none of its cyclic permutations contain a subword of the form $(P_iQ_{i,\ell}Q_{i,r}R_i)^{\pm1}$ or a mirror copy.

\begin{lemma} \label{strongly defective (3)}

Let $\pazocal{C}:W_0\to\dots\to W_t$ be a minimal computation with strongly defective base $B$ and step history $(3)$.  
Then there exists a two-letter subword $UV$ of $B$ such that:

\begin{enumerate}[label=(\alph*)]

\item There exists a rule in the history of $\pazocal{C}$ which multiplies the $UV$-sector by a letter on the left or on the right.

\item Letting $\pazocal{C}':W_0'\to\dots\to W_t'$ be the restriction of $\pazocal{C}$ to the base $UV$, we have $t\leq 12n^2+2n$ for $n=\max(|W_0'|_a,|W_t'|_a)$.

\end{enumerate}

\end{lemma}

\begin{proof}

Note that since every rule of the submachine $\textbf{E}_{\textbf{S},k}(3)$ locks the $P_iQ_{i,\ell}$- and $Q_{i,r}R_i$-sectors, the strongly defective condition precludes $B$ from having any subword of the form $(Q_{i,\ell}Q_{i,r})^{\pm1}$.  In just the same way, $B$ cannot have any subword of the form $(Q_{i,\ell}'Q_{i,r}')^{\pm1}$.

Now, if $B$ contains a letter of the form $P_i^{\pm1}$, then since every rule locks the $P_iQ_{i,\ell}$-sector it follows that $B$ must have a subword of the form $Q_{i,\ell}Q_{i,\ell}^{-1}$.  Similarly, the presence of a letter of the form $(P_i')^{\pm1}$ necessitates the presence of a subword of the form $Q_{i,\ell}'(Q_{i,\ell}')^{-1}$.  

In the same way, since every rule locks the $Q_{i,r}R_i$-sector, if $B$ has a letter of the form $R_i^{\pm1}$ or $(R_i')^{\pm1}$ then it also has a subword of the form $Q_{i,r}^{-1}Q_{i,r}$ or $(Q_{i,r}')^{-1}Q_{i,r}'$.

Finally, if $B$ has a letter of the form $\{t\}^{\pm1}$, then since every rule of $\textbf{E}_{\textbf{S},k}$ locks the $\{t\}P_0$-sector we see as above that $Q_{0,\ell}Q_{0,\ell}^{-1}$ is a subword of $B$.

Hence, $B$ has a two-letter subword of the form $Q_{i,\ell}Q_{i,\ell}^{-1}$, $Q_{i,r}^{-1}Q_{i,r}$, or a mirror copy of one of these.  Thus, $\pazocal{C}$ satisfies the hypotheses of \Cref{unreduced base quadratic}, implying the statement.

%
%
%
%

\end{proof}

\begin{lemma} \label{strongly defective (34) R}

Let $\pazocal{C}:W_0\to\dots\to W_t$ be a minimal computation with strongly defective base $B$.  Suppose the history of $\pazocal{C}$ contains a letter of the form $\sigma(34)^{\pm1}$ which is neither the first nor the last letter.  Then there exists a cyclic permutation $B'$ of $B$ with a factorization $$B'\equiv B_1'C_1'\dots B_m'C_m'$$ such that:

\begin{enumerate}[label=(\alph*)]

\item Each subword $B_j'$ is of the form $R_i^{-1}Q_{i,r}^{-1}Q_{i,r}R_i$ or a mirror copy.

\item The subwords $C_j'$ do not contain any letters of the form $Q_{i,r}^{\pm1}$, $(Q_{i,r}')^{\pm1}$, $R_i^{\pm1}$, or $(R_i')^{\pm1}$ except perhaps for the last letter of $C_m'$.

\item Any occurrence in $C_j'$ of a letter of the form $Q_{i,\ell}^{\pm1}$ or $(Q_{i,\ell}')^{\pm1}$ is part of a subword of the form $Q_{i,\ell}Q_{i,\ell}^{-1}$ or $Q_{i,\ell}'(Q_{i,\ell}')^{-1}$.

\end{enumerate}


\end{lemma}

\begin{proof}

Perhaps passing to the inverse computation, we may assume the history $H$ of $\pazocal{C}$ can be factored $H'\sigma(34)H''$.  By construction, $H''$ has a non-empty prefix $H_4$ which is the history of a maximal subcomputation $\pazocal{C}_4$ with step history $(4)$.

Note that $\sigma(34)$ locks all sectors of the standard base other than those of the form $Q_{i,\ell}Q_{i,r}$ and $(Q_{i,r}')^{-1}(Q_{i,\ell}')^{-1}$, while the first connecting rule of $\textbf{E}_{\textbf{S},k}(4)$ locks these exceptional sectors.  So, since $B$ is unreduced, \Cref{locked sectors} implies $H$ does not contain a connecting rule.

If $B$ has no letter of the form $Q_{i,r}^{\pm1}$ or $(Q_{i,r}')^{\pm1}$, then no rule of $\textbf{E}_{\textbf{S},k}(4)$ alters the tape words of any admissible word with base $B$.  So, the first and last admissible words of $\pazocal{C}_4$ are identical.  But then $H_4$ may be removed from $H$ to produce a shorter computation between $W_0$ and $W_t$, contradicting the minimality of $\pazocal{C}$.  Hence, $B$ must contain a letter of the form $Q_{i,r}^{\pm1}$ or $(Q_{i,r}')^{\pm1}$.

Now, since $\sigma(34)$ locks all sectors of the form $P_iQ_{i,\ell}$, $Q_{i,r}R_i$, and their mirror copies, $B$ being strongly defective implies it cannot contain a subword of the form $(Q_{i,\ell}Q_{i,r})^{\pm1}$ or a mirror copy.  In particular, the presence of a letter of the form $Q_{i,r}^{\pm1}$ necessitates a subword of a cyclic permutation of $B$ of the form $R_i^{-1}Q_{i,r}^{-1}Q_{i,r}R_i$ containing this letter, while the presence of a letter of the form $(Q_{i,r}')^{\pm1}$ necessitates such a subword which is a mirror copy.  The same is true for any occurrence of a letter of the form $R_i^{\pm1}$ or $(R_i')^{\pm1}$.

Let $B'$ be a cyclic permutation of $B$ such that a subword of one of these forms is a prefix.  Then it follows from the discussion above that $B'$ can be factored in such a way to satisfy (a) and (b).  Analogous arguments may be applied for occurrences of $Q_{i,\ell}^{\pm1}$ or $(Q_{i,\ell}')^{\pm1}$ in $B$, implying (c).

%
%
%

\end{proof}

\begin{lemma} \label{strongly defective (34) L}

Let $\pazocal{C}:W_0\to\dots\to W_t$ be a minimal computation with strongly defective base $B$.  Suppose the history of $\pazocal{C}$ contains a letter of the form $\sigma(34)^{\pm1}$ which is neither the first nor the last letter and $B$ contains a letter of the form $Q_{i,\ell}^{\pm1}$ or $(Q_{i,\ell}')^{\pm1}$.  Then there exists a subword $B''$ of a cyclic permutation of $B$ such that:

\begin{itemize}

\item $\|B''\|=7$

\item Letting $\pazocal{C}'':W_0''\to\dots\to W_t''$ be the restriction to $B''$ of the corresponding cyclic permutation of $\pazocal{C}$, we have $t\leq3\max(|W_0''|_a,|W_t''|_a)$.

\end{itemize}

\end{lemma}

\begin{proof}

Let $B'\equiv B_1'C_1'\dots B_m'C_m'$ be the cyclic permutation of $B$ with factorization given by \Cref{strongly defective (34) R}.  By hypothesis and \Cref{strongly defective (34) R}(c), there exists $j$ such that $C_j'$ has a subword of the form $Q_{i,\ell}Q_{i,\ell}^{-1}$ or a mirror copy.  

Let $UV$ be the first such subword in $C_j'$, {\frenchspacing i.e. given that $U_0UV$ is a prefix of $C_j'$, $U_0$ has no such subword}.  Note then that \Cref{strongly defective (34) R}(c) tells us that $U_0$ contains no letter of the form $Q_{i,\ell}^{\pm1}$ or $(Q_{i,\ell}')^{\pm1}$.  Since $\sigma(34)$ locks the $P_iQ_{i,\ell}$- and $(Q_{i,\ell}')^{-1}(P_i')^{-1}$-sectors, $U_0$ is non-empty and its last letter $T$ is either $P_i$ or $P_i'$.  

Suppose $T$ is preceded in $B'$ by the letter $P_i^{-1}$ or $(P_i')^{-1}$.  That $\sigma(34)$ locks the $P_iQ_{i,\ell}$- and $(Q_{i,\ell}')^{-1}(P_i')^{-1}$-sectors then implies that this letter is preceded by $Q_{i,\ell}^{-1}$ or $(Q_{i,\ell}')^{-1}$.  But then these letters must be part of $C_j'$ and so $U_0$, yielding a contradiction.  Similarly, if $T$ is preceded by the letter $\{t\}^{\pm1}$, then since $\sigma(34)$ locks the sectors of the form $(Q_{0,\ell}')^{-1}(P_0')^{-1}$, $(P_0')^{-1}\{t\}$, $\{t\}P_0$, $P_0Q_{0,\ell}$, and their mirror copies, we can again find an occurrence of $Q_{0,\ell}^{-1}$ or $(Q_{0,\ell}')^{-1}$ in $U_0$.

Hence, the letter preceding $T$ in $B'$ must be of the form $R_i$ or $R_i'$.  But then this is the last letter of $B_j'$, so that $U_0\equiv T$. 

Let $B''\equiv B_j'TUV$.  Note then that $\|B''\|=7$.

Now, as in the proof of \Cref{strongly defective (34) R}, we may assume that the history $H$ of $\pazocal{C}$ can be factored $H'\sigma(34)H''$ and so that $H''$ has a prefix which is the history of a maximal subcomputation $\pazocal{C}_4:W_{y+1}\to\dots\to W_z$ with step history $(4)$.  The restriction $\pazocal{C}_4^{(j)}:W_{y+1}^{(j)}\to\dots\to W_z^{(j)}$ of $\pazocal{C}_4$ to the base $B_j'$ then satisfies the hypotheses of \Cref{primitive unreduced}, so that $z=t$ and $t-y-1\leq|W_t^{(j)}|_a$.

In much the same way, $H'$ has a non-empty suffix which is the history of a maximal subcomputation $\pazocal{C}_3:W_x\to\dots\to W_y$ with step history $(3)$.  Then the restriction $\pazocal{C}_3':W_x'\to\dots\to W_y'$ of $\pazocal{C}_3$ to the $UV$-sector satisfies the hypotheses of \Cref{one alphabet historical words unreduced}, so that $x=0$ and $y\leq|W_0'|_a$.

Thus, $t\leq|W_0''|_a+|W_t''|_a+1$, so that the statement follows from noting $|W_0''|_a,|W_t''|_a\geq1$ since $B''$ is unreduced.

\end{proof}

\begin{lemma} \label{strongly defective (34)}

Let $\pazocal{C}:W_0\to\dots\to W_t$ be a minimal computation with strongly defective base $B$.  Suppose the history of $\pazocal{C}$ contains a letter of the form $\sigma(34)^{\pm1}$ which is neither the first nor the last letter.  Then there exists a subword $B''$ of a cyclic permutation of $B$ such that:

\begin{itemize}

\item $\|B''\|\leq7$

\item Letting $\pazocal{C}'':W_0''\to\dots\to W_t''$ be the restriction to $B''$ of the corresponding cyclic permutation of $\pazocal{C}$, we have $t\leq12n^2+6n$ for $n=\max(|W_0''|_a,|W_t''|_a)$.

\end{itemize}

\end{lemma}

\begin{proof}

By \Cref{strongly defective (34) L} it suffices to assume $B$ has no letter of the form $Q_{i,\ell}^{\pm1}$ or $(Q_{i,\ell}')^{\pm1}$.  Moreover, since $\sigma(34)$ locks the $P_iQ_{i,\ell}$-sectors and their mirror copies, $B$ also has no letter of the form $P_i^{\pm1}$ or $(P_i')^{\pm1}$.

Let $B'\equiv B_1'C_1'\dots B_m'C_m'$ be the cyclic permutation with factorization given by \Cref{strongly defective (34) R}.  As in the proof of that statement, we may assume the history $H$ of $\pazocal{C}$ can be factored $H\equiv H'\sigma(34)H''$.

As in previous proofs, let $H_4$ be the non-empty prefix of $H''$ which is the history of a maximal subcomputation $\pazocal{C}_4:W_{y+1}\to\dots\to W_z$ with step history $(4)$ and $H_3$ be the non-empty suffix of $H'$ which is the history of a maximal subcomputation $\pazocal{C}_3:W_x\to\dots\to W_y$ with step history $(3)$.  \Cref{strongly defective (3)} then implies there exists a two-letter subword $UV$ of $B$ such that for the restriction $\pazocal{C}_3':W_x'\to\dots\to W_y'$ of $\pazocal{C}_3$ to $UV$, there exists some transition of $\pazocal{C}_3'$ which multiplies the tape word in the $UV$-sector on the left or right and $y-x\leq12m^2+2m$ for $m=\max(|W_x'|_a,|W_y'|_a)$.

Recall that the rules of $\textbf{E}_{\textbf{S},k}(3)$ only operate in the working sectors, the $Q_{i,\ell}Q_{i,r}$-sectors, and the $(Q_{i,r}')^{-1}(Q_{i,\ell}')^{-1}$-sectors.  So, since $B$ has no letter of the form $P_i^{\pm1}$, $Q_{i,\ell}^{\pm1}$, or a mirror copy, $UV$ must be of the form $Q_{i,r}^{-1}Q_{i,r}$, $R_iR_i^{-1}$, or a mirror copy of one of these.  But $\sigma(34)$ locks the $R_iP_{i+1}$-sectors, the $R_s(R_s')^{-1}$-sector, and the $(P_{i+1}')^{-1}(R_i')^{-1}$-sectors, so that $B$ cannot contain a subword of the form $R_iR_i^{-1}$ or a mirror copy.

Hence, $UV$ is of the form $Q_{i,r}^{-1}Q_{i,r}$ or a mirror copy.  But then by \Cref{strongly defective (34) R} $UV$ is a part of a subword $B_j'$ of $B'$.  Setting $B''=B_j'$, then the above argument implies $y-x\leq12m^2+2m$ for $m=\max(|W_x''|_a,|W_y''|_a)$.



Let $\pazocal{C}_4'':W_{y+1}''\to\dots\to W_z''$ be the restriction of $\pazocal{C}_4$ to the subword $B''$.  Then $\pazocal{C}_4''$ may be viewed as a computation of a primitive machine satisfying the hypotheses of the analogue of \Cref{primitive unreduced} adapted for $\textbf{LR}_k$.  As a result, $W_z''$ is neither $\sigma(45)$- or $\sigma(43)$-admissible, so that $z=t$.  Further, $|W_{y+1}''|_a\leq|W_t''|_a$ and $t-y-1\leq|W_t''|_a$.

Hence, it suffices to assume $x>0$, as otherwise $t\leq12n^2+3n+1\leq12n^2+4n$ since $B''$ being unreduced implies $n\geq1$.

Now, since $W_y$ is $\sigma(34)$-admissible, its admissible subword $W_y'$ with base $UV$ must have non-empty tape word consisting entirely of letters from the corresponding right historical alphabet.  The inverse computation of $\pazocal{C}_3$ conjugates this word by a word over this same alphabet, and so the admissible subword $W_x'$ of $W_x$ with this base also has non-empty tape word consisting entirely of letters from the right historical alphabet.  

Hence, $W_x$ is not $\sigma(32)$-admissible, {\frenchspacing i.e. the transition $W_{x-1}\to W_x$ is given by $\sigma(43)$}.

But then applying the same arguments as above implies the subcomputation $W_0\to\dots\to W_{x-1}$ has step history $(4)$ with $x-1\leq|W_0''|_a$ and $|W_x''|_a\leq|W_0''|_a$.  Thus, $t\leq12n^2+4n+2\leq12n^2+6n$.

\end{proof}

\begin{lemma} \label{strongly defective (23) L}

Let $\pazocal{C}:W_0\to\dots\to W_t$ be a minimal computation with strongly defective base $B$.  Suppose the history of $\pazocal{C}$ contains a letter of the form $\sigma(23)^{\pm1}$ which is neither the first nor the last letter.  Then there exists a cyclic permutation $B'$ of $B$ with a factorization $$B'\equiv B_1'C_1'\dots B_m'C_m'$$ such that:

\begin{enumerate}[label=(\alph*)]

\item Each subword $B_j'$ is of the form $P_iQ_{i,\ell}Q_{i,\ell}^{-1}P_i^{-1}$ or a mirror copy.

\item The subwords $C_j'$ do not contain any letters of the form $Q_{i,\ell}^{\pm1}$, $(Q_{i,\ell}')^{\pm1}$, $P_i^{\pm1}$, or $(P_i')^{\pm1}$ except perhaps for the last letter of $C_m'$.

\item Any occurrence in $C_j'$ of a letter of the form $Q_{i,r}^{\pm1}$ or $(Q_{i,r}')^{\pm1}$ is part of a subword of the form $Q_{i,r}^{-1}Q_{i,r}$ or $(Q_{i,r}')^{-1}Q_{i,r}'$.

\end{enumerate}

%
%
%

\end{lemma}

\begin{proof}

The proof follows in much the same way as \Cref{strongly defective (34) R}: If $B$ has no letter of the form $Q_{i,\ell}^{\pm1}$ or $(Q_{i,\ell}')^{\pm1}$, then any maximal subcomputation with step history $(2)$ may be removed to contradict the minimality of $\pazocal{C}$, and so the statement follows from the structure of strongly defective bases.

\end{proof}


\begin{lemma} \label{strongly defective (23) R}

Let $\pazocal{C}:W_0\to\dots\to W_t$ be a minimal computation with strongly defective base $B$.  Suppose the history of $\pazocal{C}$ contains a letter of the form $\sigma(23)^{\pm1}$ which is neither the first nor the last letter and $B$ contains a letter of the form $Q_{i,r}^{\pm1}$ or $(Q_{i,r}')^{\pm1}$.  Then there exists a subword $B''$ of a cyclic permutation of $B$ such that:

\begin{itemize}

\item $\|B''\|=7$

\item Letting $\pazocal{C}'':W_0''\to\dots\to W_t''$ be the restriction to $B''$ of the corresponding cyclic permutation of $\pazocal{C}$, we have $t\leq3\max(|W_0''|_a,|W_t''|_a)$.

\end{itemize}

\end{lemma}

\begin{proof}

The proof follows in much the same way as that of \Cref{strongly defective (34) L}: 

Let $B'\equiv B_1'C_1'\dots B_m'C_m'$ be the factorization of a cyclic permutation of $B$ given by \Cref{strongly defective (23) L}.  Then by hypothesis there exists a $j$ such that $C_j'$ has a subword of the form $Q_{i,r}^{-1}Q_{i,r}$ or a mirror copy.  Letting $UV$ be the first such subword of $C_j'$, we argue that there exists a letter $T$ which is either of the form $R_i^{-1}$ or of the form $(R_i')^{-1}$ so that $TUV$ is a prefix of $C_j'$.  Setting $B''\equiv B_j'TUV$, the statement follows from applications of \Cref{primitive unreduced} and \Cref{one alphabet historical words unreduced}.

\end{proof}

\begin{lemma} \label{strongly defective (23)}

Let $\pazocal{C}:W_0\to\dots\to W_t$ be a minimal computation with strongly defective base $B$.  Suppose the history of $\pazocal{C}$ contains a letter of the form $\sigma(23)^{\pm1}$ which is neither the first nor the last letter.  Then there exists a subword $B''$ of a cyclic permutation of $B$ such that:

\begin{itemize}

\item $\|B''\|\leq7$

\item Letting $\pazocal{C}'':W_0''\to\dots\to W_t''$ be the restriction to $B''$ of the corresponding cyclic permutation of $\pazocal{C}$, we have $t\leq12n^2+6n$ for $n=\max(|W_0''|_a,|W_t''|_a)$.

\end{itemize}

\end{lemma}

\begin{proof}

The proof begins in much the same that for \Cref{strongly defective (34)}:

By \Cref{strongly defective (23) R} we may assume $B$ has no letter of the form $Q_{i,r}^{\pm1}$ or $(Q_{i,r}')^{\pm1}$, and so no letter of the form $R_i^{\pm1}$ or $(R_i')^{\pm1}$.  Applying \Cref{unreduced base quadratic} to the maximal subcomputation with step history $(3)$ produces a two-letter subword $UV$ of $B$ on which this subcomputation operates.

Let $B'\equiv B_1'C_1'\dots B_m'C_m'$ be a cyclic permutation of $B$ with factorization given by \Cref{strongly defective (23) L}.  If $UV$ is contained in a factor $B_j'$, then setting $B''\equiv B_j'$ the proof is completed in just the same way as in \Cref{strongly defective (34)}.

However, since $\sigma(23)$ does not lock the input sectors, $UV$ need not be a subword of a factor $B_j'$.  But if it is not, then it is of the form $P_i^{-1}P_i$ or a mirror copy, and so there exists $j$ such that $UV$ is formed by the last letter of $B_j'$ and the first letter of $B_{j+1}'$.  In this case, set $B''$ to be the subword of $B'$ obtained by appending the first letter of $B_{j+1}'$ to the end of $B_j'$.  Hence, in this case $\|B''\|=5$ and the bound is again given by applying \Cref{primitive unreduced} and \Cref{strongly defective (3)}.

\end{proof}

\begin{lemma} \label{strongly defective (5)}

Let $\pazocal{C}:W_0\to\dots\to W_t$ be a non-empty minimal computation with strongly defective base $B$ and step history $(5)$.  
Then there exists a subword $B''$ of a cyclic permutation of $B$ such that:

\begin{enumerate}[label=(\alph*)]

\item $B''$ is of the form $R_i^{-1}Q_{i,r}^{-1}Q_{i,r}R_i$ or a mirror copy.

\item Letting $\pazocal{C}'':W_0''\to\dots\to W_t''$ be the restriction of the corresponding cyclic permutation of $\pazocal{C}$ to the base $B''$, we have $t\leq 12n^2+2n$ for $n=\max(|W_0''|_a,|W_t''|_a)$.

\end{enumerate}

\end{lemma}

\begin{proof}

If $B$ contains no letter of the form $Q_{i,r}^{\pm1}$ or $(Q_{i,r}')^{\pm1}$, then $W_0\equiv W_t$ so that $t=0$.  

Otherwise, since every rule of the submachine $\textbf{E}_{\textbf{S},k}(5)$ locks the sectors of the form $P_iQ_{i,\ell}$, $Q_{i,r}R_i$, and mirror copies, $B$ cannot contain a subword of the form $(Q_{i,\ell}Q_{i,r})^{\pm1}$ or a mirror copy.  In particular, $B$ must have a subword of the form $Q_{i,r}^{-1}Q_{i,r}$ or a mirror copy.

Hence, $\pazocal{C}$ satisfies the hypotheses of \Cref{unreduced base quadratic}, and thus has a subword $UV$ which it operates on for which the restriction satisfies the desired quadratic bound.  Since the submachine $\textbf{E}_{\textbf{S},k}(5)$ only operates on the $Q_{i,\ell}Q_{i,r}$-sectors on the right, though, it follows from above that $UV$ is of the form $Q_{i,r}^{-1}Q_{i,r}$ or a mirror copy.  Since the rules lock the $Q_{i,r}R_i$-sectors and their mirror copies, it follows that $UV$ is contained in a subword $B''$ of a cyclic permutation of $B$ satisfying the statement.

\end{proof}


\begin{lemma} \label{strongly defective (45)}

Let $\pazocal{C}:W_0\to\dots\to W_t$ be a minimal computation with strongly defective base.  Suppose the history of $\pazocal{C}$ contains a letter of the form $\sigma(45)^{\pm1}$ which is neither the first nor the last letter.  Then there exists a subword $B''$ of a cyclic permutation of $B$ such that:

\begin{enumerate}[label=(\alph*)]

\item $B''$ is of the form $R_i^{-1}Q_{i,r}^{-1}Q_{i,r}R_i$ or a mirror copy.

\item Letting $\pazocal{C}'':W_0''\to\dots\to W_t''$ be the restriction of the corresponding cyclic permutation of $\pazocal{C}$ to the base $B''$, we have $t\leq 12n^2+6n$ for $n=\max(|W_0''|_a,|W_t''|_a)$.

\end{enumerate}

\end{lemma}

\begin{proof}

Perhaps passing to the inverse computation, we may assume the history $H$ of $\pazocal{C}$ can be factored $H'\sigma(54)H''$.  Then, $H'$ has a non-empty suffix which is the history of a maximal subcomputation $\pazocal{C}_5:W_x\to\dots\to W_y$ with step history $(5)$ and $H''$ has a non-empty prefix which is the history of a maximal subcomputation $\pazocal{C}_4:W_{y+1}\to\dots\to W_z$ with step history $(4)$.

Applying \Cref{strongly defective (5)} to $\pazocal{C}_5$ immediately implies the existence of $B''$ satisfying (a).  The restriction $\pazocal{C}'':W_0''\to\dots\to W_t''$ of $\pazocal{C}$ to $B''$ then satisfies $y-x\leq12m^2+2m$ for $m=\max(|W_x''|_a,|W_y''|_a)$.

Further, the restriction of $\pazocal{C}_4$ to the base $B''$ satisfies the hypotheses of the analogue of \Cref{primitive unreduced} for $\textbf{LR}_k$, so that $z=t$, $|W_{y+1}''|_a\leq|W_t''|_a$, and $t-y-1\leq|W_t''|_a$.  

Hence, if $x=0$ then we have $t\leq12n^2+3n+1\leq12n^2+4n$ since $n\geq1$.

Otherwise, the transition $W_{x-1}\to W_x$ must be given by $\sigma(45)$.  But then the same argument as above implies the subcomputation $W_0\to\dots\to W_{x-1}$ has step history $(4)$ with $|W_{x-1}''|_a\leq|W_0''|_a$ and $x-1\leq|W_0''|_a$.  Thus, in this case $t\leq12n^2+4n+2\leq12n^2+6n$.

\end{proof}

The next pair of statements then follow from analogous arguments as those presented for Lemmas \ref{strongly defective (5)} and \ref{strongly defective (45)}.

\begin{lemma} \label{strongly defective (1)}

Let $\pazocal{C}:W_0\to\dots\to W_t$ be a non-empty minimal computation with strongly defective base $B$ and step history $(1)$.  
Then there exists a subword $B''$ of a cyclic permutation of $B$ such that:

\begin{enumerate}[label=(\alph*)]

\item $B''$ is of the form $P_iQ_{i,\ell}Q_{i,\ell}^{-1}P_i^{-1}$ or a mirror copy.

\item Letting $\pazocal{C}'':W_0''\to\dots\to W_t''$ be the restriction of the corresponding cyclic permutation of $\pazocal{C}$ to the base $B''$, we have $t\leq 12n^2+2n$ for $n=\max(|W_0''|_a,|W_t''|_a)$.

\end{enumerate}

\end{lemma}


\begin{lemma} \label{strongly defective (12)}

Let $\pazocal{C}:W_0\to\dots\to W_t$ be a minimal computation with strongly defective base.  Suppose the history of $\pazocal{C}$ contains a letter of the form $\sigma(12)^{\pm1}$ which is neither the first nor the last letter.  Then there exists a subword $B''$ of a cyclic permutation of $B$ such that:

\begin{enumerate}[label=(\alph*)]

\item $B''$ is of the form $P_iQ_{i,\ell}Q_{i,\ell}^{-1}P_i^{-1}$ or a mirror copy.

\item Letting $\pazocal{C}'':W_0''\to\dots\to W_t''$ be the restriction of the corresponding cyclic permutation of $\pazocal{C}$ to the base $B''$, we have $t\leq 12n^2+6n$ for $n=\max(|W_0''|_a,|W_t''|_a)$.

\end{enumerate}

\end{lemma}

\begin{lemma} \label{strongly defective (4) no connecting}

Let $\pazocal{C}:W_0\to\dots\to W_t$ be a non-empty minimal computation with strongly defective base $B$ and step history $(4)$.  If the history of $\pazocal{C}$ contains no connecting rule, then there exists a two-letter subword $UV$ of $B$ such that for $\pazocal{C}':W_0'\to\dots\to W_t'$ the restriction of $\pazocal{C}$ to the base $UV$, we have $t\leq 12n^2+2n$ for $n=\max(|W_0'|_a,|W_t'|_a)$.

\end{lemma}

\begin{proof}

If $B$ has no letter of the form $Q_{i,r}^{\pm1}$ or $(Q_{i,r}')^{\pm1}$, then $W_0\equiv W_t$ so that $t=0$.  Hence, we assume $B$ has such a letter.

Since every rule of $\textbf{E}_{\textbf{S},k}(4)$ locks the $P_iQ_{i,\ell}$-sectors and their mirror copies, the strongly defective condition implies no cyclic permutation of $B$ contains a subword of the form $(Q_{i,\ell}Q_{i,r}R_i)^{\pm1}$ or a mirror copy.  In particular, $B$ must have a subword of the form $Q_{i,r}^{-1}Q_{i,r}$, $Q_{i,r}Q_{i,r}^{-1}$, or a mirror copy of one of these.  In any case, we apply \Cref{unreduced base quadratic} to obtain the statement.

\end{proof}

An analogous proof implies the following statement.

\begin{lemma} \label{strongly defective (2) no connecting}

Let $\pazocal{C}:W_0\to\dots\to W_t$ be a non-empty minimal computation with strongly defective base $B$ and step history $(2)$.  If the history of $\pazocal{C}$ contains no connecting rule, then there exists a two-letter subword $UV$ of $B$ such that for $\pazocal{C}':W_0'\to\dots\to W_t'$ the restriction of $\pazocal{C}$ to the base $UV$, we have $t\leq 12n^2+2n$ for $n=\max(|W_0'|_a,|W_t'|_a)$.

\end{lemma}

Thus, we arrive at the desired statement regarding minimal computations.

\begin{lemma}[Compare with Lemma 3.30 of \cite{GW}] \label{universal complexity}

Suppose $(W_1,W_2)\in\REACH_{\textbf{E}_{\textbf{S},k}}^{uni}$ such that $W_1$ and $W_2$ have defective base $B$.  Then there exists a subword $B''$ of a cyclic permutation of $B$ such that:

\begin{enumerate}[label=(\alph*)]

\item $\|B''\|\leq N+1$ 

\item Letting $W_i''$ be the admissible subword of the corresponding cyclic permutation of $W_i$ with base $B''$, the length of a minimal computation between $W_1$ and $W_2$ is at most $12n^2+c_1n+c_2$ for $n=\max(|W_1''|_a,|W_2''|_a)$.

\end{enumerate}

%
%
%


\end{lemma}

\begin{proof}

Note that since $B$ is defective, \Cref{enhanced controlled} implies the history of a minimal computation between $W_1$ and $W_2$ cannot contain a controlled subword.  If a cyclic permutation of $B$ has a reduced revolving subword, then taking $B''$ to be this subword, $\|B''\|=N+1$ and, by \Cref{reduced revolving controlled}, the length of any reduced computation between $W_1$ and $W_2$ is at most $c_1n+c_1(N+1)$.  Hence, the statement is given by the parameter choices $c_2>>c_1>>N$.

Next, if $B$ is not strongly defective, then we may take $B''$ to be a subword of a cyclic permutation of $B$ of the form $(P_iQ_{i,\ell}Q_{i,r}R_i)^{\pm1}$ or a mirror copy.  The statement then follows from \Cref{Defective PQQR}.


Hence, we assume $B$ is strongly defective.

Let $\pazocal{C}:W_1\equiv V_0\to\dots\to V_t\equiv W_2$ be a minimal computation between $W_1$ and $W_2$.  Note that the application of a transition or connecting rule does not alter the tape word of any sector.  So, taking $c_2\geq2$, it suffices to assume that neither the first nor the last letter of the history of $\pazocal{C}$ is such a rule and show that $t\leq12n^2+c_1n$.

Taking $N\geq6$, the parameter assignment $c_1>>N$ and Lemmas \ref{strongly defective (3)}--\ref{strongly defective (2) no connecting} then imply the statement in all cases except:

\begin{enumerate}[label=(\roman*)]

\item The step history of $\pazocal{C}$ is $(2)$ and the history $H$ of $\pazocal{C}$ has a connecting rule which is neither the first nor the last letter.

\item The step history of $\pazocal{C}$ is $(4)$ and the history $H$ of $\pazocal{C}$ has a connecting rule which is neither the first nor the last letter.

\end{enumerate}

Suppose $\pazocal{C}$ satisfies (i).  So, there exists a factorization $H\equiv H_1\zeta^{\pm1} H_2$ where $\zeta$ is the connecting rule of $\textbf{E}_{\textbf{S},k}(2)$ and $H_1,H_2$ are non-empty.  

If $B$ has no occurrence of a letter of the form $Q_{i,\ell}^{\pm1}$ or $(Q_{i,\ell}')^{\pm1}$, then no rule of $\pazocal{C}$ alters any tape word of the admissible words.  So, any rule of $H$ which does not alter the state letters may be removed to reduce the length of the computation, so that the minimality of $\pazocal{C}$ implies no such rule exists.  But then every letter of $H$ must be a connecting rule, contradicting the setup of (i).

So, we assume $B$ contains a letter of the form $Q_{i,\ell}^{\pm1}$ or $(Q_{i,\ell}')^{\pm1}$.  Noting that every rule of $\textbf{E}_{\textbf{S},k}(2)$ locks the $Q_{i,r}R_i$-sectors and their mirror copies, the strongly defective condition implies $B$ contains no subword of the form $(P_iQ_{i,\ell}Q_{i,r})^{\pm1}$ or a mirror copy.  Since $\zeta$ locks the $Q_{i,\ell}Q_{i,r}$-sectors and their mirror copies, though, there must be a subword $B''$ of a cyclic permutation of $B$ of the form $Q_{i,r}^{-1}Q_{i,\ell}^{-1}Q_{i,\ell}Q_{i,r}$ or a mirror copy.

Let $\pazocal{C}'':V_0''\to\dots\to V_t''$ be the restriction of the corresponding cyclic permutation of $\pazocal{C}$ to the subword $B''$.  Then, the subcomputation of $\pazocal{C}''$ with history $H_2$ satisfies the hypotheses of \Cref{primitive unreduced}, so that $\|H_2\|\leq|V_t''|_a$.  Similarly, the subcomputation of the inverse of $\pazocal{C}''$ with history $H_1^{-1}$ satisfies the hypotheses of \Cref{primitive unreduced}, so that $\|H_1\|\leq|V_0''|_a$.  Hence, $t\leq2n+1\leq3n$.

If $\pazocal{C}$ satisfies (ii), then note that the connecting rules of $\textbf{E}_{\textbf{S},k}(4)$ come in two types: (1) Those that lock the $Q_{i,\ell}Q_{i,r}$-sectors and their mirror copies, and (2) those that lock the $Q_{i,r}R_i$-sectors and their mirror copies.  Since all other sectors are locked by every rule, that $B$ is unreduced implies only one type of connecting rule can appear in $H$.

As in the argument above for (i), since $B$ is strongly defective, we may find a subword $B''$ of a cyclic permutation of $B$ of the form:

\begin{itemize}

\item $Q_{i,\ell}Q_{i,r}Q_{i,r}^{-1}Q_{i,\ell}^{-1}$ or a mirror copy if $H$ has a connecting rule of type (1).

\item $R_i^{-1}Q_{i,r}^{-1}Q_{i,r}R_i$ or a mirror copy if $H$ has a connecting rule of type (2).

\end{itemize}

In either case, applications of \Cref{primitive unreduced} as above yield $t\leq3n$.

\end{proof}

\begin{remark} \label{rmk-quadratic}

As in \cite{GW}, the quadratic term in the statement of \Cref{universal complexity} accounts for the quadratic term in the upper bound of \Cref{main-theorem}.  See Remark 3.31 of \cite{GW} for a discussion of the difficulty in removing this factor from the estimates.

\end{remark}

\medskip


\subsection{Faulty bases} \label{sec-faulty} \

The goal of this section is to achieve a linear bound on the space of a reduced computation with faulty base in terms of the $a$-lengths of its initial and terminal words. The quadratic term in \Cref{universal complexity} makes it unsuitable for achieving this immediately.  

However, one case can be addressed with a previous statement.

\begin{lemma} \label{faulty PQQR}

Let $\pazocal{C}:W_0\to\dots\to W_t$ be a reduced computation of $\textbf{E}_{\textbf{S},k}$ with faulty base $B$.  Suppose a cyclic permutation of $B$ has a subword of the form $(P_jQ_{j,\ell}Q_{j,r}R_j)^{\pm1}$ or a mirror copy.  Then $|W_i|_a\leq c_1\max(|W_0|_a,|W_t|_a)$ for all $0\leq i\leq t$.

\end{lemma}

\begin{proof}

Applying \Cref{Defective PQQR} to a cyclic permutation of $\pazocal{C}$ whose base has the relevant subword yields $t\leq3c_0\max(|W_0|_a,|W_t|_a)+2c_0$.   \Cref{simplify rules} then implies $|W_i|_a\leq|W_0|_a+2\|B\|t$ for all $i$.  But faulty bases have length at most $2N+1$, so that $$|W_i|_a\leq7(2N+1)c_0\max(|W_0|_a,|W_t|_a)+2(2N+1)c_0$$  As unreduced sectors of admissible words must have non-empty tape words, the statement then follows by the parameter choices $c_1>>c_0>>N$.

\end{proof}

\begin{lemma} \label{faulty no transition/connecting}

Let $\pazocal{C}:W_0\to\dots\to W_t$ be a reduced computation of $\textbf{E}_{\textbf{S},k}(j)$ with faulty base $B$.  Suppose either:

\begin{itemize}

\item $j\in\{1,5\}$

\item $j\in\{2,4\}$ and the history of $\pazocal{C}$ has no occurrence of a connecting rule.

\end{itemize}

Then $|W_i|_a\leq2\max(|W_0|_a,|W_t|_a)$ for all $0\leq i\leq t$.

\end{lemma}

\begin{proof}

Let $B\equiv U_0U_1\dots U_m$ where each $U_j$ is a single letter.  Then for $j\in\{1,\dots,m\}$, let $\pazocal{C}^{(j)}:W_0^{(j)}\to\dots\to W_t^{(j)}$ be the restriction of $\pazocal{C}$ to the base $U_{j-1}U_j$.

In any case, each $\pazocal{C}^{(j)}$ either has fixed tape word, satisfies the hypotheses of \Cref{multiply one letter}, or satisfies the hypotheses of \Cref{unreduced base}.  Hence, for all $0\leq i\leq t$ and $1\leq j\leq m$ we have $|W_i^{(j)}|_a\leq\max(|W_0^{(j)}|_a,|W_t^{(j)}|_a)$.  Thus, for all $i$ we have
\begin{align*}
|W_i|_a&\leq\sum_{j=1}^m|W_i^{(j)}|_a\leq\sum_{j=1}^m\max(|W_0^{(j)}|_a,|W_t^{(j)}|_a) \\
&\leq\sum_{j=1}^m(|W_0^{(j)}|_a+|W_t^{(j)}|_a)=|W_0|_a+|W_t|_a \\
&\leq2\max(|W_0|_a,|W_t|_a)
\end{align*}

\end{proof}

\begin{lemma} \label{faulty (3)}

Let $\pazocal{C}:W_0\to\dots\to W_t$ be a reduced computation of $\textbf{E}_{\textbf{S},k}(3)$ with faulty base $B$.  Then $|W_i|_a\leq18\max(|W_0|_a,|W_t|_a)$ for all $0\leq i\leq t$.

\end{lemma}

\begin{proof}

By the definition of $B_{std}$, we may factor $B\equiv B_1\dots B_m$ where each $B_j$ is either (1) a word over $B_0^{\pm1}$, (2) a word over $(B_0')^{\pm1}$, or (3) a single letter of the form $\{t\}^{\pm1}$.  For $1\leq j\leq m$ let $\pazocal{C}^{(j)}:W_0^{(j)}\to\dots\to W_t^{(j)}$ be the restriction of $\pazocal{C}$ to the base $B_j$.

Since the tape alphabets of the $R_s(R_s')^{-1}$-, $(P_0')^{-1}\{t\}$-, and $\{t\}P_0$-sectors are all empty, it follows that $|W_i|_a=\sum|W_i^{(j)}|_a$.

If $B_j$ is a single letter, then of course $|W_i^{(j)}|_a=0$ for all $i$.  In the other cases, $\pazocal{C}_j$ can be identified with a reduced computation $\pazocal{C}_j'$ of $\textbf{S}_h$ with base $B_j'$.  If $\|B_j'\|\geq3$, then \Cref{M_2 bound} implies $|W_i^{(j)}|_a\leq 9(|W_0^{(j)}|_a+|W_t^{(j)}|_a)$ for all $0\leq i\leq t$.  Otherwise, the makeup of the standard base dictates $B_j'=Q_{0,\ell}Q_{0,\ell}^{-1}$ or $Q_{s,r}^{-1}Q_{s,r}$, so that $\pazocal{C}_j'$ satisfies the hypotheses of \Cref{unreduced base}, meaning $|W_i^{(j)}|_a\leq\max(|W_0^{(j)}|_a,|W_t^{(j)}|_a)$ for all $i$.  Thus, for all $i$ we have
$$|W_i|_a=\sum|W_i^{(j)}|_a\leq\sum9(|W_0^{(j)}|_a+|W_t^{(j)}|_a)\leq9(|W_0|_a+|W_t|_a)\leq18\max(|W_0|_a,|W_t|_a)$$

\end{proof}

Thus, we arrive at the main objective of the section, achieving a linear bound on the space of a reduced computation with faulty base.

\begin{lemma} \label{faulty}

Let $\pazocal{C}:W_0\to\dots\to W_t$ be a reduced computation of $\textbf{E}_{\textbf{S},k}$ with faulty base $B$.  Then $|W_i|_a\leq c_1\max(|W_0|_a,|W_t|_a)$ for all $0\leq i\leq t$.

\end{lemma}

\begin{proof}

It is obvious that we may assume $t\geq2$, as otherwise $W_i\equiv W_0$ or $W_i\equiv W_t$ for all $i$.

By an inductive argument, we may assume $|W_j|_a>\max(|W_0|_a,|W_t|_a)$ for all $0<j<t$, as otherwise we may consider the shorter computations $W_0\to\dots\to W_j$ and $W_j\to\dots\to W_t$.  

In particular, since transition rules and connecting rules do not alter the tape words of an admissible word, we may assume that neither the first nor the last letter of the history $H$ of $\pazocal{C}$ is a letter corresponding to such a rule.  However, taking $c_0\geq18$, it follows from Lemmas \ref{faulty no transition/connecting} and \ref{faulty (3)} that we may assume $H$ does have some letter corresponding to such a rule, and so in particular we must have $t\geq3$.

Moreover, by \Cref{faulty PQQR} it suffices to assume that $B$ is strongly defective.

We now proceed in cases.

\textbf{1.} Suppose $H$ has a letter in $\{\sigma(34)^{\pm1},\sigma(45)^{\pm1}\}$ and $B$ has a letter of the form $Q_{i,r}^{\pm1}$ or $(Q_{i,r}')^{\pm1}$.

Note that both $\sigma(34)$ and $\sigma(45)$ lock the sectors of the standard base of the form $P_iQ_{i,\ell}$, $Q_{i,r}R_i$, and the mirror copies of these.  So, as it is strongly defective, $B$ cannot contain a subword of the form $(Q_{i,\ell}Q_{i,r})^{\pm1}$ or a mirror copy.

Hence, as we assume that $B$ has a letter of the form $Q_{i,r}^{\pm1}$ or $(Q_{i,r}')^{\pm1}$, the proof of \Cref{strongly defective (34) R} may be repeated to show there exists a cyclic permutation $B'\equiv B_1'C_1'\dots B_m'C_m'$ such that:

\begin{itemize}

\item Each subword $B_j'$ is of the form $R_i^{-1}Q_{i,r}^{-1}Q_{i,r}R_i$ or a mirror copy.

\item The subwords $C_j'$ do not contain any letters of the form $Q_{i,r}^{\pm1}$, $(Q_{i,r}')^{\pm1}$, $R_i^{\pm1}$, or $(R_i')^{\pm1}$ except perhaps for the last letter of $C_m'$.

\end{itemize}

Perhaps passing to the inverse computation we may assume there is a non-empty subcomputation $\pazocal{C}_4:W_x\to\dots\to W_y$ with step history $(4)$ such that the transition $W_{x-1}\to W_x$ is given by $\sigma(45)$ or $\sigma(43)$.  For each $1\leq j\leq m$, let $\pazocal{C}_4^{(j)}:W_x^{(j)}\to\dots\to W_y^{(j)}$ be the restriction of the corresponding cyclic permutation of $\pazocal{C}_4$ to the subword $B_j'$.

Then each $\pazocal{C}_4^{(j)}$ may be viewed as a reduced computation of a primitive machine satisfying the hypotheses of \Cref{primitive unreduced}.  Since $m\geq1$, it then follows that $y=t$ and $|W_x^{(j)}|_a\leq|W_t^{(j)}|_a$.

By construction the restriction to any two-letter subword not contained in $B_j'$ has constant tape word throughout $\pazocal{C}_4$, so that $|W_x|_a\leq|W_t|_a$.  But $0<x<t$, so that this contradicts the assumptions above.

\textbf{2.} Suppose $H$ has a letter in $\{\sigma(12)^{\pm1},\sigma(23)^{\pm1}\}$ and $B$ has a letter of the form $Q_{i,\ell}^{\pm1}$ or $(Q_{i,\ell}')^{\pm1}$.

Then we reach a contradiction in much the same way as in the previous case: 

Since $\sigma(12)$ and $\sigma(23)$ lock the same sectors, that $B$ is strongly defective allows us to find a cyclic permutation of $B$ in the same form as \Cref{strongly defective (23) L}.  So, assuming without loss of generality that there exists a non-empty subcomputation $W_x\to\dots\to W_y$ with step history $(2)$ such that the transition $W_{x-1}\to W_x$ is given by $\sigma(23)$ or $\sigma(21)$, \Cref{primitive unreduced} implies $y=t$ and $|W_x|_a\leq|W_t|_a$.

\textbf{3.} Suppose $H$ has a letter of the form $\sigma(45)^{\pm1}$.

In light of the first case, we may assume $B$ has no letter of the form $Q_{i,r}^{\pm1}$ or $(Q_{i,r}')^{\pm1}$.

Note that $\sigma(45)$ locks all sectors of the standard base except those of the form $Q_{i,\ell}Q_{i,r}$ and their mirror copies, while the `last' connecting rule of $\textbf{E}_{\textbf{S},k}(4)$ locks these remaining sectors.  So, since $B$ is unreduced, \Cref{locked sectors} implies $H$ has no letter corresponding to a connecting rule.  Hence, the step history of $\pazocal{C}$ is a concatenation of the letters $(4)$ and $(5)$.

But by the makeup of these submachines, in this case no rule of $H$ alters the tape word of any admissible word with base $B$, implying $|W_0|_a=\dots=|W_t|_a$.

\textbf{4.} Suppose $H$ has a letter of the form $\sigma(12)^{\pm1}$.

From the second case, we may assume $B$ has no occurrence of the letters $Q_{i,\ell}^{\pm1}$ or $(Q_{i,\ell}')^{\pm1}$.  So, if the step history of $\pazocal{C}$ is a concatenation of the letters $(1)$ and $(2)$, then we reach the contradiction $|W_0|_a=\dots=|W_t|_a$ in the same way as in the previous case.

Hence, it suffices to assume the step history has an occurrence of the letter $(3)$, and so a subword of the form $(12)(2)(23)$.

In particular, $H$ must have a letter corresponding to a connecting rule of $\textbf{E}_{\textbf{S},k}(2)$.  This rule locks the $Q_{i,\ell}Q_{i,r}$-sectors and their mirror copies, while $\sigma(12)$ locks the $P_iQ_{i,\ell}$-sectors, the $Q_{i,r}R_i$-sectors, and their mirror copies.  The strongly defective condition along with \Cref{locked sectors} then implies $B$ can have only letters of the form $\{t\}^{\pm1}$.  

But the tape alphabets of the $\{t\}P_0$- and $(P_0')^{-1}\{t\}$-sectors are both empty, so that we cannot form a faulty base with just letters of the form $\{t\}^{\pm1}$.

\textbf{5.} Suppose $H$ has a letter of the form $\sigma(34)^{\pm1}$.

Again, the first case allows us to assume $B$ has no letter of the form $Q_{i,r}^{\pm1}$ or $(Q_{i,r}')^{\pm1}$.  So since $\sigma(34)$ locks every sector of the standard base except for those of the form $Q_{i,\ell}Q_{i,r}$ and their mirror copies, $B$ must have a two-letter subword $B''$ of the form $Q_{i,\ell}Q_{i,\ell}^{-1}$ or a mirror copy (and indeed all unreduced two-letter subwords of $B$ must be of this form).

Perhaps passing to the inverse computation, we assume $H$ has a subword $H_3\sigma(34)H_4$ where $H_3$ and $H_4$ are non-empty histories of maximal subcomputations with step history $(3)$ and $(4)$, respectively.

Let $\pazocal{C}_3:W_x\to\dots\to W_y$ be the subcomputation with history $H_3$ and let $\pazocal{C}_3'':W_x''\to\dots\to W_y''$ be its restriction to the subword $B''$.  Since $W_y$ is $\sigma(34)$-admissible, the tape word of $W_y''$ is non-empty and comprised of letters from the corresponding right historical alphabet.  But then the inverse computation of $\pazocal{C}_3''$ satisfies the hypotheses of \Cref{one alphabet historical words unreduced}, so that $x=0$ and $y\leq|W_0''|_a$.  Thus, since faulty bases can have length at most $2N+1$, \Cref{simplify rules} then implies that for all $0\leq i\leq y$
$$|W_i|_a\leq|W_0|_a+2(\|B\|-1)i\leq|W_0|_a+4Ny\leq c_0|W_0|_a$$ 
by the parameter assignment $c_0>>N$.

Next, let $\pazocal{C}_4:W_{y+1}\to\dots\to W_z$ be the subcomputation with history $H_4$.  Since $B$ has no letter of the form $Q_{i,r}^{\pm1}$ or $(Q_{i,r}')^{\pm1}$, we necessarily have $|W_y|_a=|W_{y+1}|_a=\dots=|W_z|_a$.  So, we must have $z<t$.  Moreover, the presence of the subword $B''$ implies $H_4$ can have no connecting rule, so that $W_z$ cannot be $\sigma(45)$-admissible.  Hence, the transition $W_z\to W_{z+1}$ must be given by $\sigma(43)$.

But then we may apply the same argument as above to the maximal subcomputation with step history $(3)$ starting with $W_{z+1}$, obtaining $|W_i|_a\leq c_0|W_t|_a$ for all $z\leq i\leq t$.  

Thus, the desired bound is given by the parameter assignment $c_1>>c_0$.

\textbf{6.} Suppose $H$ has a letter of the form $\sigma(23)^{\pm1}$.

Again, the second case allows us to assume $B$ has no letter of the form $Q_{i,\ell}^{\pm1}$ or $(Q_{i,\ell}')^{\pm1}$.  

If $B$ does have a subword $B''$ of the form $Q_{i,r}^{-1}Q_{i,r}$ or a mirror copy, then we may repeat the argument of the previous case, using \Cref{one alphabet historical words unreduced} to obtain $|W_i|_a\leq c_0\max(|W_0|_a,|W_t|_a)$ for all $i$.

Otherwise, then as in the fourth case the strongly defective condition implies $B$ can consist only of letters of the form $\{t\}^{\pm1}$, which is impossible.


\textbf{7.} Suppose $H$ has a connecting rule $\zeta$ of $\textbf{E}_{\textbf{S},k}(4)$ which locks the $Q_{i,\ell}Q_{i,r}$-sectors.

By the previous cases, we may assume the step history of $\pazocal{C}$ is $(4)$.  So, if $B$ has no letter of the form $Q_{i,r}^{\pm1}$ or $(Q_{i,r}')^{\pm1}$, then $|W_0|_a=\dots=|W_t|_a$.  Hence, we may assume $B$ contains such a letter.

Noting that consecutive connecting rules of $\textbf{E}_{\textbf{S},k}(4)$ combine to lock every sector of the standard base, \Cref{locked sectors} implies every every occurrence of a connecting rule in $H$ is of the form $\zeta^{\pm1}$.  As such, there exists a subword $H_1\zeta H_2$ of $H$ such that $H_1$ and $H_2$ are non-empty maximal subwords with no connecting rules.  Let $\pazocal{C}_2:W_y\to\dots\to W_z$ be the subcomputation with history $H_2$.

Since $\zeta$ locks the $P_iQ_{i,\ell}$-sectors and their mirror copies, then since $B$ is strongly defective any occurrence of $Q_{i,r}^{\pm1}$ belongs to a subword $B''$ of a cyclic permutation of $B$ of the form $Q_{i,\ell}Q_{i,r}Q_{i,r}^{-1}Q_{i,\ell}^{-1}$.  For $\pazocal{C}_2'':W_y''\to\dots\to W_z''$ the restriction of the corresponding cyclic permutation of $\pazocal{C}_2$ to such a subword, $\pazocal{C}_2''$ may be viewed as a reduced computation of a primitive machine satisfying the hypotheses of \Cref{primitive unreduced}.  So if such a $B''$ exists, $z=t$ and $|W_y''|_a\leq|W_t''|_a$.

Similarly, any occurrence in $B$ of $(Q_{i,r}')^{\pm1}$ belongs to a subword of a cyclic permutation of $B$ of the form $Q_{i,\ell}'Q_{i,r}'(Q_{i,r}')^{-1}(Q_{i,\ell}')^{-1}$, for which we can make the analogous conclusion.

Hence, as we assume $B$ contains a letter of the form $Q_{i,r}^{\pm1}$ or $(Q_{i,r}')^{\pm1}$, we must have $z=t$ and, as the tape word of any other sector is fixed, $|W_y|_a\leq|W_t|_a$.  But $H_2$ is non-empty, so that $0<y<t$.

\textbf{8.} Suppose $\pazocal{C}$ has step history $(4)$.

In light of the previous case, we may assume $H$ has no connecting rule which locks the $Q_{i,\ell}Q_{i,r}$-sectors.  By all the previous cases, though, $H$ must contain a connecting rule of some sort, and so must contain one that locks the $Q_{i,r}R_i$-sectors.

But then we reach a contradiction in much the same way as in the previous case: $B$ must contain a letter of the form $Q_{i,r}^{\pm1}$ or $(Q_{i,r}')^{\pm1}$, while the strongly defective condition implies such a letter must be contained in a subword $B''$ of a cyclic permutation of $B$ of the form $R_i^{-1}Q_{i,r}^{-1}Q_{i,r}R_i$ or a mirror copy.  By \Cref{primitive unreduced} we may then find $0<y<t$ with $|W_y|_a\leq|W_t|_a$.

\textbf{9.} Thus, we may assume $\pazocal{C}$ has step history $(2)$ and $H$ has a connecting rule $\zeta$.

But this leads to a contradiction in much the same way as in the previous two cases:

If $B$ has no letter of the form $Q_{i,\ell}^{\pm1}$ or $(Q_{i,\ell}')^{\pm1}$, then no rule of $\textbf{E}_{\textbf{S},k}(2)$ alters the tape words of an admissible word with base $B$, so that $|W_0|_a=\dots=|W_t|_a$.

Since $\zeta$ locks the $Q_{i,\ell}Q_{i,r}$- and $Q_{i,r}R_i$-sectors, the strongly defective condition implies any occurrence in $B$ of a letter of the form $Q_{i,\ell}^{\pm1}$ must be contained in a subword of a cyclic permutation of $B$ of the form $P_iQ_{i,\ell}Q_{i,\ell}^{-1}P_i^{-1}$.  Similarly, any occurrence of $(Q_{i,\ell}')^{\pm1}$ is contained in the mirror copy of such a subword.  Taking a maximal subcomputation whose history has no connecting rule, the restriction of a cyclic permutation to such a subword can be viewed as a reduced computation of a primitive machine satisfying the hypotheses of \Cref{primitive unreduced}, yielding $0<y<t$ such that $|W_y|_a\leq|W_t|_a$ as in the previous cases.

\end{proof}

\medskip


\subsection{Main machine} \label{sec-main-machine} \

The main machine $\textbf{M}_\textbf{S}$ of our construction is the concatenation of the $k$-enhanced machine $\textbf{E}_{\textbf{S},k}$ a number of times.  This is achieved in much the same way as in \cite{O18} and \cite{OS20}, but with notation closer to the more general constructions in \cite{WMal}, \cite{WEmb}, and \cite{W}.

Specifically, let $B_{std}(1),\dots,B_{std}(L)$ be $L$ copies of the standard base of the mirror machine $\textbf{E}_{\textbf{S},k}^1$ (where $L$ is the parameter mentioned in \Cref{sec-parameters}).  Then the standard base of $\textbf{M}_\textbf{S}$ is
$$\{t(1)\}B_{std}(1)\{t(2)\}B_{std}(2)\dots\{t(L)\}B_{std}(L)$$
As in the construction of the $k$-enhanced machine, the parts of the form $\{t(i)\}$ are singletons; a state letter from such a part (or its inverse) is called a \textit{$t$-letter}.  Hence, the standard base of $\textbf{M}_\textbf{S}$ may be viewed as a concatenation of $L$ copies of the standard base of $\textbf{E}_{\textbf{S},k}$.

The tape alphabets of the sectors arising from a single copy of the standard base of $\textbf{E}_{\textbf{S},k}$ is a copy of the corresponding tape alphabet of the $k$-enhanced machine.  That of any other sector (including that of the $(P_0'(L))^{-1}\{t(1)\}$-sector) is taken to be empty.  As such, any sector of the standard base formed by a $t$-letter has empty tape alphabet.

The start and end state letters of each sector are those corresponding to the start and end letters of the $k$-enhanced machines, while the input sectors are taken to be all the copies of such sectors.  Note than that the number of input sectors of $\textbf{M}_\textbf{S}$ is $L$ times the number of such sectors in $\textbf{E}_{\textbf{S},k}$ or $\textbf{E}_{\textbf{S},k}^1$, and $2L$ times the number in $\textbf{E}_{\textbf{S},k}^0$ or $\textbf{S}$.

The software of $\textbf{M}_\textbf{S}$ is identified with that of $\textbf{E}_{\textbf{S},k}$, with each rule operating on every copy of $\{t\}B_{std}$ as the corresponding rule and, of course, locking every other sector.

Given a configuration $W$ of $\textbf{M}_\textbf{S}$, the \textit{$i$-th component} of $W$ is the admissible subword $W(i)$ with base $\{t(i)\}B_{std}(i)$.  Note that the makeup of the tape alphabets implies $W\equiv W(1)\cdots W(L)$.

If $V$ is an admissible word with base $B$ where every letter of $B$ is in $(\{t(i)\}B_{std}(i))^{\pm1}$, then a \textit{coordinate shift} of $V$ is an admissible word $V'$ whose base $B'$ consists of letters from $\{t(j)\}B_{std}(j))^{\pm1}$ obtained by changing the `coordinate' of each state letter from $i$ to $j$ and taking the natural copy of the tape words over the corresponding tape alphabets.  Note that the parallel nature of the rules implies that if $W$ is an accepted configuration, then $W(i)$ and $W(j)$ are coordinate shifts of one another.  The accept configuration of $\textbf{M}_\textbf{S}$ is denoted $W_{ac}$, while the input configuration with input $w$ is denoted $I(w)$.

As opposed to the constructions in \cite{WMal}, \cite{WEmb}, and \cite{W} (where one component operated in a different manner to the others), the computational makeup of $\textbf{M}_\textbf{S}$ is immediately deduced from that of $\textbf{E}_{\textbf{S},k}$.  For example, the notion of `almost-extending' a computation introduced in \cite{W} (see Lemma 5.12) may be strengthened: Given a reduced computation $V_0\to\dots\to V_t$ with base $\{t(j)\}B_{std}(j)$, there exists a reduced computation $W_0\to\dots\to W_t$ in the standard base such that $W_i(j)\equiv V_i$ for all $i$.  Conversely, any reduced computation of $\textbf{M}_\textbf{S}$ can be identified with a reduced computation of $\textbf{E}_{\textbf{S},k}$ by removing any coordinates and taking the natural copies of the tape words.

\bigskip


\section{Groups Associated to an \texorpdfstring{$S$}--machine and their Diagrams} \label{sec-groups}

\subsection{The groups} \label{sec-associated-groups} \

As in previous literature (for example \cite{O18}, \cite{OS20}, \cite{WCubic}, \cite{WEmb}, \cite{W}), we now associate a pair of finitely presented groups to any recognizing cyclic $S$-machine $\textbf{M}$. 


Denote the hardware of $\textbf{M}$ by $(Y,Q)$ with $Q=\sqcup_{i=0}^s Q_i$ and $Y=\sqcup_{i=1}^{s+1} Y_i$.  Further, denote the software by $\Theta=\Theta^+\sqcup\Theta^-$. 

By Lemma \ref{simplify rules}, we may consider all elements of $\Theta^+$ as adding at most one $a$-letter to each side of a $q$-letter.  In other words, any $\theta\in\Theta^+$ is of the form $$\theta=[q_0\to v_{s+1}q_0'u_1, \ q_1\to v_1q_1'u_2, \ \dots, \ q_{s-1}\to v_{s-1}q_{s-1}'u_s, \ q_s\to v_sq_s'u_{s+1}]$$ where $q_i,q_i'\in Q_i$ and $u_i,v_i\in Y_i(\theta)^{\pm1}\cup \{1\}$. Also, some sectors may be locked, in which case some of the arrows above take the form $\xrightarrow{\ell}$. 

Define $\Theta^+_*=\{\theta_i: \theta\in\Theta^+,0\leq i\leq s\}$. 
  The group $M(\textbf{S})$ is then defined to be the finitely presented group with generating set $\pazocal{X}=Y\cup Q\cup \Theta^+_*$ and set of relators given by
\begin{equation}\notag\label{thetaRels}\begin{matrix}\theta_iy& = & y\theta_i &&&&\text{for all }&i\in \{1,\ldots, s\}, &y\in Y_i(\theta)\\ q_{i}\theta_{i+1}&=&\theta_{i}v_iq_i'u_{i+1}&&&&\text{for all }&i\in \{1,\ldots, s\},  &\theta\in \Theta^+ \end{matrix}\end{equation}
where, by convention, arithmetic in all subscripts is performed modulo $s+1$. 

We define \textit{$q$-letters} and \textit{$a$-letters} in these groups analogously to their counterparts in the computational realm. Letters from $(\Theta^+_*)^{\pm1}=\Theta^+_*\sqcup\Theta^-_*$ are called \textit{$\theta$-letters}. The relations of the form $q_i\theta_{i+1}=\theta_iv_iq_i'u_{i+1}$ are called \textit{$(\theta,q)$-relations}, while those of the form $\theta_ia=a\theta_i$ are called \textit{$(\theta,a)$-relations}.
Note that the number of $a$-letters in any part of $\theta$, and so in any defining relation of $M(\textbf{M})$, is at most two.  Further, note that if $\theta$ locks the $i$-th sector, then $Y_i(\theta)=\emptyset$ so that each $\theta_j$ has no relation with the elements of $Y_i$.

Finally, we form the finitely presented group $G(\textbf{M})$ from $M(\textbf{M})$ by simply adding a single relation, called the \textit{hub relation}, which sets the accept configuration $W_{ac}$ of $\textbf{M}$ to $1$.  In other words, $G(\textbf{M})\cong M(\textbf{M})/\gen{\gen{W_{ac}}}$.

\medskip


\subsection{Bands and annuli} \label{sec-bands-annuli} \

The arguments presented in the forthcoming sections rely heavily on van Kampen (circular) and Schupp (annular) diagrams over (presentations of) the groups $M(\textbf{M}_\textbf{S})$ and $G(\textbf{M}_\textbf{S})$ introduced in \Cref{sec-associated-groups}.  It is assumed that the reader is well acquainted with these notions; for reference, see \cite{Lyndon-Schupp} and \cite{O}.




As we exclusively consider diagrams on spaces homeomorphic to a disk or an annulus, we are able to adopt the convention that the boundary of any cell and of any boundary component of a diagram or subdiagram is traced in the counterclockwise direction.

Fix a cyclic recognizing $S$-machine $\textbf{M}$.  For any diagram over (the canonical presentations of) $M(\textbf{M})$ and $G(\textbf{M})$, every (positive) edge is labelled by an $a$-letter, a $q$-letter, or a $\theta$-letter.  As such, these edges are called \textit{$a$-edges}, \textit{$q$-edges}, and \textit{$\theta$-edges}, respectively.


Given a path $\textbf{p}$ in diagram over $M(\textbf{M})$ or $G(\textbf{M})$, the (combinatorial) length of $\textbf{p}$ ({\frenchspacing i.e. the number of $a$-, $q$-, or $\theta$-edges comprising it) is denoted $\|\textbf{p}\|$.  The \textit{$a$-length} $|\textbf{p}|_a$ and \textit{$q$-length} $|\textbf{p}|_q$ are then defined in the much the same way as they were for admissible words, while the \textit{$\theta$-length} $|\textbf{p}|_\theta$ is defined analogously. Cells corresponding to $(\theta,q)$-relations, $(\theta,a)$-relations, or the hub relation are called \textit{$(\theta,q)$-cells}, \textit{$(\theta,a)$-cells}, or \textit{hubs}, respectively.

In the general setting of a reduced diagram $\Delta$ over an arbitrary presentation $\pazocal{P}=\gen{X\mid\pazocal{R}}$, let $\pazocal{Z}\subseteq X$.  An edge of $\Delta$ is called a $\pazocal{Z}$-edge if its label is in $\pazocal{Z}^{\pm1}$.  For $m\geq1$, a sequence of (distinct) cells $\pazocal{B}=(\Pi_1,\dots,\Pi_m)$ in $\Delta$ is called a \textit{$\pazocal{Z}$-band} of length $m$ if:

\begin{itemize}

\item every two consecutive cells $\Pi_i$ and $\Pi_{i+1}$ have a common boundary $\pazocal{Z}$-edge $\textbf{e}_i$, 

\item $\textbf{e}_{i-1}^{-1}$ and $\textbf{e}_i$ are the only two $\pazocal{Z}$-edges of $\partial\Pi_i$, and

\item $\text{Lab}(\textbf{e}_i)$ are either all positive or all negative.

\end{itemize}

A $\pazocal{Z}$-band is \textit{maximal} if it is not a subsequence of another such band. Extending the definition so that a $\pazocal{Z}$-edge is a $\pazocal{Z}$-band of length $0$, it follows that that every $\pazocal{Z}$-edge in $\Delta$ is contained in a maximal $\pazocal{Z}$-band.

\begin{figure}[H]
\centering
\begin{subfigure}[b]{0.48\textwidth}
\centering
\raisebox{0.5in}{\includegraphics[scale=1.25]{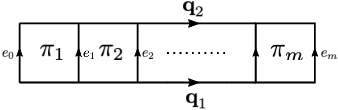}}
\caption{Non-annular $\pazocal{Z}$-band of length $m$}
\end{subfigure}\hfill
\begin{subfigure}[b]{0.48\textwidth}
\centering
\includegraphics[scale=1.4]{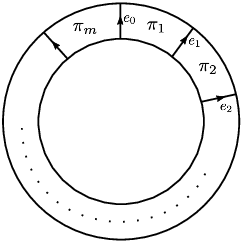}
\caption{Annular $\pazocal{Z}$-band of length $m$ \cite{WEmb}}
\end{subfigure}
\caption{ \ }
\label{fig-bands}
\end{figure}

In any $\pazocal{Z}$-band $\pazocal{B}$ of length $m\geq1$, there exists a closed path $\textbf{e}_0^{-1}\textbf{q}_1\textbf{e}_m\textbf{q}_2^{-1}$ such that $\textbf{q}_1$ and $\textbf{q}_2$ are simple (perhaps closed) paths consisting entirely of edges from the boundaries of the cells comprising $\pazocal{B}$ (see \Cref{fig-bands}). In this case, $\textbf{q}_1$ is called the \textit{bottom side} of $\pazocal{B}$, denoted $\textbf{bot}(\pazocal{B})$, while $\textbf{q}_2$ is called the \textit{top side} of $\pazocal{B}$ and denoted $\textbf{top}(\pazocal{B})$. 

If $\textbf{e}_0=\textbf{e}_m$ in a $\pazocal{Z}$-band $\pazocal{B}$ of length $m\geq1$, then $\pazocal{B}$ is called \textit{$\pazocal{Z}$-annulus}.  Note that in this case $\pazocal{B}$ forms an annular subdiagram of $\Delta$ (see \Cref{fig-bands}(b)); $\pazocal{B}$ is then called a \textit{contractible} $\pazocal{Z}$-annulus if either side forms a contractible loop in $\Delta$.

If $\pazocal{B}$ is a non-annular $\pazocal{Z}$-band of length $m\geq1$, then the closed path $\textbf{e}_0^{-1}\textbf{q}_1\textbf{e}_m\textbf{q}_2^{-1}$ bounds a circular subdiagram of $\Delta$ consisting of the cells comprising $\pazocal{B}$ (see \Cref{fig-bands}(a)).  In this case, we identify this subdiagram with $\pazocal{B}$ and call $\textbf{e}_0^{-1}\textbf{q}_1\textbf{e}_m\textbf{q}_2^{-1}$ is called the \textit{standard factorization} of the contour of $\pazocal{B}$.  If either $(\textbf{e}_0^{-1}\textbf{q}_1\textbf{e}_m)^{\pm1}$ or $(\textbf{e}_m\textbf{q}_2^{-1}\textbf{e}_0^{-1})^{\pm1}$ is a subpath of $\partial\Delta$, then $\pazocal{B}$ is called a \textit{rim $\pazocal{Z}$-band} in $\Delta$.



A $\pazocal{Z}_1$-band and a $\pazocal{Z}_2$-band \textit{cross} if they have a common cell and $\pazocal{Z}_1\cap\pazocal{Z}_2=\emptyset$.




%

In our setting of interest, {\frenchspacing i.e. diagrams over the presentations of groups associated to $\textbf{M}$}, we have $a$-, $q$-, and $\theta$-bands that arise from the corresponding edges.  In particular:

\begin{itemize}

\item An $a$-band is given by taking $\pazocal{Z}=\{a\}$ for some $a\in Y$.

\item A $q$-band is given by taking $\pazocal{Z}=Q_i$ for some $i$.  

\item A $\theta$-band is given by taking $\pazocal{Z}=\{\theta_i\}$ for some $\theta\in\Theta^+$.

\end{itemize}

Note that \Cref{simplify rules} implies an $a$-band consists entirely of $(\theta,a)$-cells, while the definition of the accept configuration implies a $q$-band consists entirely of $(\theta,q)$-cells.  Hence, all of these bands are formed over (the presentation of) the group $M(\textbf{M})$.

It is clear from the definition of the relations that no distinct maximal $a$-bands can intersect, nor can distinct maximal $q$-bands or distinct maximal $\theta$-bands.


If a maximal $a$-band contains a cell with an $a$-edge that is also on the contour of a $(\theta,q)$-cell, then the $a$-band is said to \textit{end} (or \textit{start}) on that $(\theta,q)$-cell and the corresponding $a$-edge is said to be the \textit{end} (or \textit{start}) of the band.  In the analogous way, a maximal $q$-band may end on a hub, while any maximal band ($q$-, $a$-, or $\theta$-) may end on (a component of) the boundary of the diagram. Note that a maximal band that ends in one part of the diagram must also end in another part.  In an annular diagram, a band that has one end on each boundary component is called \textit{radial}. 

The natural projection of the label of the top (or bottom) of a $q$-band onto $F(\Theta^+)$ is called the \textit{history} of the band. The history of an $a$-band can be defined similarly.  The natural projection (without reduction) of the top (or bottom) of a $\theta$-band onto the alphabet $\{Q_0,\dots,Q_s\}^{\pm1}$ is called the \textit{base} of the band.

Suppose the sequence of cells $(\pi_0,\pi_1,\dots,\pi_m)$ comprises a $\theta$-band $\pazocal{T}$ and $(\gamma_0,\gamma_1,\dots,\gamma_\ell)$ a $q$-band $\pazocal{Q}$ such that $\pi_0=\gamma_0$, $\pi_m=\gamma_\ell$, and no other cells are shared. Suppose further that $\partial\pi_0$ and $\partial\pi_m$ both contain $\theta$-edges on the outer countour of the annulus bounded by the two bands. Then the union of these two bands is called a \textit{$(\theta,q)$-annulus} (see \Cref{fig-annular-band}). 

A $(\theta,a)$-annulus is defined similarly.


\begin{figure}[H] 
\centering
\includegraphics[scale=1.45]{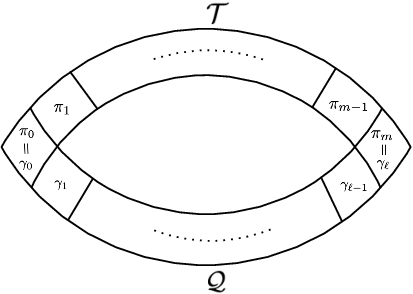} 
\caption{$(\theta,q)$-annulus with defining $\theta$-band $\pazocal{T}$ and $q$-band $\pazocal{Q}$} \label{fig-annular-band}
\end{figure}

The following statement is proved in a more general setting in \cite{O97}:

\begin{lemma}[Lemma 6.1 of \cite{O97}] \label{M(S) annuli}

A reduced circular diagram over $M(\textbf{M})$ contains no:

\begin{enumerate}[label=({\arabic*})]

\item $(\theta,q)$-annuli

\item $(\theta,a)$-annuli

\item $a$-annuli

\item $q$-annuli

\item $\theta$-annuli

\end{enumerate}

\end{lemma}

As a result, in a reduced circular diagram $\Delta$ over $M(\textbf{M})$, if a maximal $\theta$-band and a maximal $q$-band (respectively $a$-band) cross, then their intersection is exactly one $(\theta,q)$-cell (respectively $(\theta,a)$-cell). Further, every maximal $\theta$-band and maximal $q$-band ends on $\partial\Delta$ in two places.

Indeed, the next statement is proved in a more general setting in \cite{W}.

\begin{lemma}[Lemma 8.1 of \cite{W}] \label{G(S) annuli}

A reduced circular diagram over $G(\textbf{M})$ contains no:

\begin{enumerate}

\item $(\theta,q)$-annuli

\item $(\theta,a)$-annuli

\item $a$-annuli

\item $q$-annuli

\end{enumerate}

\end{lemma}

The next statement follows immediately from the ``sewing and detaching procedure" described in \cite{Lyndon-Schupp} (see {\frenchspacing pp. 150-151}) or the concept of ``$0$-cells" introduced in \cite{O}.

\begin{lemma}[Compare with Lemma 4.3 of \cite{GW}] \label{exciseTwoPaths}
Let $\gamma_1$ and $\gamma_2$ be two disjoint simple loops in a circular {\frenchspacing(resp. annular)} diagram $\Delta$ over some group presentation $\pazocal{P}$.  Suppose $\gamma_1$ and $\gamma_2$ have freely conjugate labels and bound an annulus $\Gamma$ that contains at least one (positive) cell.  Then there exists a circular {\frenchspacing(resp. annular)} diagram $\Delta'$ over $\pazocal{P}$ with the same boundary label(s) as $\Delta$ formed by removing any cells `between' $\gamma_1$ and $\gamma_2$.
    
\end{lemma}

%
%

\medskip


\subsection{Trapezia} \label{sec-trapezia} \

These diagrams will be critical in our argument, and we follow \cite[Section 7.3]{WEmb} nearly exactly.

Let $\Delta$ be a reduced circular diagram over $M(\textbf{M})$ whose contour can be factored as $\textbf{p}_1^{-1}\textbf{q}_1\textbf{p}_2\textbf{q}_2^{-1}$, where $\textbf{p}_1$ and $\textbf{p}_2$ are sides of $q$-bands and $\textbf{q}_1$ and $\textbf{q}_2$ are maximal parts of the sides of $\theta$-bands whose labels start and end with $q$-letters. Then $\Delta$ is called a \textit{trapezium}.

In this case, $\textbf{q}_1$ and $\textbf{q}_2$ are called the \textit{bottom} and \textit{top} of the trapezium, respectively, denoted $\textbf{bot}(\Delta)$ and $\textbf{top}(\Delta)$. Further, $\textbf{p}_1$ and $\textbf{p}_2$ are called the \textit{left} and \textit{right} sides of the trapezium. 


The \textit{history} of the trapezium is the history of the rim $q$-band with side $\textbf{p}_2$ and the length of this history is the trapezium's \textit{height}. The base of $\text{Lab}(\textbf{q}_1)$ is called the \textit{base} of the trapezium.

Note that a non-annular $\theta$-band $\pazocal{T}$ of positive length whose first and last cells are $(\theta,q)$-cells can be viewed as a trapezium of height 1.   In this case, the bottom and top of the corresponding trapezium are called the \textit{trimmed} bottom and top of the band, denoted $\textbf{tbot}(\pazocal{T})$ and $\textbf{ttop}(\pazocal{T})$, respectively.

\begin{figure}[H]
\centering
\includegraphics[scale=1.8]{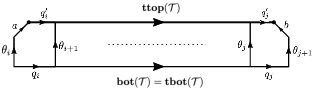}
\caption{$\theta$-band $\pazocal{T}$ with trimmed top}
\end{figure}

Conversely, \Cref{M(S) annuli} implies any trapezium $\Delta$ of height $h\geq1$ is the union of a sequence of $\theta$-bands $\pazocal{T}_1,\dots,\pazocal{T}_h$ connecting the left and right sides of the trapezium, with $\textbf{tbot}(\pazocal{T}_1)$ and $\textbf{ttop}(\pazocal{T}_h)$ making up the bottom and top of $\Delta$, respectively. The top of $\pazocal{T}_i$ is, moreover, essentially adjacent to the bottom of $\pazocal{T}_{i+1}$; formally stated, we mean that $\textbf{ttop}(\pazocal{T}_i)=\textbf{tbot}(\pazocal{T}_{i+1})$ for all $1\leq i\leq h-1$.  In this case, the bands $\pazocal{T}_1,\dots,\pazocal{T}_h$ are said to be \textit{enumerated from bottom to top}.

The following two statements are proved in more generality in \cite{WMal} and exemplify how the group $M(\textbf{M})$ simulates the work of the $S$-machine $\textbf{M}$:

\begin{lemma} \label{trapezia are computations}

Let $\Delta$ be a trapezium with history $H\equiv\theta_1\dots\theta_h$ for $h\geq1$ and maximal $\theta$-bands $\pazocal{T}_1,\dots,\pazocal{T}_h$ enumerated from bottom to top. Letting $W_{j-1}\equiv\lab(\textbf{tbot}(\pazocal{T}_j))$ for $1\leq j\leq h$ and letting $W_h\equiv\lab(\textbf{ttop}(\pazocal{T}_h))$, then there exists a reduced computation $W_0\to\dots\to W_h$ of $\textbf{M}$ with history $H$.

\end{lemma}

\begin{lemma} \label{computations are trapezia}

For any non-empty reduced computation $W_0\to\dots\to W_t$ of $\textbf{M}$ with history $H$, there exists a trapezium $\Delta$ such that: 

\begin{enumerate} [label=(\alph*)]

\item $\lab(\textbf{bot}(\Delta))\equiv W_0$

\item $\lab(\textbf{top}(\Delta))\equiv W_t$

\item The history of $\Delta$ is  $H$

\item $\text{Area}(\Delta)\leq t\max(\|W_0\|,\dots,\|W_t\|)$

\end{enumerate}

\end{lemma}

The next statement, introduced by the authors in \cite{GW}, is vital for the necessary arguments pertaining to annular diagrams, specifically for the case of `spirals'.

\begin{lemma}[Lemma 4.6 of \cite{GW}] \label{history as powers}
Let $\Delta$ be a trapezium with base $B$, history $H\equiv H_0^k$ for $k\geq1$, and maximal $\theta$-bands $\pazocal{T}_1,\dots,\pazocal{T}_h$ enumerated from bottom to top.  Then for $\ell=\|H_0\|$, either:

\begin{enumerate}

\item $\lab(\textbf{tbot}(\pazocal{T}_i))\equiv \lab(\textbf{tbot}(\pazocal{T}_{i+\ell}))$ for all $i$, or 

\item There exists a two-letter subword $UV$ of $B$ such that for $w_0,w_h$ the tape words of $\lab(\textbf{bot}(\Delta)),\lab(\textbf{top}(\Delta))$, respectively, in the $UV$-sector, $k\leq \|w_0\|+\|w_h\|+2\ell$.

\end{enumerate}

\end{lemma}
%

\medskip


\subsection{Modified length function}\label{subsec-modified-length} \

To assist with the main arguments spanning the rest of this article, we now modify the length function on words over the generators of the groups associated to $S$-machines. This is done in the same way as in \cite{O18}, \cite{OS20}, \cite{WEmb}, and \cite{W}. 

The standard length of a word/path is henceforth called its \textit{combinatorial length} and this modified length simply its \textit{length}. 

A word consisting of one $\theta$-letter, no $q$-letters, and at most one $a$-letter is called a \textit{$(\theta,a)$-syllable}. Then, define the length of:

\begin{itemize}
	
	\item any $q$-letter as 1
	
	\item any $(\theta,a)$-syllable as $1$

	\item any $a$-letter as the parameter $\delta$ (see \Cref{sec-parameters})
	
\end{itemize}

Given a word $w$ over $\pazocal{X}$, define a \textit{decomposition} of $w$ to be a factorization of it into a product of letters and $(\theta,a)$-syllables.  The length of a decomposition is then the sum of the lengths of its factors.  The length of $w$, denoted $|w|$, is then taken to be the minimal length of one of its decompositions. 

Of course, the length of a path in a diagram over any presentation with generating set $\pazocal{X}$ is the length of its label.  Further, given $g\in G(\textbf{M})$, the \textit{length} $|g|$ is the minimal length of word that represents $g$.



The next statement gives some basic properties of the length function.

\begin{lemma}[Lemma 6.2 of \cite{OS20}] \label{lengths}

Let \textbf{s} be a path in a diagram $\Delta$ over $G(\textbf{M})$ consisting of $c$ $\theta$-edges and $d$ $a$-edges. Then:

\begin{enumerate}[label=({\alph*})]

\item $|\textbf{s}|\geq\max(c,c+(d-c)\delta)$

\item $|\textbf{s}|=c$ if $\textbf{s}$ is the top or a bottom of a $q$-band

\item For any product $\textbf{s}=\textbf{s}_1\textbf{s}_2$ of two paths in a diagram,
$$|\textbf{s}_1|+|\textbf{s}_2|-\delta\leq|\textbf{s}|\leq|\textbf{s}_1|+|\textbf{s}_2|$$

\item Let $\pazocal{T}$ be a $\theta$-band with base of length $l_b$. If $\textbf{top}(\pazocal{T})$ (or $\textbf{bot}(\pazocal{T})$) has $l_a$ $a$-edges, then the number of cells in $\pazocal{T}$ is between $l_a-l_b$ and $l_a+3l_b$.

\end{enumerate}

\end{lemma}

\begin{lemma}[Lemma 4.8 of \cite{GW}] \label{minBoundaryLengths} 

Let $\Delta$ be a reduced diagram over $G(\textbf{M})$.  Suppose $\Delta$ contains a $q$-band with two ends on the component $\textbf{p}$ of $\partial\Delta$ such that there are no $(\theta,q)$-cells between $\textbf{p}$ and one side of the $q$-band.  Then there exists a diagram $\Delta'$ over $G(\textbf{M})$ with a corresponding boundary component $\textbf{p}'$ such that:

\begin{enumerate}[label=(\alph*)]

\item $\lab(\textbf{p}')$ and $\lab(\textbf{p})$ represent the same element of $G(\textbf{M})$

\item $|\textbf{p}|-|\textbf{p}'|\geq2$

\item For any other component of $\partial\Delta$, there exists a corresponding component of $\partial\Delta'$ with identical label.

\end{enumerate}

\end{lemma}

%
%
%

\begin{lemma}[Compare with Lemma 4.7 of \cite{GW}] \label{minBoundaryLengths-theta}

Let $\Delta$ be a reduced diagram over $G(\textbf{M})$.  Suppose $\Delta$ contains a rim $\theta$-band $\pazocal{T}$ with two ends on the component $\textbf{p}$ of $\partial\Delta$.  If the base of $\pazocal{T}$ has length at most $K$, then there exists a diagram $\Delta'$ over $G(\textbf{M})$ with a corresponding boundary component $\textbf{p}'$ such that:

\begin{enumerate}[label=(\alph*)]

\item $\lab(\textbf{p}')$ and $\lab(\textbf{p})$ represent the same element of $G(\textbf{M})$

\item $|\textbf{p}|-|\textbf{p}'|\geq1$

\item For any other component of $\partial\Delta$, there exists a corresponding component of $\partial\Delta'$ with identical label.

\end{enumerate}

\end{lemma}

\begin{proof}

This follows in just the same way as its analogue in \cite{GW}, using the parameter assignment $\delta^{-1}>>K$.


\end{proof}



\medskip


\subsection{Disks and weights} \label{sec-disks-weights} \

Next we add extra relations to the group $G(\textbf{M})$, all of which can be seen as generalizations of the hub relation.  The addition of these relations produces a presentation which is convenient for studying the associated diagrams, but crucially does not change the group.

Specifically, the relations that we add, termed \textit{disk relations}, are of the form $W=1$ for any configuration $W$ accepted by $\textbf{M}$.  The next essential statement shows that these relations may be deduced from the finite presentation defining $G(\textbf{M})$.

\begin{lemma}[Lemma 7.2 of \cite{OS20}] \label{disks are relations}

If the configuration $W$ of $\textbf{M}$ is accepted by the machine $\textbf{M}$ and $Y_{s+1}=\emptyset$, then the word $W$ is trivial over the group $G(\textbf{M})$.

\end{lemma}

\Cref{disks are relations} shows that the presentation obtained by adding the disk relations is indeed a presentation for $G(\textbf{M})$.  This presentation is called the \textit{disk presentation} of $G(\textbf{M})$, while the defining presentation is henceforth called the \textit{canonical presentation}.

However, as part of our goal is to study the areas of diagrams over finite presentations of the groups, it is important to note that the disk presentation is not necessarily (and almost always definitively not) a finite presentation.  As such, the typical measure of diagram area is not useful in the study of diagrams over the disk presentation.  

To fix this, we generalize the concept of area by defining the \textit{weight} of a diagram over the disk presentation of $G(\textbf{M})$.  For this, we first define a few auxiliary functions.

Recall that in the statement of \Cref{main-theorem}, we assume that the time function $\TM_\textbf{S}$ is asymptotically bounded above by a positive superadditive function $f:\N\to[0,\infty)$.  We now define the three auxiliary functions $g_{div},g_{max},g:\N\to[0,\infty)$ by:

\begin{itemize}

\item $g_{div}(0)=0$ and $g_{div}(n)=\frac{f(n)}{n}$ for all $n\geq1$

\item $g_{max}(n)=\max_{0\leq i\leq n} \lfloor g_{div}(i)\rfloor$ for all $n\in\N$

\item $g(n)=g_{max}(n)^2$ for all $n\in \N$.

\end{itemize}

The next statement sheds light on the assignments of these functions.

\begin{lemma} \label{mixture quotient}

The function $g:\N\to[0,\infty)$ is non-decreasing with $g(n)\geq1$ for all $n\geq1$ and $n^2g(n)\sim f(n)^2$.

\end{lemma}

\begin{proof}

It is clear from the construction that $g$ is non-decreasing.

Fix $n\geq1$.  Since we assume $f(1)\geq1$, it follows from the superadditivity of $f$ that $f(n)\geq n$, and so $g_{div}(n)\geq1$.  This implies $g_{max}(n)\geq1$, and so $g(n)\geq1$.

Now let $m$ be the maximal index in $\{1,\dots,n\}$ such that $g_{max}(n)=\lfloor g_{div}(m)\rfloor$.  Note that since $f$ is superadditive, $g_{div}(2m)=\frac{f(2m)}{2m}\geq\frac{f(m)}{m}=g_{div}(m)$, and so $2m>n$.  Hence,
\begin{align*}
f(n)^2&=n^2g_{div}(n)^2\leq n^2(g_{max}(n)+1)^2\leq4n^2g_{max}(n)^2\leq4(2m)^2g_{div}(m)^2=16f(m)^2\leq16f(n)^2
\end{align*}
Hence, $\frac{1}{4}f(n)^2\leq n^2g(n)\leq 4f(n)^2$, and thus the statement follows.

\end{proof}

Per the hypotheses of \Cref{main-theorem} and the conclusion of \Cref{mixture quotient}, we choose the parameter $c_0$ to be large enough so that for all $n\in\N$:
\begin{itemize}

\item $f(n)^2\leq c_0n^2g(c_0n)+c_0n+c_0$

\item $n^2g(n)\leq c_0f(c_0n)^2+c_0n+c_0$


\item $\TM_\textbf{S}(n)\leq c_0f(c_0n)+c_0n+c_0$

\end{itemize}

However, note that the length function defined in \Cref{subsec-modified-length} does not take integer values, making it inconvenient to restrict to functions with domain $\N$.  To rectify this, we extend $g$ in any way so that it is a non-decreasing function on the non-negative reals with $g(x)\geq1$ for all $x\neq0$ (for example, define the value at $x$ to be $g(\lceil x \rceil)$).  Finally, define the function $\phi$ on the non-negative reals by $\phi(x)=x^2g(x)$.

\begin{lemma} \label{phi properties}

Fix non-negative reals $x\geq y$.

\begin{enumerate}

\item $\phi(x+y)\geq\phi(x)+\phi(y)$

\item $\phi(x)-\phi(x-y)\geq xyg(x)$

\end{enumerate}

\end{lemma}

\begin{proof}

(1) $\phi(x+y)=(x+y)^2g(x+y)\geq(x^2+y^2)g(x+y)\geq x^2g(x)+y^2g(y)=\phi(x+y)$.

(2) Since $(x-y)^2=x^2-xy-y(x-y)\leq x^2-xy$, we have $$\phi(x-y)\leq (x^2-xy)g(x-y)\leq(x^2-xy)g(x)=\phi(x)-xyg(x)$$

\end{proof}

Now set the weight of a (positive) cell $\Pi$ in a diagram over the disk presentation of $G(\textbf{M})$ to be:

$
\begin{array}{ll}
      \bullet \ \text{wt}(\Pi)=C_1\phi(C_1|\partial\Pi|) & \ \text{if $\Pi$ is a disk} \\
      \bullet \ \text{wt}(\Pi)=1 & \ \text{otherwise} \\
   \end{array}
$

We then define the \textit{weight} of a diagram $\Delta$ over the disk presentation of $G(\textbf{M})$ to be the sum $\text{wt}(\Delta)$ of the weights of all of its (positive) cells.

\medskip


\subsection{Mixtures} \

We now recall an invariant of the relevant reduced circular diagrams, first introduced in \cite{OS12}, that is invaluable for the numerical estimates pertaining to the Dehn function that follow. This technical tool is used in various related previous literature, {\frenchspacing e.g. \cite{O18}, \cite{OS20}, \cite{WEmb}, and \cite{W}}.

Consider a circle $O$ in the Euclidean plane.  We place an auxiliary structure on $O$ by introducing a finite set of points partitioned into two sets.  With this, we call $O$ a \textit{necklace} and the corresponding points \textit{white beads} and \textit{black beads}.

	
	Given distinct white beads $o_1$ and $o_2$ on $O$, define $\ell(o_1,o_2)$ to be the number of black beads in the counterclockwise subarc $O$ from $o_1$ to $o_2$.  We may extend this to define a function on all pairs of white beads by setting $\ell(o,o)=0$ for any white bead $o$.  Note that, in general, $\ell(o_1,o_2)\neq\ell(o_2,o_1)$.
	
	For each $j\in\N$, define the set  $P_j=\{(o_1,o_2):\ell(o_1,o_2)\geq j\}$.  Finally, for $J$ the parameter specified in \Cref{sec-parameters}, define the \textit{$J$-mixture} of the necklace $O$ to be $\mu_J(O)=\sum\limits_{j=1}^J \#P_j$.
		
	
	
	\begin{lemma}[Lemma 6.1 of \cite{OS12}] \label{mixtures}
		
		Let $O$ be a necklace with $x$ white beads and $y$ black beads.
		
		\begin{enumerate}
			
			\item $\mu_J(O)\leq J(x^2-x)$
			
			\item If $O'$ is a necklace obtained from $O$ through the removal of one white bead, then for every $j$, $\# P_j-2x<\#P_j'\leq\#P_j$, and so $\mu_J(O)-2Jx<\mu_J(O')\leq\mu_J(O)$
			
			\item If $O'$ is a necklace obtained from $O$ through the removal of one black bead, then for every $j$, $\#P_j'\leq\#P_j$, and so $\mu_J(O')\leq\mu_J(O)$
			
			\item Suppose $v_1,v_2,v_3$ are three black beads on $O$ such that the counterclockwise arc from $v_1$ to $v_3$, $v_1 - v_3$, has at most $J$ black beads (excluding $v_1$ and $v_3$). Let $y_1$ and $y_2$ be the number of white beads on the counterclockwise arcs $v_1 - v_2$ and $v_2 - v_3$, respectively. If $O'$ is the necklace obtained from $O$ through the removal of $v_2$, then $\mu_J(O')\leq\mu_J(O)-y_1y_2$.
			
		\end{enumerate}
		
	\end{lemma}
	
	%
	%
	%
	%
	%
	
	Let $\textbf{M}$ be a cyclic recognizing $S$-machine and $\Delta$ be a reduced circular diagram over a presentation with generating set $\pazocal{X}$ (see \Cref{sec-groups}).  Let $O$ be a circle partitioned by subarcs corresponding to the edges of $\partial\Delta$.  Place the white beads (respectively black beads) at the midpoints of subarcs corresponding to $\theta$-edges (respectively $q$-edges). Then, define the \textit{mixture on $\Delta$}, $\mu(\Delta)$, to be the $J$-mixture of the corresponding necklace, i.e $\mu(\Delta)=\mu_J(O)$.

\bigskip


\section{Circular diskless diagrams} \label{sec-diskless-circular}

In this section we study circular diagrams over $M(\textbf{M}_\textbf{S})$ with the goal of bounding the `area' of such a diagram in terms of its perimeter (measured with the length function defined in \Cref{subsec-modified-length}).  However, this is not achieved with respect to the area of the diagram, but rather in terms of the `$G$-area', a concept developed in \cite{O18} and \cite{OS20}.

The argument outlined in this section proceeds in much the same way as Section 6 of \cite{OS20}.  As such, many of the proofs are either omitted with reference to the proof of their analogue or shortened to a brief explanation of any deviations in the estimates.

\medskip


\subsection{Combs} \label{sec-combs} \

In this section, we introduce a type of circular diagram which helps apply inductive arguments for our combinatorial goals. Our exposition follows \cite[Section 9.6]{WEmb}.

\begin{definition}[c.f. Section 9.6 of \cite{WEmb}]
	
	Let $\Gamma$ be a reduced circular diagram over $M(\textbf{M}_\textbf{S})$ containing a maximal $q$-band $\pazocal{Q}$ such that $\textbf{bot}(\pazocal{Q})$ is a subpath of $\partial\Gamma$ and every maximal $\theta$-band of $\Gamma$ ends at an edge of $\textbf{bot}(\pazocal{Q})$. Then $\Gamma$ is called a \textit{comb} and $\pazocal{Q}$ its \textit{handle}.  The \textit{height} of $\Gamma$ is the length of its handle, while the \textit{basic width} is the maximal base length of one of its maximal $\theta$-bands. 
Given a maximal $\theta$-band $\pazocal{T}$ in a comb $\Gamma$, let $B$ be the base read toward the handle.  Then $\pazocal{T}$ is called:

\begin{itemize}

\item \textit{simple} if $B$ is a product of distinct letters

\item \textit{tight} if $B$ has a revolving suffix and any proper prefix of $B$ is simple

\end{itemize}

If every maximal $\theta$-band of $\Gamma$ is either simple or tight with at least one being tight, then $\Gamma$ is called a \textit{tight comb}.

\end{definition}

\begin{lemma}[Lemma 6.7 of \cite{OS20}] \label{comb area}

Let $\Gamma$ be a comb with height $h$, basic width $b$, and $|\partial\Gamma|_a=\a$. Let $\pazocal{T}_1,\dots,\pazocal{T}_h$ be the maximal $\theta$-bands of $\Gamma$ enumerated from bottom to top. Factor $\partial\Gamma=\textbf{y}\textbf{x}\textbf{z}$, where $\textbf{z}$ is the bottom of the handle of $\Gamma$ and $\textbf{x}$ is the maximal subpath below $\pazocal{T}_1$. Then:

\begin{enumerate}[label=({\arabic*})]

\item $|\textbf{y}|_a\leq\frac{1}{2}\a+bh$

\item $\text{Area}(\Gamma)\leq c_0bh^2+2\a h$

\end{enumerate}

\end{lemma}

\begin{definition}

Let $\Delta$ be a reduced circular diagram over the disk presentation of $G(\textbf{M}_\textbf{S})$.  Suppose the maximal $q$-band $\pazocal{Q}$ has two ends on $\partial\Delta$ and one of the two subdiagrams obtained from cutting along a side of $\pazocal{Q}$ is a comb $\Gamma$.  Then the subdiagram $\Gamma$ is called a \textit{subcomb} of $\Delta$.

\end{definition}

\begin{lemma}[Lemma 6.10 of \cite{OS20}] \label{tight subcomb existence}

Let $\Delta$ be a reduced circular diagram of non-zero area over the disk presentation of $G(\textbf{M}_\textbf{S})$ such that every maximal $\theta$-band of $\Delta$ has base of length at least $K$.  Suppose either:

\begin{enumerate}

\item $\Delta$ is a diagram over $M(\textbf{M}_\textbf{S})$, or

\item $\Delta$ has a subcomb with basic width at least $K_0$.

\end{enumerate}

Then $\Delta$ has a tight subcomb $\Gamma$.

\end{lemma}

\begin{definition}

Let $\Gamma$ be a comb with handle $\pazocal{Q}$ and let $\pazocal{Q}'$ be a maximal $q$-band of $\Gamma$ distinct from $\pazocal{Q}$.  Then there exists a subdiagram $\Gamma'$ of $\Gamma$ which is a comb with handle $\pazocal{Q}'$, obtained by cutting along a side of $\pazocal{Q}'$.  If any $\theta$-band connecting $\pazocal{Q}'$ to $\pazocal{Q}$ has no $(\theta,q)$-cells (besides where it crosses $\pazocal{Q}'$ and $\pazocal{Q}$), then $\Gamma'$ is called a \textit{derivative subcomb} of $\Gamma$.

\end{definition}

\medskip


\subsection{Big trapezia and $G$-area} \label{sec-big-trapezia-G-weight} \

We now alter how we count the area of a circular diagram over the disk presentation of $G(\textbf{M}_\textbf{S})$.  This alteration reflects the ability to `replace' certain subdiagrams with diagrams over the finite presentation of $G(\textbf{M}_\textbf{S})$ whose area is lower than the subdiagram's weight.

\begin{definition}
	
	A trapezium $\Gamma$ over $M(\textbf{M}_\textbf{S})$  with revolving base and history containing a controlled subword is called \textit{big}.  In this case, the \textit{$G$-area} of $\Gamma$, denoted $\text{Area}_G(\Gamma)$, is defined to be the minimum of half its area and the value $3h+C_2\phi(C_2M)$, where $h$ is the height of the trapezium and $M=\max(\|\textbf{bot}(\Gamma)\|,\|\textbf{top}(\Gamma)\|)$.
	
\end{definition}

Now, given a reduced diagram $\Delta$ over the disk presentation of $G(\textbf{M}_\textbf{S})$, consider a family of subdiagrams $\textbf{P}$ where:

\begin{itemize}

\item if $P\in\textbf{P}$, then $P$ is either a single cell or a big trapezium,

\item every cell of $\Delta$ belongs to an element of $\textbf{P}$, and

\item if there exist $P_1,P_2\in\textbf{P}$ with nonempty intersection, then both $P_1$ and $P_2$ are big trapezia and this intersection is a $q$-band.

\end{itemize}

In this case, $\textbf{P}$ is called a \textit{covering} of $\Delta$. The $G$-area of $\textbf{P}$, $\text{Area}_G(\textbf{P})$, is defined to be the sum of the $G$-area of its elements, where the $G$-area of a single cell is taken to be its weight.

Note that any reduced diagram over the disk presentation of $G(\textbf{M}_\textbf{S})$ has a covering, namely the one given by its cells. So, we may define the $G$-area of $\Delta$, $\text{Area}_G(\Delta)$, to be the minimum of the $G$-areas of its coverings.  

Further, any cell of a reduced diagram over $G(\textbf{M}_\textbf{S})$ belongs to at most two elements of a covering, in which case these elements are big trapezia.  So, since the $G$-area of a big trapezium does not exceed half of its weight, the $G$-area of any covering is at most the weight of the diagram.


The next statement is a crucial tool in achieving upper bounds on the $G$-area, and is proved in a more general setting in \cite{W}.

\begin{lemma}[Lemma 6.15 of \cite{OS20}] \label{G-area subdiagrams}

Let $\Delta$ be a reduced diagram over $G(\textbf{M}_\textbf{S})$ and suppose every cell $\pi$ of $\Delta$ belongs in one of the subdiagrams $\Delta_1,\dots,\Delta_m$, where any nonempty intersection $\Delta_i\cap\Delta_j$ is a $q$-band. Then $\text{Area}_G(\Delta)\leq\sum_{i=1}^m\text{Area}_G(\Delta_i)$.

\end{lemma}

The purpose of this modification is to allow for a uniform estimate for the area of a trapezium with revolving base.  This uniform bound is explicitly given by the next statement.

\begin{lemma} \label{G-area revolving trapezia}

If $\Gamma$ is a trapezium with revolving base and height $h$, $\text{Area}_G(\Gamma)\leq c_3hM+C_2\phi(C_2M)$ for $M=\max(\|\textbf{bot}(\Gamma)\|,\|\textbf{top}(\Gamma)\|)$.

\end{lemma}

\begin{proof}

If $\Gamma$ is big, then the bound is given by taking the covering consisting only of $\Gamma$ (and assuming $c_3\geq3$).  Hence, we may assume $\Gamma$ is not big.  

Let $\pazocal{C}:W_0\to\dots\to W_h$ be the reduced computation of $\textbf{M}_\textbf{S}$ given by \Cref{trapezia are computations}.  If the history of $\Gamma$ has a controlled subword, then the history of $\pazocal{C}$ also does.  But then \Cref{enhanced controlled} implies the base of $\Gamma$ is reduced, and so $\Gamma$ is big.  

Note that $\pazocal{C}$ can be identified with a reduced computation of the $k$-enhanced machine $\textbf{E}_{\textbf{S},k}$ with circular base $B$.  Let $\pazocal{C}':W_0'\to\dots\to W_h'$ be the restriction of this computation to a revolving base of the $k$-enhanced machine.

If the base of $\pazocal{C}'$ is reduced, then \Cref{reduced revolving controlled} implies $h\leq c_1\max(\|W_0'\|,\|W_h'\|)$.  Since the base of this computation has length $N+1$, the parameter choices $c_2>>c_1>>N$ and \Cref{simplify rules} imply $\|W_i'\|\leq c_2\max(\|W_0'\|,\|W_h'\|)$ for all $i$.

Otherwise, $\pazocal{C}'$ is a reduced computation of $\textbf{E}_{\textbf{S},k}$ with a faulty base, so that \Cref{faulty} implies $|W_i'|_a\leq c_1\max(|W_0'|_a,|W_h'|_a)$ for all $i$, so that again $\|W_i'\|\leq c_2\max(\|W_0'\|,\|W_h'\|)$ for all $i$.

Removing the subwords $W_i'$ except for the first (or last) state letter then produces admissible words with shorter circular base forming a corresponding reduced computation.  We then iterate this process until there is only a single state letter remaining.  As it takes at most $\|B\|\leq\max(\|W_0\|,\|W_t\|)$ iterations to complete, we find that $\|W_i\|\leq 4c_2\max(\|W_0\|,\|W_h\|)$ for all $i$.  \Cref{computations are trapezia}(d) thus implies $\text{Area}(\Gamma)\leq 4c_2hM$, so that the statement is given by the parameter choice $c_3>>c_2$.

\end{proof}

\medskip


\subsection{Upper bound on $G$-area} \label{sec-diskless-upper-bound} \

We now work toward obtaining an upper bound on the $G$-area of a reduced circular diagram over $M(\textbf{M}_\textbf{S})$.  Specifically, our goal through the rest of this section is to show that the $G$-area of a reduced circular diagram $\Delta$ over $M(\textbf{M}_\textbf{S})$ is bounded by the inequality $$\text{Area}_G(\Delta)\leq N_2\phi(N_2|\partial\Delta|)+N_1\mu(\Delta)g(N_1|\partial\Delta|)$$  To show this, we consider a `minimal counterexample' diagram $\Delta$, that is, a reduced circular diagram $\Delta$ over $M(\textbf{M}_\textbf{S})$ with $n=|\partial\Delta|$ minimal such that
$$\text{Area}_G(\Delta)> N_2\phi(N_2n)+N_1\mu(\Delta)g(N_1n)$$

For the sake of the combinatorial arguments that follow, it should be noted that necessarily $n>0$.  Hence, $g(Cn)\geq1$ for any positive constant $C$.

As previously mentioned, the argument proceeds in much the same way as in \cite{OS20}; however, the estimates deviate throughout, reflecting the different goal: The desired bound in the setting of \cite{OS20} can be interpreted as being given for the particular functions $\phi(n)=n^2$ and $g(n)=1$.

Note that the estimates made here are closer to those presented in Section 6 of \cite{O18}.  Moreover, this strategy is employed in a more general setting in \cite{W} and \cite{WEmb}.

\begin{lemma}[See Lemma 9.24 of \cite{WEmb}] \label{counterexample combs} \

\begin{enumerate}[label=({\arabic*})]

\item $\Delta$ has no two disjoint subcombs $\Gamma_1$ and $\Gamma_2$ of basic widths at most $K$ with handles $\pazocal{B}_1$ and $\pazocal{B}_2$ such that some ends of these handles are connected by a subpath $\textbf{x}$ of $\partial\Delta$ with $|\textbf{x}|_q\leq N$.

\item If $\Gamma$ is a subcomb of $\Delta$ with basic width $s\leq K$, $|\partial\Gamma|_q=2s$.

\end{enumerate}

\end{lemma}

\renewcommand\thesubfigure{\arabic{subfigure}}
\begin{figure}[H]
\centering
\begin{subfigure}[b]{0.48\textwidth}
\centering
\includegraphics[scale=1]{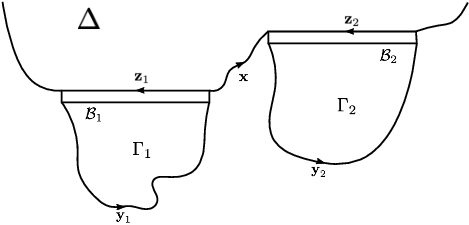}
\caption{ \ }
\end{subfigure}\hfill
\begin{subfigure}[b]{0.48\textwidth}
\centering
\includegraphics[scale=1.225]{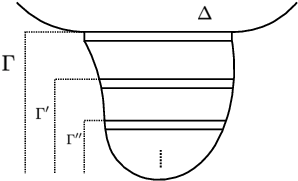}
\caption{ \ }
\end{subfigure}
\caption{Lemma \ref{counterexample combs}}
\end{figure}

\begin{proof}

As in the proof of the analogous statement in \cite{OS20}, we prove the statements simultaneously, inducting on the value $A$ defined to be $A=\text{Area}(\Gamma_1)+\text{Area}(\Gamma_2)$ for (1) and $A=\text{Area}(\Gamma)$ for (2).

Assume there exists a minimal counterexample of form (1).  By the inductive hypothesis, $\Gamma_1$ and $\Gamma_2$ both satisfy (2).  Let $h_i$ be the height of $\Gamma_i$ and assume without loss of generality that $h_1\leq h_2$.

Let $\Delta'$ be the subdiagram of $\Delta$ obtained by cutting along the bottom of the handle of $\Gamma_1$.  Then as in the proof of \cite{OS20}, \Cref{comb area}, \Cref{mixtures}, and the parameter assignments imply:

\begin{itemize}

\item $\mu(\Delta)-\mu(\Delta')\geq h_1h_2$

\item $\text{Area}(\Gamma_1)\leq Jh_1^2+2h_1|\partial\Gamma_1|_a$

\item $|\partial\Delta|-|\partial\Delta'|\geq\gamma=\max(2,\delta(|\partial\Gamma_1|_a-2h_1))$

\end{itemize}

As $\Delta$ is a minimal counterexample, we have $\text{Area}_G(\Delta')\leq N_2\phi(N_2(n-\gamma))+N_1\mu(\Delta')g(N_1(n-\gamma))$.  As $n\geq\gamma$, \Cref{phi properties}(2) then implies $\phi(N_2(n-\gamma))\leq\phi(N_2n)-N_2^2n\gamma g(N_2n)$.

So, since $g$ is non-decreasing, $\text{Area}_G(\Delta')\leq N_2\phi(N_2n)-N_2^3n\gamma g(N_2n)+N_1(\mu(\Delta)-h_1h_2)g(N_1n)$.  As such, $\text{Area}_G(\Delta')\leq \text{Area}_G(\Delta)-N_2^3n\gamma g(N_2n)-N_1h_1h_2g(N_1n)$.

By \Cref{G-area subdiagrams}, though, we have $\text{Area}_G(\Delta)\leq\text{Area}_G(\Delta')+\text{Area}_G(\Gamma_1)\leq\text{Area}_G(\Delta')+\text{Area}(\Gamma_1)$.  Hence, as $h_1\leq h_2$ we reach a contradiction if 
\begin{equation}\label{eqn-two-comb}
Jh_1^2+2h_1|\partial\Gamma_1|_a<N_2^3n\gamma g(N_2n)+N_1h_1^2g(N_1n)
\end{equation}

If $|\partial\Gamma_1|_a\leq4h_1$, then the left side of (\ref{eqn-two-comb}) is at most $(J+8)h_1^2$.  But $g(N_1n)\geq1$, so that the inequality follows from the parameter choice $N_1>>J$.

Otherwise, assuming $|\partial\Gamma_1|_a\geq4h_1$, then $n\geq\gamma\geq\frac{1}{2}\delta|\partial\Gamma_1|_a$.  As such, $g(N_2n)\geq1$ and the parameter choice $N_2>>\delta^{-1}$ yield $$N_2^3n\gamma g(N_2n)\geq \frac{N_2^3\delta^2}{4}|\partial\Gamma_1|_a^2\geq N_2^3\delta^2h_1|\partial\Gamma_1|_a\geq 2h_1|\partial\Gamma_1|_a$$
Hence, (\ref{eqn-two-comb}) again follows from noting $g(N_1n)\geq1$ and the parameter choice $N_1>>J$.

Thus, we may assume there exists a minimal counterexample of form (2).  But then as in the proof of the analogous statement in \cite{OS20}, there is at most one derivative subcomb $\Gamma'$ in $\Gamma$, then at most one derivative subcomb $\Gamma''$ in $\Gamma'$, etc.  The statement then follows immediately.

\end{proof}

\begin{lemma}[Compare with Lemma 6.17 of \cite{OS20}] \label{lem-bigsubcomb}

Suppose $\Gamma$ is a subcomb of $\Delta$ whose basic width is at most $K_0$ and whose handle $\pazocal{B}$ has length $\ell$. If $\Gamma'$ is a subcomb of $\Gamma$ with handle $\pazocal{B}'$ of length $\ell'$, then $\ell'>\ell/2$.

\end{lemma}

\begin{proof}

Arguing toward contradiction, pick $\Gamma'$ so that $\ell'\leq\ell/2$ is minimal.  As such, $\Gamma'$ must have basic width 1.  Hence, letting $\Delta'$ be the subdiagram obtained from $\Delta$ by removing $\Gamma'$, then as in \cite{OS20} we have:

\begin{itemize}

\item $\text{Area}(\Gamma')\leq c_0(\ell')^2+2\ell'|\partial\Gamma'|_a$

\item $|\partial\Delta|-|\partial\Delta'|\geq\gamma=\max(2,\delta(|\partial\Gamma'|_a-2\ell'))$

\item $\mu(\Delta)-\mu(\Delta')\geq(\ell')^2$

\end{itemize}

\begin{figure}[H]
\centering
\includegraphics[scale=0.85]{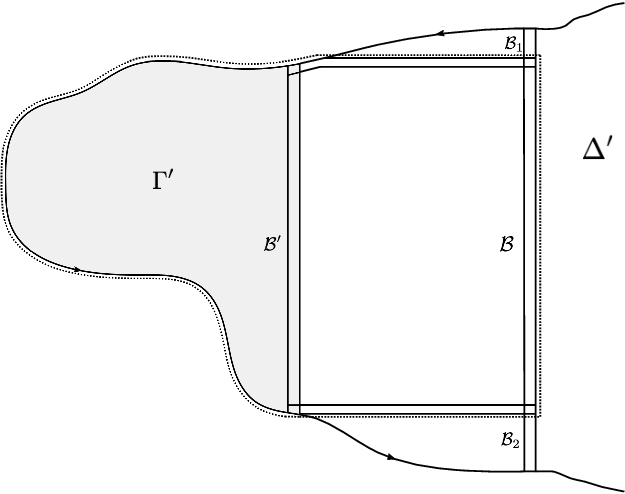}
\caption{Lemma \ref{lem-bigsubcomb}}
\label{fig-bigsubcomb}
\end{figure}

So, using the minimality of $\Delta$ as a counterexample,
$$\text{Area}_G(\Delta')\leq
N_2\phi(N_2(n-\gamma))+N_1(\mu(\Delta)-(\ell')^2)g(N_1n)$$
As in the proof of the last statement, \Cref{phi properties}(2) implies $\phi(N_2(n-\gamma))\leq\phi(N_2n)-N_2^2n\gamma g(N_2n)$, so that it suffices to show
\begin{equation}\label{eqn-bigsubcomb}
c_0(\ell')^2+2\ell'|\partial\Gamma'|_a<N_2^3n\gamma g(N_2n)+N_1(\ell')^2g(N_1n)
\end{equation}
If $|\partial\Gamma'|_a\leq4\ell'$, then the left side is at most $(c_0+8)(\ell')^2$, so that the inequality is given by the parameter assignment $N_1>>c_0$ and the observation $g(N_1n)\geq1$.

Otherwise, $n\geq\gamma\geq\frac{1}{2}\delta|\partial\Gamma'|_a$, so that the parameter assignment $N_2>>\delta^{-1}$ yields
$$N_2^3n\gamma g(N_2n)\geq\frac{N_2^3\delta^2}{4}|\partial\Gamma'|_a\geq N_2^3\delta^2\ell'|\partial\Gamma'|_a\geq2\ell'|\partial\Gamma'|_a$$
and thus (\ref{eqn-bigsubcomb}) is again given by $N_1>>c_0$.

\end{proof}

\begin{lemma}[Compare to Lemma 6.18 of \cite{OS20}] \label{short theta-bands}

If $\pazocal{T}$ is a rim $\theta$-band in $\Delta$, then the base of $\pazocal{T}$ has length $s>K$.

\end{lemma}

\begin{proof}

If $s\leq K$, then \Cref{minBoundaryLengths-theta} implies the diagram $\Delta'$ obtained from cutting off $\pazocal{T}$ satisfies $|\partial\Delta|-|\partial\Delta'|\geq1$.  As such, the minimality of $\Delta$ as a counterexample implies
$$\text{Area}_G(\Delta')\leq N_2\phi(N_2(n-1))+N_1\mu(\Delta)g(N_1(n-1))\leq\text{Area}_G(\Delta)-N_2^3ng(N_2n)$$
So, since $g(N_2n)\geq1$, it follows from \Cref{G-area subdiagrams} that it suffices to show that the length of $\pazocal{T}$ is less than $N_2^3n$.

But as in \cite{OS20}, \Cref{lengths} implies the length of $\pazocal{T}$ is at most $3s+\delta^{-1}(n-s)\leq\delta^{-1}n$, so that the desired inequality follows from the parameter choice $N_2>>\delta^{-1}$.

\end{proof}

We now reach the desired contradiction, obtaining the desired upper bound on $G$-areas.

\begin{lemma}[Compare with Lemma 6.19 of \cite{OS20}] \label{circular diskless}

The counterexample diagram $\Delta$ does not exist.

\end{lemma}

\begin{proof}

\Cref{short theta-bands} implies the counterexample diagram $\Delta$ satisfies the hypotheses of \Cref{tight subcomb existence}, and so there exists a tight subcomb $\Gamma$.  Note that the basic width of $\Gamma$ is then at most $2LN+1\leq K_0$ by the parameter choices $K_0>>L>>N$.

As in \cite{OS20}, we decompose $\Gamma$ into four subdiagrams $\Gamma_1,\Gamma_2,\Gamma_3,\Gamma_4$ where:

\begin{itemize}

\item $\Gamma_2$ is a trapezium with revolving base,

\item $\Gamma_3$ and $\Gamma_4$ are combs,

\item $\Gamma_1$ is a comb whose handle, $\pazocal{Q}'$, is one of the boundary maximal $q$-bands of $\Gamma_2$, and

\item The handle $\pazocal{Q}_3$ of $\Gamma_3$, the handle $\pazocal{Q}_4$ of $\Gamma_4$, and the boundary maximal $q$-band $\pazocal{Q}_2$ of $\Gamma_2$ distinct from $\pazocal{Q}'$ comprise a decomposition of the handle $\pazocal{Q}$ of $\Gamma$ (see \Cref{fig-Mcounterexample}).

\end{itemize}

\begin{figure}[H]
\centering
\includegraphics[height=3.5in]{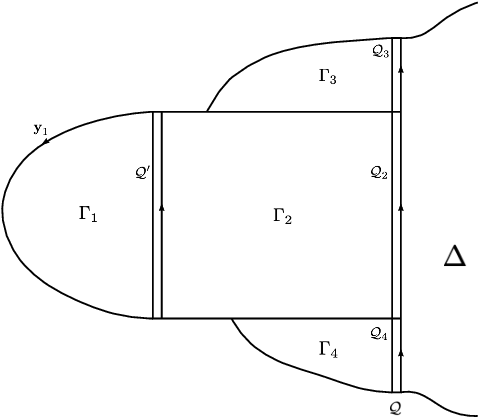}
\caption{Tight subcomb $\Gamma$}
\label{fig-Mcounterexample}
\end{figure}

It should be noted, however, that the subdiagrams $\Gamma_3$ and $\Gamma_4$ need not exist should the band $\pazocal{Q}_2$ have an end on $\partial\Delta$.

Let $\ell$ be the length of $\pazocal{Q}$, $\ell'$ be the length of $\pazocal{Q}'$, and $\ell_i$ be the length of $\pazocal{Q}_i$ for $i=2,3,4$.  We then have $\ell'=\ell_2$, $\ell=\ell'+\ell_3+\ell_4$, and $\ell'>\ell/2$ by \Cref{lem-bigsubcomb}.

Now let $A_i$ be the $G$-area of $\Gamma_i$ and $\a_i=|\partial\Gamma_i|_a$.  Then \Cref{comb area} implies $A_i\leq c_0K_0\ell_i^2+2\ell_i\a_i$ for $i=3,4$, while \Cref{G-area revolving trapezia} implies $A_2\leq c_3\ell'M+C_2\phi(C_2M)$ for $M=\max(\|\textbf{bot}(\Gamma_2)\|,\|\textbf{top}(\Gamma_2)\|)$.

Let $d_3=|\textbf{top}(\Gamma_2)|_a$, $d_4=|\textbf{bot}(\Gamma_2)|_a$, and $d=d_3+d_4$.  Further, let $d_3'$ and $d_4'$ are the number of such $a$-edges of $\textbf{top}(\Gamma_2)$ and $\textbf{bot}(\Gamma_2)$, respectively, which are also boundary edges of $\Delta$.  We then have $M\leq d+K_0$, so that the parameter choices $J>>K_0>>c_3$ yield $c_3M\leq J(d+1)$.  Similarly, $C_3>>C_2>>K_0$ yields $C_2M\leq C_3(d+1)$.  Hence, $A_2\leq J\ell'(d+1)+C_2\phi(C_3(d+1))$.

We then perform the following surgery: We cut $\Gamma$ off $\Delta$ along the top of the handle $\pazocal{Q}$ to produce the diagram $\Delta_1$, then paste $\Gamma_1$ to $\Delta_1$ by identifying $\pazocal{Q}'$ with $\pazocal{Q}_2$ (note this is possible since the base of $\Gamma_2$ is revolving).  The resulting diagram is denoted $\Delta_0$ (see \Cref{fig-diskless}).

\begin{figure}[H]
\centering
\includegraphics[height=3in]{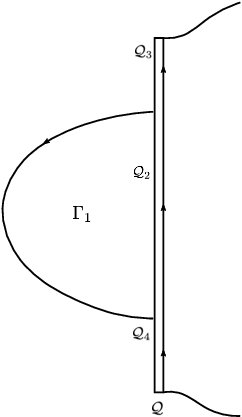}
\caption{The construction of $\Delta_0$}
\label{fig-diskless}
\end{figure}

An identical argument to Lemma 6.20 of \cite{OS20} implies $\text{Area}_G(\Delta_0)\geq \text{Area}_G(\Delta_1)+A_1-\ell'$.  As a result, \Cref{G-area subdiagrams} implies
$$\text{Area}_G(\Delta)\leq\text{Area}_G(\Delta_1)+A_1+A_2+A_3+A_4\leq \text{Area}_G(\Delta_0)+A_2+A_3+A_4+\ell'$$
As in the proof in \cite{OS20}, we have $|\partial\Delta|-|\partial\Delta_0|\geq2$.  Indeed, for $n_0=|\partial\Delta_0|$, it follows from \Cref{comb area}(1) that
\begin{equation}\label{eqn-diskless-perimeter}
n-n_0\geq2+\delta\max(0,d_3'-1,|\partial\Gamma_3|_a-(d_3-d_3')-2\ell_3)+\delta\max(0,d_4'-1,|\partial\Gamma_4|_a-(d_4-d_4')-2\ell_4)
\end{equation}
Hence, we may apply the inductive hypothesis to $\Delta_0$ to estimate $$\text{Area}_G(\Delta_0)\leq N_2\phi(N_2n_0)+N_1\mu(\Delta_0)g(N_1n_0)\leq N_2\phi(N_2n)-N_2^3n(n-n_0)g(N_2n)+N_1\mu(\Delta_0)g(N_1n)$$
But also $\mu(\Delta_0)\leq\mu(\Delta)-\ell'(\ell-\ell')$, so that 
$$\text{Area}_G(\Delta_0)\leq\text{Area}_G(\Delta)-N_2^3n(n-n_0)g(N_2n)-N_1\ell'(\ell-\ell')g(N_1n)$$

Hence, we've reached a contradiction if we show
\begin{equation} \label{eqn-diskless1}
N_1\ell'(\ell-\ell')g(N_1n)+N_2^3n(n-n_0)g(N_2n)>A_2+A_3+A_4+\ell'
\end{equation}

Since $\ell'>\ell/2$ implies $\ell'>\ell_3+\ell_4$, we have $\ell'(\ell-\ell')>(\ell_3+\ell_4)^2\geq\ell_3^2+\ell_4^2$.  So, using the parameter choices $N_1>>K_0>>c_0$, we have $N_1\ell'(\ell-\ell')\geq2c_0K_0(\ell_3^2+\ell_4^2)$.  As a result, $$N_1\ell'(\ell-\ell')g(N_1n)-A_3-A_4>\frac{1}{2}N_1\ell'(\ell-\ell')g(N_1n)-2\ell_3\a_3-2\ell_4\a_4$$
Hence, accounting for this in (\ref{eqn-diskless1}) it suffices to show
\begin{equation} \label{eqn-diskless2}
\frac{1}{2}N_1\ell'(\ell-\ell')g(N_1n)+N_2^3n(n-n_0)g(N_2n)\geq A_2+2\ell_3\a_3+2\ell_4\a_4+\ell'
\end{equation}
Further, since $n-n_0\geq2$, $g(N_2n)\geq1$, and $n\geq2\ell'$, the parameter choice $N_2>>J$ allows us to assume $\frac{1}{2}N_2^3n(n-n_0)g(N_2n)\geq J\ell'+\ell'$.  As such,
$$N_2^3n(n-n_0)g(N_2n)-A_2-\ell'\geq\frac{1}{2}N_2^3n(n-n_0)g(N_2n)-J\ell'd-C_2\phi(C_3(d+1))$$
In light of (\ref{eqn-diskless2}), assuming without loss of generality that $\a_3\geq\a_4$, it suffices to show
\begin{equation} \label{eqn-diskless3}
N_1\ell'(\ell-\ell')g(N_1n)+N_2^3n(n-n_0)g(N_2n)\geq 4(\ell-\ell')\a_3+2J\ell'd+2C_2\phi(C_3(d+1))
\end{equation}
We now proceed in two cases:

\textbf{1.} Suppose $\a_3\leq2J(\ell-\ell')$

Note that $d_i\leq\a_i+d_i'$ for $i=3,4$, so that $d\leq2\a_3+d_3'+d_4'\leq4J(\ell-\ell')+d_3'+d_4'$.  

By (\ref{eqn-diskless-perimeter}), $n-n_0\geq2+\delta(d_3'+d_4'-2)\geq\delta(d_3'+d_4')$, and so $d\leq4J(\ell-\ell')+\delta^{-1}(n-n_0)$.  So, $$2J\ell'd\leq8J^2\ell'(\ell-\ell')+2J\delta^{-1}\ell'(n-n_0)$$

Moreover, $4(\ell-\ell')\a_3\leq8J(\ell-\ell')^2\leq8J\ell'(\ell-\ell')$.  So, since $g(N_1n)\geq1$, the parameter choice $N_1>>J$ implies
$$N_1\ell'(\ell-\ell')g(N_1n)-4(\ell-\ell')\a_3-2J\ell'd\geq\frac{1}{2}N_1\ell'(\ell-\ell')g(N_1n)-2J\delta^{-1}\ell'(n-n_0)$$
As also $n\geq2\ell'$ and $g(N_2n)\geq1$, the parameter choices $N_2>>\delta^{-1}>>J$ imply $$N_2^3n(n-n_0)g(N_2n)-2J\delta^{-1}\ell'(n-n_0)\geq\frac{1}{2}N_2^3n(n-n_0)g(N_2n)$$
So noting that $d+1\leq4J(\ell-\ell')+2\delta^{-1}(n-n_0)$, then by (\ref{eqn-diskless3}), it suffices to show that
$$N_1\ell'(\ell-\ell')g(N_1n)+N_2^3n(n-n_0)g(N_2n)\geq4C_2\phi(4C_3J(\ell-\ell')+2C_3\delta^{-1}(n-n_0))$$
Now, as $n\geq2\ell$, the parameter choices $N_1>>C_3>>\delta^{-1}>>J$ imply $$4C_3J(\ell-\ell')+2C_3\delta^{-1}(n-n_0)\leq N_1n$$  
So, since $g(N_1n)\geq1$ and we choose $N_2>>N_1$, it suffices to show
\begin{equation} \label{eqn-diskless4}
N_1\ell'(\ell-\ell')+N_2^3n(n-n_0)\geq 16C_2C_3^2(2J(\ell-\ell')+\delta^{-1}(n-n_0))^2
\end{equation}
Since $\ell'>\ell-\ell'$, the parameter choices $N_1>>C_3>>C_2>>J$ imply 
$$N_1\ell'(\ell-\ell')\geq16C_2C_3^2(2J(\ell-\ell'))^2$$  
Meanwhile, $n\geq\ell'>\ell-\ell'$, so that for $N_2\geq2$ we have
$$N_2^3n(n-n_0)\geq N_2^2(n-n_0)(2n)\geq N_2^2(n-n_0)^2+N_2^2(n-n_0)(\ell-\ell')$$
Hence, (\ref{eqn-diskless4}) follows from the parameter choices $N_2>>C_3>>C_2>>\delta^{-1}>>J$.

\textbf{2.} Otherwise, suppose $\a_3\geq2J(\ell-\ell')$.

Then $\ell_3\leq\ell_3+\ell_4=\ell-\ell'\leq\frac{1}{2J}\a_3$, so that \Cref{comb area}(1) and the parameter choice $J>>K_0$ imply $d_3-d_3'\leq\frac{1}{2}\a_3+K_0\ell_3\leq\frac{2}{3}\a_3$.  An identical argument implies $d_4-d_4'\leq\frac{2}{3}\a_3$, and so $d\leq\frac{4}{3}\a_3+d_3'+d_4'$.  

As in the previous case, though, (\ref{eqn-diskless-perimeter}) implies $d_3'+d_4'\leq\delta^{-1}(n-n_0)$, and so $d\leq\frac{4}{3}\a_3+\delta^{-1}(n-n_0)$.

Moreover, a parameter choice for $J$ implies $\a_3-(d_3-d_3')-2\ell_3\geq\frac{1}{3}\a_3-\frac{1}{J}\a_3\geq\frac{1}{4}\a_3$ and so by (\ref{eqn-diskless-perimeter}) we have $n-n_0\geq\frac{1}{4}\delta\a_3$.  Hence, $d\leq 7\delta^{-1}(n-n_0)$, and so $d+1\leq8\delta^{-1}(n-n_0)$.

Now, this implies $\ell'd\leq7\delta^{-1}\ell'(n-n_0)$ and $(\ell-\ell')\a_3\leq4\delta^{-1}\ell'(n-n_0)$.  As such, the parameter assignments $N_2>>\delta^{-1}>>J$ imply $$4(\ell-\ell')\a_3+2J\ell'd\leq\frac{1}{2}N_2^3n(n-n_0)g(N_2n)$$
since $g(N_2n)\geq1$ and $n\geq2\ell'$.  Hence by (\ref{eqn-diskless3}) it suffices to show 
\begin{equation} \label{eqn-diskless5}
N_2^3n(n-n_0)g(N_2n)\geq4C_2\phi(8C_3\delta^{-1}(n-n_0))
\end{equation}
But the parameter choices $N_2>>C_3>>\delta^{-1}$ allow us to assume $N_2\geq8C_3\delta^{-1}$, so that $$\phi(8C_3\delta^{-1}(n-n_0))\leq\phi(N_2(n-n_0))\leq N_2^2(n-n_0)^2g(N_2(n-n_0))\leq N_2^2n(n-n_0)g(N_2n)$$
Thus, (\ref{eqn-diskless5}) follows by the parameter choice $N_2>>C_2$.

\end{proof}

\bigskip


\section{Annular diskless diagrams}

In this section, we show that two words admitting an annular diagram with no hubs must have a quadratic-length length conjugator in the group $G(\textbf{M}_\textbf{S})$. 





\begin{lemma} \label{annular diskless}

Let $u$ and $v$ be two words over $\pazocal{X}$ which represent conjugate elements of $M(\textbf{M}_\textbf{S})$.  Then either:

\begin{enumerate}

\item $u$ and $v$ both represent the identity in $G(\textbf{M}_\textbf{S})$.

\item There exists a word $w$ with $wuw^{-1}=v$ in $M(\textbf{M}_\textbf{S})$ and $|w|\leq N_1n^2+N_1n+N_1$ for $n=\max(|u|,|v|)$.

\end{enumerate}

\end{lemma}

\begin{proof}
 Without loss of generality, we assume $u$ and $v$ are the shortest words over $\pazocal{X}$ representing their respective elements of $M(\textbf{M}_\textbf{S})$.  By van Kampen's Lemma, there exists a reduced annular diagram $\Delta$ over $M(\textbf{M}_\textbf{S})$ with outer boundary label $u$ and inner boundary label $v$.

By \Cref{M(S) annuli} $\Delta$ has no contractible $q$- or $\theta$-annuli.  Lemma \ref{M(S) annuli}(1) and (2) in particular will be implicitly used in this proof many times without reference.

Suppose $\Delta$ has a rim $q$-band $\pazocal{Q}$.  Through $0$-refinement, we may then obtain a reduced annular diagram $\Delta'$ by cutting $\pazocal{Q}$ off of $\Delta$.  Only one of the two boundary components is affected by this procedure; suppose without loss of generality that it is the outer component.  Letting $u'$ be the word labelling the outer boundary component of $\Delta'$, van Kampen's Lemma implies $u'$ and $u$ are equal in $M(\textbf{M}_\textbf{S})$.  But also \Cref{minBoundaryLengths} implies $|u'|<|u|$, contradicting the choice of $u$.

By a similar argument (using \Cref{minBoundaryLengths-theta} in place of \Cref{minBoundaryLengths}), any maximal $\theta$-band of $\Delta$ which has two ends on the same boundary component must have base of length at least $K$.

Moreover, given the parameter choice $N_1>>N$, Lemmas 5.1 and 5.2 of \cite{GW} allow us to assume without loss of generality that $\Delta$ has at least one maximal $q$-band, and that all such maximal $q$-bands are radial.  Fix a radial $q$-band $\pazocal{Q}$ in $\Delta$.

If $\Delta$ has no $\theta$-annulus, then we may assume it contains maximal $\theta$-bands that `spiral', hitting all the radial $q$-bands more than once.  But then an identical argument to the second case of Lemma 5.3 of \cite{GW} implies the existence of a conjugator $w$ with $|w|\leq(\delta^{-1}+4)n^2+n$, so that the statement is given by the parameter choice $N_1>>\delta^{-1}$.

Hence, we assume now that there exists a $\theta$-annulus in $\Delta$.  As maximal $\theta$-bands cannot cross, the $\theta$-annuli form an annular subdiagram $\Delta_2$ of $\Delta$.  Letting $\pazocal{Q}_2$ be the maximal $q$-band of $\Delta_2$ corresponding to a subband of $\pazocal{Q}$, let $\Gamma_2$ be the circular diagram obtained from $\Delta_2$ by cutting along a side of $\pazocal{Q}_2$ and pasting a copy of $\pazocal{Q}_2$ to the side on which it is removed.  Observe then that $\Gamma_2$  is a trapezium with circular base.  Let $\pazocal{D}:U_0\to\dots\to U_t$ be the reduced computation of $\textbf{M}_\textbf{S}$ corresponding to $\Gamma_2$ by \Cref{trapezia are computations}.  By the parallel construction of the machine $\textbf{M}_\textbf{S}$, by `forgetting the coordinates' we may identify $\pazocal{D}$ with a reduced computation $\pazocal{C}:W_0\to\dots\to W_t$ of $\textbf{E}_{\textbf{S},k}$ with circular base $B$.  

Suppose $B$ has a reduced revolving subword $B'$, and let $\pazocal{C}':W_0'\to\dots\to W_t'$ be the restriction of $\pazocal{C}$ to this subword.  Then \Cref{reduced revolving controlled} implies either $t\leq c_1\max(\|W_0'\|,\|W_t'\|)$ or there is a cyclic permutation of $W_0^{\pm1}$ of the form $V^\ell t$ for some accepted configuration $V$ of $\textbf{E}_{\textbf{S},k}$ and $\ell\geq1$.  

In the latter case, since the base of $\pazocal{D}$ is also circular, `remembering the coordinates' we see that the corresponding cyclic permutation of $U_0^{\pm1}$ is of the form $W^m\{t(j)\}$ for some $m\geq1$ and $j\in\{1,\dots,L\}$, where $W$ is a cyclic permutation of an accepted configuration of $\textbf{M}_\textbf{S}$.  Note then that by the formation of $\Gamma_2$, the label of one boundary component of $\Delta_2$ is a cyclic permutation of $W^m$.  But this label is then trivial in $G(\textbf{M}_\textbf{S})$, so that van Kampen's Lemma implies $u$ and $v$ both represent the identity in $G(\textbf{M}_\textbf{S})$.

So, assuming $u$ and $v$ do not represent the identity in $G(\textbf{M}_\textbf{S})$, the existence of a reduced revolving subword $B'$ of $B$ implies $t\leq c_1\max(\|U_0'\|,\|U_t'\|)$, where $U_i'$ is the admissible subword of $U_i$ whose base corresponds to $B'$.  

If no such subword of $B$ exists, then the base of $\Gamma_2$ is necessarily defective.  In this case, there exists a reduced computation $\pazocal{E}$ between $\lab(\textbf{bot}(\Gamma_2))$ and $\lab(\textbf{top}(\Gamma_2))$ satisfying \Cref{universal complexity}.  Letting $\Lambda_2$ be the trapezium corresponding to $\pazocal{E}$ by \Cref{computations are trapezia}, then since the base is circular we may paste $\Lambda_2$ along its side $q$-bands to obtain an annular diagram $\Phi_2$ over $M(\textbf{M}_\textbf{S})$ with the same boundary labels as $\Delta_2$.  But then we may excise $\Delta_2$ from $\Delta$ and paste $\Phi_2$ in its place.  Hence, we may assume $\pazocal{D}$ itself satisfies \Cref{universal complexity}.

So, in any case there exists a subword $B''$ of a cyclic permutation of the base of $\Gamma_2$ such that 

\begin{itemize}

\item $\|B''\|\leq N+1$

\item Letting $\pazocal{D}'':U_0''\to\dots\to U_t''$ be the restriction of a cyclic permutation of $\pazocal{D}$ to $B''$, $t\leq 12m^2+c_1m+c_2$ for $m=\max(|U_0''|_a,|U_t''|_a)$.

\end{itemize}

The desired bound is then given by an identical argument to that of the first case of Lemma 5.3 of \cite{GW}.

\end{proof}

\bigskip


\section{Minimal diagrams} \label{sec-minimal-diagrams}

In this section we introduce a class of diagrams over the disk presentation of $G(\textbf{M}_\textbf{S})$ and observe immediate consequences regarding their makeup.  These diagrams, termed `minimal diagrams', were studied similarly in the circular case in \cite{O18} and \cite{OS20}, while generalizations were studied in \cite{WMal}, \cite{WEmb}, and \cite{W}.  Crucially, this class is `generic' in a particular sense, but also specific enough to facilitate the necessary combinatorial arguments.

\medskip

\subsection{$t$-letters and Minimal diagrams} \label{sec-t-minimal} \

Recall that in the construction of the machine $\textbf{M}_\textbf{S}$, the standard base consists of $L$ copies of the standard base of the $k$-enhanced machine $\textbf{E}_{\textbf{S},k}$, including $L$ singletons of the form $\{t(i)\}$.  These singleton parts of the hardware function simply as a placeholder to divide the components.

A state letter of the form $t(i)^{\pm1}$ is henceforth called a \textit{$t$-letter}, and any $(\theta,q)$-relation involving a $t$-letter is called a \textit{$(\theta,t)$-relation}.  Observe that for any positive $t$-letter $t(i)$ and any positive rule $\theta$, the corresponding $(\theta,t)$-relation is of the form $\theta_jt(i)=t(i)\theta_{j+1}$ for the appropriate index $j$.

These definitions adapt naturally to the setting of a reduced diagram over the disk presentation of $G(\textbf{M}_\textbf{S})$: A $q$-edge is a \textit{$t$-edge} if its label is a $t$-letter, a cell is a \textit{$(\theta,t)$-cell} if its boundary is labelled by a $(\theta,t)$-relation, and a $q$-band is called a \textit{$t$-band} if it corresponds to the singleton part $\{t(i)\}$ of the hardware.  Note that the sides of a $t$-band are both labelled by copies of the band's history.

Now, given a reduced diagram $\Delta$ over the disk presentation of $G(\textbf{M}_\textbf{S})$, we define the \textit{signature} of $\Delta$ to be the pair $s(\Delta)=(s_1(\Delta),s_2(\Delta))$ where:

\begin{itemize}

\item $s_1(\Delta)$ is the number of disks in $\Delta$

\item $s_2(\Delta)$ is the number of $(\theta,t)$-cells

\end{itemize}

We order signatures of reduced diagrams over the disk presentation of $G(\textbf{M}_\textbf{S})$ lexicographically. In particular, if $\Delta$ and $\Gamma$ are such diagrams, then $s(\Delta)\leq s(\Gamma)$ means $s_1(\Delta)\leq s_1(\Gamma)$ and, if $s_1(\Delta)=s_1(\Gamma)$, then $s_2(\Delta)\leq s_2(\Gamma)$.

\begin{definition}

If $\Delta$ is a reduced circular or annular diagram over $G(\textbf{M}_\textbf{S})$, then it is called \textit{minimal} if for any diagram $\Gamma$ of the same type with identical boundary labels, then $s(\Delta)\leq s(\Gamma)$.  

\end{definition}

The following statement functions as van Kampen's Lemma for minimal diagrams, and applies to the more general setting of graded presentations (see Section 13.1 of \cite{O}).

\begin{lemma} \label{vK-minimal} \

\begin{enumerate}

\item A word $W$ over $\pazocal{X}$ represents the identity in $G(\textbf{M}_\textbf{S})$ if and only if there exists a minimal circular diagram $\Delta$ over the disk presentation of $G(\textbf{M}_\textbf{S})$ with $\lab(\partial\Delta)\equiv W$.

\item Two words $U$ and $V$ over $\pazocal{X}$ represent conjugate elements in $G(\textbf{M}_\textbf{S})$ if and only if there exists a minimal annular diagram $\Gamma$ over the disk presentation of $G(\textbf{M}_\textbf{S})$ whose boundary components are labelled by $U$ and $V$.

\end{enumerate}

\end{lemma}

\medskip


\subsection{$t$-spokes} \label{sec-spokes} \

In a reduced diagram over the disk presentation of $G(\textbf{M}_\textbf{S})$, a maximal $q$-band or $t$-band with one end on a disk $\Pi$ is called a \textit{spoke} or \textit{$t$-spoke} of $\Pi$, respectively.  Note that any $t$-spoke is necessarily a spoke.  Of course, two $t$-spokes of $\Pi$ are \textit{consecutive} if there exists a subpath of $\partial\Pi$ between their ends which consists of no other $t$-edges.

The next statement is crucial for the estimates of minimal diagrams.  An identical argument is presented in \cite{O18} and \cite{OS20}, while a generalization is presented in \cite{WMal}, \cite{WEmb}, and \cite{W}.

\begin{lemma} \label{t-spokes between disks}

Let $\Delta$ be a reduced diagram over the disk presentation of $G(\textbf{M}_\textbf{S})$. Suppose there exist two disks $\Pi_1$ and $\Pi_2$ in $\Delta$ so that $\pazocal{Q}_1$ and $\pazocal{Q}_2$ are consecutive $t$-spokes of both. If these $t$-spokes and disks bound a (circular) subdiagram $\Psi$ containing no disks, then there exists a reduced diagram $\Gamma$ with the same boundary labels such that $s_1(\Gamma)\leq s_1(\Delta)-2$.

\end{lemma}

\begin{figure}[H]
\centering
\includegraphics{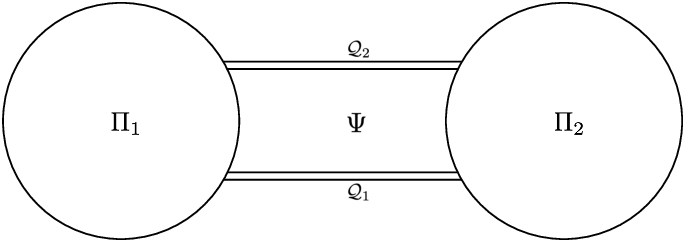}
\caption{\Cref{t-spokes between disks}}
\label{fig-t-spoke-removal}
\end{figure}

\begin{proof}

Since $\Psi$ contains no disks, by \Cref{M(S) annuli} it is a trapezium.  \Cref{trapezia are computations} then associates to $\Psi$ a reduced computation $\pazocal{C}$ with base $\{t(i)\}B_{std}(i)\{t(i+1)\}$ between admissible words $V_1$ and $V_2$ such that:

\begin{itemize}

\item $V_1$ is a subword of $(\partial\Pi_1)^{\pm1}$

\item $V_2$ is a subword of $(\partial\Pi_2)^{\mp1}$

\item The history of $\pazocal{C}$ is the same as that of $\Psi$.

\end{itemize}

By the parallel construction of $\textbf{M}_\textbf{S}$, we may extend $\pazocal{C}$ to a reduced computation $\pazocal{D}$ of a reduced revolving base between $(\partial\Pi_1)^{\pm1}t(1)$ and $(\partial\Pi_2)^{\mp1}t(1)$.  Letting $\Gamma_\pazocal{D}$ be the trapezium corresponding to $\pazocal{D}$ through \Cref{computations are trapezia}, then $\Gamma_\pazocal{D}$ contains a copy of $\Psi$ (or a mirror copy of it).  As the tape alphabets of sectors bounded by $t$-letters are empty, we may then cut this copy of $\Psi$ from $\Gamma_\pazocal{D}$ and glue up the remaining pieces to produce a trapezium $\Gamma'$.

Letting $\Phi$ be the (circular) subdiagram of $\Delta$ consisting of $\Psi$, $\Pi_1$, and $\Pi_2$, then by construction $\lab(\partial\Gamma')\equiv\lab(\partial\Phi)^{\pm1}$.  But then letting $\Gamma$ be the reduced diagram obtained from $\Delta$ by excising $\Phi$ and pasting in its place $\Gamma'$ or its mirror image (and then making any necessary cancellations) satisfies the statement.

\end{proof}

Now, let $\Delta$ be a reduced circular diagram over the disk presentation of $G(\textbf{M}_\textbf{S})$ such that $s_1(\Delta)$ is minimal among all such diagrams with the same boundary label.  We then form the underlying graph $\Gamma=\Gamma(\Delta)$ as follows:

\begin{enumerate}[label=({\arabic*})]

\item $V(\Gamma)=\{v_0,v_1,\dots,v_\ell\}$ where each $v_i$ is a disk in $\Delta$ save $v_0$, which is regarded as an exterior vertex.

\item For $i,j\geq1$, each $t$-spoke running between disks $v_i$ and $v_j$ is recorded as an $(v_i,v_j)\in E(\Gamma)$

\item For $i\geq1$, each $t$-spoke running between disk $v_i$ and $\partial\Delta$ is recorded as an edge $(v_0,v_i)\in E(\Gamma)$

\end{enumerate}

Note that $\Gamma$ has no $1$-gons nor (by \Cref{t-spokes between disks}) $2$-gons on interior vertices.  The degree of each interior vertex of $\Gamma$ is $L$, so for $L$ sufficiently large we have the following statement

\begin{figure}[H]
\centering
\includegraphics[scale=1]{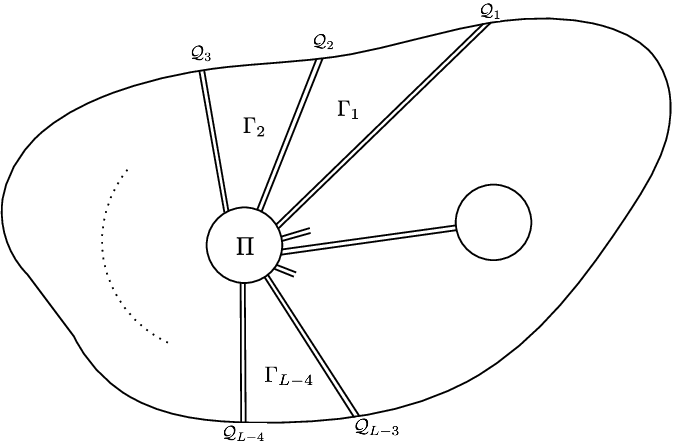}
\caption{Lemma \ref{circular-graph}}
\label{fig-graph}
\end{figure}

\begin{lemma}[Lemma 7.5 of \cite{OS20}, inter alia] \label{circular-graph}

Suppose $\Delta$ is a reduced circular diagram over the disk presentation of $G(\textbf{M}_\textbf{S})$ such that $s_1(\Delta)\geq1$ is minimal among all such diagrams with the same boundary label. Then $\Delta$ contains a disk $\Pi$ such that $L-3$ consecutive $t$-spokes $\pazocal{Q}_1,\dots,\pazocal{Q}_{L-3}$ of $\Pi$ end on $\partial\Delta$ and every subdiagram $\Gamma_i$ bounded by $\pazocal{Q}_i$, $\pazocal{Q}_{i+1}$, $\partial\Pi$, and $\partial\Delta$ ($i=1,\dots,L-4$) contains no disks (see \Cref{fig-graph}).

\end{lemma}

On the other hand, let $\Delta$ be a reduced annular diagram over the disk presentation of $G(\textbf{M}_\textbf{S})$ such that $s_1(\Delta)$ is minimal among all such diagrams with the same boundary labels.  Then we form the underlying graph $\Gamma=\Gamma(\Delta)$ in much the same way, but without any exterior vertex (and so without `external' edges).  Again, this graph has no $1$-gons or $2$-gons.  Hence, as long as there is at least one vertex, an appeal to the Euler characteristic of an annulus (see Lemma 10.1 of \cite{O}) implies there exists a vertex with degree at most 18.  

\begin{lemma} \label{annular-graph}

Suppose $\Delta$ is a reduced annular diagram over the disk presentation of $G(\textbf{M}_\textbf{S})$ such that $s_1(\Delta)\geq1$ is minimal among all such diagrams with the same boundary labels.  Then $\Delta$ contains a disk $\Pi$ such that at least $L-18$ $t$-spokes of $\Pi$ end on the boundary of $\Delta$.

\end{lemma}

\medskip


\subsection{Transposition} \label{sec-transposition} \

The next operation we outline is a procedure to move a disk past a $\theta$-band.  As will be evidenced in the forthcoming arguments, this operation is useful in numerous ways.

Let $\Delta$ be a reduced diagram over the disk presentation of $G(\textbf{M}_\textbf{S})$.  Let $\Pi$ be a disk and $\pazocal{T}$ be a maximal $\theta$-band in $\Delta$ which crosses $\ell\geq2$ consecutive $t$-spokes $\pazocal{Q}_1,\dots,\pazocal{Q}_\ell$ of $\Pi$.  Let $\pazocal{T}'$ be the minimal subband of $\pazocal{T}$ which crosses these $t$-spokes and suppose there are no cells between $\textbf{bot}(\pazocal{T}')$ and $\Pi$.  So, $\Pi$ and $\pazocal{T}$ form a subdiagram $\Gamma$ of $\Delta$, while $\Pi$ and $\pazocal{T}'$ form a subdiagram $\Gamma'$.

Let $\theta$ be the rule of $\textbf{M}_\textbf{S}$ corresponding to $\pazocal{T}$ and $W$ be the configuration corresponding to $\partial\Pi$.  Since $\ell\geq2$, there exists $i\in\{1,\dots,L\}$ such that $W(i)$ is $\theta$-admissible.  So, the parallel nature of the machine implies $W$ is $\theta$-admissible, with $W\cdot\theta$ a new disk relator.  Construct a disk $\bar{\Pi}$ with $\lab(\partial\bar{\Pi})\equiv W\cdot\theta$.

\Cref{computations are trapezia} produces a $\theta$-band $\pazocal{T}'$ corresponding to the single-rule computation $(W\cdot\theta)\cdot\theta^{-1}\equiv W$.  Since the tape alphabet of any sector bounded by a $t$-letter is empty, we may then paste a subband $\pazocal{T}_0'$ of $\pazocal{T}'$ to the disk $\bar{\Pi}$ to produce a diagram $\bar{\Gamma}'$ with the same boundary label as $\Gamma'$.

Finally, note that the ends of $\pazocal{T}_0'$ have the same label as the ends of $\pazocal{T}'$.  Hence, we excise $\Gamma'$ from $\Delta$ and paste in its place $\bar{\Gamma}'$, making any necessary reductions to produce a reduced diagram $\bar{\Delta}$ with corresponding subdiagram $\bar{\Gamma}$.

This process is called the \textit{transposition of the disk $\Pi$ with the $\theta$-band $\pazocal{T}$}.

\renewcommand\thesubfigure{\alph{subfigure}}
\begin{figure}[H]
\centering
\begin{subfigure}[b]{0.48\textwidth}
\centering
\raisebox{0.2in}{\includegraphics[width=3in]{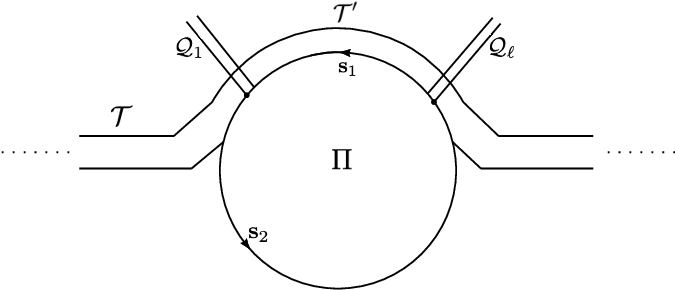}}
\caption{The subdiagram $\Gamma$}
\end{subfigure}\hfill
\begin{subfigure}[b]{0.48\textwidth}
\centering
\includegraphics[width=3in]{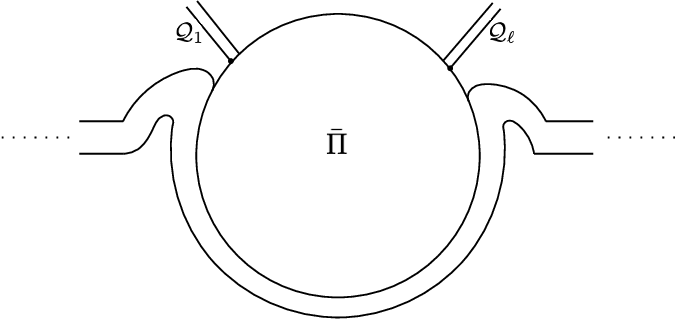}
\caption{The resulting subdiagram $\bar{\Gamma}$}
\end{subfigure}
\caption{The transposition of a $\theta$-band with a disk}
\end{figure}

The next statement follows immediately from this construction and \Cref{circular-graph}, and is proved in a more general setting in \cite{W}.

\begin{lemma} \label{minimal annuli}

A minimal circular diagram over the disk presentation of $G(\textbf{M}_\textbf{S})$ contains no $\theta$-annuli.

\end{lemma}

Indeed, this statement can be strengthened considerably:

\begin{lemma}[Lemma 7.7(2) in \cite{OS20}] \label{two theta-bands about disk}

Let $\Pi$ be a disk in a reduced diagram $\Delta$ over the disk presentation of $G(\textbf{M}_\textbf{S})$.  Suppose $\pazocal{T}$ and $\pazocal{T}'$ are disjoint $\theta$-bands in $\Delta$ such that each has a side that is shared with the boundary of $\Pi$.  Suppose further that these bands correspond to the same letter of $\Theta(\textbf{M}_\textbf{S})$.  If $\pazocal{T}$ and $\pazocal{T}'$ cross $\ell$ and $\ell'$ $t$-spokes of $\Pi$, respectively, then $\ell+\ell'\leq L/2$.


\end{lemma}

\medskip


\subsection{Quasi-trapezia} \label{sec-quasi-trapezia} \

As discussed in \Cref{sec-trapezia}, the crossings of $q$-bands with $\theta$-bands in a reduced circular diagram over $M(\textbf{M}_\textbf{S})$ form a trapezium whose structure is totally informed by the computational makeup of the machine.  In a reduced circular diagram over the disk presentation of $G(\textbf{M}_\textbf{S})$, though, the same cannot be deduced, as disks may be present between the $\theta$-bands.  However, the process of transposition described in the previous section helps understand this case, moving the disks about the bands to form a trapezium.

To make this precise, a \textit{quasi-trapezium} over a minimal diagram over the disk presentation of $G(\textbf{M}_\textbf{S})$ is a (circular) subdiagram whose contour may be factored $\textbf{p}_1^{-1}\textbf{q}_1\textbf{p}_2\textbf{q}_2^{-1}$ where the paths $\textbf{p}_i$ are the sides of maximal $q$-bands and the paths $\textbf{q}_i$ are maximal subpaths of the sides of maximal $\theta$-bands which start and end with $q$-letters.  Hence, a comparison to the definitions of \Cref{sec-trapezia} reveals that quasi-trapezia are simply trapezia which may contain some disks.

Much of the terminology of trapezia carries over to quasi-trapezia, {\frenchspacing e.g. history and standard factorization}.  The next statement crucially uses the transposition operation and functions to `remove' the disks from the quasi-trapezium to produce a trapezium.

\begin{lemma}[Lemma 7.9 of \cite{OS20}] \label{quasi-trapezia}

Let $\Gamma$ be a quasi-trapezium with standard factorization of its contour $\textbf{p}_1^{-1}\textbf{q}_1\textbf{p}_2\textbf{q}_2^{-1}$. Then there exists a minimal diagram $\Gamma'$ such that:

\begin{enumerate}[label=({\arabic*})]

\item $\partial\Gamma'=(\textbf{p}_1')^{-1}\textbf{q}_1'\textbf{p}_2'(\textbf{q}_2')^{-1}$, where $\lab(\textbf{p}_j')\equiv\lab(\textbf{p}_j)$ and $\lab(\textbf{q}_j')\equiv\lab(\textbf{q}_j)$ for $j=1,2$

\item the number of disks in $\Gamma'$ is the same as the number of disks in $\Gamma$

\item there exists a simple path $\textbf{s}_1$ (respectively $\textbf{s}_2$) connecting the vertices $(\textbf{p}_1')_-$ and $(\textbf{p}_2')_-$ (respectively $(\textbf{p}_1')_+$ and $(\textbf{p}_2')_+$) such that

\begin{enumerate}

\item $(\textbf{p}_1')^{-1}\textbf{s}_1\textbf{p}_2'\textbf{s}_2^{-1}$ is the standard factorization of the boundary of a trapezium $\Gamma_2$ and

\item any cell above $\textbf{s}_2$ or below $\textbf{s}_1$ is a disk

\end{enumerate}

\item there exists $m\in\N$ such that any maximal $\theta$-band of $\Gamma$ contains $m$ $(\theta,t)$-cells and any maximal $\theta$-band of $\Gamma{\color{red}'}$ contains $m$ $(\theta,t)$-cells.

\end{enumerate}

\end{lemma}

\medskip


\subsection{Shafts and Designs} \label{sec-shafts} \

We now recall an auxiliary measure which is invaluable for the study of circular diagrams over the disk presentation of $G(\textbf{M}_\textbf{S})$ and the related combinatorial arguments that follow.

Let $\pazocal{D}$ be a topological disk in the Euclidean plane.  Fix a finite set of disjoint chords $\textbf{T}$ in $\pazocal{D}$ and a finite set $\textbf{Q}$ of disjoint simple paths in $\pazocal{D}$.  We assume that the elements of $\textbf{Q}$, called \textit{arcs}, belong to the open disk and that any arc and any chord cross at most once, in which case they cross transversely.

With these assumptions, the pair $(\textbf{T},\textbf{Q})$ is called a \textit{design} on the disk.

Given an arc $C\in\textbf{Q}$, the length $|C|$ is the number of chords it crosses. We then denote the length of the design to be $\ell(\textbf{Q})=\sum|C|$.

A \textit{subarc} $D$ of $C$ is defined in the obvious way, with the understanding that $D$ need not be an element of $\textbf{Q}$ in order to be a subarc.  The length of the subarc $D$ is naturally the number of chord crossings in this portion of $C$, so that $|D|\leq|C|$.

For two subarcs $D_1,D_2$ of arcs $C_1,C_2\in\textbf{Q}$, respectively, we say that $D_1$ is \textit{parallel} to $D_2$ if every chord crossing $D_1$ also crosses $D_2$.  In this case, we write $D_1 \| D_2$.  It's important to note that this relation is not symmetric.

Given a positive integer $m$ and any $\zeta<1/2$, a design $(\textbf{T},\textbf{Q})$ is said to \textit{satisfy property $P(\zeta,m)$} if for any collection of $m$ distinct arcs $C_1,\dots,C_m\in\textbf{Q}$, there are no subarcs $D_1,\dots,D_m$, respectively, such that $D_i\| D_{i+1}$ and $|D_i|>(1-\zeta)|C_i|$ for all $i$.

\begin{lemma}[Lemma 8.2 of \cite{O18}] \label{design}

There exists a positive constant $C_{\zeta,m}$ dependent only on $\zeta$ and $m$ such that for any design $(\textbf{T},\textbf{Q})$ satisfying property $P(\zeta,m)$, $\ell(\textbf{Q})\leq C_{\zeta,m}|\textbf{T}|$.

\end{lemma}

Now, consider a reduced circular diagram $\Delta$ over the disk presentation of $G(\textbf{M}_\textbf{S})$.  Fix a disk $\Pi$ in $\Delta$ and a $t$-spoke $\pazocal{Q}$ of $\Pi$ with history $H$ (read starting at $\Pi$).  Given a prefix $H'$ of $H$, suppose that for every factorization $H'\equiv H_1'H_2'H_3'$ with $\|H_1'\|+\|H_3'\|<\zeta\|H'\|$, $H_2'$ has a controlled subword.  If also $\lab(\partial\Pi)$ is $H'$-admissible, then the subband $\pazocal{Q}'$ with history $H'$ is called a \textit{$\zeta$-shaft} at $\Pi$.

For every disk $\Pi$ and $t$-spoke $\pazocal{Q}$ of $\Pi$, we fix a maximal (perhaps empty) subband $\pazocal{Q}_\Pi^\zeta$ of $\pazocal{Q}$ which is a $\zeta$-shaft at $\Pi$.  Then, we define $\sigma_\zeta(\Delta)$ to be the sum of the lengths of the bands $\pazocal{Q}_\Pi^\zeta$ for all $\Pi$ and $\pazocal{Q}$.

\begin{lemma}[Lemma 8.15 of \cite{O18}] \label{G design}

For any minimal circular diagram $\Delta$ and any $\zeta<1/2$, we have $\sigma_\zeta(\Delta)\leq C_{\zeta,2L+1}|\partial\Delta|$.

\end{lemma}

Given the importance of \Cref{G design} and our goal to generalize it in the next section, we provide a sketch of its proof below.

From $\Delta$ we construct a design by adding middle lines to the maximal $\theta$-bands and to maximal $\zeta$-shafts.  By \Cref{minimal annuli} we may construct the middle lines of the maximal $\theta$-bands so that they form chords which intersect the arcs given by the maximal $\zeta$-shafts transversely and at most once.  The only potential pitfalls in this construction are that the arcs may end on the boundary and that two arcs may overlap if a maximal $t$-band ends on two disks.  These are both easily remedied, though: We cut off the very end of any arc going to the boundary and `make room' within any $\zeta$-shaft so that two arcs can fit disjointly.

Suppose we have a counterexample to property $P(\zeta,2L+1)$.  This means there exists a quasi-trapezium in $\Delta$ with at least $L+1$ distinct maximal $t$-bands (recall that two arcs may correspond to the same maximal $t$-band), at least one of which is a subband $\pazocal{Q}''$ of a maximal $\zeta$-shaft $\pazocal{Q}'$ at a disk $\Pi$ with $|\pazocal{Q}''|>(1-\zeta)|\pazocal{Q}'|$.  As such, the history $H''$ of $\pazocal{Q}''$, and so of the quasi-trapezium, contains a controlled subword.  

We can then perform the transpositions as in \Cref{quasi-trapezia}, so that we may assume $\pazocal{Q}''$ is a maximal $t$-band of a big trapezium $\Gamma''$ in $\Delta$.  Let $\Gamma_0''$ be the subdiagram obtained from $\Gamma''$ by chopping off one of its rim $t$-bands.  Note that since the tape alphabet of any sector adjacent a $t$-letter is empty, $\Gamma_0''$ is itself a trapezium with history $H''$.

Let $H_1'H''$ be a prefix of the history $H'$ of $\pazocal{Q}'$.  By the definition of shafts, $\lab(\partial\Pi)$ is $H'$-admissible, and so in particular is $H_1'H''$-admissible.  Since $H''$ has a controlled subword, \Cref{enhanced controlled} then implies that, when read starting at the correct vertex, $\lab(\partial\Pi)\cdot H_1'$ is the inverse of $\lab(\textbf{bot}(\Gamma_0''))$.

Now, let E be trapezium given by \Cref{computations are trapezia} corresponding to the application $\lab(\partial\Pi)^{-1}\cdot H_1'$ and $\exists$ be its `mirror copy' corresponding to the inverse transition.  Pasting the top of $\exists$ to the bottom of E then produces a circular diagram $\Phi$ with freely trivial boundary label.  As such, we may introduce $\Phi$ into $\Delta$ without affecting the boundary label, pasting the top of E to the bottom of $\Gamma_0''$.  

Note then that $\Pi$, E, and $\Gamma_0''$ form a subdiagram whose boundary label is freely equal to the disk relator $\lab(\textbf{top}(\Gamma_0''))^{-1}\equiv\lab(\partial\Pi)\cdot H_1'H''$.   So, we may replace these subdiagrams with a single disk, producing a circular diagram with the same boundary label and number of disks as $\Delta$.

However, note that in this process we've added at most $\zeta|\pazocal{Q}'|$ $(\theta,t)$-cells through the introduction of $\exists$ and deleted at least $(1-\zeta)|\pazocal{Q}'|$ through the removal of $\Gamma_0''$.  Hence, as $\zeta<1/2$, this contradicts the minimality of $\Delta$.

\medskip


\subsection{Spears and Directed Designs} \label{sec-spears} \

We now introduce a generalization of the concept of shafts and designs that will aid in the arguments that follow.

Before we introduce the definitions, first some motivation: To adapt the techniques of previous literature efficiently, it would be convenient to have an analogous measure for minimal annular diagrams.  However, if the design is defined in a similar way, no analogue of the bound given in \Cref{G design} can be found for $\sigma_\zeta(\Delta)$.  Indeed, while \Cref{G(S) annuli} tells us the sides of a $\theta$-band and $q$-band cannot bound a contractible path, it does not preclude the possibility of a $\theta$-band `spiraling' and crossing the same $q$-band several times, or even to wrap around on itself and form a $\theta$-annulus. 

To begin to resolve this, we want to restrict our attention to the $t$-spokes which only cross maximal $\theta$-bands that have ends on the diagram's boundary, and cross these bands at most a fixed number of times.  Even then, though, we end up with an auxiliary structure on an annulus rather than a disk, meaning we cannot apply \Cref{design}.  For this, we must impose a stronger condition than $P(\zeta,m)$, but one that is still applicable in the relevant settings.

To this end, let $X$ be an annulus in the Euclidean plane.  As in the previous setup, let $\textbf{T}$ be a finite set of disjoint line segments in $X$ whose endpoints are on boundary components of $X$.  For consistency, we still call the elements of $\textbf{T}$ \textit{chords}, but also call an element of $\textbf{T}$ a \textit{radial chord} if its endpoints are on distinct boundary components.


However, in this case we replace the arcs with a (finite) set $\textbf{R}$ of \textit{directed arcs}.  These are defined in much the same way as arcs are defined for designs, but with an assigned orientation.  In particular, $\textbf{R}$ is a finite set of disjoint oriented simple paths contained in the interior of $X$ each of which crosses any chord transversely and at most twice.  Moreover, if the directed arc $C\in\textbf{R}$ crosses the chord $T'\in\textbf{T}$ twice, then we assume that:

\begin{enumerate}


\item There exists a non-contractible loop in $X$ made up of a subsegment $D$ of $C$ and a subsegment of $T'$,

\item $T'$ is not a radial chord, and

\item There exists a radial chord $T\in\textbf{T}$ which crosses $D$.  (see \Cref{fig-directed-design})

\end{enumerate}


In this case, $(\textbf{T},\textbf{R})$ is called a \textit{directed design} on $X$.

\begin{figure}
\centering
\includegraphics[scale=1]{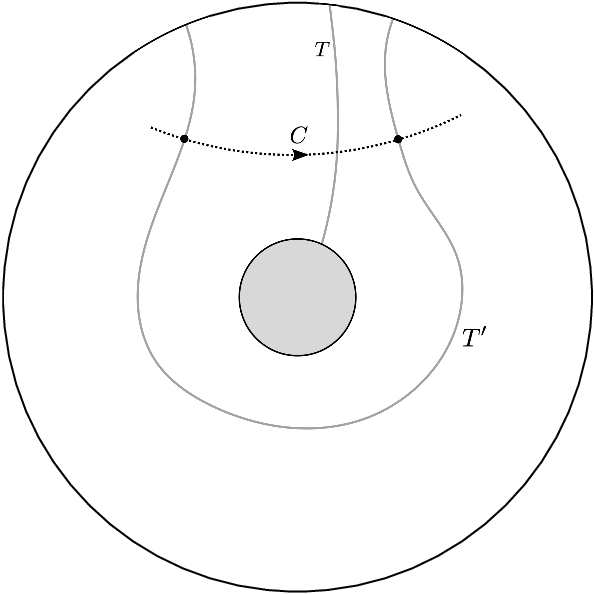}
\caption{A directed arc $C$ which crosses a chord $T'$ twice, where $D$ is the subarc of $C$ bounded by the two auxiliary nodes.}
\label{fig-directed-design}
\end{figure}

Analogous to the treatment of arcs in designs, we define the length of a directed arc of a directed design to be the number of chords it crosses.  Similarly, we take the length of the directed design to be $\ell(\textbf{R})=\sum_{C\in\textbf{R}}|C|$.  

Subarcs are taken to be unoriented subsegments of directed arcs analogous to subarcs of arcs of an undirected design. 
While subarcs are undirected, though, an \textit{initial subarc} is one that would be an initial segment should it have inherited the same orientation.  A \textit{terminal subarc} is defined symmetrically.

We now strengthen the notion of `parallel subarcs' to properly address this new setting.

\begin{definition}

Given subarcs $D_1,D_2$ of directed arcs $C_1,C_2\in\textbf{R}$, respectively, $D_1$ is said to be \textit{aligned to} $D_2$, denoted $D_1 \|\| D_2$, if:

\begin{itemize}

\item Every chord that crosses $D_1$ does so exactly once, and also crosses $D_2$ exactly once.  

\item Given distinct chords $T_1,T_2\in\textbf{T}$ which cross $D_1$, let $S_1,S_2$ be the minimal subsegments that cross both $D_1$ and $D_2$.  Then $S_1,S_2$ and the corresponding subsegments of $D_1,D_2$ bound a disk in $X$.


\end{itemize}

\end{definition}

Again, it is important to note that this relation is not symmetric, as a chord that crosses $D_2$ need not cross $D_1$.

With this, the condition $P(\zeta,m)$ can be adapted for a directed design by replacing `parallel' with `aligned'.  We introduce the following stronger condition, however, that (1) can be shown to possess an analogue of the bound in \Cref{design}, and (2) still applies in the relevant arguments of the next section.

\begin{definition}

Given a positive integer $m$ and real number $\zeta<1/2$, the directed design $(\textbf{T},\textbf{R})$ is said to \textit{satisfy condition $P'(\zeta,m)$} if for any collection of $m$ distinct directed arcs $C_1,\dots,C_m\in\textbf{R}$, given initial subarcs $C_i'$ of $C_i$ with $|C_i'|\geq\frac{1}{2}\zeta|C_i|$ for all $i$, there are no subsegments $D_1,\dots,D_m$ of $C_1',\dots,C_m'$, respectively, such that $D_i\|\| D_{i+1}$ and $|D_i|>(1-\zeta)|C_i'|$ for all $i$.


\end{definition}

\begin{lemma} \label{directed design}

There exists a constant $K_{\zeta,m}$ dependent only on $\zeta$ and $m$ such that for any directed design $(\textbf{T},\textbf{R})$ satisfying property $P'(\zeta,m)$, $\ell(\textbf{R})\leq K_{\zeta,m}|\textbf{T}|$.

\end{lemma}

\begin{proof}

Note that if $|\textbf{T}|=0$, then the statement is given by any choice of constant since then $\ell(\textbf{R})=0$.  As such, we assume $|\textbf{T}|\geq1$.

We now proceed in two cases:

\textbf{1.} Suppose there exists a radial chord $T\in\textbf{T}$.

Then we construct a design on a disk as follows:

\begin{itemize}

\item Cut along $T$ to turn the annulus into a disk, removing the directed arcs that cross $T$ and forgetting the orientations of the others.

\item Add two new chords on each of the two new portions of the boundary arising from $T$ disjoint from all other chords and arcs already present (this is possible since we've removed $T$ and all directed arcs that crossed it).

\item If $C_1\in\textbf{R}$ crosses $T$ and has an initial subarc of length at least $\frac{1}{2}\zeta|C_1|$ which does not cross $T$
, then add an arc $C_1'$ along a minimal initial subarc of $C_1$ that crosses $T$, where the crossing with $T$ is given by a crossing with the corresponding new chord.

\item If $C_2\in\textbf{R}$ crosses $T$ but doesn't satisfy the previous condition, then add an arc $C_2'$ corresponding to a minimal terminal subarc of $C_2$ that crosses $T$, where the crossing with $T$ is given by a crossing with the corresponding new chord.

\end{itemize}

See \Cref{fig-radial-cut} for a sketch of this construction.

\renewcommand\thesubfigure{\alph{subfigure}}
\begin{figure}
\centering
\begin{subfigure}[b]{0.48\textwidth}
\centering
\includegraphics[width=3in]{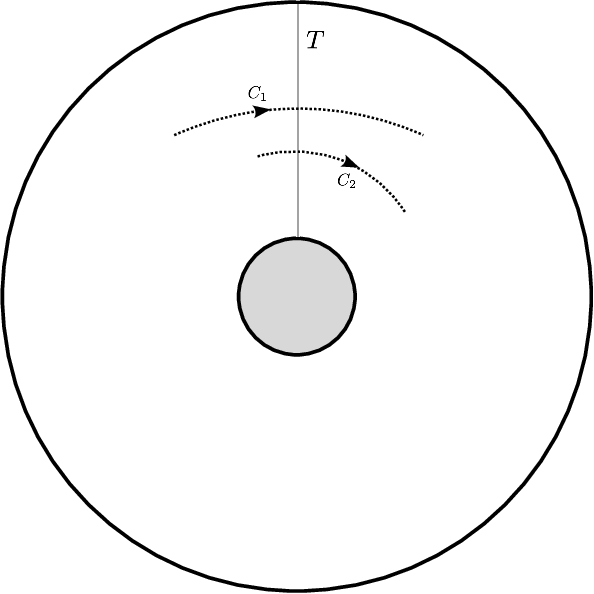}
\caption{The radial chord $T$ in $X$ crossing two directed arcs $C_1$ and $C_2$.}
\end{subfigure}\hfill
\begin{subfigure}[b]{0.48\textwidth}
\centering
\includegraphics[width=3in]{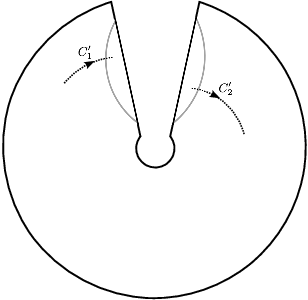}
\caption{The disk formed from cutting along $T$, with arcs $C_1'$ and $C_2'$ crossing the new chords.}
\end{subfigure}
\caption{The formation of the design $(\textbf{T}',\textbf{Q})$ using a radial chord $T$.}
\label{fig-radial-cut}
\end{figure}

First, we argue that the resulting pair $(\textbf{T}',\textbf{Q})$ is in fact a design.  For this, we must in particular verify that the corresponding chords and arcs cross at most once.  

Since no directed arc crosses a radial chord twice, this condition is immediately true for the crossings involving the two new chords arising from the radial chord $T$.  Indeed, no arc crosses both of these two new chords.

Any other chord $T'\in\textbf{T}'$ corresponds to a chord of $\textbf{T}$.  If in the initial directed design it crosses a directed arc $C\in\textbf{R}$ twice, then since $T$ is radial it must cross the non-contractible loop formed by the subsegments of $C$ and $T'$.  But chords in a directed design are disjoint, meaning $T$ crosses $C$ along the subsegment between its two crossings with $T'$.  The arc $Q\in\textbf{Q}$ corresponding to $C$ is formed only from one side of its crossing with $T$, and so can only cross $T'$ at most once.

All other pairs clearly cross at most once, and thus $(\textbf{T}',\textbf{Q})$ is indeed a design on a disk.

Note that $\textbf{T}'$ consists of one more chord than $\textbf{T}$, while $\textbf{Q}$ has the same number of arcs as $\textbf{R}$ has directed arcs.  Moreover, for each $C\in\textbf{R}$, the arc $Q\in\textbf{Q}$ arising from $C$ satisfies $|Q|\geq\frac{1}{2}\zeta|C|$, so that $\ell(\textbf{R})\leq2\zeta^{-1}\ell(\textbf{Q})$.

Fix $m$ distinct arcs $Q_1,\dots,Q_m\in\textbf{Q}$ along with subarcs $D_1,\dots,D_m$, respectively, such that $|D_i|>(1-\frac{1}{2}\zeta)|Q_i|$, and suppose $D_i \| D_{i+1}$ for all $i$.  

Letting $C_i\in\textbf{R}$ be the directed arc corresponding to $Q_i$, by construction each $D_i$ is identified with a subarc $E_i$ of $C_i$.  Of course, any crossing of $E_i$ with a chord other than $T$ corresponds to a crossing of $D_i$ with the same chord of $\textbf{T}'$.  Similarly, any crossing of $E_i$ with $T$ corresponds to a crossing of $D_i$ with one of the new chords.  In particular, since $D_i \| D_{i+1}$, every chord of $\textbf{T}$ that crosses $E_i$ also crosses $E_{i+1}$, and crosses each exactly once.  Moreover, for any two distinct chords which cross both $E_i$ and $E_{i+1}$, the minimal subsegments which cross both can be viewed as subsegments of the corresponding chords of $\textbf{T}'$ that also cross both $D_i$ and $D_{i+1}$ in the same way.  So, since $(\textbf{T}',\textbf{Q})$ is formed on a disk, it follows that these subsegments and the corresponding subarcs of $E_i,E_{i+1}$ bound a disk in the annulus.  Hence, $E_i \| \| E_{i+1}$.

If $C_i$ has an initial subarc of length at least $\frac{1}{2}\zeta|C_i|$ which doesn't cross $T$, then $Q_i$ may be identified with an initial subarc $C_i'$ of $C_i$ with $|C_i'|\geq\frac{1}{2}\zeta|C_i|$.  In this case, $E_i$ is a subarc of $C_i'$ with $|E_i|>(1-\frac{1}{2}\zeta)|C_i'|\geq(1-\zeta)|C_i'|$.

Otherwise, $Q_i$ may be identified with a terminal subarc of $C_i$ with $|Q_i|\geq(1-\frac{1}{2}\zeta)|C_i|$.  In this case, we take $C_i'=C_i$, so that $|E_i|=|D_i|>(1-\frac{1}{2}\zeta)|Q_i|\geq(1-\frac{1}{2}\zeta)^2|C_i'|>(1-\zeta)|C_i'|$.  

In either case, $C_i'$ is an initial subarc of $C_i$ with $|C_i'|\geq\frac{1}{2}\zeta|C_i|$, while $E_i$ is a subarc of $C_i'$ with $|E_i|>(1-\zeta)|C_i'|$.  But since $E_i \| \| E_{i+1}$, this provides a counterexample to condition $P'(\zeta,m)$.  

This contradiction implies $(\textbf{T}',\textbf{Q})$ satisfies condition $P(\zeta/2,m)$, so that \Cref{G design} implies $$\ell(\textbf{R})\leq2\zeta^{-1}\ell(\textbf{Q})\leq2\zeta^{-1}C_{\zeta/2,m}|\textbf{T}'|\leq4\zeta^{-1}C_{\zeta/2,m}|\textbf{T}|$$

\textbf{2.} Hence, we may assume $\textbf{T}$ has no radial chord.

By condition (3) in the definition of a directed design, it thus follows that no chord and directed arc cross twice.

For any $T\in\textbf{T}$, cutting along $T$ separates the annulus into two components, one of which is a disk and the other of which is another annulus.  Let $T_1,\dots,T_n\in\textbf{T}$ be such that this cutting results in a maximal disk, {\frenchspacing i.e. for all $j$ there is no $T\in\textbf{T}$ such that the disk component arising from cutting along $T$ contains $T_j$}.

For each $j\in\{1,\dots,n\}$, we construct a design $(\textbf{T}_j,\textbf{Q}_j)$ on a disk as follows:

\begin{itemize}

\item Cut along $T_j$ and consider the component $\pazocal{D}_j$ homeomorphic to a disk, removing the directed arcs that cross $T_j$ and forgetting the orientations of the others.

\item Add one new chord on the portion of the boundary arising from $T_j$ disjoint from all other chords and arcs already present (again, this is possible since we've removed all directed arcs that crossed $T_j$).

\item If $C_1\in\textbf{R}$ crosses $T_j$ and has an initial subarc $C_1'$ of length at least $\frac{1}{2}\zeta|C_1|$ contained in $\pazocal{D}_j$, then add $C_1'$ along with a crossing with the new chord (corresponding to the crossing of $C_1$ and $T_j$).

\item If $C_2\in\textbf{R}$ crosses $T_j$ and has a terminal subarc $C_2'$ of length at least $(1-\frac{1}{2}\zeta)|C|$ contained in $\pazocal{D}_j$, then add $C_2'$ along with a crossing with the new chord (corresponding to the crossing of $C_2$ and $T_j$).

\end{itemize}

See \Cref{fig-non-radial-cut} for a sketch of this construction.  Observe that $C_3$ and $C_4$ in this figure are examples of directed arcs that have a subarc in $\pazocal{D}_j$, but which have no corresponding arc in $\textbf{Q}_j$.

\renewcommand\thesubfigure{\alph{subfigure}}
\begin{figure}
\centering
\begin{subfigure}[b]{0.48\textwidth}
\centering
\includegraphics[width=3in]{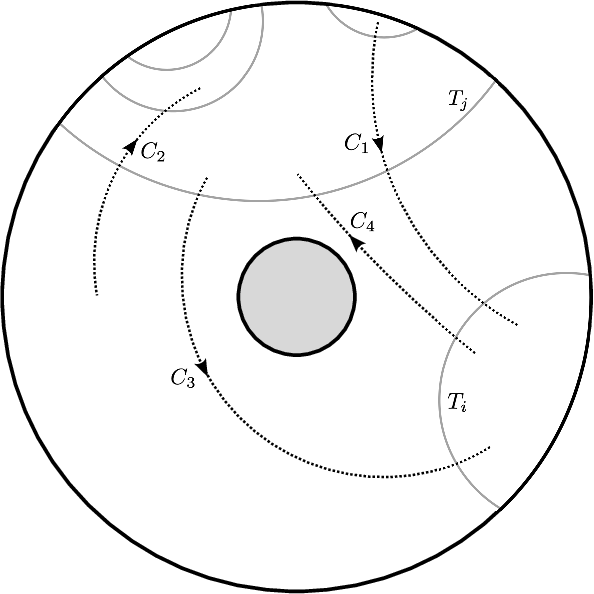}
\caption{The maximal chord $T_j$ in $X$ crossing directed arcs $C_1,\dots,C_4$.}
\end{subfigure}\hfill
\begin{subfigure}[b]{0.48\textwidth}
\centering
\raisebox{1.65in}{\includegraphics[width=3in]{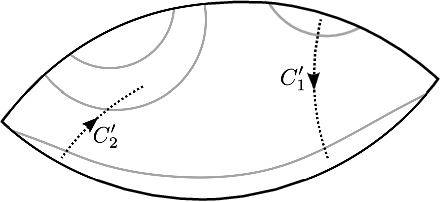}}
\caption{The disk formed from cutting along $T_j$, with arcs $C_1'$ and $C_2'$ crossing the new chord.}
\end{subfigure}
\caption{The formation of the design $(\textbf{T}_j,\textbf{Q}_j)$ using the maximal chord $T_j$.}
\label{fig-non-radial-cut}
\end{figure}

As in the previous case, fix $Q_1,\dots,Q_m\in\textbf{Q}_j$ with corresponding subarcs $D_1,\dots,D_m$ such that $|D_i|>(1-\frac{1}{2}\zeta)|Q_i|$, and suppose $D_i \| D_{i+1}$ for all $i$.  Letting $C_1,\dots,C_m\in\textbf{R}$ be the directed arcs from which $Q_1,\dots,Q_m$ arise, then again either:

\begin{itemize}

\item $Q_i$ corresponds to an initial subarc $C_i'$ of $C_i$ with $|C_i'|\geq\frac{1}{2}\zeta|C_i|$, while $D_i$ corresponds to a subarc $E_i$ of $C_i'$ with $|E_i|>(1-\frac{1}{2}\zeta)|C_i'|$.

\item $Q_i$ corresponds to a terminal subarc of $C_i$ with $|Q_i|>(1-\frac{1}{2}\zeta)|C_i|$.  In this case, taking $C_i'=C_i$, $E_i$ is a subarc of $C_i'$ with $|E_i|>(1-\frac{1}{2}\zeta)^2|C_i'|>(1-\zeta)|C_i'|$.

\end{itemize} 

In either case, $C_i'$ is an initial subarc of $C_i$ with $|C_i'|\geq\frac{1}{2}\zeta|C_i|$, while $E_i$ is a subarc of $C_i'$ with $|E_i|>(1-\zeta)|C_i'|$.  But since no chord and directed arc cross twice and all of this is contained in the subspace $\pazocal{D}_j$ homeomorphic to a disk, $D_i \| D_{i+1}$ immediately implies $E_i \| \| E_{i+1}$, providing a contradiction to condition $P'(\zeta,m)$.

Hence, $(\textbf{T}_j,\textbf{Q}_j)$ satisfies condition $P(\zeta/2,m)$, and so \Cref{G design} implies $\ell(\textbf{Q}_j)\leq C_{\zeta/2,m}|\textbf{T}_j|$.

Now, fix $C\in\textbf{R}$ such that $C$ has a subarc contained in $\pazocal{D}_j$.  If there is an arc $Q\in\textbf{Q}_j$ arising from $C$, then by construction we have $|Q|\geq2\zeta^{-1}|C|$.  Otherwise, either:

\begin{enumerate}

\item $C$ begins in $\pazocal{D}_j$ and its maximal initial subarc contained in $\pazocal{D}_j$ has length less than $\frac{1}{2}\zeta|C|$.

\item $C$ begins outside of $\pazocal{D}_j$ and its maximal terminal subarc contained in $\pazocal{D}_j$ has length less than $(1-\frac{1}{2}\zeta)|C|$.

\end{enumerate}

Since these directed arcs do not cross the same chord twice, though, then by maximality $C$ intersects at most one other disk $\pazocal{D}_k$.  In either of the two outstanding cases, $C$ must intersect such a disk, in which case $C$ crosses $T_k$ and has:

\begin{enumerate}

\item  a terminal subarc of length at least $(1-\frac{1}{2}\zeta)|C|$ contained in $\pazocal{D}_k$.

\item an initial subarc of length at least $\frac{1}{2}\zeta|C|$ contained in $\pazocal{D}_k$.

\end{enumerate}

Hence, in either case there is an arc in $\textbf{Q}_k$ arising from $C$.  So, $$\ell(\textbf{R})\leq2\zeta^{-1}\sum\ell(\textbf{Q}_j)\leq2\zeta^{-1}\sum C_{\zeta/2,m}|\textbf{T}_j|=2\zeta^{-1}C_{\zeta/2,m}|\textbf{T}|$$
and thus the statement is given for $K_{\zeta,m}\geq4\zeta^{-1}C_{\zeta/2,m}$.

\end{proof}

We now apply this to the diagrams of interest, generalizing the concept of shaft.

\begin{definition}

Let $\Delta$ be a reduced annular diagram over $G(\textbf{M}_\textbf{S})$.  Fix a disk $\Pi$ in $\Delta$ and a $t$-spoke $\pazocal{Q}$ of $\Pi$ with history $H$ (read starting at $\Pi$).  Given a prefix $H'$ of $H$, suppose:

\begin{enumerate}[label=(\roman*)]

\item $\lab(\partial\Pi)$ is $H'$-admissible.

\item Each cell of the subband $\pazocal{Q}'$ of $\pazocal{Q}$ with history $H'$ belongs to a non-annular maximal $\theta$-band of $\Delta$.

\item Any maximal $\theta$-band of $\Delta$ crosses $\pazocal{Q}'$ at most twice.  In the case that there are exactly two crossings, these crossings correspond to inverse letters in $H'$.

\item Given a prefix $H''$ of $H'$ with $\|H''\|\geq\frac{1}{2}\zeta\|H'\|$, for any factorization $H''\equiv H_1''H_2''H_3''$ with $\|H_1''\|+\|H_3''\|<\zeta\|H''\|$, $H_2''$ contains a controlled subword.

\end{enumerate}

Then $\pazocal{Q}'$ is called a \textit{$\zeta$-spear} of $\Pi$.

For every disk $\Pi$ and $t$-spoke $\pazocal{Q}$ of $\Pi$, fix a maximal subband $\pazocal{Q}_\Pi^\zeta$ of $\pazocal{Q}$ (perhaps empty) which is a $\zeta$-spear of $\Pi$.  Then, define $\rho_\zeta(\Delta)$ to be the sum of the lengths of the bands $\pazocal{Q}_\Pi^\zeta$ for all $\Pi$ and $\pazocal{Q}$.

\end{definition}

While we have defined everything for any arbitrary parameter $\zeta<1/2$, we now restrict our attention to the case $\zeta=\lambda$, where $\lambda$ is as given by the parameter assignment of $\lambda^{-1}$ in \Cref{sec-parameters}.  We thus reach the following analogue of \Cref{G design} for any minimal annular diagram $\Delta$, taking $|\partial\Delta|$ to be the sum of the lengths of the boundary components of $\Delta$.

\begin{lemma} \label{G directed design}

For any minimal annular diagram $\Delta$, $\rho_\lambda(\Delta)\leq K|\partial\Delta|$.

\end{lemma}

\begin{proof}

We create the directed design $(\textbf{T},\textbf{R})$ as an auxiliary structure to $\Delta$ by:

\begin{itemize}

\item Adding a chord through the middle of any non-annular maximal $\theta$-band.

\item For any maximal $\lambda$-spear $\pazocal{Q}_\Pi^\lambda$, adding a directed arc along a middle line directed away from $\Pi$ (again, making room accordingly in case the same maximal $t$-band contains two overlapping $\lambda$-spears at distinct disks).

\end{itemize}

We first verify that this in fact defines a directed design.  As the definition of $\lambda$-spear indicates that a chord and directed arc can cross at most twice, what is left to verify is the condition imposed on such a double crossing.


Assume there exists a non-annular maximal $\theta$-band $\pazocal{T}$ which crosses a maximal $\lambda$-spear $\pazocal{Q}$ twice.  \Cref{G(S) annuli} implies these bands cannot form a $(\theta,q)$-annulus.  So, since the two crossings correspond to inverse letters of the history of $\pazocal{Q}$, their crossings must circle the inner hole.  In particular, letting $\pazocal{Q}'$ and $\pazocal{T}'$ be the subbands of $\pazocal{Q}$ and $\pazocal{T}$, respectively, between the two crossings, $\pazocal{Q}'$ and $\pazocal{T}'$ bound an annular subdiagram of $\Delta$.  As such, the subsegments of the directed arc and chord corresponding to these subbands form a non-contractible loop, implying condition (1).

Since $\pazocal{T}$ cannot cross itself or $\pazocal{Q}'$ again, it must not be a radial $\theta$-band.  Of course, this means the corresponding chord is also not radial, and so condition (2) is satisfied.

Now, suppose every maximal $\theta$-band that crosses $\pazocal{Q}'$ crosses it just once.  Since it also cannot cross $\pazocal{T}$, such a $\theta$-band must be radial, and hence correspond to the crossing of a radial chord with the relevant subsegment of the directed arc corresponding to $\pazocal{Q}$.  Hence, condition (3) is satisfied as long as $\pazocal{Q}'$ is crossed by some $\theta$-band; but if it is not, then the two cells which correspond to the crossing with $\pazocal{T}$ are adjacent, and so cancellable as they have inverse histories.

So, we may assume there exists a maximal $\theta$-band that crosses $\pazocal{Q}'$ twice.  Then the corresponding $\theta$-band and $\pazocal{Q}$ can be studied in the same way, with a subband $\pazocal{Q}''$ of $\pazocal{Q}'$ playing the analogous role.  If all $\theta$-bands that cross $\pazocal{Q}''$ cross it just once, then the above argument implies the existence of a radial $\theta$-band that crosses it, and so crosses $\pazocal{Q}'$; otherwise, we pass to a shorter subband $\pazocal{Q}'''$, so that an inductive argument applies.  

Thus, condition (3) is satisfied, and so $(\textbf{T},\textbf{R})$ is a directed design on $\Delta$.

Now, fix distinct directed arcs $C_1,\dots,C_m\in\textbf{R}$ with initial subarcs $C_1',\dots,C_m'$, respectively, such that $|C_i'|\geq\frac{1}{2}\lambda|C_i|$ for all $i$.  Suppose there exist subarcs $D_i$ of $C_i'$ with $|D_i|>(1-\lambda)|C_i'|$ and $D_i \| \| D_{i+1}$ for all $i$.

For each $i$, let $H_i$ be the history of the $t$-spoke $\pazocal{Q}_i$ corresponding to $C_i$, $H_i'$ be the prefix which is the history of the subband $\pazocal{Q}_i'$ corresponding to $C_i'$, and $H_i''$ be the subword of $H_i'$ which is the history of the $t$-band $\pazocal{Q}_i''$ corresponding to $D_i$.  

By the definition of $\lambda$-spear, each $H_i''$ contains a controlled subword.  Further, the definition of the aligned relation implies every maximal $\theta$-band that crosses $\pazocal{Q}_1''$ does so once, and also crosses $\pazocal{Q}_i''$ exactly once.  Assuming $\pazocal{Q}_1''$ and $\pazocal{Q}_i''$ are distinct bands, the second condition defining the aligned relation further implies that the minimal subbands of any pair of such maximal $\theta$-bands which crosses both $\pazocal{Q}_1''$ and $\pazocal{Q}_i''$ bound a quasi-trapezium $\Gamma$ in $\Delta$ with history $H_1''$.

If $\Gamma$ has at least $L+1$ distinct maximal $t$-bands, then we find a contradiction to the minimality of $\Delta$ in just the same way as in the proof of \Cref{G design} (see the sketch of the proof provided after the statement).  So, since at most two of $\pazocal{Q}_i''$ correspond to the same $t$-band, it follows that $m\leq 2L$.  In particular, this means $(\textbf{T},\textbf{R})$ satisfies condition $P'(\lambda,2L+1)$.

The statement thus follows by applying \Cref{directed design} and the parameter choices $K>>L>>\lambda^{-1}$.

\end{proof}

\begin{remark}

It is clear that the condition of being $\lambda$-spear is substantially stronger than that of being a $\lambda$-shaft, and the conditions for a structure to be a directed design are more stringent than for one to be a design.  Hence, while $\sigma_\lambda$ proved a useful invariant in studying circular diagrams in previous literature, the prospect of $\rho_\lambda$ being put to similar use may not seem promising.  

However, note that \Cref{long history controlled} provides a powerful condition which can be used to verify that a particular spoke satisfies condition (iv) in the definition of $\lambda$-spear, and indeed in previous literature the analogue of this condition is one of the main tools used to verify a particular spoke is a $\lambda$-shaft.  Moreover, in the particular setting where this concept will prove useful, {\frenchspacing i.e. the `big scopes' in the next section}, the shafts considered automatically satisfy the rest of the properties of being a $\lambda$-spear.   Thus, we will be able to use analogues of the arguments of previous literature with this new invariant, helping us to study the structure of a generic class of annular diagrams.

\end{remark}

\bigskip


\section{Scopes} \label{sec-scopes}

In this section, we study a type of the subdiagram that arises from Lemmas \ref{circular-graph} and \ref{annular-graph}.  This notion was introduced in \cite{WMal} and is useful in simultaneously treating the relevant arguments for circular and annular diagrams.  In particular, in this section we both:

\begin{enumerate} 

\item Achieve an upper bound on the $G$-area of a certain generic class of circular diagrams over the disk presentation of $G(\textbf{M}_\textbf{S})$, which we will subsequently prove is sufficient for an upper bound on the Dehn function of the group; and

\item Rule out the presence of such a subdiagram in a generic class of annular diagrams over the disk presentation of $G(\textbf{M}_\textbf{S})$, a critical first step toward understanding the conjugator length function of the group.

\end{enumerate}

\subsection{Definition of scope} \

Let $\Pi$ be a disk in a reduced circular or annular diagram $\Delta$ over the disk presentation of $G(\textbf{M}_\textbf{S})$.  Suppose $\Pi$ has $\ell\geq2$ consecutive $t$-spokes $\pazocal{Q}_1,\dots,\pazocal{Q}_\ell$ such that:

\begin{itemize}

\item The $t$-spokes $\pazocal{Q}_1,\dots,\pazocal{Q}_\ell$ all end on the same component $\textbf{C}$ of $\partial\Delta$

\item A subpath of $\partial\Pi$, a subpath of $\textbf{C}$, and the $t$-spokes $\pazocal{Q}_1,\dots,\pazocal{Q}_\ell$ bound a (circular) subdiagram $\Psi$ of $\Delta$ which contains no disks.

\end{itemize}

Then $\Psi$ is called a \textit{scope of $\Pi$ on $\mathbf{C}$} in $\Delta$ with \textit{width} $\ell$.

\begin{remark} \label{rmk-scopes}

Let $\Delta$ be a reduced circular or annular diagram over the disk presentation of $G(\textbf{M}_\textbf{S})$ such that $s_1(\Delta)\geq1$ is minimal among all such diagrams with the same boundary label(s).

If $\Delta$ is circular, then \Cref{circular-graph} implies the existence of a scope on (the only component of) $\partial\Delta$ with width $\geq L-3$.  As such, \Cref{fig-graph} is an apt picture of a scope (though the width of a scope is not generally required to be as large as the pictured $L-3$).

On the other hand, if $\Delta$ is annular, then \Cref{annular-graph} implies the existence of a disk $\Pi$ with almost all of its spokes going to the boundary, but does not immediately imply the existence of a scope on a boundary component with width $>L/2$.

\end{remark}


\medskip

\subsection{Weakly minimal diagrams} \

As alluded to at the start of this section, our goal in the rest of this section is to simultaneously address the existence of a scope of large width in both a particular class of circular and a particular class of annular diagrams over the disk presentation of $G(\textbf{M}_\textbf{S})$. 

For circular diagrams, our goal is to bound the $G$-area.  However, as in \cite{O18}, \cite{OS20}, \cite{WEmb}, and \cite{W}, this bound is not achieved for minimal diagrams, but for a slightly wider class of diagrams that allows for the necessary surgeries.

Let $\Delta$ be a reduced circular diagram over the disk presentation of $G(\textbf{M}_\textbf{S})$ containing at least one disk.  A maximal $q$-band $\pazocal{Q}$ in $\Delta$ is said to be \textit{cutting} if it has two ends on $\partial\Delta$.  This naming is indicative of the fact that the $q$-band cuts $\Delta$ into two connected components, each of which can be viewed (using $0$-refinement) as a circular subdiagram.  If one such subdiagram contains some positive cell and no disks, then it is called a \textit{crown} of $\Delta$.  Removing such a crown, we obtain a reduced diagram containing just as many disks as $\Delta$.  Iterating this process until no crown exists, we obtain the \textit{stem} $\Delta^*$.  It is important to note that the stem is well-defined, {\frenchspacing i.e. it does not} depend on the order in which we remove the different crowns.

\begin{definition}

A reduced circular diagram $\Delta$ over the disk presentation of $G(\textbf{M}_\textbf{S})$ containing at least one disk is said to be \textit{weakly minimal} if its stem $\Delta^*$ is minimal.

\end{definition}

Of course, as any subdiagram of a minimal diagram is necessarily minimal, it is immediate that minimal diagrams are weakly minimal.  The next statement collects some less obvious consequences of the definition.

\begin{lemma}[Lemma 7.17 of \cite{OS20}] \label{weakly-minimal} \

\begin{enumerate}[label=({\alph*})]

\item If $\Delta_1$ is a subdiagram of a weakly minimal diagram $\Delta$ and contains a disk, then $\Delta_1$ is weakly minimal, $\Delta_1^*\subset\Delta^*$, and $\sigma_\lambda(\Delta_1^*)\leq\sigma_\lambda(\Delta^*)$.

\item For every weakly minimal diagram $\Delta$, $\sigma_\lambda(\Delta^*)\leq K|\partial\Delta|$.

\item A weakly minimal diagram $\Delta$ contains no $\theta$-annuli.

\item Let $\pazocal{C}$ be a cutting $q$-band of a reduced diagram $\Delta$ over the disk presentation of $G(\textbf{M}_\textbf{S})$ and let $\Delta_1$, $\Delta_2$ be the components of $\Delta\setminus\pazocal{C}$. Suppose $\Delta_1\cup\pazocal{C}$ is a reduced diagram over $M(\textbf{M}_\textbf{S})$ and $\Delta_2\cup\pazocal{C}$ is weakly minimal. Then $\Delta$ is weakly minimal.

\end{enumerate}

\end{lemma}

The next statement also follows quickly from the definition.

\begin{lemma} \label{weakly-minimal-1-signature}

Let $\Delta$ and $\Gamma$ be reduced circular diagrams over $G(\textbf{M}_\textbf{S})$ with $\lab(\partial\Delta)\equiv\lab(\partial\Gamma)$.  If $\Delta$ is weakly minimal, then $s_1(\Delta)\leq s_1(\Gamma)$.

\end{lemma}

\begin{proof}

Assume to the contrary that $s_1(\Gamma)<s_1(\Delta)$.  We then attach mirror copies of the crowns of $\Delta$ to the appropriate portions of $\partial\Gamma$, producing a diagram $\Gamma'$ with the same boundary label as $\Delta^*$.  But no crown contains a disk, so that $s_1(\Gamma')=s_1(\Gamma)<s_1(\Delta)=s_1(\Delta^*)$, contradicting the minimality of the stem.

\end{proof}

\subsection{Counterexample diagrams} \label{sec-counterexamples} \

The arguments that span the rest of this section pertain to diagrams of different `type', and so require completely separate setups.  However, there is somewhat of a common framework: Our goal is to verify some property for all diagrams of the given type; to do so, we define a parameter $n(\Gamma)$ for all diagrams $\Gamma$ of the given type (the definition of which varies by case) and fix a potential `minimal counterexample' diagram $\Delta$ of this type which both violates the desired property and has minimal $n(\Delta)$ for all such counterexamples (note that this latter assumption is not necessary in the annular case, but is convenient for uniformity).  The arguments spanning the rest of the section are then all in service of showing that $\Delta$ cannot exist, thus proving that all diagrams satisfy the given property.

We now detail the exact setup for each case.

\subsubsection{Circular counterexample diagram} \

For circular diagrams, our general aim is to find an upper bound on the $G$-area of weakly minimal diagrams in terms of their perimeters.  To achieve this, given a weakly minimal diagram $\Gamma$ define $$n(\Gamma)=|\partial\Gamma|+\sigma_\lambda(\Gamma^*)$$ 
Our goal is to then show that $\text{Area}_G(\Gamma)\leq N_4\phi(N_4n(\Gamma))+N_3\mu(\Gamma)g(N_3n(\Gamma))$ for any weakly minimal diagram $\Gamma$.  Hence, as in \Cref{sec-diskless-upper-bound} the study of a `minimal counterexample' amounts to fixing a weakly minimal diagram $\Delta$ such that
$$\text{Area}_G(\Delta)>N_4\phi(N_4n)+N_3\mu(\Delta)g(N_3n)$$
for $n=n(\Delta)$, while the desired inequality holds for any weakly minimal diagram $\Gamma$ with $n(\Gamma)<n$.

It is crucial to note that \Cref{weakly-minimal-1-signature} and \Cref{rmk-scopes} combine to imply $\Delta$ contains a scope $\Psi$ of width $\ell\geq L-3$.

\subsubsection{Annular counterexample diagram} \

For annular diagrams, our general aim is to rule out the existence of scopes of very large width in a generic class of minimal annular diagrams.

To begin, for a minimal annular diagram $\Gamma$ define the parameter
$$n(\Gamma)=|\partial\Gamma|+\rho_\lambda(\Gamma)$$
where, naturally, $|\partial\Gamma|$ is the sum of the lengths of its two boundary components and $\rho_\lambda(\Gamma)=0$ if $s_1(\Gamma)=0$.  The minimal annular diagram $\Gamma$ is then called \textit{$n$-minimal} if the value of $n(\Gamma)$ is minimal among all minimal annular diagrams realizing the same conjugacy relation in $G(\textbf{M}_\textbf{S})$.

It is important to note that for any pair of conjugate elements $g$ and $h$ of $G(\textbf{M}_\textbf{S})$ and any pair of words $u,v$ over $\pazocal{X}$ which represent $g,h$, respectively, van Kampen's lemma provides a reduced annular diagram over the disk presentation of $G(\textbf{M}_\textbf{S})$ with boundary labels $u$ and $v$.  Hence, there exists a minimal annular diagram for every pair $u,v$, and so there exists an $n$-minimal diagram for every pair $g,h$.

However, note that the condition of $n$-minimality depends on the pair of elements of $G(\textbf{M}_\textbf{S})$ and not the particular words over $\pazocal{X}$.  In particular, a minimal annular diagram $\Gamma$ need not be $n$-minimal even if $n(\Gamma)$ is minimal among all such diagrams with the same boundary labels, as there may be a minimal annular diagram $\Gamma'$ with $n(\Gamma')<n(\Gamma)$ whose boundary labels are different words that represent the same elements of $G(\textbf{M}_\textbf{S})$.

Our goal is then to show that any scope in an $n$-minimal diagram has width $<L-k$.  Thus, in this context the study of a `minimal counterexample' amounts to fixing an $n$-minimal (annular) diagram $\Delta$ containing a scope $\Psi$ of width $\ell\geq L-k$ on one of its components.

\medskip


\subsection{Combs and $\theta$-bands in counterexample diagrams} \

Through the rest of this section, we assume $\Delta$ is a counterexample diagram and deduce facts about its makeup until finally determining that it cannot exist in the first place.  Whether $\Delta$ is a circular diagram (and so a weakly minimal diagram) or an annular diagram (and so an $n$-minimal diagram), we may assume $\Delta$ contains a scope $\Psi$ of width $\ell=L-k$.

Many of the arguments of these sections follow the same general outline of Section 7 of \cite{OS20}.  Of course, there are several major deviations:

\begin{itemize}

\item There is no analogue of the argument for annular diagrams in \cite{OS20}.

\item The desired bounds in this setting are given by more general functions than those in \cite{OS20} (we can interpret the arguments of \cite{OS20} as applying to the specific choices $\phi(x)=x^2$ and $g(x)=1$).  In this way, the arguments of this section are closer to those of Section 11 of \cite{WEmb} or Section 9 of \cite{O18}.

\item The analogous argument in \cite{OS20} relies on the existence of a scope of width $L-3$, whereas for uniformity in the two argument paths in this setting we assume the weaker condition that $\Psi$ has width $\ell=L-k$.

\end{itemize}

Thus, the proofs generally proceed in two steps: 

\begin{enumerate}[label=(\arabic*)]

\item An argument for why the statement holds if $\Delta$ is an annular diagram.

\item How to adapt the numerical arguments of \cite{OS20} to show that the statement holds if $\Delta$ is a circular diagram.

\end{enumerate}

We begin with a basic example that demonstrates this method, which uses \Cref{minBoundaryLengths-theta} to deduce an important property of the rim $\theta$-bands in $\Delta$.

\begin{lemma}[Lemma 7.19 of \cite{OS20}] \label{counterexample-long-theta}

The base of any rim $\theta$-band has length greater than $K$.

\end{lemma}

\begin{proof}

Suppose $\Delta$ has a rim $\theta$-band $\pazocal{T}$ whose base has length $\leq K$.  

(1) Using $0$-refinement, we may remove $\pazocal{T}$ from $\Delta$ to obtain a reduced annular diagram $\Delta'$ over the disk presentation of $G(\textbf{M}_\textbf{S})$.  Note that since $\Delta$ is minimal, $\Delta'$ must also be, as otherwise we can add a copy of $\pazocal{T}$ to contradict the minimality of $\Delta$.  \Cref{minBoundaryLengths-theta} then implies that the boundary labels of $\Delta'$ represent the same elements of $G(\textbf{M}_\textbf{S})$ as those of $\Delta$ and that $|\partial\Delta'|<|\partial\Delta|$.  

Now, for any disk $\Pi$ in $\Delta$, removing $\pazocal{T}$ removes a non-annular maximal $\theta$-band which crosses any $t$-spoke of $\Pi$ at most once.  So, since $\Pi$ is also contained in $\Delta'$, a $\lambda$-spear of $\Pi$ in $\Delta'$ is also a $\lambda$-spear of $\Pi$ in $\Delta$.  Hence, $\rho_\lambda(\Delta')\leq\rho_\lambda(\Delta)$.


But then $n(\Delta')<n(\Delta)$, contradicting the $n$-minimality of $\Delta$.

(2) The proof follows in much the same way as that of \Cref{short theta-bands}: 

Lemmas \ref{minBoundaryLengths-theta} and \ref{weakly-minimal}(a) imply the reduced diagram $\Delta'$ obtained from $\Delta$ by removing $\pazocal{T}$ is weakly minimal with $n(\Delta')\leq n(\Delta)-1$.  Applying the inductive hypothesis, we then arrive at the contradiction $\text{wt}_G(\Delta)\leq N_4\phi(N_4n)+N_3\mu(\Delta)g(N_3n)$ by using \Cref{G-area subdiagrams}, \Cref{lengths}, and the parameter assignment $N_4>>\delta^{-1}$.

\end{proof}

A more advanced example is next, based on a reworking of the proofs of \Cref{sec-diskless-upper-bound}.

\begin{lemma}[Compare with Lemma 7.21 of \cite{OS20}] \label{minimal-subcombs}

Any subcomb of $\Delta$ contained in the scope $\Psi$ has basic width at most $K_0$.

\end{lemma}

\begin{proof}

(1) Cutting along the bottom of the handle of any subcomb (and using $0$-refinement) produces a reduced annular diagram $\Delta'$ over the disk presentation of $G(\textbf{M}_\textbf{S})$.  Again, $\Delta'$ must be minimal.  Further, by \Cref{minBoundaryLengths} we again have that the boundary labels of $\Delta'$ represent the same elements of $G(\textbf{M}_\textbf{S})$ as those of $\Delta$, while $|\partial\Delta'|<|\partial\Delta|$.  

As in the proof of \Cref{counterexample-long-theta}, any disk $\Pi$ in $\Delta$ is also contained in $\Delta'$.  This time, though, the $t$-spokes of $\Pi$ in $\Delta'$ have the same length as those of $\Pi$ in $\Delta$.  Moreover, given a maximal $\theta$-band $\pazocal{T}$ of $\Delta$ which crosses a $t$-spoke $\pazocal{Q}$ of $\Pi$, the corresponding maximal $\theta$-band $\pazocal{T}'$ of $\Delta'$ is either unaffected by the surgery or obtained by removing any subbands which enter the subcomb through its handle.  As such, $\pazocal{T}$ is annular if and only if $\pazocal{T}'$ is, and these bands cross $\pazocal{Q}$ the same number of times and in the same way.  Hence, a $\lambda$-spear of $\Pi$ in $\Delta'$ is also a $\lambda$-spear of $\Pi$ in $\Delta$, so that $\rho_\lambda(\Delta')\leq\rho_\lambda(\Delta)$.

But as in the proof of \Cref{counterexample-long-theta} this contradicts the $n$-minimality of $\Delta$.

(2) For this case, we repeat the proofs of \Cref{sec-diskless-upper-bound}: 

\Cref{tight subcomb existence} implies the existence of a tight subcomb of $\Delta$; analogues of \Cref{counterexample combs} and \Cref{lem-bigsubcomb} can then be proved for this subcomb, replacing the parameters $N_1,N_2$ with $N_3,N_4$, respectively; then, applying \Cref{counterexample-long-theta} in place of \Cref{short theta-bands}, we prove the analogue of \Cref{circular diskless} by using \Cref{weakly-minimal}(d).

\end{proof}

\begin{remark} \label{rmk-annular-combs}

\Cref{minimal-subcombs}(2) illustrates the importance of considering weakly minimal diagrams: 

The surgery necessary for the proof of \Cref{circular diskless} combines two subdiagrams of the counterexample diagram.  If we were to consider minimal diagrams in this section, then the minimality of our counterexample diagram would imply the two subdiagrams are also minimal, but there is no reason to assume that their combination should result in a minimal diagram.  

Instead, \Cref{weakly-minimal}(d) fills this role, implying the combination of weakly minimal diagrams for this purpose results in another weakly minimal diagram.

Further, the proof of \Cref{minimal-subcombs}(1) can obviously be applied to a stronger statement: In the annular case, the existence of any subcomb of $\Delta$ being contained in the scope $\Psi$ can be ruled out.

\end{remark}

\medskip


\subsection{The makeup of the scope} \

Let $\Pi$ be the disk and $\pazocal{Q}_1,\dots,\pazocal{Q}_\ell$ be the consecutive $t$-spokes of $\Pi$ defining the scope $\Psi$. For any $1\leq i<j\leq\ell$, define the scope $\Psi_{ij}$ to be the subdiagram of $\Psi$ bounded by $\pazocal{Q}_i$ and $\pazocal{Q}_j$.

For fixed $i,j$, let $\textbf{p}_{ij}$ be the subpath of $\partial\Delta$ shared with $\partial\Psi_{ij}$.  Letting $\textbf{u}_{ij}$ be the maximal subpath of $\partial\Pi$ which is not shared with $\partial\Psi_{ij}$, define the path $\bar{\textbf{p}}_{ij}$ homotopic to $\textbf{p}_{ij}$ consisting of (the inverse of) a side of $\pazocal{Q}_i$, a side of $\pazocal{Q}_j$, and $\textbf{u}_{ij}^{-1}$.  Note then that cutting along $\bar{\textbf{p}}_{ij}$ separates $\Delta$ into two connected components: A circular subdiagram $\bar{\Delta}_{ij}$ consisting of $\Psi_{ij}$ and $\Pi$ and another, $\Psi_{ij}'$, which is of the same `type' as $\Delta$ (see \Cref{fig-clove-paths}).  Note that in either case, $\Psi_{ij}'$ is necessarily minimal.

In the case where $i=1$ and $j=\ell$, it is convenient to drop indices, so that $\textbf{p}$ is the subpath of $\partial\Psi$ shared with $\partial\Delta$ and is homotopic to the path $\bar{\textbf{p}}$, which has a maximal subpath $\textbf{u}^{-1}$ shared with $(\partial\Pi)^{-1}$ and separates $\Delta$ into two connected components: A circular subdiagram $\bar{\Delta}$ and another, $\Psi'$, of the same `type' as $\Delta$.

\begin{figure}
\centering
\includegraphics[scale=0.75]{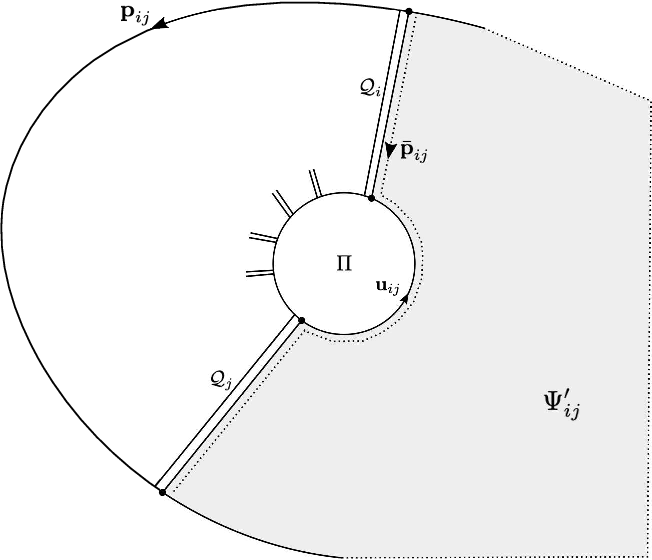}
\caption{Subdiagrams and paths in $\Delta$}
\label{fig-clove-paths}
\end{figure}

The next statement then follows in exactly the same way as its analogue in \cite{OS20}, with the main ingredient of the proof given by \Cref{two theta-bands about disk}.

\begin{lemma}[Lemma 7.22 of \cite{OS20}] \label{scope-theta-bands} \

\begin{enumerate}[label=({\arabic*})]

\item Every maximal $\theta$-band of $\Psi$ crosses either $\pazocal{Q}_1$ or $\pazocal{Q}_\ell$.

\item There exists an $r$ satisfying $L/2-k=\ell-L/2\leq r\leq L/2$ such that the maximal $\theta$-bands of $\Psi$ that cross $\pazocal{Q}_\ell$ do not cross $\pazocal{Q}_r$ and those that cross $\pazocal{Q}_1$ do not cross $\pazocal{Q}_{r+1}$.

\end{enumerate}

\end{lemma}

Let $H_i$ be the history of $\pazocal{Q}_i$ read starting at $\Pi$ and $h_i=\|H_i\|$.  Then \Cref{scope-theta-bands} implies $H_{i+1}$ is a prefix of $H_i$ for all $1\leq i\leq r-1$, while $H_i$ is a prefix of $H_{i+1}$ for $r+1\leq i\leq \ell-1$.  Hence,
$$h_1\geq h_2\geq\dots \geq h_r; \ \ h_{r+1}\leq\dots\leq h_\ell$$

Further, note that \Cref{scope-theta-bands} immediately implies $|\textbf{p}|_\theta$ is the sum of the lengths of $\pazocal{Q}_1$ and $\pazocal{Q}_\ell$.  

\begin{lemma}[Lemma 7.23 of \cite{OS20}] \label{scope-short-i,i+1}

For $1\leq i\leq \ell-1$, $|\textbf{p}_{i,i+1}|_q<3K_0$.

\end{lemma}

\begin{proof}

(1) As noted in \Cref{rmk-annular-combs}, the proof of \Cref{minimal-subcombs} may be adapted in the annular case to show that no subcomb of $\Delta$ is contained in $\Psi$.  As such, every maximal $q$-band of $\Psi$ connects $\textbf{p}$ to $\partial\Pi$.  It then follows immediately from the definition of the disk relations that $|\textbf{p}_{i,i+1}|_q=N+1$, so that the statement follows from the parameter choice $K_0>>N$.

(2) The proof for the circular case follows in just the same way as its analogue in \cite{OS20}, using an adaptation of \Cref{counterexample combs} as discussed in the proof of \Cref{minimal-subcombs}.

\end{proof}

Let $W$ be the disk word corresponding to $\Pi$, {\frenchspacing i.e. so that $\lab(\partial\Pi)\equiv W^{\pm1}$}.  Note that the makeup of the disk words implies $W(i)$ and $W(j)$ are coordinate shifts of one another for all $1\leq i,j\leq L$, and are indeed copies of some accepted configuration $V$ of the $k$-enhanced machine.

The next two statements then follow in just the same way as their analogues in \cite{OS20}.

\begin{lemma}[Lemma 7.24 of \cite{OS20}] \label{scope-path-inequalities} 

Suppose $i\leq r$ and $j\geq r+1$.

\begin{enumerate}[label=({\arabic*})]

\item $|\textbf{p}_{ij}|\geq|\textbf{p}_{ij}|_\theta+|\textbf{p}_{ij}|_q\geq h_i+h_j+N(j-i)+1$

\item $|\bar{\textbf{p}}_{ij}|\leq h_i+h_j+N(L-j+i)+(L-j+i)\delta|V|_a-1$

\end{enumerate}

\end{lemma}

\begin{lemma}[Lemma 7.25 of \cite{OS20}] \label{scope-mixture}

If $\Delta$ is circular and $1\leq i<j\leq \ell-1$ with $j-i\geq L/2$, then $$\mu(\Delta)-\mu(\Psi_{ij}')>-2J|\partial\Delta|(h_i+h_j)\geq-2J|\partial\Delta||\textbf{p}_{ij}|$$

\end{lemma}

The arguments that follow require slightly different setups based on case.  For uniformity, we introduce the notation $\tau_\lambda$ given by:

\begin{enumerate}

\item In the case that $\Delta$ is annular, then:

\begin{itemize}

\item For any annular diagram $\Gamma\subset\Delta$, we set $\tau_\lambda(\Gamma)=\rho_\lambda(\Gamma)$.

\item For any circular diagram $\Gamma\subset\Delta$, we set $\tau_\lambda(\Gamma)$ to be the sum of the lengths of the maximal $\lambda$-spears in $\Delta$ which are contained in $\Gamma$.

\end{itemize}

\item In the case that $\Delta$ is circular, then for any subdiagram $\Gamma\subset\Delta$, we set: 

\begin{itemize}

\item $\tau_\lambda(\Gamma)=\sigma_\lambda(\Gamma^*)$ if $\Gamma$ has a disk.

\item $\tau_\lambda(\Gamma)=0$ if $\Gamma$ has no disks.

\end{itemize}

\end{enumerate}


%
%
%
%
%
%

\begin{lemma}[Compare with Lemma 7.26 of \cite{OS20}] \label{p_ij upper bound}

If $1\leq i<j\leq \ell-1$ with $j-i\geq L/2$, then $$|\textbf{p}_{ij}|+\tau_\lambda(\bar{\Delta}_{ij})\leq|\textbf{p}_{ij}|+\tau_\lambda(\Delta)-\tau_\lambda(\Psi_{ij}')<(1+\eps)|\bar{\textbf{p}}_{ij}|$$ for $\eps=1/\sqrt{N_4}$.

\end{lemma}

\begin{proof}

If $\Delta$ is circular then since $\bar{\Delta}_{ij}$ and $\Psi_{ij}'$ have no common $t$-spokes, \Cref{weakly-minimal}(a) implies $\tau_\lambda(\Delta)\geq\tau_\lambda(\bar{\Delta}_{ij})+\tau_\lambda(\Psi_{ij}')$.

If $\Delta$ is annular, on the other hand, it is necessary to verify that $\Psi_{ij}'$ does not have any $\lambda$-spears which are not already $\lambda$-spears in $\Delta$.  This is generally possible when passing to an annular subdiagram of $\Delta$ if, for example, we cut a maximal $\theta$-band so that it is no longer annular or separated into two different bands.  By \Cref{scope-theta-bands}, though, every maximal $\theta$-band of $\Psi_{ij}$ crosses exactly one of $\pazocal{Q}_i$ and $\pazocal{Q}_j$, and so has an end on $\textbf{p}_{ij}$.  As such, any maximal $\theta$-band of $\Delta$ that crosses one of these spokes does so at most twice and continues to $\textbf{p}_{ij}$, so that we are simply removing a `tail' of this $\theta$-band and having it still hit the boundary.  Thus, any $\lambda$-spear of $\Psi_{ij}'$ is already a $\lambda$-spear of $\Delta$, and so again since $\bar{\Delta}_{ij}$ and $\Psi_{ij}'$ have no common $t$-spokes we have $\tau_\lambda(\Delta)\geq\tau_\lambda(\bar{\Delta}_{ij})+\tau_\lambda(\Psi_{ij}')$.


Hence, letting $y=|\textbf{p}_{ij}|+\tau_\lambda(\Delta)-\tau_\lambda(\Psi_{ij}')$ and $d=y-|\bar{\textbf{p}}_{ij}|$, it suffices to assume $d\geq\eps|\bar{\textbf{p}}_{ij}|$ and argue toward a contradiction.  As in the proof in \cite{OS20}, this implies $d\geq\frac{\eps y}{2}$.

Let $\textbf{C}$ be the component of $\partial\Delta$ containing $\textbf{p}_{ij}$ and $\textbf{s}$ be the complement of $\textbf{p}_{ij}$ in $\textbf{C}$.  Then \Cref{lengths} implies $|\partial\Delta|=|\textbf{p}_{ij}|+|\textbf{s}|$ and $|\partial\Psi_{ij}'|\leq|\bar{\textbf{p}}_{ij}|+|\textbf{s}|$.  Hence,
$$(|\partial\Delta|+\tau_\lambda(\Delta))-(|\partial\Psi_{ij}'|+\tau_\lambda(\Psi_{ij}'))\geq|\textbf{p}_{ij}|-|\bar{\textbf{p}}_{ij}|+\tau_\lambda(\bar{\Delta}_{ij})\geq d>0$$
But then if $\Delta$ is annular we have $n(\Psi_{ij}')\leq n(\Delta)-d<n(\Delta)$ and $\Psi_{ij}'$ is obtained from $\Delta$ by removing the subdiagram $\bar{\Delta}_{ij}$, contradicting the $n$-minimality of $\Delta$.  Thus, we may assume henceforth that $\Delta$ is weakly minimal.

If $\Psi_{ij}'$ contains a disk, then as above $n(\Psi_{ij}')\leq n(\Delta)-d$, and so we may apply the inductive hypothesis to obtain an upper bound on its $G$-area; otherwise, \Cref{circular diskless} provides an analogous upper bound with even lower parameter constants.  In either case, for $n=n(\Delta)$ we have
$$\text{Area}_G(\Psi_{ij}')\leq N_4\phi(N_4(n-d))+N_3\mu(\Psi_{ij}')g(N_3(n-d))$$
Noting that $d\leq n$, from \Cref{phi properties}(2) we then have $\phi(N_4(n-d))\leq \phi(N_4n)-N_4^2ndg(N_4n)$.  So, since \Cref{scope-mixture} implies $\mu(\Psi_{ij}')\leq\mu(\Delta)+2J|\partial\Delta||\textbf{p}_{ij}|$, we have
$$\text{Area}_G(\Psi_{ij}')\leq N_4\phi(N_4n)+N_3\mu(\Delta)g(N_3n)-N_4^3ndg(N_4n)+2N_3J|\partial\Delta||\textbf{p}_{ij}|g(N_3n)$$
Noting that $d>0$ implies $y\geq|\bar{\textbf{p}}_{ij}|$, we have $|\partial\Pi|\leq L|\bar{\textbf{p}}_{ij}|\leq Ly$.  So, by the parameter choices $C_2>>C_1>>L$ and the assignment of weights we have $\text{wt}(\Pi)\leq C_1\phi(C_1Ly)\leq C_1\phi(C_2y)$.

Moreover, since $j-i\geq L/2$ implies $i\leq r$ and $j\geq r+1$, \Cref{scope-path-inequalities} implies $|\bar{\textbf{p}}_{ij}|\leq|\textbf{p}_{ij}|+|\partial\Pi|$.  So, $|\partial\Psi_{ij}|=|\textbf{p}_{ij}|+|\bar{\textbf{p}}_{ij}|\leq2|\textbf{p}_{ij}|+|\partial\Pi|\leq(L+2)y$, and hence \Cref{circular diskless} implies
$$\text{Area}_G(\Psi_{ij})\leq N_2\phi(N_2(L+2)y)+N_1\mu(\Psi_{ij})g(N_1(L+2)y)\leq N_2\phi(N_3y)+N_1\mu(\Psi_{ij})g(N_2y)$$
Hence, by \Cref{G-area subdiagrams} it suffices to show that:
\begin{equation} \label{eqn-p_ij-1}
N_4^3ndg(N_4n)\geq 2N_3J|\partial\Delta||\textbf{p}_{ij}|g(N_3n)+C_1\phi(C_2y)+N_2\phi(N_3y)+N_1\mu(\Psi_{ij})g(N_2y)
\end{equation}
Note that $n=|\partial\Delta|+\tau_\lambda(\Delta)\geq|\textbf{p}_{ij}|+\tau_\lambda(\Delta)$, so that $n\geq\max(|\partial\Delta|,y)$.  So, since $N_4d\geq N_4\frac{\eps y}{2}\geq y$, we have 
$$N_4ndg(N_4n)\geq nyg(N_4n)\geq |\partial\Delta||\textbf{p}_{ij}|g(N_4n)$$
The parameter assignments $N_4>>N_3>>J$ then allow us to assume $$N_4^3ndg(N_4n)\geq 4N_3J|\partial\Delta||\textbf{p}_{ij}|g(N_3n)$$
Hence, in place of (\ref{eqn-p_ij-1}) it suffices to show
\begin{equation}\label{eqn-p_ij-2}
N_4^3ndg(N_4n)\geq2C_1\phi(C_2y)+2N_2\phi(N_3y)+2N_1\mu(\Psi_{ij})g(N_2y)
\end{equation}
Further, the parameter choices $N_3>>N_2>>C_2>>C_1$ imply $$2C_1\phi(C_2y)+2N_2\phi(N_3y)\leq N_3\phi(N_3y)=N_3^3y^2g(N_3y)$$
while the parameter choice $N_4>>N_3$ implies $N_4^3ndg(N_4n)\geq2N_3y^2g(N_3y)$.  Hence, in place of (\ref{eqn-p_ij-2}) it suffices to show
\begin{equation}\label{eqn-p_ij-3}
N_4^3ndg(N_4n)\geq4N_1\mu(\Psi_{ij})g(N_2y)
\end{equation}
Now, \Cref{mixtures} implies $\mu(\Psi_{ij})\leq J|\partial\Psi_{ij}|_\theta^2\leq J|\partial\Psi_{ij}|^2\leq J(L+2)^2y^2$.  So, since $N_4nd\geq y^2$, (\ref{eqn-p_ij-3}) follows if $N_4^2g(N_4n)\geq4N_1J(L+2)^2g(N_2y)$.  But this is given by the parameter choices $N_4>>N_2>>N_1>>J>>L$.

\end{proof}

For $1\leq i\leq\ell-1$, let $\textbf{q}_{i,i+1}$ be a shortest path (with respect to length) in $\Psi_{i,i+1}$ homotopic to $\textbf{p}_{i,i+1}$ and having the same first and last edges.  

\begin{lemma} \label{scope annular geodesic}

If $\Delta$ is annular, then we may take $\textbf{q}_{i,i+1}=\textbf{p}_{i,i+1}$.

\end{lemma}

\begin{proof}

Suppose $\textbf{q}_{i,i+1}$ is chosen to have smaller length than $\textbf{p}_{i,i+1}$.  Cutting along $\textbf{q}_{i,i+1}$ then produces a minimal annular diagram $\Delta'$ whose boundary labels represent the same elements of $G(\textbf{M}_\textbf{S})$ as those of $\Delta$, but with $|\partial\Delta'|<|\partial\Delta|$.  

Further, as in the proof of \Cref{p_ij upper bound} we see that passing from $\Delta$ to $\Delta'$ edits the maximal $\theta$-bands by simply removing the `tails' of some that cross one of $\pazocal{Q}_1$ or $\pazocal{Q}_\ell$ and proceeds to $\textbf{p}$.  As such, $\rho_\lambda(\Delta')\leq\rho_\lambda(\Delta)$.

But then $n(\Delta')<n(\Delta)$, contradicting the $n$-minimality of $\Delta$.

\end{proof}

We extend the above terminology to define $\textbf{q}_{ij}$ to be the path obtained from concatenating $\textbf{q}_{i,i+1},\dots,\textbf{q}_{j-1,j}$ along their shared $q$-edges.  We then define $\Psi_{ij}^0$ to be the subdiagram of $\Psi_{ij}$ obtained by replacing $\textbf{p}_{ij}$ by $\textbf{q}_{ij}$ and removing any cells between these paths.  

Note that in the case where $\Delta$ is annular, \Cref{scope annular geodesic} immediately implies $\textbf{q}_{ij}=\textbf{p}_{ij}$, so that $\Psi_{ij}^0=\Psi_{ij}$.

As in previous settings, we drop indices in the case where $i=1$ and $j=\ell$, yielding the path $\textbf{q}$ and the subdiagram $\Psi^0$ of $\Psi$.

The next statement is then an analogue of \Cref{scope-path-inequalities} for the paths $\textbf{q}_{ij}$, providing lower bounds on theiir lengths as in (1).

\begin{lemma}[Lemma 7.27 of \cite{OS20}] \label{q-path-inequalities} 

If $i\leq r$ and $j\geq r+1$, then $|\textbf{q}_{ij}|\geq h_i+h_j+N(j-i)+1$.

\end{lemma}

\begin{proof}

(1) If $\Delta$ is annular, then this inequality is simply given by Lemmas \ref{scope annular geodesic} and \ref{scope-path-inequalities}(1).

(2) If $\Delta$ is circular, then the statement follows in just the same way as the proof of the analogous statement in \cite{OS20}.

\end{proof}

\begin{lemma}[Label 7.28 of \cite{OS20}] \label{q-combs-and-thetas} \

\begin{enumerate}[label=({\alph*})]

\item Every maximal $q$-band of $\Psi^0$ corresponds to a spoke of $\Pi$.

\item No two $\theta$-edges of $\textbf{q}_{i,i+1}$ are part of the same $\theta$-band of $\Psi_{i,i+1}$.

\end{enumerate}

\end{lemma}

\begin{proof}

(1) As noted in \Cref{rmk-annular-combs}, if $\Delta$ is annular then no subcomb of $\Delta$ can be contained in the scope $\Psi=\Psi^0$.  Both statements then follow quickly.

(2) Similar to the proof of the previous statement, the case where $\Delta$ is circular follows in just the same way as the proof of the analogous statement in \cite{OS20}.

\end{proof}


\subsection{Trapezia and combs in scopes} \

For $1\leq i\leq r-1$, every maximal $\theta$-band of $\Psi_{i,i+1}$ which crosses $\pazocal{Q}_{i+1}$ also crosses $\pazocal{Q}_i$.  By \Cref{q-combs-and-thetas}(b), these $\theta$-bands comprise a trapezium $\Gamma_i$ contained in $\Psi_{i,i+1}^0$ with history $H_{i+1}$.  By construction, the complement of $\Gamma_i$ in $\Psi_{i,i+1}$ {\frenchspacing (resp. $\Psi_{i,i+1}^0$)} is a comb $E_i$ {\frenchspacing (resp. $E_i^0$)} of height $h_i-h_{i+1}$ whose handle is the subband of $\pazocal{Q}_i$ which is obtained by removing the first $h_{i+1}$ cells.

For $r+1\leq i\leq \ell-1$, there similarly exists a trapezium $\Gamma_i$ contained in $\Psi_{i,i+1}^0$ with history $H_i$, so that we can find combs $E_i$ and $E_i^0$ of height $h_{i+1}-h_i$.

In any case, define the bottom and top of the trapezium $\Gamma_i$ to be $\textbf{y}_i$ and $\textbf{z}_i$, respectively.  Then $\textbf{y}_i^{-1}$ is a subpath of $\partial\Pi$ and has label $W(j_i)\{t(j_i+1)\}$ for some $j_i\in\{1,\dots,L\}$ dependent on $i$ (with $L+1$ taken to be $1$).  Note that then the boundary of the comb $E_i$ {\frenchspacing (resp. $E_i^0$)} is comprised of the bottom of the handle, the path $\textbf{z}_i$, and the path $\textbf{p}_{i,i+1}$ {\frenchspacing (resp. $\textbf{q}_{i,i+1}$)}.

The next statement follows in just the same way as its analogue in \cite{OS20}.

\begin{lemma}[Compare with Lemma 7.29 of \cite{OS20}] \label{nabla combs}

For $2\leq i\leq r-1$, any subdiagram $\nabla$ of $\Psi_{i,i+1}^0$ has a copy in $\Gamma_{i-1}$.  In particular, there exists a copy $\Gamma_i'$ of $\Gamma_i$ contained in $\Gamma_{i-1}$ which is a trapezium with bottom $\textbf{y}_{i-1}$ and top $\textbf{z}_i'$ such that $\lab(\textbf{z}_i')$ is a coordinate shift of $\lab(\textbf{z}_i)$.

\end{lemma}

\begin{figure}
\centering
\includegraphics[scale=1.25]{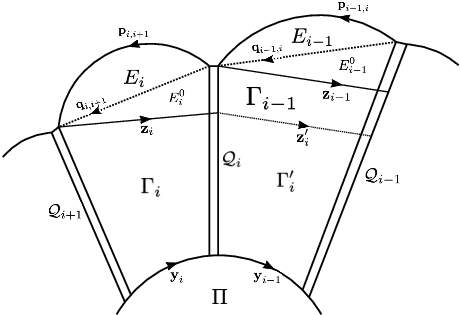}
\caption{Trapezia and combs in $\Psi_{i-1,i}$ and $\Psi_{i,i+1}$.}
\end{figure}

The next statement then follows immediately from \Cref{simplify rules}; the reason for the disparity in the estimate is the choice for convenience to allow a rule to insert two letters in a sector.

\begin{lemma}[Compare with Lemma 7.30 of \cite{OS20}] \label{a-bands in scopes}

At most $2N$ $a$-bands starting on the path $\textbf{y}_i$ (or $\textbf{z}_i$) can end on $(\theta,q)$-cells of the same $\theta$-band.

\end{lemma}

Based on the parameter assignments $L>>L_0>>k$, we may assume by \Cref{scope-theta-bands} that $L_0\leq r-1$ and $\ell-L_0\geq r+2$.  With this, we assume without loss of generality that $h\equiv h_{L_0+1}\geq h_{\ell-L_0}$ (otherwise, we pass to the `mirror' diagram so that the assumption holds).

\begin{lemma}[Compare with Lemma 7.31 of \cite{OS20}] \label{number of trapezia} 

Let $I$ be the subset of the set of indices $i\in[L_0+1,r-1]\cup[r+1,\ell-L_0-1]$ such that $|\textbf{z}_i|_a\geq|V|_a/c_2N$. If $h\leq L_0^2|V|_a$, then $\# I\leq L/5$.

\end{lemma}

\begin{proof}

The proof proceeds in exactly the same manner as for its analogue in \cite{OS20}, with only minor alterations amounting to different uses of parameters.  In summary, the assumptions lead to a lower bound on $|\textbf{p}_{L_0+1,\ell-L_0-1}|$ and an upper bound on $|\bar{\textbf{p}}_{L_0+1,\ell-L_0-1}|$ which imply $\frac{|\textbf{p}_{L_0+1,\ell-L_0-1}|}{|\bar{\textbf{p}}_{L_0+1,\ell-L_0-1}|}>1+\eps$, contradicting \Cref{p_ij upper bound} since, by \Cref{scope-theta-bands} $L_0\leq r-1$ and $\ell-L_0-1\geq r+1$.

\end{proof}

\begin{lemma}[Lemma 7.32 of \cite{OS20}] \label{first rules different}

If $h\leq L_0^2|V|_a$, then the first letter of $H_1$ differs from the first letter of $H_\ell$.

\end{lemma}

\begin{proof}

Let $\pazocal{T}$ be the maximal $\theta$-band of $\Delta$ corresponding to the first letter of $H_1$ and $\pazocal{S}$ be that corresponding to the first letter of $H_\ell$.

For any $1\leq i\leq r$, \Cref{scope-theta-bands} implies that either $H_i$ is empty or $\pazocal{T}$ crosses $\pazocal{Q}_i$.  Similarly, for $r+1\leq i\leq\ell$, either $H_i$ is empty or $\pazocal{S}$ crosses $\pazocal{Q}_i$.

Note that for $i\in[L_0+1,r-1]$, $H_{i+1}$ being empty implies $i\in I$; further, for $i\in[r+1,\ell-L_0-1]$, $H_i$ being empty implies $i\in I$.  So, by \Cref{number of trapezia} the sum of the number of $t$-spokes of $\Pi$ that $\pazocal{T}$ crosses and the number that $\pazocal{S}$ crosses is at least $\ell-2L_0-1-L/5=4L/5-2L_0-k-1$.  

But the parameter choices $L>>L_0>>k$ imply this is at least $L/2$, so that the statement follows from \Cref{two theta-bands about disk}.

\end{proof}

\begin{lemma}[Lemma 7.33 of \cite{OS20}] \label{V lower bound}

If $h\leq L_0^2|V|_a$, then $\displaystyle|V|_a>\frac{LN}{4L_0}$.

\end{lemma}

\begin{proof}

As in the proof of \Cref{number of trapezia}, the statement proceeds exactly in the same way as its analogue in \cite{OS20}, with an argument that $\frac{|\textbf{p}_{L_0+1,\ell}|}{|\bar{\textbf{p}}_{L_0+1,\ell}|}>1+\eps$ to contradict \Cref{p_ij upper bound}.

\end{proof}

\begin{lemma}[Lemma 7.34 of \cite{OS20}] \label{h lower bound}

We have $h>L_0^2|V|_a$.

\end{lemma}

\begin{proof}

Assume to the contrary that $h\leq L_0^2|V|_a$.

By \Cref{number of trapezia}, we have at least $(\ell-2L_0-2)-L/5$ indices $i\in[L_0+1,r-1]\cup[r+1,\ell-L_0-1]$ such that $|\textbf{z}_i|_a<|V|_a/c_2N$.  By the parameter choices $L>>L_0>>k$, the number of indices is greater than $2L/3$.  As such, similar parameter estimates yield $\a\in\{L_0+1,\dots,r-1\}$ and $\b\in\{r+1,\dots,\ell-L_0-1\}$ such that $|\textbf{z}_\a|_a,|\textbf{z}_\b|_a\leq|V|_a/c_2N$ and $\b-\a\geq3L/5$.

Let $\pazocal{C}_\a$ and $\pazocal{C}_\b$ be the reduced computations corresponding to the trapezia $\Gamma_\a$ and $\Gamma_\b$, respectively, which are given by \Cref{trapezia are computations}.  As $\lab(\textbf{y}_\a^{-1})\equiv W(i_\a)t(i_\a)$ and $\lab(\textbf{y}_\b^{-1})\equiv W(i_\b)t(i_\b)$ for some $i_\a,i_\b\in\{1,\dots,L\}$, the initial admissible subwords of $\pazocal{C}_\a$ and $\pazocal{C}_\b$ are coordinate shifts of one another (up to adding an extra $t$-letter).  Hence, \Cref{first rules different} implies we may take the coordinate shift of the inverse computation of $\pazocal{C}_\a$ to obtain a reduced computation $\pazocal{C}:U_0\to\dots\to U_t$ with history $H_{\a+1}^{-1}H_\b$ such that $U_0$ is a coordinate shift of $\lab(\textbf{z}_\a)$ and $U_t\equiv\lab(\textbf{z}_\b)$.

Note that for $j=0,t$, we have $|V|_a-|U_j|_a>(1-1/c_2N)|V|_a\geq|V|_a/2$.  So since the application of any rule changes the $a$-length of an admissible word with the relevant base by at most $2N$, we have $h_{\a+1},h_\b\geq|V|_a/4N$.  In particular, \Cref{V lower bound} and the parameter assignments $L>>L_0$ imply $t\geq|V|_a/2N>L/8L_0>2L_0$, and so $t>L_0+|V|_a/4N$.

Meanwhile, $\max(|U_0|_a,|U_t|_a)\leq|V|_a/c_2N$, and so 
$$t>L_0+\frac{c_2}{4}\max(|U_0|_a,|U_t|_a)>c_1\max(\|U_0\|,\|U_t\|)$$ 
by the parameter choice $L_0>>c_2>>c_1>>N$.  Hence, $\pazocal{C}$ satisfies the hypotheses of \Cref{long history controlled}.

In particular, assuming without loss of generality that $h_{\a+1}\geq h_\b$, then any subband of $\pazocal{Q}_{\a+1}$ of length at least $\frac{1}{2}\lambda h_{\a+1}\geq\frac{1}{4}\lambda t$ satisfies the factorization condition of \Cref{long history controlled}.  Moreover, since $\Gamma_\a$ has history $H_{\a+1}$, the parallel construction of $\textbf{M}_\textbf{S}$ implies $\lab(\partial\Pi)$ is $H_{\a+1}$-admissible.  So, by the structure of the scope $\Psi$, it follows that it is both a $\lambda$-spear and a $\lambda$-shaft.


The rest of the proof proceeds just as in the analogue in \cite{OS20}: That $\pazocal{Q}_{\a+1}$ is a $\lambda$-spear implies $|\textbf{p}_{\a+1,\b}|+\tau_\lambda(\bar{\Delta}_{\a+1,\b})\geq1.1|\bar{\textbf{p}}_{\a+1,\b}|$, so that taking $N_4\geq100$ yields a contradiction to \Cref{p_ij upper bound}.

\end{proof}

\begin{lemma}[Lemma 7.35 of \cite{OS20}] \label{h_i lower bound}

For $i\in\{1,\dots,L_0\}$, $h_i>\delta^{-1}$.

\end{lemma}

\begin{proof}

The proof proceeds in much the same way as the proof of its analogue in \cite{OS20}: Assuming toward contradiction, \Cref{h lower bound} implies $\delta^{-1}\geq h_i\geq|V|_a$, so that \Cref{scope-path-inequalities} quickly yields $\frac{|\textbf{p}_{i,\ell-L_0}|}{|\bar{\textbf{p}}_{i,\ell-L_0}|}>1+\delta$.  But the parameter assignment $N_4>>\delta^{-1}$ then implies this is a contradiction to \Cref{p_ij upper bound}.

\end{proof}

\begin{lemma}[Compare with Lemma 7.36 of \cite{OS20}] \label{no lambda-spear}

For $i\in\{1,\dots,L_0\}$, $\pazocal{Q}_i$ does not contain a $\lambda$-spear or a $\lambda$-shaft of length $\delta h$.

\end{lemma}

\begin{proof}

Note that in passing from $\Delta$ to $\Psi_{L_0+1,\ell-L_0}'$, the disk $\Pi$ is removed, and so $\pazocal{Q}_i$ is no longer a $t$-spoke of any disk.  Hence, arguing as in the proof of \Cref{p_ij upper bound} that any $\lambda$-shaft or $\lambda$-spear of $\Psi_{L_0+1,\ell-L_0}'$ corresponds to one in $\Delta$, in either case we have $\tau_\lambda(\Delta)-\tau_\lambda(\Psi_{L_0+1,\ell-L_0}')\geq\delta h$.  Repeating the arguments of \cite{OS20}, \Cref{scope-path-inequalities} implies $|\textbf{p}_{L_0+1,\ell-L_0}|+\delta h\geq(1+\eps)|\bar{\textbf{p}}_{L_0+1,\ell-L_0}|$, thus again yielding a contradiction to \Cref{p_ij upper bound}.

\end{proof}

\begin{lemma}[Compare with Lemma 7.37 of \cite{OS20}] \label{z_i lower bound}

For $i\in\{1,\dots,L_0-1\}$, $|\textbf{z}_i|_a>h_{i+1}/c_2$.

\end{lemma}

\begin{proof}

\Cref{h_i lower bound} and the parameter choices $\delta^{-1}>>c_2>>N$ imply $h_{i+1}/c_2>\delta^{-1}/c_2>N+1$.  So, assuming toward contradiction that $|\textbf{z}_i|_a\leq h_{i+1}/c_2$, we have $$\|\textbf{z}_i\|=|\textbf{z}_i|_a+N+1\leq2h_{i+1}/c_2<h_{i+1}/c_1$$
by the parameter choice $c_2>>c_1$.

Similarly, \Cref{h lower bound} implies $\|\textbf{y}_i\|=|V|_a+N+1<h/L_0^2+h_{i+1}/c_2\leq h_{i+1}/c_1$ as $L_0>>c_2>>c_1$.  Hence, the reduced computation corresponding to $\Gamma_i$ through \Cref{trapezia are computations} satisfies the hypotheses of \Cref{long history controlled}.

But then as in the proof of \Cref{h lower bound}, $\pazocal{Q}_{i+1}$ is both a $\lambda$-spear and a $\lambda$-shaft of $\Pi$ with length $h_{i+1}\geq h$, contradicting \Cref{no lambda-spear}.

\end{proof}

\begin{lemma}[Compare with Lemma 7.38 of \cite{OS20}] \label{h_i shrinkage}

For $i\in\{1,\dots,L_0-1\}$, $h_{i+1}<(1-\frac{1}{20c_2N})h_i$.

\end{lemma}

\begin{proof}

Assuming the statement is false, the height of the comb $E_i$ is at most $\frac{1}{20c_2N}h_i$.  \Cref{a-bands in scopes} then implies at most $\frac{1}{10c_2}h_i$ maximal $a$-bands in $E_i$ have ends on both $\textbf{z}_i$ and on a $(\theta,q)$-cell.  So, by \Cref{z_i lower bound} and parameter assignments, the number of maximal $a$-bands in $E_i$ which have an end on both $\textbf{z}_i$ and on $\partial\Delta$ is at least
$$|\textbf{z}_i|_a-h_i/10c_2>h_{i+1}/c_2-h_i/10c_2\geq\left(1-\frac{1}{20c_2N}\right)h_i/c_2-h_i/10c_2\geq2h_i/5$$
The rest of the proof proceeds in much the same way as in the analogue in \cite{OS20}: \Cref{scope-theta-bands} implies $|\textbf{p}_{i,i+1}|_\theta\leq\frac{1}{10c_2}h_i$, so that Lemmas \ref{lengths} and \ref{scope-path-inequalities} imply inequalities that yield $\frac{|\textbf{p}_{i,\ell-L_0}|}{|\bar{\textbf{p}}_{i,\ell-L_0}|}\geq1+\delta/20c_2$; the parameter choices $N_4>>\delta^{-1}>>c_2$ then provide a contradiction to \Cref{p_ij upper bound}.

\end{proof}

\begin{lemma}[Compare with Lemma 7.39 of \cite{OS20}] \label{z_i bounded by h_i}

For $i\in\{2,\dots,L_0-1\}$, $|\textbf{z}_i|_a<4Nh_i$.

\end{lemma}

\begin{proof}

By \Cref{a-bands in scopes} and \Cref{scope-theta-bands}, at most $2Nh_i$ maximal $a$-bands of $E_i$ have one end on $\textbf{z}_i$ and another on a $(\theta,q)$-cell.  So, assuming $|\textbf{z}_i|_a\geq4Nh_i$, at least $2Nh_i$ maximal $a$-bands of $E_i$ have one end on $\textbf{z}_i$ and another on $\partial\Delta$.

The rest of the proof proceeds in much the same way as for the analogue of \cite{OS20} (and much the same way as for \Cref{h_i shrinkage}): $|\textbf{p}_{i,i+1}|_\theta\leq h_i$, so that Lemmas \ref{lengths} and \ref{scope-path-inequalities} yield a contradiction to \Cref{p_ij upper bound}.

\end{proof}

\begin{lemma}[Compare with Lemma 7.40 of \cite{OS20}] \label{one-step scope}

Fix $i\in\{2,\dots,L_0-2\}$ and let $H_i=H_{i+1}H'=H_{i+2}H''H'$. Further, let $\pazocal{C}$ be the reduced computation corresponding to the trapezium $\Gamma_{i-1}$ by \Cref{trapezia are computations} and $\pazocal{D}$ be the subcomputation of $\pazocal{C}$ with history $H''H'$.  If $\pazocal{D}$ has step history of length $1$, then there is no two-letter subword $Q'Q$ of the base of $\Gamma_{i-1}$ such that every rule of $\pazocal{D}$ inserts one letter to the left of $Q$.

\end{lemma}

\begin{proof}

Assume to the contrary that the two-letter subword $Q'Q$ of the base of $\Gamma_{i-1}$ exists as in the statement.

As in analogous settings in previous literature, let $\pazocal{Q}'$ and $\pazocal{Q}$ be the maximal $q$-bands of $E_i^0$ which are subbands of the $q$-spokes of $\Pi$ corresponding to the appropriate coordinate shifts of $Q'$ and $Q$, respectively, and let $\textbf{x}$ be the subpath of $\textbf{z}_i$ between $\pazocal{Q}'$ and $\pazocal{Q}$.

Let $\Gamma_{i+1}'$ be the copy of $\Gamma_{i+1}$ in $\Gamma_i$ given by \Cref{nabla combs}.  Note then that the complement $\Gamma_i\setminus\Gamma_{i+1}'$ is a trapezium with history $H''$ and top $\textbf{z}_i$, and so by hypothesis we have $|\textbf{x}|_a\geq\|H''\|=h_{i+1}-h_{i+2}$.  Hence, \Cref{h_i shrinkage} implies $|\textbf{x}|_a>h_{i+1}/20c_2N$.

Let $\nabla$ be the comb contained in $E_i^0$ bounded by $\pazocal{Q}'$, $\pazocal{Q}$, $\textbf{x}$, and $\textbf{q}_{i,i+1}$.  \Cref{nabla combs} tells us a copy of $\nabla$ exists in $\Gamma_{i-1}$ as part of the subtrapezium with history $H'$, so that our hypothesis implies every corresponding cell of $\pazocal{Q}$ is the end of a maximal $a$-band of $\nabla$ which also ends on either $\pazocal{Q}'$ or $\textbf{q}_{i,i+1}$.  But \Cref{scope-theta-bands} implies the length of $\pazocal{Q}$ is at least that of $\pazocal{Q}'$ and the number of $\theta$-edges of $\textbf{q}_{i,i+1}$ is the difference between these two values.  \Cref{lengths} then implies there are at least $|\textbf{x}|_a$ $a$-edges of the portion of $\textbf{q}_{i,i+1}$ shared with $\partial\nabla$ which contribute $\delta$ to $|\textbf{q}_{i,i+1}|$.  Hence, we have
$$|\textbf{p}_{i,\ell-L_0}|\geq|\textbf{q}_{i,\ell-L_0}|\geq h_i+h_{\ell-L_0}+N(\ell-L_0-i)+\frac{\delta h_{i+1}}{20c_2N}$$
On the other hand, note that $L-\ell=k$, so that \Cref{scope-path-inequalities} implies $$|\bar{\textbf{p}}_{i,\ell-L_0}|\leq h_i+h_{\ell-L_0}+N(k+L_0+i)+(k+L_0+i)\delta|V|_a$$
Since $i\leq L_0$, the parameter choices $L>>L_0>>k$ then imply
$$\ell-L_0-i\geq\ell-2L_0=L-2L_0-k\geq k+2L_0\geq k+L_0+i$$
so that \Cref{h lower bound} and the parameter choices $L_0>>c_2>>N$ imply
$$|\textbf{p}_{i,\ell-L_0}|-|\bar{\textbf{p}}_{i,\ell-L_0}|\geq\delta\left(\frac{h_{i+1}}{20c_2N}-3L_0|V|_a\right)>\delta\left(\frac{h_{i+1}}{20c_2N}-\frac{3h}{L_0}\right)\geq\delta h_{i+1}/40c_2N$$
Cutting along $\bar{\textbf{p}}_{i,\ell-L_0}$ to remove $\bar{\Delta}_{i,\ell-L_0}$, this implies $|\partial\Psi_{i,\ell-L_0}'|<|\partial\Delta|-\delta h_{i+1}/40c_2N$.

\Cref{p_ij upper bound} then implies $n(\Psi_{i,\ell-L_0}')<n(\Delta)-\delta h_{i+1}/40c_2N$.  In particular, if $\Delta$ is annular, then $\Psi_{i,\ell-L_0}'$ contradicts the $n$-minimality of $\Delta$, and thus we may assume henceforth that $\Delta$ is a (weakly minimal) circular diagram.

So, letting $n=n(\Delta)$, either $\Psi_{i,\ell-L_0}'$ is diskless or is weakly minimal diagram that satisfies $n(\Psi_{i,\ell-L_0}')<n-\delta h_{i+1}/40c_2N$, and so either the minimality of $\Delta$ as a counterexample or \Cref{circular diskless} implies
\begin{align*}
\text{Area}_G(\Psi_{i,\ell-L_0}')&\leq N_4\phi(N_4(n-\delta h_{i+1}/40c_2N))+N_3\mu(\Psi_{i,\ell-L_0}')g(N_3(n-\delta h_{i+1}/40c_2N)) \\
&\leq N_4\phi(N_4n)-\frac{N_4^3\delta}{40c_2N}h_{i+1}ng(N_4n)+N_3\mu(\Psi_{i,\ell-L_0}')g(N_3n)
\end{align*}
Many of the estimates necessary for the rest of the proof follow in just the same way as in \cite{OS20}.  For example, we have
$$\mu(\Delta)-\mu(\Psi_{i,\ell-L_0}')\geq-2Jn(h_{i+1}+h_{\ell-L_0})+\frac{1}{20c_2N}h_{i-1}(h_i+h_{\ell-L_0})$$
where the only deviations come from the use of $\ell$ in place of $L-3$ and the difference between the estimate of \Cref{h_i shrinkage} and that of its analogue.  Further, note that $h_{i+1}\geq h\geq h_{\ell-L_0}$ by assumption, and so 
$$\mu(\Delta)-\mu(\Psi_{i,\ell-L_0}')\geq-4Jh_{i+1}n+\frac{1}{20c_2N}h_{i-1}h_i$$

Further, similar parameter estimates as in \cite{OS20} yield $|\bar{\textbf{p}}_{i,\ell-L_0}|\leq2.1h_i$, so that taking $N_4$ large enough means \Cref{p_ij upper bound} implies $|\textbf{p}_{i,\ell-L_0}|\leq2.2h_i$.

Conversely, $|\partial\Pi|=LN+L\delta|V|_a$, so that \Cref{h lower bound}, \Cref{h_i lower bound}, and the parameter choices $\delta^{-1}>>L>>N$ imply $|\partial\Pi|\leq\delta^{-1}/4+h_i/L_0^2\leq h_i/2$.

As a result, $|\partial\Psi_{i,\ell-L_0}|\leq|\textbf{p}_{i,\ell-L_0}|+|\bar{\textbf{p}}_{i,\ell-L_0}|+|\partial\Pi|\leq5h_i$, so that \Cref{circular diskless} implies
$$\text{Area}_G(\Psi_{i,\ell-L_0})\leq N_2\phi(5N_2h_i)+N_1\mu(\Psi_{i,\ell-L_0})g(5N_1h_i)$$
while the assignment of weights implies
$$\text{wt}(\Pi)\leq C_1\phi(C_1h_i/2)$$
Note that \Cref{mixtures} implies $\mu(\Psi_{i,\ell-L_0})\leq J|\partial\Psi_{i,\ell-L_0}|^2\leq25Jh_i^2$, so that the parameter choice $N_1>>J$ implies $N_1\mu(\Psi_{i,\ell-L_0})g(5N_1h_i)\leq N_1\phi(5N_1h_i)$.  In particular, the parameter choices $N_2>>N_1>>C_1$ yield
$$\text{Area}_G(\Psi_{i,\ell-L_0})+\text{wt}(\Pi)\leq2N_2\phi(5N_2h_i)$$
Hence, as \Cref{G-area subdiagrams} implies $\text{Area}_G(\Delta)\leq\text{Area}_G(\Psi_{i,\ell-L_0}')+\text{Area}_G(\Psi_{i,\ell-L_0})+\text{wt}(\Pi)$, it suffices to show that
\begin{equation}\label{eqn-one-step-scope-1} 
\frac{N_4^3\delta}{40c_2N}h_{i+1}ng(N_4n)+\frac{N_3}{20c_2N}h_{i-1}h_ig(N_3n)\geq4N_3Jh_{i+1}ng(N_3n)+2N_2\phi(5N_2h_i)
\end{equation}
First, the parameter choices $N_4>>N_3>>\delta^{-1}>>J>>c_2>>N$ allow us to assume $$\frac{N_4^3\delta}{40c_2N}h_{i+1}ng(N_4n)\geq4N_3Jh_{i+1}ng(N_3n)$$
Thus, in place of \Cref{eqn-one-step-scope-1} it suffices to show that
\begin{equation}\label{eqn-one-step-scope-2}
\frac{N_3}{20c_2N}h_{i-1}h_ig(N_3n)\geq2N_2\phi(5N_2h_i)
\end{equation}
But $\phi(5N_2h_i)=25N_2^2h_i^2g(5N_2h_i)$, so that this amounts to showing 
$$N_3h_{i-1}h_ig(N_3n)\geq1000N_2^3c_2Nh_i^2g(5N_2h_i)$$
which follows from the parameter choices $N_3>>N_2>>c_2>>N$ since $n\geq h_{i-1}\geq h_i$.

\end{proof}

We now reach the ultimate contradiction of the section, proving the statements in \Cref{sec-counterexamples}.

\begin{lemma}[Compare with Lemma 7.41 of \cite{OS20}] \label{scope-contradiction}

The counterexample diagrams do not exist.

\end{lemma}

\begin{proof}

As in similar settings in previous literature, we begin by fixing a positive integer $\eta$ such that $(1-\frac{1}{20c_2N})^\eta<\frac{1}{12c_2N}$.  While $\eta$ is not a listed parameter, we may take $L_0>>\eta$ since $\eta$ is dependent only on $c_2$ and $N$ while we take $L_0>>c_2>>N$.

For $1\leq i<j\leq L_0-1$ with $j-i-1\geq\eta$, \Cref{h_i shrinkage} implies $h_j\leq(1-\frac{1}{20c_2N})^\eta h_{i+1}<\frac{1}{12c_2N}h_{i+1}$.  \Cref{z_i lower bound} then implies $|\textbf{z}_i|_a>h_{i+1}/c_2>12Nh_j$, while \Cref{z_i bounded by h_i} implies $|\textbf{z}_j|_a<4Nh_j$.  Hence, $|\textbf{z}_i|_a>3|\textbf{z}_j|_a$.

Using $L_0>>\eta$, fix indices $2\leq j_1<j_2<\dots<j_m\leq L_0-1$ such that $m\geq c_0$ and $j_{i+1}-j_i-1\geq\eta$ for all $i$.  So, $|\textbf{z}_{j_i}|_a>3|\textbf{z}_{j_{i+1}}|_a$ and $h_{j_i+1}\geq12c_2Nh_{j_{i+1}}$.

Let $\pazocal{C}:W_0\to\dots\to W_t$ be the reduced computation corresponding to the trapezium $\Gamma_{j_2}$.  By \Cref{nabla combs}, $\Gamma_{j_2}$ contains a copy of $\Gamma_j$ for all $j\geq j_2$.  So, there exist words $V_i$ in $\pazocal{C}$ for $i=2,\dots,m$ which are coordinate shifts of the labels of $\textbf{z}_{j_i}$.  As above, $|V_i|_a>3|V_{i+1}|_a$ for each $i$.

If for some $i$ the subcomputation $V_{i+2}\to\dots\to V_i$ of $\pazocal{C}$ has step history of length 1, then \Cref{standard one-step} yields a two-letter subword $Q'Q$ of the base of $\pazocal{C}$ such that each transition of the subcomputation inserts a letter to the left of $Q$ and increases the sector's length.  But then $\eta\geq2$ implies the existence of a subcomputation which contradicts \Cref{one-step scope}.

Hence, for any $i$ the subcomputation $V_{i+20}\to\dots\to V_i$ of $\pazocal{C}$ has at least 10 distinct maximal one-step subcomputations.  Lemmas \ref{E primitive step history} and \ref{E run step history} then imply that the step history of such a subcomputation has a subword of the form $(12)(2)(23)(3)(34)(45)$ or $(54)(4)(43)(3)(32)(2)(21)$.

Taking $c_0$ sufficiently large, consider the subcomputation $\pazocal{D}:V_m\to\dots\to V_{m-41}$ of $\pazocal{C}$ and let $\pazocal{D}':W_r\to\dots\to W_s$ be the maximal subcomputation of $\pazocal{D}$ such that the first letter of the step history of $\pazocal{D}'$ is either $(12)$ or $(54)$ and the last letter is either $(45)$ or $(21)$.  Note then that $s\geq h_{j_{m-21}}$ and $r\leq h_{j_{m-20}}$, so that $s\geq12c_2Nr$.

Applying \Cref{(1) to (5) mostly (4)} to $\pazocal{D}'$ then implies the length of its subcomputations whose step histories are either $(34)(4)(45)$ or $(54)(4)(43)$ is at least $(1-\frac{5}{2k})(s-r)$.  So, considering the subcomputation $\pazocal{E}:W_0\to\dots\to W_s$ the parameter choices $c_2>>N>>k$ then allow us to assume the sum of the subcomputations of $\pazocal{E}$ with step history $(34)(4)(45)$ or $(54)(4)(43)$ is at least $(1-\frac{3}{k})s$.  

As in the proof of \Cref{long history controlled}, this implies that for any fixed subcomputation $\pazocal{F}:W_x\to\dots\to W_y$ of $\pazocal{E}$ with $y-x\geq\frac{1}{4}\lambda s$, any subword of the history of $\pazocal{F}$ of length at least $(1-\lambda)(y-x)$ has a controlled subword.

In particular, the subband of $\pazocal{Q}_{j_2}$ corresponding to the history of $\pazocal{E}$ satisfies condition (iv) of being a $\lambda$-spear.  As $\Gamma_{j_2-1}$ shows that condition (i) is satisfied and (ii) and (iii) are given by the properties of the scope $\Psi$, it then follows that the band is indeed a $\lambda$-spear.

Hence, $\pazocal{Q}_{j_2}$ contains a maximal $\lambda$-spear (and $\lambda$-shaft) of length at least $s\geq h_{j_{m-21}}\geq h$.  But this contradicts \Cref{no lambda-spear}, completing the proof.

\end{proof}

\begin{remark}

Note that these final arguments crucially use the mirror construction of the machine $\textbf{M}_\textbf{S}$: We assumed throughout the arguments that $h_{L_0+1}\geq h_{\ell-L_0}$, which we claimed did not lose generality since otherwise we could pass to the mirror diagram (see the discussion preceding \Cref{number of trapezia}).  However, it is crucial for the proof of \Cref{one-step scope} that one letter is being inserted to the right of the sector, and so taking a mirror diagram would require one letter to be inserted to the left of the sector.  The mirror construction of the machine achieves this, as we can then pass to the mirror copy of the sector in question.

\end{remark}

\bigskip


\section{Annular diagrams with disks}

While the arguments outlined in the previous section are sufficient for our study of circular diagrams, more work is needed for annular diagrams.  

To lay the framework for the arguments that follow, recall from the previous section (see \Cref{sec-counterexamples}) that a minimal annular diagram $\Delta$ is called $n$-minimal if the value $n(\Delta)=|\partial\Delta|+\rho_\lambda(\Delta)$ is minimal among all minimal annular diagrams realizing the same conjugacy relation in $G(\textbf{M}_\textbf{S})$.  Fixing an $n$-minimal diagram $\Delta$ and assuming $s_1(\Delta)\geq1$, our goal through the rest of this section is to find a path $\textbf{t}$ in $\Delta$ between the two boundary components with $|\textbf{t}|\leq N_4|\partial\Delta|+N_4$.

To present this efficiently, we fix a potential counterexample, that is, an (annular) $n$-minimal diagram $\Delta$ with $s_1(\Delta)\geq1$ such that any path $\textbf{t}$ between its two boundary components satisfies $$|\textbf{t}|>N_4|\partial\Delta|+N_4$$  
Through the rest of this section, we study the makeup of this hypothetical diagram $\Delta$, culminating with the conclusion that it cannot exist in the first place (see \Cref{annular disks}).

\medskip


\subsection{Basic properties} \

The next pair of statements, our first observations about the counterexample diagram $\Delta$, are immediate consequences of the definition of $n$-minimal diagrams.

\begin{lemma} \label{n-minimal trivial path}

There is no non-contractible loop $\textbf{s}$ in $\Delta$ with $\lab(\textbf{s})=_{G(\textbf{M}_\textbf{S})}1$.

\end{lemma}

\begin{proof}

Suppose $\textbf{s}$ is a counterexample to the statement.  Through $0$-refinement we may assume this loop is disjoint from the boundary components of $\Delta$.  Further, as combinatorial homotopies do not alter the label's value in $G(\textbf{M}_\textbf{S})$, we may assume $\textbf{s}$ is simple.

So, cutting along $\textbf{s}$ separates $\Delta$ into two annular diagrams, each of which has one boundary component labelled by $\lab(\textbf{s})$ and the other labelled by that of a component of $\partial\Delta$.  van Kampen's Lemma then implies the label of each component of $\partial\Delta$ is conjugate to $1$ in $G(\textbf{M}_\textbf{S})$, and so is itself trivial in $G(\textbf{M}_\textbf{S})$.

As a result, an annular diagram $\Lambda$ consisting entirely of $0$-cells and with trivial boundary labels realizes the same conjugacy relation as $\Delta$.  Since $n(\Lambda)=|\partial\Lambda|+\rho_\lambda(\Lambda)=0$, we must also have $n(\Delta)=0$.  But since we assume $s_1(\Delta)\geq1$, \Cref{annular-graph} implies $|\partial\Delta|\geq L-18>0$.

\end{proof}

\begin{lemma} \label{n-minimal boundary ends}

If the maximal $\theta$-band $\pazocal{T}$ of $\Delta$ crosses the same $q$-band at least twice, then it has ends on each boundary component of $\Delta$.

\end{lemma}

\begin{proof}

Let $\pazocal{Q}$ be a maximal $q$-band of $\Delta$ which $\pazocal{T}$ crosses at least twice.  Then, let $\pi_1$ and $\pi_2$ be two $(\theta,q)$-cells of $\pazocal{Q}$ which are crossed by $\pazocal{T}$ and such that the maximal subband $\pazocal{Q}'$ of $\pazocal{Q}$ between these two cells does not cross $\pazocal{T}$.  

Let $\pazocal{T}'$ be the minimal subband of $\pazocal{T}$ containing $\pi_1$ and $\pi_2$.  The combination of $\pazocal{T}'$ and $\pazocal{Q}'$ then form a sequence of cells $\pazocal{S}$ which can be viewed as an annular band whose `sides' are simple loops in $\Delta$.  If either of these loops is contractible, then there must exist a $(\theta,q)$-annulus formed by subbands of $\pazocal{T}$ and $\pazocal{Q}$, contradicting \Cref{G(S) annuli}.  

Hence, cutting along one `side' of $\pazocal{S}$ separates $\Delta$ into two annular subdiagrams, one of which, $\Delta_1$, has outer component arising from that of $\Delta$ and contains $\pazocal{T}'$.  Let $\pazocal{T}_1$ be the maximal $\theta$-band in $\Delta_1$ for which $\pazocal{T}'$ as a subband.  Then $\pazocal{T}_1$ can be identified with a subband of $\pazocal{T}$ in $\Delta$.  Note that $\pazocal{T}'$ has one end on the inner component, while $\pazocal{T}_1$ does not have another end on this component since it cannot cross itself and does not cross $\pazocal{Q}'$.  Hence, $\pazocal{T}_1$ has an end on the outer component of $\Delta_1$, and so $\pazocal{T}$ has an end on the outer component of $\Delta$.

Cutting along the other `side' of $\pazocal{S}$, a symmetric argument implies $\pazocal{T}$ also has an end on the inner component of $\Delta$. 

\end{proof}

Now, note that \Cref{scope-contradiction} implies $\Delta$ has no scope of width at least $L-k$, while \Cref{annular-graph} promises a disk $\Pi$ such that at least $L-18$ $t$-spokes of $\Pi$ end on $\partial\Delta$.  Letting $\textbf{C}$ be one of the two boundary components of $\Delta$, let $\pazocal{Q}_1,\dots,\pazocal{Q}_\ell$ be the $t$-spokes of $\Pi$ which end on $\textbf{C}$.  As there is just one hole in the annulus, these $t$-bands, a subpath of $\partial\Pi$, and a subpath of $\textbf{C}$ bound a circular subdiagram $\Psi$ of $\Delta$.  

Suppose $\ell\geq2$ and $\Psi$ contains a disk.  Then \Cref{circular-graph} yields a disk $\Pi'$ with $L-3$ consecutive $t$-spokes all ending on $\partial\Psi$ and no disks between them.  \Cref{t-spokes between disks} implies at most two of these $t$-spokes end on $\partial\Pi$, and so since $q$-bands cannot cross at least $L-5$ of them end on $\textbf{C}$.  But then these spokes bound a scope of $\Pi'$ on $\textbf{C}$ in $\Delta$, contradicting \Cref{scope-contradiction} if we take $k\geq6$.

Hence, $\Psi$ contains no disk, and so is itself a scope of $\Pi$ on $\textbf{C}$ in $\Delta$.  As its width is $\ell$, we must then have $\ell<L-k$.  Taking $k\geq20$, at least two $t$-spokes of $\Pi$ must then end on each component of $\partial\Delta$.  As above, these $t$-spokes bound scopes $\Psi_1,\Psi_2$ of $\Pi$ on the boundary components $\textbf{C}_1,\textbf{C}_2$ in $\Delta$, respectively.  Letting $\ell_1,\ell_2$ be the widths of these scopes, we have $\ell_1+\ell_2\geq L-18$ and $\ell_j\geq k-18\geq2$.

\begin{lemma} \label{n-minimal subcombs}

No subcomb of $\Delta$ is contained in $\Psi_j$.

\end{lemma}

\begin{proof}

The proof proceeds in exactly the same way as the annular case for \Cref{minimal-subcombs}: If such a subcomb existed, then cutting along the bottom of its handle produces an annular diagram $\Delta'$ representing the same conjugacy relation in $G(\textbf{M}_\textbf{S})$ with $s_1(\Delta')=s_1(\Delta)$, $|\partial\Delta'|<|\partial\Delta|$, and $\rho_\lambda(\Delta')\leq\rho_\lambda(\Delta)$.  But this contradicts the $n$-minimality of $\Delta$.

\end{proof}

For $j=1,2$, enumerate the consecutive $t$-spokes $\pazocal{Q}_{1,j},\dots,\pazocal{Q}_{\ell_j,j}$ of $\Pi$ bounding the scope $\Psi_j$ on the boundary component $\textbf{C}_j$ (see \Cref{fig-n-minimal}).

\begin{lemma} \label{n-minimal theta-bands}

Every maximal $\theta$-band in $\Psi_j$ crosses $\pazocal{Q}_{1,j}$ or $\pazocal{Q}_{\ell_j,j}$.

\end{lemma}

\begin{proof}

An identical proof to that of \Cref{counterexample-long-theta} implies the base of any rim $\theta$-band in $\Psi_j$ has length greater than $K$.  In particular, if a maximal $\theta$-band of $\Psi_j$ has two ends on $\textbf{C}_j$, then the length of its base must be greater than $K$.  

By \Cref{M(S) annuli}, any $\theta$-band and $q$-band in $\Psi_j$ can cross at most once, while \Cref{n-minimal subcombs} implies every maximal $q$-band of $\Psi_j$ has an end on $\partial\Pi$.  But this implies the length of the base of any $\theta$-band in $\Psi_j$ is at most $LN\leq K$ by the  parameter choices $K>>L>>N$.  

Hence, no maximal $\theta$-band can have two ends on $\textbf{C}_j$, and so must have at least one end on the side of the bounding $t$-spokes since $|\partial\Pi|_\theta=0$.

\end{proof}

\begin{figure}
\centering
\includegraphics[scale=2.5]{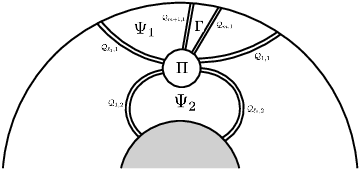} 
\caption{} \label{fig-n-minimal}
\end{figure}

Let $H_{i,j}$ be the history of the band $\pazocal{Q}_{i,j}$, $h_{i,j}=\|H_{i,j}\|$ be the length of the band, and $h=\max(h_{i,j})$.  Perhaps passing to a mirror copy of $\Delta$, we may assume without loss of generality that $h=h_{m,1}$ for some $m\leq\ell_1-1$, where $\textbf{C}_1$ is the outer boundary component.  Denote by $\Gamma$ the subdiagram of $\Psi_1$ bounded by $\pazocal{Q}_{m,1}$ and $\pazocal{Q}_{m+1,1}$.

Let $W$ be the accepted configuration of $\textbf{M}_\textbf{S}$ corresponding to $\lab(\partial\Pi)$.  As in the previous section, the components of $W$ are all coordinate shifts of one another, and all a copy of the same accepted configuration $V$ of $\textbf{E}_{\textbf{S},k}$.  Then, there exists a factorization $\partial\Gamma=\textbf{u}^{-1}\textbf{p}^{-1}\textbf{r}\textbf{q}$ where:

\begin{itemize}

\item $\textbf{u}$ is a side of $\pazocal{Q}_{m+1,1}$.

\item $\textbf{r}$ is a side of $\pazocal{Q}_{m,1}$.

\item $\textbf{p}$ is a subpath of $\partial\Pi$ with $\lab(\textbf{p})^{\pm1}\equiv W(i_m)t(i_m+1)$ for some $i_m$.

\item $\textbf{q}$ is a subpath of $\textbf{C}_1$.

\end{itemize}

\begin{lemma} \label{n-minimal first bound}

$2h+\delta L|V|_a>N_4|\partial\Delta|+3N_3$.

\end{lemma}

\begin{proof}

Let $\textbf{x}$ be the side of any $t$-spoke $\pazocal{Q}_{i,2}$ starting at $\Pi$.  Then there exists a subpath $\textbf{p}'$ of $\partial\Pi$ such that $\textbf{r}^{-1}\textbf{p}'\textbf{x}$ forms a path between $\textbf{C}_1$ and $\textbf{C}_2$. Hence, \Cref{lengths} implies $$2h+\delta L|V|_a+NL\geq 2h+|\partial\Pi|\geq|\textbf{r}|+|\textbf{p}'|+|\textbf{x}|\geq|\textbf{r}^{-1}\textbf{p}'\textbf{x}|>N_4|\partial\Delta|+N_4$$
The parameter choices $N_4>>N_3>>L>>N$ then implies the statement.

\end{proof}

\begin{lemma} \label{n-minimal first step} \

\begin{enumerate}[label=(\alph*)]

\item $|V|_a\leq c_0h$

\item $h>N_3|\partial\Delta|+N_3$

\end{enumerate}

\end{lemma}

\begin{proof}

(a) By \Cref{n-minimal subcombs}, $\Gamma$ has at most $N+1$ maximal $q$-bands, two of which are $t$-bands.  Further, \Cref{n-minimal theta-bands} implies every maximal $\theta$-band of $\Gamma$ crosses at least one of $\pazocal{Q}_{m,1}$ or $\pazocal{Q}_{m+1,1}$, while \Cref{M(S) annuli} implies each such $\theta$-band crosses any of the maximal $q$-bands at most once.  This means there are at most $2h$ distinct maximal $\theta$-bands in $\Gamma$, and hence at most $2Nh$ $(\theta,q)$-cells which are not $(\theta,t)$-cells.

Suppose $|V|_a>c_0h$.  By the structure of the rules and \Cref{simplify rules}, at most $4Nh$ maximal $a$-bands in $\Gamma$ end on a $(\theta,q)$-cell.  In particular, since no maximal $a$-band of $\Gamma$ can have two ends on $\textbf{p}^{-1}$, at least $|V|_a-4Nh\geq(1-\frac{4N}{c_0})|V|_a$ maximal $a$-bands in $\Gamma$ have one end on $\textbf{p}^{-1}$ and one end on $\textbf{q}$.  In particular, $|\textbf{q}|_a\geq(1-\frac{4N}{c_0})|V|_a$ and $|\textbf{q}|_\theta\leq2h$.  Hence, by the parameter choice $c_0>>N$ and \Cref{lengths}, at least $\frac{1}{2}|V|_a$ $a$-edges of $\textbf{q}$ contribute $\delta$ to $|\partial\Delta|$, \frenchspacing{i.e. $|\partial\Delta|\geq\frac{\delta}{2}|V|_a$}.

This implies $2h+\delta L|V|_a\leq(\frac{2}{c_0}+\delta L)|V|_a\leq2\delta^{-1}(\frac{2}{c_0}+\delta L)|\partial\Delta|$.  But then the parameter choices $N_4>>\delta^{-1}>>L>>c_0$ imply $2h+\delta L|V|_a\leq N_4|\partial\Delta|$, contradicting \Cref{n-minimal first bound}.

(b) Part (a) implies $2h+\delta L|V|_a\leq (2+\delta Lc_0)h$.  The parameter choices $\delta^{-1}>>L>>c_0$ then allow us to assume $\delta Lc_0\leq1$, and so $2h+\delta L|V|_a\leq3h$.

\Cref{n-minimal first bound} then implies $3h>N_4|\partial\Delta|+3N_3$, so that $h>\frac{1}{3}N_4|\partial\Delta|+N_3$.  Thus, the statement follows from the parameter choice $N_4>>N_3$.

\end{proof}

\medskip


\subsection{Annular bands} \

The rest of our argument proceeds in two cases.  As a first case, we assume in this section that one of the maximal $\theta$-bands that crosses $\pazocal{Q}_{m,1}$ is annular.  The presence of this $\theta$-annulus implies the next statement regarding the structure of the bands in $\Delta$.

\begin{lemma} \label{n-minimal annuli crossings}

Any maximal $\theta$-band of $\Delta$ crosses any maximal $q$-band at most once.

\end{lemma}

\begin{proof}

Assuming to the contrary that there exists a maximal $\theta$-band $\pazocal{T}$ in $\Delta$ which crosses a $q$-band at least twice, \Cref{n-minimal boundary ends} implies $\pazocal{T}$ ends on both $\textbf{C}_1$ and $\textbf{C}_2$.  But then an annular $\theta$-band in $\Delta$ would have to cross $\pazocal{T}$, which is not possible since $\theta$-bands cannot cross.

\end{proof}

Let $\pazocal{T}_1,\pazocal{T}_2$ be two annular $\theta$-bands in $\Delta$ which cross $\pazocal{Q}_{m,1}$ at the $(\theta,t)$-cells $\pi_1,\pi_2$, respectively.  Suppose $\pi'$ is a cell of $\pazocal{Q}_{m,1}$ between $\pi_1$ and $\pi_2$ and let $\pazocal{T}'$ be the maximal $\theta$-band of $\Delta$ crossing $\pi'$.  Since $\theta$-bands cannot cross, $\pazocal{T}'$ must then also be annular.  Hence, the cells of $\pazocal{Q}_{m,1}$ which are crossed by annular $\theta$-bands form a subband $\pazocal{Q}_{m,1}'$.

By construction, the maximal $\theta$-bands of $\Gamma$ which cross $\pazocal{Q}_{m,1}'$ also cross $\pazocal{Q}_{m+1,1}$.  These $\theta$-bands, which are subbands of the $\theta$-annuli which cross $\pazocal{Q}_{m,1}$, form a trapezium $\Gamma'$ which is a subdiagram of $\Gamma$.  One side $q$-band of $\Gamma'$ is $\pazocal{Q}_{m,1}'$, while the other is a subband $\pazocal{Q}_{m+1,1}'$ of $\pazocal{Q}_{m+1,1}$.  We may then assume the standard factorization of $\partial\Gamma'$ is $\textbf{s}_{m+1}^{-1}\textbf{y}\textbf{s}_m\textbf{z}^{-1}$ of $\partial\Gamma'$ such that $\textbf{s}_m$ and $\textbf{s}_{m+1}$ are sides of $\pazocal{Q}_{m,1}'$ and $\pazocal{Q}_{m+1,1}'$, respectively. 

\begin{lemma} \label{n-minimal annuli h'}

Let $H'$ be the history of $\Gamma'$ and $h'=\|H'\|$.

\begin{enumerate}[label=(\alph*)]

\item $h>N_3(h-h')$

\item $h'>N_2|\partial\Delta|+N_2$

\item $H'$ has no controlled subword.

\end{enumerate}

\end{lemma}

\begin{proof}

(a) \Cref{n-minimal annuli crossings} implies the cells of $\pazocal{Q}_{m,1}$ not contained in $\pazocal{Q}_{m,1}'$ correspond to distinct maximal $\theta$-bands of $\Delta$ that end on the boundary components.  As such, $h-h'\geq|\partial\Delta|$, and hence the statement is given by \Cref{n-minimal first step}(b).

(b) By \Cref{n-minimal first step}(b) and part (a), $h'>(1-\frac{1}{N_3})h>(N_3-1)(|\partial\Delta|+1)$.  The statement thus follows by the parameter choice $N_3>>N_2$.

(c) Suppose $H'$ has a controlled subword $H''$ and let $\pazocal{Q}_{m,1}''$ be the subband with this history.  Let $\Delta''$ be the annular subdiagram of $\Delta$ bounded by the annular $\theta$-bands which cross $\pazocal{Q}_{m,1}''$.  Then let $\Gamma''$ be the circular diagram obtained by cutting along a side of $\pazocal{Q}_{m,1}''$ and pasting a copy of this $t$-band to the side on which it is removed.

By construction, $\Gamma''$ is then a quasi-trapezium with circular base and history $H''$.  \Cref{quasi-trapezia} then provides a minimal diagram $\Sigma$ with the same boundary label as $\Gamma''$ consisting of a trapezium $\Sigma_0$ with the same base and history as $\Gamma''$ along with some disks attached to the top and bottom.

Applying \Cref{enhanced controlled} to the reduced computation corresponding to $\Sigma_0$ through \Cref{trapezia are computations}, up to inversion the bottom label of $\Sigma_0$ is the power of a cyclic permutation of a disk relation with the appropriate $t$-letter appended at the end.  As the $t$-bands remain intact when passing from $\Gamma''$ to $\Sigma$, it follows that the label of the bottom of $\Sigma$ (and so of $\Gamma''$) is of the form $(t(i_m)w''t(i_m))^{\pm1}$ such that $t(i_m)w''=_{G(\textbf{M}_\textbf{S})}1$.  

Hence, as $\Delta''$ is obtained from $\Gamma''$ by pasting the side $t$-bands together, the label of its inner boundary component is trivial in $G(\textbf{M}_\textbf{S})$.

But then this is a non-contractible loop in $\Delta$, yielding a contradicting to \Cref{n-minimal trivial path}.

\end{proof}

\begin{lemma} \label{n-minimal annuli h' bounds} \

\begin{enumerate}[label=(\alph*)]

\item $h'>c_1\|\textbf{z}\|$

\item $h'\leq c_1\|\textbf{y}\|$

\item $h'\leq c_2|V|_a+c_2$

\item $|V|_a>N_1|\partial\Delta|+N_1$

\end{enumerate}

\end{lemma}

\begin{proof}

(a) Suppose to the contrary that $h'\leq c_1\|\textbf{z}\|$.  Since $\|\textbf{z}\|=|\textbf{z}|_a+N+1$, \Cref{n-minimal annuli h'}(b) implies $c_1|\textbf{z}|_a+c_1(N+1)>N_2|\partial\Delta|+N_2$, so that the parameter choices $N_2>>N_1>>c_1>>N$ yield $|\textbf{z}|_a>N_1|\partial\Delta|+N_1$.

In particular, $|\textbf{z}|_a>N_1$, so that the parameter choice $N_1>>N$ implies $$\|\textbf{z}\|=|\textbf{z}|_a+N+1\leq2|\textbf{z}|_a$$

Now, \Cref{n-minimal annuli h'}(a) implies $h'>(1-\frac{1}{N_3})h$, so that $h-h'<\frac{1}{N_3}h<\frac{1}{N_3-1}h'$.  As a result, we have $\|\textbf{z}\|\geq\frac{1}{c_1}h'>\frac{N_3-1}{c_1}(h-h')$, so that the parameter choices $N_3>>N_2>>c_1$ imply $\|\textbf{z}\|>2N_2(h-h')$.  Hence, $|\textbf{z}|_a>N_2(h-h')$.

Let $\Gamma_1$ be the subdiagram of $\Gamma$ bounded by $\textbf{z}$ and the subpath $\textbf{q}$ of $\textbf{C}_1$.  As in the proof of \Cref{n-minimal first step}(a), the number of maximal $a$-bands of $\Gamma_1$ which have an end on a $(\theta,q)$-cell is at most $2N(h-h')$, while $|\textbf{q}|_\theta\leq2(h-h')$.  The parameter choice $N_2>>N$ then implies at least $\frac{1}{2}|\textbf{z}|_a$ $a$-edges of $\textbf{C}_1$ contribute $\delta$ to $|\partial\Delta|$.  In particular, $|\partial\Delta|\geq\frac{\delta}{2}|\textbf{z}|_a$.  

But then the parameter choice $N_1>>\delta^{-1}$ implies $|\textbf{z}|_a\leq N_1|\partial\Delta|$, yielding a contradiction.

(b) Assuming the statement is false, $h'>c_1\max(\|\textbf{y}\|,\|\textbf{z}\|)$.  In particular, the reduced computation corresponding to $\Gamma'$ by \Cref{trapezia are computations} satisfies the hypotheses of \Cref{long history controlled}.  But then $H'$ must have a controlled subword, contradicting \Cref{n-minimal annuli h'}(c).

(c) Let $\Gamma_2$ be the subdiagram of $\Gamma$ bounded by $\textbf{y}$ and the subpath $\textbf{p}$ of $\partial\Pi$.  As in part (a) above, at least half of the maximal $a$-bands of $\Gamma_2$ with one end on $\textbf{y}$ have another end on $\textbf{p}$, so that $|V|_a=|\textbf{p}|_a\geq\frac{1}{2}|\textbf{y}|_a$.  

Hence, (b) implies $h'\leq 2c_1|V|_a+c_1(N+1)$, and thus the statement is given by the parameter choices $c_2>>c_1>>N$.

(d) Combining \Cref{n-minimal annuli h'}(b) and (c) we have $c_2|V|_a+c_2>N_2|\partial\Delta|+N_2$.  

Thus, the statement is given by the parameter choices $N_2>>N_1>>c_2$.

\end{proof}

%
%
%
%
%

For $j\in\{1,2\}$ and $1\leq i\leq \ell_j-1$, let $\Psi_{i,j}$ be the subdiagram of $\Psi_j$ bounded by $\pazocal{Q}_{i,j}$ and $\pazocal{Q}_{i+1,j}$.  Further, let $h_{i,j}''=\max(h_{i,j},h_{i+1,j})$.

\begin{lemma} \label{n-minimal annuli r} 

For all $j\in\{1,2\}$ and $1\leq i\leq \ell_j-1$, we have:

\begin{enumerate}[label=(\alph*)]

\item $|V|_a\leq c_0h_{i,j}''$.

\item $h_{i,j}''>C_3|\partial\Delta|+C_3$.

\end{enumerate}

\end{lemma}

\begin{proof}

(a) As in the proof of \Cref{n-minimal first step}(a), the number of maximal $a$-bands of $\Psi_{i,j}$ which have an end on a $(\theta,q)$-cell is at most $N(h_{i,j}+h_{i+1,j})\leq2Nh_{i,j}''$.

Assuming to the contrary that $|V|_a>c_0h_{i,j}''$ for some $i,j$, then arguments analogous to those in previous statements apply to $\Psi_{i,j}$.  In particular, the parameter choice $c_0>>N$ implies at least $\frac{1}{2}|V|_a$ contribute $\delta$ to $|\partial\Delta|$, and so $|V|_a\leq2\delta^{-1}|\partial\Delta|$.  

But then the parameter choice $N_1>>\delta^{-1}$ leads to a contradiction of \Cref{n-minimal annuli h' bounds}(d).

(b) Combining \Cref{n-minimal annuli h' bounds}(d) with (a) we have $h_{i,j}''\geq\frac{1}{c_0}|V|_a>\frac{N_1}{c_0}(|\partial\Delta|+1)$.  Hence, the parameter choices $N_1>>C_3>>c_0$ imply the statement.

\end{proof}

\begin{lemma} \label{n-minimal 2-annuli}

For any $1\leq i\leq \ell_2-1$, at least one $\theta$-annulus crosses $\pazocal{Q}_{i,2}$.

\end{lemma}

\begin{proof}

Assuming this is not the case, then \Cref{n-minimal annuli crossings} implies the $h_{i,2}''$ cells of $\pazocal{Q}_{i,2}$ (or of $\pazocal{Q}_{i+1,2}$ depending on which one is longer) are crossed by distinct maximal $\theta$-bands, each of which has ends on $\partial\Delta$.  But then $h_{i,2}''\leq|\partial\Delta|$, contradicting \Cref{n-minimal annuli r}(b) for $C_3\geq1$.

\end{proof}

Just as in the discussion of $\pazocal{Q}_{m,1}$, the cells of any of the spokes $\pazocal{Q}_{i,j}$ which are crossed by $\theta$-annuli form a subband $\pazocal{Q}_{i,j}'$.  Indeed, any $\theta$-annulus that crosses $\pazocal{Q}_{n,j}'$ for some $n$ crosses $\pazocal{Q}_{i,j}'$ for all $i$.  So, for $j=1,2$, these $\theta$-annuli bound a trapezium $\Gamma_j$ in $\Psi_j$.

Let $\partial\Gamma_j=\textbf{t}_j^{-1}\textbf{y}_j\textbf{s}_j\textbf{z}_j^{-1}$ be the standard factorization of the trapezium such that $\textbf{s}_j$ is a side of $\pazocal{Q}_{1,j}'$ and $\textbf{t}_j$ is a side of $\pazocal{Q}_{\ell_j,j}'$.  Further, let $H_j'$ be the history of $\Gamma_j$.

\begin{lemma} \label{n-minimal transposition}

For $j=1,2$, we have:

\begin{enumerate}[label=(\alph*)]

\item $\textbf{y}_j^{-1}$ is a subpath of $\partial\Pi$

\item The first letters of $H_1'$ and $H_2'$ are not the same.

\item $H_j'$ does not have a controlled subword.

\end{enumerate}

\end{lemma}

\begin{proof}

(a) Let $\pazocal{T}_j$ be the $\theta$-annulus for which $\textbf{y}_j$ is a subpath of one side.  Then $\pazocal{T}_1$ and $\pazocal{T}_2$ bound an annular subdiagram $\Delta_a$ of $\Delta$.  As $\theta$-bands cannot cross, any other maximal $\theta$-band of $\Delta_a$ must be annular.  But then it must cross $\pazocal{Q}_{i,j}$ closer to $\Pi$ than $\pazocal{T}_j$, contradicting the definition of $\pazocal{T}_j$.  Hence, the crossing of $\pazocal{T}_j$ with $\pazocal{Q}_{i,j}$ must be the first cell of the band, implying the statement.

(b) Assuming the statement is false, (a) implies $\pazocal{T}_1$ and $\pazocal{T}_2$ satisfy the hypotheses of \Cref{two theta-bands about disk}.  But these bands cross at least $L-18$ of the spokes of $\Pi$ in total, so that we reach a contradiction by taking $L\geq37$.

(c) follows from an identical proof to \Cref{n-minimal annuli h'}(c): A controlled subword would yield a non-contractible loop with label that represents the identity in $G(\textbf{M}_\textbf{S})$, contradicting \Cref{n-minimal trivial path}.

\end{proof}

For $i=1,\dots,\ell_j-1$, let $\Gamma_{i,j}$ be the subdiagram of $\Gamma_j$ shared with $\Psi_{i,j}$.  Note then that $\Gamma_{i,j}$ is also trapezium with history $H_j'$.  

Let $\partial\Gamma_{i,j}=\textbf{t}_{i,j}^{-1}\textbf{y}_{i,j}\textbf{s}_{i,j}\textbf{z}_{i,j}^{-1}$ be the standard factorization of the trapezium where $\textbf{y}_{i,j}$ and $\textbf{z}_{i,j}$ are subpaths of $\textbf{y}_j$ and $\textbf{z}_j$, respectively.

Note then that $\Gamma_{m,1}=\Gamma$, so that $H_1'=H'$.  Let $h_j'=\|H_j'\|$, so that $h_1'=h'$.


\begin{lemma} \label{n-minimal annuli r'}

For $j\in\{1,2\}$ and $1\leq i\leq \ell_j-1$, we have:

\begin{enumerate}[label=(\alph*)]

\item $C_3(h_{i,j}''-h_j')<h_{i,j}''$.

\item $h_j'>C_2|\partial\Delta|+C_2$.

\item $h_j'>c_1\|\textbf{z}_{i,j}\|$.

\end{enumerate}

\end{lemma}

\begin{proof}

(a) By \Cref{n-minimal annuli crossings}, the $h_{i,j}''-h_j'$ $(\theta,t)$-cells of $\pazocal{Q}_{i,j}$ (or of $\pazocal{Q}_{i+1,j}$ if it is longer) not crossed by $\theta$-annuli are part of distinct maximal $\theta$-bands that end on $\partial\Delta$.  So, assuming the inequality is false, $h_{i,j}''\leq C_3(h_{i,j}''-h_j')\leq C_3|\partial\Delta|$.  But then this contradicts \Cref{n-minimal annuli r}(b).

(b) Combining (a) with \Cref{n-minimal annuli r}(b), we have $h_j'>(1-\frac{1}{C_3})h_{i,j}''>(C_3-1)(|\partial\Delta|+1)$, so that the statement follows from the parameter choice $C_3>>C_2$.

(c) follows in much the same way as the proof of \Cref{n-minimal annuli h' bounds}(a): 

Supposing $h_j'\leq c_1\|\textbf{z}_{i,j}\|=c_1|\textbf{z}_{i,j}|_a+c_1(N+1)$, the parameter choices $C_2>>C_1>>c_1>>N$ and (b) yield $|\textbf{z}_{i,j}|_a>C_1|\partial\Delta|+C_1$.  In particular, $|\textbf{z}_{i,j}|_a>C_1$, so that $\|\textbf{z}_{i,j}\|\leq2|\textbf{z}_{i,j}|_a$.

The inequality given in part (a) implies $h_j'>(1-\frac{1}{C_3})h_{i,j}''$, so that $h_{i,j}''-h_j'<\frac{1}{C_3}h_{i,j}''<\frac{1}{C_3-1}h_j'$.  So, $\|\textbf{z}_{i,j}\|\geq\frac{1}{c_1}h_j'>\frac{C_3-1}{c_1}(h_{i,j}''-h_j')$, and hence the parameter choices $C_3>>C_2>>c_1$ imply $|\textbf{z}_{i,j}|_a>C_2(h_{i,j}''-h_j')$.

Letting $\Gamma_{i,j}'$ be the subdiagram of $\Gamma_{i,j}$ bounded by $\textbf{z}_{i,j}$ and a subpath of $\textbf{C}_j$, then the number of maximal $a$-bands of $\Gamma_{i,j}'$ which have an end on a $(\theta,q)$-cell is at most $2N(h_{i,j}''-h_j')$.  The parameter choice $C_2>>N$ then implies at least $\frac{1}{2}|\textbf{z}_{i,j}|_a$ $a$-edges of $\textbf{C}_j$ contribute $\delta$ to $|\partial\Delta|$, so that $|\partial\Delta|\geq\frac{\delta}{2}|\textbf{z}_{i,j}|_a$.  But then the parameter choice $C_1>>\delta^{-1}$ implies $|\textbf{z}_{i,j}|_a\leq2\delta^{-1}|\partial\Delta|\leq C_1|\partial\Delta|$, yielding a contradiction.

\end{proof}

We finally reach a contradiction to our assumption that there exists an annular $\theta$-band in $\Delta$ which crosses $\pazocal{Q}_{m,1}$.

\begin{lemma} \label{n-minimal annular contradiction}

No annular $\theta$-band in $\Delta$ crosses $\pazocal{Q}_{m,1}$.

\end{lemma}

\begin{proof}

Assuming the statement is false, we have trapezia $\Gamma_{1,1}$ and $\Gamma_{1,2}$ with histories $H_1'$ and $H_2'$, respectively.  Let $\pazocal{C}_1$ and $\pazocal{C}_2$ be the reduced computations corresponding to the trapezia through \Cref{trapezia are computations}.  By \Cref{n-minimal transposition}(a), the initial admissible words of these computations are coordinate shifts of one another (with the appropriate $t$-letter appended).  So, \Cref{n-minimal transposition}(b) implies the inverse computation of $\pazocal{C}_1$ can be concatenated with the appropriate coordinate shift of $\pazocal{C}_2$ to produce a reduced computation $\pazocal{C}'':U_0\to\dots\to U_t$ with history $(H_1')^{-1}H_2'$.  

By construction, $\|U_0\|=\|\textbf{z}_{i,1}\|$ and $\|U_t\|=\|\textbf{z}_{i,2}\|$.  So, \Cref{n-minimal annuli r'}(c) implies $\pazocal{C}''$ can be identified with a reduced computation of the $k$-enhanced machine satisfying the hypotheses of \Cref{long history controlled}.  But then assuming without loss of generality that $h_1'\geq h_2'$, it follows that $H_1'$ has a controlled subword, contradicting \Cref{n-minimal transposition}(c).

\end{proof}


\subsection{Spirals} \label{sec-n-minimal-spirals} \

The arguments of the previous section allow us to assume now that every maximal $\theta$-band in $\Delta$ that crosses $\pazocal{Q}_{m,1}$ has two ends on $\partial\Delta$.  

\begin{lemma} \label{n-minimal spiral existence}

There exists a maximal $\theta$-band of $\Delta$ which crosses $\pazocal{Q}_{m,1}$ more than $N_3$ times.

\end{lemma}

\begin{proof}

Assuming the statement is false, at least $h/N_3$ distinct $\theta$-bands cross $\pazocal{Q}_{m,1}$.  But since these $\theta$-bands have distinct ends on $\partial\Delta$, we have $|\partial\Delta|\geq h/N_3$, contradicting \Cref{n-minimal first step}(b).

\end{proof}

Let $\pazocal{T}$ be the maximal $\theta$-band of $\Delta$ satisfying \Cref{n-minimal spiral existence} which crosses $\pazocal{Q}_{m,1}$ closest to $\Pi$.  That is, for any `initial' subband $\pazocal{U}$ of $\pazocal{Q}_{m,1}$ {\frenchspacing(i.e. which ends on $\Pi$)} that is not crossed by $\pazocal{T}$, any maximal $\theta$-band of $\Delta$ that does cross $\pazocal{U}$ crosses $\pazocal{Q}_{m,1}$ less than $N_3$ times.

\begin{lemma} \label{n-minimal spiral no annuli}

There are no annular $\theta$-bands in $\Delta$.

\end{lemma}

\begin{proof}

\Cref{n-minimal boundary ends} implies $\pazocal{T}$ has one end on $\textbf{C}_1$ and another on $\textbf{C}_2$.  But then any $\theta$-annulus in $\Delta$ would have to cross $\pazocal{T}$, which is impossible since $\theta$-bands cannot cross.

\end{proof}

Let $\pazocal{T}_1,\dots,\pazocal{T}_n$ be the minimal subbands of $\pazocal{T}$ which cross both $\pazocal{Q}_{m,1}$ and $\pazocal{Q}_{m+1,1}$, enumerated by proximity to $\Pi$ {\frenchspacing (i.e. there exists an initial subband of $\pazocal{Q}_{m,1}$ which $\pazocal{T}_i$ crosses but $\pazocal{T}_{i+1}$ does not)}.  Note that a subband of $\pazocal{T}$ which crosses $\pazocal{Q}_{m,1}$ twice must cross $\pazocal{Q}_{m+1,1}$, and so \Cref{n-minimal spiral existence} implies $n$ may be taken sufficiently large (with respect to $N_3$).

So, each $\pazocal{T}_i$ is a maximal $\theta$-band in $\Gamma$.  Indeed, these bands bound a trapezium $\Gamma'$ in $\Gamma$ with standard factorization $\partial\Gamma'=\textbf{s}_{m+1}^{-1}\textbf{y}\textbf{s}_m\textbf{z}^{-1}$ such that:

\begin{itemize}

\item $\textbf{s}_j$ is the side of the minimal subband $\pazocal{Q}_{j,1}'$ of $\pazocal{Q}_{j,1}$ which is crossed by each $\pazocal{T}_i$.


\item $\textbf{y}$ is a side of $\pazocal{T}_1$.

\item $\textbf{z}$ is a side of $\pazocal{T}_n$.

\end{itemize}

\begin{lemma} \label{n-minimal spiral h'}

Let $H'$ be the history of $\Gamma'$ and $h'=\|H'\|$.

\begin{enumerate}[label=(\alph*)]

\item $h>\frac{1}{2}N_3(h-h')$

\item $h'>N_2|\partial\Delta|+N_2$

\end{enumerate}

\end{lemma}

\begin{proof}

(a) Since $\pazocal{T}$ has an end on each boundary component, cutting along a side this $\theta$-band produces a circular diagram $\Delta'$.  Since $\theta$-bands do not cross, the maximal $\theta$-bands of $\Delta'$ can be identified with those of $\Delta$.  On the other hand, the cells of $\pazocal{Q}_{m,1}$ are split into several bands corresponding to the times $\pazocal{T}$ crosses it in $\Delta$.  By the maximality of $n$, though, $\pazocal{T}$ crosses the two components of $\pazocal{Q}_{m,1}\setminus\pazocal{Q}_{m,1}'$ at most once, so that these components correspond to at most four distinct maximal $q$-bands in $\Delta'$.  As such, \Cref{G(S) annuli} implies any maximal $\theta$-band of $\Delta$ crosses $\pazocal{Q}_{m,1}\setminus\pazocal{Q}_{m,1}'$ at most four times (note that this bound can be sharpened, but such improvement is immaterial).

Hence, at least $(h-h')/4$ maximal $\theta$-bands cross $\pazocal{Q}_{m,1}\setminus\pazocal{Q}_{m,1}'$, each of which has two ends on $\Delta$.  But then $|\partial\Delta|>(h-h')/2$, so that the inequality is given by \Cref{n-minimal first step}(b).

(b) Combining (a) with \Cref{n-minimal first step}(b) we have $h'>(1-\frac{2}{N_3})h>(N_3-2)(|\partial\Delta|+1)$.  Hence, the inequality is given by the parameter choice $N_3>>N_2$.

\end{proof}

%
%
%
%

\begin{lemma} \label{n-minimal spiral H_0}

There exists a word $H_0\in F(\Theta^+)$ such that:

\begin{enumerate}[label=(\alph*)]

\item $H'\equiv H_0^{n-1}\theta$, where $\theta$ is the first letter of $H_0$.

\item $H_0^2$ has no controlled subword

\end{enumerate}

\end{lemma}

\begin{proof}

(a) For $1\leq i\leq n-1$, let $\pazocal{T}_i'$ be the maximal subband of $\pazocal{T}$ between $\pazocal{T}_i$ and $\pazocal{T}_{i+1}$.  Note that $\pazocal{T}_i'$ then shares an end with both $\pazocal{T}_i$ and $\pazocal{T}_{i+1}$, with one of these ends an edge of the side $\textbf{r}$ of $\pazocal{Q}_{m,1}$ and the other an edge of the side $\textbf{u}$ of $\pazocal{Q}_{m+1,1}$ (see the definition of $\Gamma$ with \Cref{fig-n-minimal}).

If the end shared with $\pazocal{T}_i$ is an edge of $\textbf{u}$, then let $\pazocal{T}_i''$ to be the subband of $\pazocal{T}$ consisting of the `concatenation' of the subbands $\pazocal{T}_i$ and $\pazocal{T}_i'$.  Otherwise, the end shared with $\pazocal{T}_{i+1}$ is an edge of $\textbf{u}$, and so we let $\pazocal{T}_i''$ be the subband of $\pazocal{T}$ consisting of the `concatenation' of $\pazocal{T}_i'$ and $\pazocal{T}_{i+1}$.

Note that in either case, both ends of $\pazocal{T}_i''$ are edges of $\textbf{r}$.
Let $\textbf{s}_i'$ be the maximal subpath of $\textbf{r}$ between the ends of $\pazocal{T}_i$ and $\pazocal{T}_{i+1}$, {\frenchspacing i.e. the maximal subpath between the two ends of $\pazocal{T}_i''$.}

Now, for $1\leq i\leq n-2$, let $\Lambda_i$ be the circular subdiagram of $\Delta$ bounded by $\textbf{s}_i'$, $\textbf{s}_{i+1}'$, and the appropriate sides of $\pazocal{T}_i''$ and $\pazocal{T}_{i+1}''$ (see \Cref{fig-n-minimal-Lambda}).  As $\theta$-bands cannot cross, the maximal $\theta$-bands in $\Lambda_i$ all have ends on both $\textbf{s}_i'$ and $\textbf{s}_{i+1}'$.  But since also these $\theta$-bands cross $\pazocal{Q}_{m,1}$ at most once, this means $\lab(\textbf{s}_i')\equiv\lab(\textbf{s}_{i+1}')$.

In particular, the history of the subband of $\pazocal{Q}_{m,1}$ between any consecutive crossings of $\pazocal{T}$ is identical.  Letting $\theta$ be the history of $\pazocal{T}$, the statement then follows.

\begin{figure}
\centering
\includegraphics[scale=1.5]{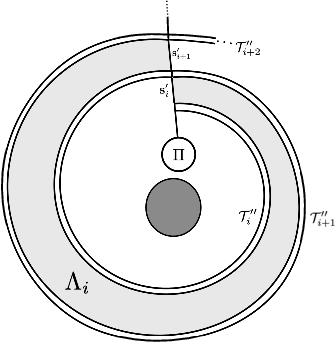} 
\caption{\Cref{n-minimal spiral H_0}(a)} \label{fig-n-minimal-Lambda}
\end{figure}

(b) Since $\lab(\textbf{s}_i')\equiv\lab(\textbf{s}_{i+1}')$, we may paste a copy of the subband of $\pazocal{Q}_{m,1}'$ contained in $\Lambda_i$ to the `other side', {\frenchspacing i.e. to the} subpath of $\partial\Lambda_i$ corresponding to $\textbf{s}_i'$ or $\textbf{s}_{i+1}'$ (depending on whether $\pazocal{T}_i''$ is the concatenation of $\pazocal{T}_i'$ with $\pazocal{T}_i$ or with $\pazocal{T}_{i+1}$).  Let $\Lambda_i'$ be the resulting circular diagram.

Further, letting $\pazocal{T}_i'''$ be the subband of $\pazocal{T}$ given by extending $\pazocal{T}_i''$ by the single $(\theta,t)$-cell corresponding to the crossing of $\pazocal{T}_i$ or $\pazocal{T}_{i+1}$ with $\pazocal{Q}_{m,1}$ (again, depending on how $\pazocal{T}_i''$ is formed), we may paste $\pazocal{T}_i'''$ to the `bottom' of $\Lambda_i'$ and $\pazocal{T}_{i+1}'''$ to the `top'.  Let $\bar{\Lambda}_i$ be the resulting circular diagram.

Now, let $\bar{\Lambda}$ be the circular diagram obtained from pasting together $\bar{\Lambda}_1,\dots,\bar{\Lambda}_4$ along the common bands arising from $\pazocal{T}_i'''$.  By construction, $\bar{\Lambda}$ is a quasi-trapezium with circular base starting and ending with $t(i_m)^{\pm1}$ and history $H_0^4\theta$.

Applying \Cref{quasi-trapezia}, there then exists a minimal circular diagram $\Sigma$ with $\lab(\partial\Sigma)\equiv\lab(\bar{\Lambda})$ consisting of a trapezium $\Sigma_0$ with the same base and history as $\bar{\Lambda}$ along with some disks attached to the top and bottom.  Let $\pazocal{U}_i$ be the maximal $\theta$-band of $\Sigma_0$ arising from $\pazocal{T}_i'''$ from these operations.

Folding and pasting $\Sigma$ along the side $t$-bands in the appropriate way, we obtain an annular diagram with the same boundary labels as an annular subdiagram of $\Delta$.  Let $\Delta'$ be the diagram obtained from $\Delta$ by removing this annular subdiagram and pasting in its place the annular diagram arising from $\Sigma$.  Note that the minimality of both $\Delta$ and $\Sigma$ imply $\Delta'$ is also minimal.  

Now, assuming $H_0^2$ has a controlled subword, there exist distinct controlled subwords of the history of $\Sigma_0$ in the two factors of $H_0^2$.  Applying \Cref{enhanced controlled} to the appropriate subcomputations of the reduced computation corresponding to $\Sigma_0$ through \Cref{trapezia are computations}, the labels of the side of $\pazocal{U}_1$ and $\pazocal{U}_3$ are identical.

In $\Delta'$, these paths can be closed with the side of a subband of $\pazocal{Q}_{m,1}'$ with history $H_0$.  In particular, these closed paths are disjoint, have identical labels, and bound an annular subdiagram containing at least one $(\theta,t)$-cell.  But then by \Cref{exciseTwoPaths} this contradicts the minimality of $\Delta'$.

\end{proof}

\begin{lemma} \label{n-minimal spiral controlled}

$H'$ has no controlled subword.

\end{lemma}

\begin{proof}

Suppose to the contrary that $H'$ has a controlled subword $H_c$.  Recall that the definition of controlled histories implies $H_c=\zeta_1w\zeta_2$, where $\zeta_1\neq\zeta_2^{\pm1}$ and $w$ has no occurrence of $\zeta_i^{\pm1}$.

Now, \Cref{n-minimal spiral H_0}(a) implies any letter that occurs in $H'$ occurs in a single factor of $H_0$.  In particular, no factor of $H_0$ can be a proper subword of $H_c$.  So, $H_c$ must be a subword of a factor $H_0^2$ or of the suffix $H_0\theta$, which is itself a subword of $H_0^2$.  But then this contradicts \Cref{n-minimal spiral H_0}(b).

\end{proof}

\begin{lemma} \label{n-minimal spiral h' bounds} \

\begin{enumerate}[label=(\alph*)]

\item $h'>c_1\|\textbf{z}\|$

\item $h'\leq c_1\|\textbf{y}\|$

\item $h'\leq c_2|V|_a+c_2$

\item $|V|_a>N_1|\partial\Delta|+N_1$

\end{enumerate}

\end{lemma}

\begin{proof}

The argument proceeds in the same way as that presented for the proof of \Cref{n-minimal annuli h' bounds}(a), using Lemmas \ref{n-minimal spiral h'} and \ref{n-minimal spiral controlled} in place of \Cref{n-minimal annuli h'}.

\end{proof}

As in the annular case, we set $\Psi_{i,j}$ to be the subdiagram of $\Psi_j$ bounded by $\pazocal{Q}_{i,j}$ and $\pazocal{Q}_{i+1,j}$, then let $h_{i,j}''=\max(h_{i,j},h_{i+1,j})$.  The next statement then follows in just the same way as \Cref{n-minimal annuli r}.

\begin{lemma} \label{n-minimal spiral r} 

For all $j\in\{1,2\}$ and $1\leq i\leq \ell_j-1$, we have:

\begin{enumerate}[label=(\alph*)]

\item $|V|_a\leq c_0h_{i,j}''$.

\item $h_{i,j}''>C_3|\partial\Delta|+C_3$.

\end{enumerate}

\end{lemma}

For $j\in\{1,2\}$ and $1\leq i\leq \ell_j-1$, fix $p(i,j)\in\{i,i+1\}$ such that $\pazocal{Q}_{p(i,j),j}$ has length $h_{i,j}''$.  Note then that there exists a side $\textbf{r}_{i,j}$ of $\pazocal{Q}_{p(i,j),j}$ which is a subpath of $(\partial\Psi_{i,j})^{\pm1}$.  For example, in the setting above we have $p(m,1)=m$ and the side $\textbf{r}_{m,1}=\textbf{r}$ of $\pazocal{Q}_{m,1}$ is a subpath of $\partial\Psi_{m,1}=\Gamma$.

\begin{lemma} \label{n-minimal spiral existence 2}

There exists a maximal $\theta$-band of $\Delta$ which crosses $\pazocal{Q}_{p(i,j),j}$ more than $C_3$ times.

\end{lemma}

\begin{proof}

This follows in much the same way as \Cref{n-minimal spiral existence}: If the statement is false, at least $h_{i,j}''/C_3$ distinct $\theta$-bands cross $\pazocal{Q}_{p(i,j),j}$, meaning $|\partial\Delta|\geq h_{i,j}''/C_3$, contradicting \Cref{n-minimal spiral r}(c).

\end{proof}

As with $\pazocal{T}$, let $\pazocal{S}$ be the maximal $\theta$-band of $\Delta$ satisfying \Cref{n-minimal spiral existence 2} for $\pazocal{Q}_{p(1,2),2}$ which crosses the band closest to $\Pi$.

Again, let $\pazocal{S}_1,\dots,\pazocal{S}_l$ be the minimal subbands of $\pazocal{S}$ which cross both $\pazocal{Q}_{1,2}$ and $\pazocal{Q}_{2,2}$.  \Cref{n-minimal spiral existence 2} then allows $l$ to be taken to be sufficiently large (with respect to $C_3$).

Each $\pazocal{S}_i$ is then a maximal $\theta$-band in $\Psi_{1,2}$, and bound a trapezium $\Gamma_2'$ in $\Psi_{1,2}$ with standard factorization $\partial\Gamma_2'=\textbf{t}_2^{-1}\textbf{y}_2\textbf{t}_1\textbf{z}_2^{-1}$ such that:

\begin{itemize}

\item $\textbf{t}_j$ is the side of the minimal subband $\pazocal{Q}_{j,2}'$ of $\pazocal{Q}_{j,2}$ which is crossed by each $\pazocal{S}_i$.

\item $\textbf{y}_2$ is a side of $\pazocal{S}_1$.

\item $\textbf{z}_2$ is a side of $\pazocal{S}_r$.

\end{itemize}

Repeating the arguments above relating to $\Gamma'$, we quickly obtain the following analogues of Lemmas \ref{n-minimal spiral h'}--\ref{n-minimal spiral h' bounds} for $\Gamma_2'$.

\begin{lemma} \label{n-minimal spiral h' 2}

Let $H_2'$ be the history of $\Gamma_2'$ and $h_2'=\|H'\|$.

\begin{enumerate}[label=(\alph*)]

\item $h_{1,2}''>\frac{1}{2}C_3(h_{1,2}''-h_2')$

\item $h_2'>C_2|\partial\Delta|+C_2$

\end{enumerate}

\end{lemma}

\begin{lemma} \label{n-minimal spiral H_0 2}

There exists a word $H_{0,2}\in F(\Theta^+)$ such that:

\begin{enumerate}[label=(\alph*)]

\item $H_2'\equiv H_{0,2}^{n-1}\theta_2$, where $\theta_2$ is the first letter of $H_{0,2}$.

\item $H_{0,2}^2$ has no controlled subword

\end{enumerate}

\end{lemma}

\begin{lemma} \label{n-minimal spiral controlled 2}

$H_2'$ has no controlled subword.

\end{lemma}

\begin{lemma} \label{n-minimal spiral h' bounds 2} \

\begin{enumerate}[label=(\alph*)]

\item $h_2'>c_1\|\textbf{z}_2\|$

\item $h_2'\leq c_1\|\textbf{y}_2\|$

\end{enumerate}

\end{lemma}

Recall that in the proof of \Cref{n-minimal spiral H_0}, for $1\leq i\leq n-1$ we define $\pazocal{T}_i'$ to be the subband of $\pazocal{T}$ between $\pazocal{T}_i$ and $\pazocal{T}_{i+1}$, then extend $\pazocal{T}_i'$ to a subband $\pazocal{T}_i''$ which has two ends on the side $\textbf{r}$ of $\pazocal{Q}_{m,1}$.  Further, we let $\textbf{s}_i'$ be the subpath of $\textbf{r}$ between the ends of $\pazocal{T}_i''$.  Note then that a side of $\pazocal{T}_i''$, $\textbf{s}_i'$, and one end of $\pazocal{T}_i''$ form a simple, non-contractible loop $\a_i$ in $\Delta$.

We now do the same for the band $\pazocal{S}$: For $1\leq i\leq r-1$, define $\pazocal{S}_i'$ to be the subband of $\pazocal{S}$ between $\pazocal{S}_i$ and $\pazocal{S}_{i+1}$, $\pazocal{S}_i''$ to be the extension which has two ends on a side $\textbf{r}_{1,2}$ of $\pazocal{Q}_{p(1,2),2}$ (formed by concatenating $\pazocal{S}_i'$ and either $\pazocal{S}_i$ or $\pazocal{S}_{i+1}$), and $\textbf{t}_i'$ to be the subpath of $\textbf{r}_{1,2}$ between the ends of $\pazocal{S}_i''$.  Then a side of $\pazocal{S}_i''$, $\textbf{t}_i'$, and one end of $\pazocal{S}_i''$ form a simple, non-contractible loop $\b_i$ in $\Delta$.

\begin{lemma} \label{n-minimal spiral loop separation}

Let $\textbf{s}_i''$ the maximal initial subpath of $\textbf{r}$ disjoint from $\a_i$.  Similarly, let $\textbf{t}_i''$ be the maximal initial subpath of $\textbf{r}_{1,2}$ disjoint from $\b_i$.

\begin{enumerate}[label=(\alph*)]

\item $\a_i$ separates the outer component of $\Delta$ from every edge of $\textbf{r}_{1,2}$ and of $\textbf{s}_i''$.

\item $\b_i$ separates the inner component of $\Delta$ from every edge of $\textbf{r}$ and of $\textbf{t}_i''$.

\end{enumerate}

\end{lemma}

\begin{proof}

Let $\textbf{x}$ be the path formed by $\textbf{r}^{-1}$, a subpath of $\partial\Pi$, and $\textbf{r}_{1,2}$.  Then $\textbf{x}$ is a path from the outer boundary component of $\Delta$ to the inner component.

Cutting along $\textbf{x}$ then produces a circular diagram $D$ with $\partial D=\textbf{x}'\textbf{i}(\textbf{x}'')^{-1}\textbf{o}$ where $\textbf{x}'$ and $\textbf{x}''$ are identified with $\textbf{x}$, $\textbf{i}$ with the inner component of $\Delta$, and $\textbf{o}$ with the outer component.

In $D$, $\a_i$ and $\b_i$ correspond to paths between $\textbf{x}'$ and $\textbf{x}''$ which share no edges with $\textbf{i}$ and $\textbf{o}$.  The statement then follows from noting the order in which the edges of $\textbf{r}^{-1}$ and $\textbf{r}_{1,2}$ arise in $\textbf{x}$.

\end{proof}

\begin{lemma} \label{n-minimal spiral subpaths} \

\begin{enumerate}[label=(\alph*)]

\item The $(\theta,t)$-cell of $\pazocal{Q}_{m,1}$ crossed by $\pazocal{T}_1$ is adjacent to $\Pi$.

\item The $(\theta,t)$-cell of $\pazocal{Q}_{p(1,2),2}$ crossed by $\pazocal{S}_1$ is adjacent to $\Pi$.

\end{enumerate}

\end{lemma}

\begin{proof}

(a) Suppose to the contrary that there is another cell $\pi$ of $\pazocal{Q}_{m,1}$ between $\Pi$ and the cell crossed by $\pazocal{T}_1$.  Let $\textbf{e}$ be the boundary edge of $\pi$ shared with $\textbf{r}$ and $\pazocal{T}'$ be the maximal $\theta$-band which crosses $\pazocal{Q}_{m,1}$ in this cell.

First, suppose $\pazocal{T}'$ has an end on the outer boundary component.  By \Cref{n-minimal spiral loop separation}, for all $1\leq i\leq n-1$ the loop $\a_i$ separates $\textbf{e}$ from the outer boundary component of $\Delta$.  Since $\theta$-bands cannot cross, it follows that $\pazocal{T}'$ must cross $\pazocal{Q}_{m,1}$ in the subband with side path $\textbf{s}_i'$.  In particular, $\pazocal{T}'$ crosses $\pazocal{Q}_{m,1}$ at least $n\geq N_3$ times.  But it also crosses $\pazocal{Q}_{m,1}$ closer to $\Pi$ than $\pazocal{T}$, contradicting the choice of $\pazocal{T}$.  

Hence, $\pazocal{T}'$ must have two ends on the inner component of $\Delta$.

Again, though, \Cref{n-minimal spiral loop separation} implies that each loop $\b_i$ separates $\textbf{e}$ from the inner component of $\Delta$, so that $\pazocal{T}'$ must cross $\pazocal{Q}_{p(1,2),2}$ in the subband with side path $\textbf{t}_i'$.  In particular, $\pazocal{T}'$ crosses $\pazocal{Q}_{p(1,2),2}$ at least twice, contradicting \Cref{n-minimal boundary ends}.

(b) follows from a symmetric argument.

\end{proof}

\begin{lemma} \label{n-minimal spiral first rule} \

\begin{enumerate}[label=(\alph*)]

\item There exists a subband $\pazocal{T}_0$ of $\pazocal{T}$ containing $\pazocal{T}_1$ which is a maximal $\theta$-band of $\Psi_1$, crosses every maximal $q$-band of $\Psi_1$, and has a side which is a subpath of $\partial\Pi$.

\item There exists a subband $\pazocal{S}_0$ of $\pazocal{S}$ containing $\pazocal{S}_1$ which is a maximal $\theta$-band of $\Psi_2$, crosses every maximal $q$-band of $\Psi_2$, and has a side which is a subpath of $\partial\Pi$.

\item When read away from $\Pi$, the history of $\pazocal{T}_0$ is different from that of $\pazocal{S}_0$.

\end{enumerate}

\end{lemma}

\begin{proof}

(a) Let $\textbf{e}_1$ be the end of $\pazocal{T}_1$ on $\textbf{r}$.  Note that since $\textbf{s}_2'$ is the subpath of $\textbf{r}$ between the ends of $\pazocal{T}_2$ and $\pazocal{T}_3$, by construction $\textbf{e}_1$ is disjoint from $\a_2$.  In particular, $\textbf{e}_1$ is an edge of $\textbf{s}_2''$, so that \Cref{n-minimal spiral loop separation} implies $\a_2$ separates $\textbf{e}_1$ from the outer boundary component of $\Delta$.

Let $\pazocal{T}_0$ be the maximal $\theta$-band of $\Psi_1$ for which $\pazocal{T}_1$ is a subband.  Note that since $\pazocal{T}_0$ crosses the subband of $\pazocal{Q}_{m,1}$ with side $\textbf{s}_2''$, \Cref{G(S) annuli} implies it cannot cross a subband whose side is a subpath of $\a_2$.  Hence, as $\theta$-bands cannot cross, $\pazocal{T}_0$ cannot end on the subpath of $\partial\Psi_1$ shared with the outer boundary component of $\Delta$.  

As such, $\pazocal{T}_0$ must cross both $\pazocal{Q}_{1,1}$ and $\pazocal{Q}_{\ell_1,1}$, and hence by \Cref{n-minimal subcombs} must cross every maximal $q$-band of $\Psi_1$.  Finally, \Cref{n-minimal spiral subpaths} implies no $\theta$-band can sit between $\pazocal{T}_0$ and $\Pi$, and hence its side must be shared with $\partial\Pi$.

(b) follows by a symmetric argument.

(c) As in the proof of \Cref{n-minimal transposition}(b), we see from (a) and (b) that $\pazocal{T}_0$ and $\pazocal{S}_0$ combine to cross at least $L-18$ spokes of $\Pi$, and so would provide a counterexample to \Cref{two theta-bands about disk} if they were to have the same history.

\end{proof}

We finally arrive at the main contradiction, achieving the ultimate goal of this section.

\begin{lemma} \label{annular disks}

The counterexample diagram does not exist.  That is, for any $n$-minimal diagram $\Delta$ with $s_1(\Delta)\geq1$, there exists a path $\textbf{t}$ between its boundary components with $|\textbf{t}|\leq N_4|\partial\Delta|+N_4$.

\end{lemma}

\begin{proof}

The idea of the proof is the same as that for \Cref{n-minimal annular contradiction}:

By \Cref{n-minimal spiral first rule}(a),(b) the bottom of the trapezia $\Gamma'$ and $\Gamma_2'$ are subpaths of $\partial\Pi$.  Letting $\pazocal{C}'$ and $\pazocal{C}_2'$ be the reduced computations corresponding to these trapezia by \Cref{trapezia are computations}, it follows from the definition of the disk relations that the initial admissible words of these computations are coordinate shifts of one another.   

Hence, by \Cref{n-minimal spiral first rule}(c) we may construct a reduced computation $\pazocal{C}''$ with history $(H')^{-1}H_2'$ by concatenating the inverse computation of $\pazocal{C}'$ with a coordinate shift of $\pazocal{C}_2'$.  Note then that $\pazocal{C}''$ is between admissible words which are copies of the labels of $\textbf{z}$ and $\textbf{z}_2$.  

Lemmas \ref{n-minimal spiral h' bounds}(a) and \ref{n-minimal spiral h' bounds 2}(a) then imply that $\pazocal{C}''$ can be identified with a reduced computation of the $k$-enhanced machine satisfying the hypotheses of \Cref{long history controlled}.  But then whichever of $H'$ or $H_2'$ is longer must certainly contain a controlled subword, contradicting \Cref{n-minimal spiral controlled} or \Cref{n-minimal spiral controlled 2}.

\end{proof}

\bigskip


\section{Proof of Theorem \ref{main-theorem}} \label{sec-main-proof}

In this section, we put together the ingredients developed in the previous sections, showing that the finitely presented group $G(\textbf{M}_\textbf{S})$ is sufficient for the proof of Theorem \ref{main-theorem}.  As indicated by the structure of the statement, the proof proceeds in two major steps: (1) That the Dehn function is equivalent to $\TM_\textbf{S}(n)^2$, and (2) That the conjugator length is equivalent to $n^2$.

However, it is convenient to only prove the relevant upper bound on the conjugator length in place of (2).  The lower bound is then achieved by taking the direct product with an appropriate finitely presented group, producing the group satisfying the statement.

\medskip


\subsection{Dehn function} \

To understand the Dehn function of the group $G(\textbf{M}_\textbf{S})$, it suffices to get an upper bound on the area of a circular diagram over the finite presentation with an arbitrary boundary label.  The groundwork for achieving this has been laid by our study of certain circular diagrams over the disk presentation of $G(\textbf{M}_\textbf{S})$.

However, while \Cref{disks are relations} implies the disk presentation is a presentation for $G(\textbf{M}_\textbf{S})$, it has infinitely many relations, making the area of diagrams over this presentation a poor measure for the Dehn function of the group.  Fortunately, this is accounted for by the concept of $G$-area.

First, we justify the carefully chosen assignments defining the $G$-area of a diagram.

\begin{lemma} \label{disk weight}

For any configuration $W$ accepted by $\textbf{M}_\textbf{S}$, there exists a reduced circular diagram $\Delta$ over the finite presentation of $G(\textbf{M}_\textbf{S})$ such that $\lab(\partial\Delta)\equiv W$ and $\text{Area}(\Delta)\leq C_1\phi(C_1|W|)$.

\end{lemma}

\begin{proof}

By the parallel construction of the main machine, the admissible subword of $W(i)$ whose base is a copy of the standard $k$-enhanced machine is a copy of an accepted configuration $V$ of $\textbf{E}_{\textbf{S},k}^0$.  \Cref{E generalized time} then produces a reduced computation of $\textbf{E}_{\textbf{S},k}^0$ accepting $V$ of length at most $c_0\TM_\textbf{S}(2c_0\|V\|)+c_0|V|_a+2c_0$.  The parallel construction of the machine then allows us to `extend' this to a reduced computation $\pazocal{C}:W\equiv W_0\to\dots\to W_t\equiv W_{ac}$ of $\textbf{M}_\textbf{S}$ accepting $W$.

Now, recall from the proof of \Cref{mixture quotient} that $f(n)\geq n$ for all $n$.  The bound on $\TM_\textbf{S}$ in terms of $f$ (see \Cref{sec-disks-weights}) and the parameter choice $c_1>>c_0$ then imply $t\leq c_1f(c_1\|V\|)\leq c_1f(c_1\|W\|)$.

As the base of a configuration of $\textbf{M}_\textbf{S}$ has length $LN$, \Cref{simplify rules} implies the application of any rule alters the $a$-length by at most $2LN$.  So, since $|W_{ac}|_a=0$, $|W_j|_a\leq 2(t-j)LN\leq 2tLN$ for all $j$, and so the parameter choices $K>>L>>c_1>>N$ imply $\|W_j\|\leq Kf(c_1\|W\|)$.

\Cref{computations are trapezia} then produces a trapezium $\Gamma$ with $\lab(\textbf{bot}(\Gamma))\equiv W$, $\lab(\textbf{top}(\Gamma))\equiv W_{ac}$, and $\text{Area}(\Gamma)\leq c_1Kf(c_1\|W\|)^2$.

Since the $(P_0'(L))^{-1}\{t(1)\}$-sector has empty tape alphabet, the side labels of $\Gamma$ are identical.  Hence, we may paste the sides together and paste a hub in the middle, producing a circular diagram $\Delta$ over the finite presentation of $G(\textbf{M}_\textbf{S})$ with $\lab(\partial\Delta)\equiv W$ and $\text{Area}(\Delta)\leq c_1Kf(c_1\|W\|)^2+1$.

The bound on $f$ in terms of $g$ and the definition of $\phi$ (again, see \Cref{sec-disks-weights}) then yield 
\begin{align*}
\text{Area}(\Delta)&\leq c_1K(c_0c_1^2\|W\|^2g(c_0c_1\|W\|)+c_0c_1\|W\|+c_0)+1 \\
&\leq c_1K(c_1\phi(c_0c_1\|W\|)+c_0c_1\|W\|+c_0)+1
\end{align*}
Since $g(x)\geq1$ for all $x\neq0$, $\phi(n)\geq1$ for all $n\in\N\setminus\{0\}$.  Hence, the parameter choices $J>>K>>c_2>>c_1>>c_0$ imply $\text{Area}(\Delta)\leq J\phi(c_2\|W\|)$.

Finally, note from the definition of the length function that $|W|=\delta|W|_a+LN\geq\delta\|W\|$, so that $c_2\|W\|\leq c_2\delta^{-1}|W|$.  Thus, noting that \Cref{phi properties} implies $\phi$ is superadditive and so non-decreasing, the parameter choices $C_1>>\delta^{-1}>>J>>c_2$ yield $\text{Area}(\Delta)\leq C_1\phi(C_1|W|)$.

\end{proof}

\begin{lemma} \label{big trapezia G-area assignment}

Suppose $\Gamma$ is a big trapezium whose $G$-area is not assigned to be half its area.  Let $\Gamma_0$ be the subdiagram of $\Gamma$ given by cutting off the rim $q$-bands $\pazocal{Q}_1$ and $\pazocal{Q}_2$.  Then there exists a circular diagram $\Delta_0$ over the finite presentation of $G(\textbf{M}_\textbf{S})$ such that: 

\begin{itemize}

\item $\lab(\partial\Delta_0)\equiv\lab(\partial\Gamma_0)$

\item $\text{Area}(\Delta_0)\leq\text{Area}_G(\Gamma)-2h$

\end{itemize}

\end{lemma}

\begin{proof}

By the definition of $G$-area, we have $\text{Area}_G(\Gamma)=3h+C_2\phi(C_2M)$, where $h$ is the height of the trapezium and $M=\max(\|\textbf{bot}(\Gamma)\|,\|\textbf{top}(\Gamma)\|)$.


Let $\pazocal{C}:W_0'\to\dots\to W_h'$ be the reduced computation corresponding to $\Gamma$ by \Cref{trapezia are computations}.  As $\Gamma$ is big, there exist $q$-letters $q_i$ and words $U_i$ such that $W_i'\equiv q_iU_iq_i$ for all $i$.  \Cref{enhanced controlled} then implies $W_i\equiv q_iU_i$ is a cyclic permutation of a disk relation (or its inverse).


For all $0\leq i\leq h$, \Cref{disk weight} then yields a reduced circular diagram $\Sigma_i$ over the finite presentation of $G(\textbf{M}_\textbf{S})$ with $\lab(\partial\Sigma_i)\equiv W_i$ and $\text{Area}(\Sigma_i)\leq C_1\phi(C_1|W_i|)$.  In particular, since $|W_i|\leq\|W_i\|$ and $\phi$ is non-decreasing, we have $\text{Area}(\Sigma_0),\text{Area}(\Sigma_h)\leq C_1\phi(C_1M)$.

Now, let $\pazocal{Q}$ be a $q$-band corresponding to the same part of the standard base as the first or last letter of the base of $\pazocal{C}$ and whose history is the same as that of $\Gamma$.  By construction, $\pazocal{Q}$ is a copy of both $\pazocal{Q}_1$ and $\pazocal{Q}_2$.

Hence, combining $\Sigma_0$ with mirror copies of $\Sigma_h$ and of $\pazocal{Q}$ (and perhaps $0$-refining), we obtain a reduced circular diagram $\Delta_0$ with $\lab(\partial\Delta_0)\equiv\lab(\partial\Gamma_0)$ and $\text{Area}(\Delta_0)\leq2C_1\phi(C_1M)+h$.  Thus, the parameter choice $C_2>>C_1$ implies $\text{Area}(\Delta_0)\leq\text{Area}_G(\Gamma)-2h$.

\end{proof}

\begin{lemma} \label{Dehn upper bound}

For any reduced circular diagram $\Delta$ over the disk presentation of $G(\textbf{M}_\textbf{S})$, there exists a reduced circular diagram $\Delta_0$ over the finite presentation of $G(\textbf{M}_\textbf{S})$ such that:

\begin{itemize}

\item $\lab(\partial\Delta_0)\equiv\lab(\partial\Delta)$ 

\item $\text{Area}(\Delta_0)\leq2\text{Area}_G(\Delta)$

\end{itemize}

\end{lemma}

\begin{proof}

Let $\textbf{P}$ be a covering of $\Delta$ realizing the $G$-area of $\Delta$.  We then construct from $\Delta$ the (potentially unreduced) diagram $\tilde{\Delta}_0$ over the finite presentation of $G(\textbf{M}_\textbf{S})$ as follows:

\begin{itemize}

\item Excise any disk $\Pi$ and paste in its place the corresponding diagram given by \Cref{disk weight}.

\item For any big trapezium $\Gamma\in\textbf{P}$ with $\text{Area}_G(\Gamma)<\frac{1}{2}\text{Area}(\Gamma)$, excise the subdiagram $\Gamma_0$ given by cutting off the rim $q$-bands and paste in its place the diagram given by \Cref{big trapezia G-area assignment}.

\end{itemize}

Note that Lemmas \ref{disk weight} and \ref{big trapezia G-area assignment} imply that for any element $P\in\textbf{P}$, the subdiagram of $\tilde{\Delta}_0$ corresponding to $P$ has area at most $2\text{Area}_G(P)$, and so $\text{Area}(\tilde{\Delta}_0)\leq2\text{Area}_G(\textbf{P})$.  Hence, the reduced diagram $\Delta_0$ obtained from $\tilde{\Delta}_0$ by making any necessary cancellations satisfies the statement.

\end{proof}

Recall that for every input configuration $V$ of $\textbf{S}$, there exists a corresponding input configuration $W\equiv I(V)$ of $\textbf{M}_\textbf{S}$ formed by letting $W(i)$ be the concatenation of $t(i)$ with the copy of the configuration $I_0(V)$ of the standard $k$-enhanced machine $\textbf{E}_{\textbf{S},k}^0$ and its mirror copy.  By construction, $|I(V)|_a=2L|V|_a$.

\begin{lemma}[Compare with Lemma 10.2 of \cite{O18}] \label{Dehn lower bound}

Given an accepted input configuration $V$ of $\textbf{S}$, for any reduced circular diagram $\Delta$ over the finite presentation of $G(\textbf{M}_\textbf{S})$ with $\lab(\partial\Delta)\equiv I(V)$, we have $\text{Area}(\Delta)\geq \frac{1}{2}\tm_{\textbf{S}}(V)^2$.

\end{lemma}

\begin{proof}

For simplicity, denote $W\equiv I(V)$.

First, note that since $|W|_\theta=0$, \Cref{M(S) annuli} implies $W$ cannot be trivial over $M(\textbf{M}_\textbf{S})$.  Hence, any diagram $\Delta$ as in the statement must have a hub.

Now, we say that a subdiagram $\Gamma$ of such a diagram $\Delta$ is a disk subdiagram if it has exactly one hub and a number of $\theta$-annuli surrounding it.  \Cref{trapezia are computations} then implies that the boundary label of any disk subdiagram is an accepted configuration of $\textbf{M}_\textbf{S}$. 

Similar to the definition of $G$-area, we define a covering $\textbf{P}$ of $\Delta$ to be a family of subdiagrams, each of which is either a disk subdiagram or a single $(\theta,q)$- or $(\theta,a)$-cell, such that every cell of $\Delta$ belongs to exactly one member of $\textbf{P}$.  In this case, we define $A(\textbf{P})$ to be the sum of the areas of the disk subdiagrams comprising $\textbf{P}$ along with twice the number of single cells.

Fix a reduced circular diagram $\Delta_0$ with $\lab(\partial\Delta_0)\equiv W$ along with a covering $\textbf{P}_0$ of $\Delta_0$ such that $A(\textbf{P}_0)$ is minimal among all such diagrams.  Then for any $\Delta$ with $\lab(\partial\Delta)\equiv W$ and any covering $\textbf{P}$ of $\Delta$, we have $2\text{Area}(\Delta)\geq A(\textbf{P})\geq A(\textbf{P}_0)$.  Hence, it suffices to show $A(\textbf{P}_0)\geq\tm_{\textbf{S}}(V)^2$.

Suppose $\Gamma_1,\Gamma_2\in\textbf{P}_0$ are disk subdiagrams that share consecutive boundary $t$-edges so that the subdiagram bounded by the subpaths between these two $t$-edges contains no hub.  Then \Cref{M(S) annuli} implies the entire subpath between these two $t$-edges is shared.  But then the parallel construction of $\textbf{M}_\textbf{S}$ implies $\partial\Gamma_1$ and $\partial\Gamma_2$ have inverse labels, so that these two subdiagrams may be cancelled to reduce $A(\textbf{P}_0)$.

Next, assume $\textbf{P}_0$ has a single cell and let $\pazocal{T}$ be the maximal $\theta$-band of $\Delta_0$ crossing this cell.  As $|\partial\Delta|_\theta=0$, $\pazocal{T}$ must be a $\theta$-annulus, and so by \Cref{M(S) annuli} must bound a subdiagram $\Delta_1$ containing a disk subdiagram of $\textbf{P}_0$.  We may then select $\pazocal{T}$ to be `minimal' in the sense that $\Delta_1$ has no single cell of $\textbf{P}_0$.

By construction, the boundary $t$-edges of any disk subdiagram contained in $\Delta_1$ are all either shared with another such disk subdiagram or on the boundary of $\Delta_1$.  Construct the auxiliary graph to $\Delta_1$ as in the discussion of \Cref{circular-graph}, placing an interior vertex inside each disk subdiagram in $\textbf{P}_0$ and connecting an edge between two vertices if the corresponding disk subdiagrams have a common boundary $t$-edge.  This graph does not have $1$-gons or $2$-gons, so that adding an exterior vertex we may apply an analogue of \Cref{circular-graph}.  This yields a disk subdiagram $\Gamma_1$ sharing a boundary subpath with a side of $\pazocal{T}$ which contains at least $L-3$ consecutive $t$-edges.

Let $m$ be the number of cells of $\pazocal{T}$ between consecutive $(\theta,t)$-cells.  We then perform the transposition of $\pazocal{T}$ about $\Gamma_1$ as in \Cref{sec-transposition}, producing a new disk subdiagram $\Gamma_1'$ along with a new $\theta$-annulus $\pazocal{T}'$.  Note then that by construction $\text{Area}(\Gamma_1')\leq\text{Area}(\Gamma_1)+L(m+1)$, while passing from $\pazocal{T}$ to $\pazocal{T}'$ removes $(L-4)(m+1)+1$ single cells and adds back $4(m+1)-1$ single cells.  

Let $\textbf{P}_0'$ be the covering of $\Delta_0$ given replacing $\Gamma_1$ with $\Gamma_1'$ and changing the single cells accordingly.  As single cells are counted twice, we then have $$A(\textbf{P}_0')\leq A(\textbf{P}_0)+L(m+1)-2(L-4)(m+1)-2+8(m+1)-2<A(\textbf{P}_0)$$ contradicting the choice of $\textbf{P}_0$.  

Hence, $\textbf{P}_0$ can contain no single cell.  In particular, this means $A(\textbf{P}_0)=\text{Area}(\Delta_0)$, and so it suffices to show $\text{Area}(\Delta_0)\geq\tm_\textbf{S}(V)^2$.

We may then apply the same argument as above to achieve an analogue of \Cref{circular-graph}, producing a disk subdiagram sharing a boundary subpath with $\partial\Delta_0$ containing at least $L-3$ consecutive $t$-edges.  Excising this disk subdiagram and considering what remains, the presence of another disk subdiagram allows another application of this argument, producing another disk subdiagram sharing a boundary subpath with $\partial\Delta_0$ containing at least $L-5$ consecutive $t$-edges.  But then $L=|\partial\Delta_0|_t\geq2L-8$, which is false for sufficiently large $L$.

Thus, $\textbf{P}_0$ has exactly one member, which is a disk subdiagram.  In particular, $\Delta_0$ consists of a single hub surrounded by several $\theta$-annuli.

Let $\textbf{p}_1$ be the subpath of the boundary of the hub in $\Delta_0$ whose label (or its inverse) is $W_{ac}(1)$.  Then, let $\textbf{p}$ be the subpath of $\textbf{p}_1$ such that the base of $\lab(\textbf{p})$ is a copy of the standard $k$-enhanced machine $\textbf{E}_{\textbf{S},k}^0$.  Then, let $\Gamma_0$ be the subdiagram bounded by $\textbf{p}$, all $q$-bands that have an end on $\textbf{p}$, and a subpath of $\partial\Delta_0$.

Note then that $\Gamma_0$ is a trapezium. By \Cref{trapezia are computations} and the construction of $\textbf{M}_\textbf{S}$, $\Gamma_0$ corresponds to a reduced computation $\pazocal{C}:U\equiv U_0\to\dots\to U_t$ of $\textbf{E}_{\textbf{S},k}^0$ accepting the input configuration $U$.  

\Cref{E0 language} implies $U\equiv I_0(V)$.  Lemmas \ref{E primitive step history} and \ref{E run step history} further imply there exists a maximal initial subcomputation $\pazocal{C}':U_0\to\dots\to U_y$ with step history $(1)(2)(3)$.  Letting $\pazocal{C}_3':U_x\to\dots\to U_y$ be the maximal subcomputation with step history $(3)$, the operation of $\pazocal{C}_3'$ in the working sectors corresponds to a reduced computation $\pazocal{D}$ of $\textbf{S}$ accepting $V$.  As such, the tape word of $V_x$ in any $Q_{i,\ell}Q_{i,r}$-sector must be a copy of the history $H$ of $\pazocal{D}$ in the left historical alphabet.  Note then that $y-x=\|H\|$.

Let $\pazocal{C}_3''$ be the restriction of $\pazocal{C}_3'$ to any $Q_{i,\ell}Q_{i,r}$-sector. Then $\Gamma_0$ contains a copy of the trapezium $\Gamma_3''$ corresponding to $\pazocal{C}_3''$ by \Cref{computations are trapezia}.  Every maximal $\theta$-band of $\Gamma_3''$ has two $(\theta,q)$-cells, and $\|H\|-1$ $(\theta,a)$-cells corresponding to the rest of the rules of $H$.  As such, $$\text{Area}(\Delta_0)\geq\text{Area}(\Gamma_0)\geq\text{Area}(\Gamma_3'')\geq\|H\|^2$$
By the definition of the time complexity, though, $\|H\|\geq\tm_\textbf{S}(V)$, thus implying the statement.

\end{proof}


\begin{lemma} \label{Dehn}

$\TM_\textbf{S}(n)^2\preceq\delta_{G(\textbf{M}_\textbf{S})}(n)\preceq f(n)^2$.

\end{lemma}

\begin{proof}

Let $W$ be a freely reduced word over $\pazocal{X}$ which represents the identity in $G(\textbf{M}_\textbf{S})$.  By van Kampen's Lemma we may then find a minimal circular diagram $\Delta$ with $\lab(\partial\Delta)\equiv W$.  Lemmas \ref{circular diskless} and \ref{scope-contradiction} imply $\text{Area}_G(\Delta)\leq N_4\phi(N_4n)+N_3\mu(\Delta)g(N_3n)$, where $n=|\partial\Delta|+\sigma_\lambda(\Delta^*)$ (taking $\sigma_\lambda(\Delta^*)=0$ if $\Delta$ is diskless).

Now, \Cref{mixtures} implies $\mu(\Delta)\leq J|\partial\Delta|_\theta^2\leq Jn^2$, while \Cref{G design} and the definition of length imply $\sigma_\lambda(\Delta)\leq C_1|\partial\Delta|\leq\delta^{-1}C_1|W|_\pazocal{X}$.  Hence, \Cref{mixture quotient} implies that for $N_5$ sufficiently large: $$2\text{Area}_G(\Delta)\leq N_5f(N_5|W|_\pazocal{X})^2+N_5|W|_\pazocal{X}+N_5$$
But then \Cref{Dehn upper bound} provides a reduced circular diagram $\Delta_0$ over the finite presentation of $G(\textbf{M}_\textbf{S})$ such that $\lab(\partial\Delta_0)\equiv W$ and $\text{Area}(\Delta_0)\leq2\text{Area}_G(\Delta)$.  Thus, $\delta_{G(\textbf{M}_\textbf{S})}(n)\preceq f(n)^2$.

On the other hand, fix $n\in\N$.  By the definition of the time function, we fix an input configuration $V$ of $\textbf{S}$ with $|V|_a\leq n$ and $\tm_\textbf{S}(V)=\TM_\textbf{S}(n)$.  Then, the input configuration $I(V)$ of $\textbf{M}_\textbf{S}$ is a word over $\pazocal{X}$ representing the identity in $G(\textbf{M}_\textbf{S})$ with $\|I(V)\|=LN+2L|V|_a\leq 2Ln+LN$.  By \Cref{Dehn lower bound}, the area of the word $I(V)$ is at least $\frac{1}{2}\tm_\textbf{S}(V)^2=\frac{1}{2}\TM_\textbf{S}(n)^2$.  Thus, the parameter choices $K>>L>>N$ imply $\delta_{G(\textbf{M}_\textbf{S})}(Kn)\geq2\TM_\textbf{S}(n)^2$, so that $\TM_\textbf{S}(n)^2\preceq\delta_{G(\textbf{M}_\textbf{S})}(n)$.

\end{proof}

\medskip


\subsection{Conjugator length function} \

We now study the conjugator length function of $G(\textbf{M}_\textbf{S})$, achieving a quadratic upper bound.

\begin{lemma} \label{CL}

$\CL_{G(\textbf{M}_\textbf{S})}(n)\preceq n^2$

\end{lemma}

\begin{proof}

Let $u,v$ be two words over $\pazocal{X}$ which represent conjugate elements of $G(\textbf{M}_\textbf{S})$ and set $m=\|u\|+\|v\|$.  It then suffices to find a word $w$ over $\pazocal{X}$ realizing this conjugation relation with $\|w\|\leq N_5m^2+N_5m+N_5$.  

Of course, we may assume $u$ and $v$ do not represent the identity element of $G(\textbf{M}_\textbf{S})$, as otherwise the empty word is a conjugator.  So, if $u$ and $v$ are conjugate in $M(\textbf{M}_\textbf{S})$, then \Cref{annular diskless}(2) provides a word $w$ over $\pazocal{X}$ with $|w|\leq N_1n^2+N_1n+N_1$ for $n=\max(|u|,|v|)$.  By the definition of the length function, $|u|\leq\|u\|$, $|v|\leq\|v\|$, and $\|w\|\leq\delta^{-1}|w|$.  Hence, $\|w\|\leq\delta^{-1}N_1m^2+\delta^{-1}N_1m+\delta^{-1}N_1$, so that the desired inequality follows from the parameter choices $N_5>>N_1>>\delta^{-1}$.

Thus, we may assume $u$ and $v$ do not represent conjugate elements of $M(\textbf{M}_\textbf{S})$ at all.  As such, fixing an $n$-minimal diagram $\Delta$ with boundary labels $u$ and $v$, we must have $s_1(\Delta)\geq1$. \Cref{annular disks} then provides a path $\textbf{t}$ in $\Delta$ between between its boundary components with $|\textbf{t}|\leq N_4n+N_4$.  Extending this path in both directions along the boundary components so that the endpoints coincide with those of the boundary paths whose labels are $u$ and $v$, we then have a conjugator $w$ with $|w|\leq(N_4+1)n+N_4$.  The desired inequality then follows in just the same way as above through the parameter choices $N_5>>N_4>>\delta^{-1}$.

\end{proof}

\begin{proof}[Proof of \Cref{main-theorem}] \

Lemmas \ref{Dehn} and \ref{CL} imply the group $G(\textbf{M}_\textbf{S})$ suffices as long as its conjugator length is asymptotically bounded below by the quadratic function $n^2$.  While this can be shown to hold for almost any relevant choice of $\textbf{S}$, it is simpler to slightly adjust the group.

Suppose $H$ is a finitely presented group with $\delta_H(n)\sim n^2\sim \CL_H(n)$ and define $G_\textbf{S}$ to be the direct product $G(\textbf{M}_\textbf{S})\times H$.  Straightforward combination theorems {\frenchspacing(e.g. an argument of Brick for Dehn functions \cite{Brick})} then implies the Dehn and conjugator length functions for $G_\textbf{S}$ are the maxima of the corresponding functions for the two factors.  As $\TM_\textbf{S}(n)\geq n$ for all $n$ by hypothesis, $\TM_\textbf{S}(n)^2\geq n^2$, and so \Cref{Dehn} implies $\TM_\textbf{S}(n)^2\preceq\delta_{G_\textbf{S}}(n)\preceq f(n)^2$.  Moreover, \Cref{CL} implies $\CL_{G_\textbf{S}}(n)\sim n^2$.

Hence, it suffices simply to show the existence a finitely presented group $H$ with both quadratic Dehn and conjugator length functions.  But this is true for Stallings' group (\cite{DERY}, \cite{BRS}) or Thompson's group $F$ (\cite{GubaThompsonF}, \cite{BelkMatucci}).

\end{proof}

\bigskip

\bibliographystyle{plain}
\bibliography{biblio}

\end{document}